\documentclass[10pt]{amsart}
\usepackage{amsthm}
\usepackage{amsfonts}
\usepackage{amssymb}
\usepackage{stmaryrd}
\usepackage{amsmath}
\usepackage{scalerel}
\usepackage{lscape}
\usepackage{xfrac}
\usepackage{nicefrac}
\usepackage{rotating}

\input xy
\xyoption{all}
\usepackage{tikz-cd}
\usepackage[shortlabels,inline]{enumitem}
\setenumerate[1]{leftmargin=5.5ex}
\setitemize[1]{leftmargin=5.5ex}

\usepackage{amsthm}
\usepackage{amssymb}
\usepackage{amsmath}
\usepackage[shortlabels]{enumitem}

\input xy
\xyoption{all} 

\newtheorem{theorem}{Theorem}
\newtheorem*{theorem*}{Theorem}
\newtheorem*{remark*}{Remark}

\newtheorem*{example*}{Example}

\newtheorem*{examples*}{Examples}
\newtheorem*{exercise*}{Exercise}

\newtheorem{definition}[theorem]{Definition}
\newtheorem*{definition*}{Definition}
\newtheorem{lemma}[theorem]{Lemma}

\newtheorem{proposition}[theorem]{Proposition}
\newtheorem{remark}[theorem]{Remark}
\newtheorem{cor}[theorem]{Corollary}
\newtheorem*{cor*}{Corollary}

\newtheorem{question}{Question}

\newtheorem*{conjecture*}{Conjecture}
\numberwithin{theorem}{section}

\renewcommand\iff{%
\ifmmode\text{ if and only if }%
\else if and only if \fi}

\renewcommand{\and}{\wedge}

\renewcommand{\phi}{\varphi}

\newcommand{\Span}{\text{Span}}

\newcommand{\Mod}{\textnormal{Mod}}

\newcommand{\rad}{\textnormal{rad}}

\renewcommand{\P}{\mathbb{P}}

\newcommand{\Ass}{\textnormal{Ass}}
\newcommand{\Div}{\textnormal{Div}}

\newcommand{\Att}{\textnormal{Att}}

\newcommand{\ann}{\textnormal{ann}}

\newcommand{\Spec}{\textnormal{Spec}}

\newcommand{\Zg}{\textnormal{Zg}}

\newcommand{\pinj}{\textnormal{pinj}}

\newcommand{\mcal}[1]{\mathcal{#1}}

\newcommand{\mfrak}[1]{\mathfrak{#1}}
\newcommand{\st}{\ \vert \ }
\newcommand{\vertl}{\left\vert}
\newcommand{\vertr}{\right\vert}

\newcommand{\pp}{\textnormal{pp}}

\newcommand{\Q}{\mathbb{Q}}
\newcommand{\N}{\mathbb{N}}
\newcommand{\Z}{\mathbb{Z}}

\newcommand{\ldot}{\raisebox{-3pt}{$\cdot$}}

\newcommand{\DPR}{\mathrm{DPR}}
\newcommand{\PP}{\mathrm{PP}}
\newcommand{\EPP}{\mathrm{EPP}}
\newcommand{\clx}{\mathrm{clx}}
\newcommand{\exsig}{\mathtt{exsig}\, }

\let\chiorg\chi
\renewcommand\chi{{\raisebox{1.5pt}{$\chiorg$}}}

\newcommand{\nf}{\nicefrac}
\newif\iflorna\lornafalse

\usepackage[pdftex,linktocpage]{hyperref}

\newcommand\hfuzzReset{\hfuzz=3pt}
\hfuzzReset
\newcommand\toleranceReset{\tolerance=1400}
\toleranceReset
\newcommand\emergencystretchReset{\emergencystretch=1ex}
\emergencystretchReset

\title{Decidability for the theory of modules over a Pr\"ufer domain}
\author{Lorna Gregory}
\address{Dipartimento di Matematica e Fisica, Universit\`a degli Studi della Campania ``Luigi
Vanvitelli'', Viale Lincoln 5, 81100 Caserta, Italy}
\email{Lorna.Gregory@gmail.com}
\date{\today}

\begin{document}
\begin{abstract}
In this paper we give elementary conditions completely characterising when the theory of modules of a Pr\"ufer domain is decidable.
Using these results, we show that the theory of modules of the ring of integer valued polynomials is decidable.
\end{abstract}

\maketitle
\tableofcontents

\section{Introduction}

\noindent
In this article, we give a complete characterization of when the theory of modules of a recursive Pr\"ufer domain is decidable.


Pr\"ufer domains are a much studied class of rings, including many classically important rings and classes of rings. For example, they include Dedekind domains and hence rings of integers of number fields; B\'ezout domains and hence the ring of complex entire functions \cite[Thm. 9]{Helmer} and the ring of algebraic integers \cite[Thm. 102]{Kapcomm}; the ring of integer valued polynomials with rational coefficients \cite[VI.1.7]{IntPoly} and the real holomorphy ring of a formally real field \cite[2.16]{Becker}.

Pr\"ufer domains have provided a rich supply of rings for which the decidability of modules can be determined. The theory of modules of a ring $R$ is said to be decidable if there is an algorithm which decides whether a given first order sentence in the language of $R$-modules is true in all $R$-modules.

The first non-trivial example of a ring with decidable theory of modules was given by Szmielew, \cite{Szmie}, who showed that the theory of abelian groups (or equivalently $\Z$-modules) is decidable. This result was generalised by Eklof and Fischer, \cite{EkFis}, to certain Dedekind domains, among them certain rings of integers, and they showed that, for a (recursive) field $k$ with a splitting algorithm, the theory of $k[x]$-modules is decidable.

The most recent effort to understand decidability of theories of modules over Pr\"ufer domains started with a paper, \cite{PunPunTof}, of Puninski, Puninskaya and Toffalori. They showed that a recursive valuation domain with dense archimedean value group has decidable theory of modules if and only if its set of units is recursive. Proving a conjecture in \cite{PunPunTof}, we show in \cite{DecVal} that an arbitrary recursive valuation domain has decidable theory of modules if and only if the radical relation $a\in\rad bR$ is recursive.

The theory of modules of B\'ezout domains of the form $D+XQ[X]\subseteq Q[X]$, where $D$ is a principal ideal domain with field of fractions $Q$, is shown in \cite{D+M} to be decidable under certain reasonable effective conditions on $D$. In particular, it is shown that $\Z+X\Q[X]$ has decidable theory of modules. The theory of modules of the ring of algebraic integers, along with some other B\'ezout domains with Krull dimension $1$, is shown to have decidable theory of modules in \cite{AlgIntDec}.

Work towards characterising when a general Pr\"ufer domain has decidable theory of modules was started in the articles \cite{Decpruf} and \cite{Decdense}, and is finished in the present one.
We will describe the results of these articles whilst describing the main result of this article.

First a reminder of the setup for proving decidability results for theories of modules.
Thanks to the Baur-Monk theorem, if $R$ is a recursive ring then the theory of $R$-modules is decidable if and only if there exists an algorithm which, given $\nf{\phi_1}{\psi_1},\ldots,\nf{\phi_l}{\psi_l}$ pairs of pp-formulae and intervals $[n_1,m_1],\ldots,[n_l,m_l]\subseteq \N$ where $n_i,m_i\in \N\cup\{\infty\}$, answers whether there exists an $R$-module $M$ such that, for all $1 \leq i \leq l$, $\vertl\phi(M)/\psi(M)\vertr\in [n_i,m_i]$. The existence of an algorithm answering this question when $[n_i,m_i]$ are either $[1,1]$ or $[1,\infty]$ is equivalent to the existence of an algorithm deciding whether one Ziegler basic open set is contained in a finite union of other Ziegler basic open sets
(for the definition of the Ziegler spectrum see \ref{sectionModThModules}).

We characterise when the theory of modules of a Pr\"ufer domain $R$ is decidable in terms of the recursivity of three sets:
$\DPR(R)$, $\EPP(R)$ and $X(R)$. Each of these sets is a subset of $R^n\times \N_0^k$ for some $n,k\in\N_0$. For the sets $\EPP(R)$ and $X(R)$, we postpone their definitions to section \ref{Srecsets} and in this introduction we instead give some indication of their meaning.
%

For any commutative ring $R$, the set $\DPR(R)$ is defined as the set of $(a,b,c,d)\in R^4$ such that for all prime ideals $\mfrak{p},\mfrak{q}\lhd R$ with $\mfrak{p}+\mfrak{q}\neq R$, either $a\in \mfrak{p}$, $b\notin \mfrak{p}$, $c\in\mfrak{q}$ or  $d\notin \mfrak{q}$.
This set was introduced in \cite{Decpruf} as a generalisation of the radical relation $a\in\rad bR$.
For a recursive B\'ezout domain it is shown there that $\DPR(R)$ is recursive if and only if there is an algorithm deciding inclusions of Ziegler basic open sets. For recursive Pr\"ufer domains, analogous sufficient conditions were given for there to exist an algorithm deciding inclusions of Ziegler basic open sets.  Building heavily on those results, we extend the equivalence given in \cite{Decpruf} for B\'ezout domains to all recursive Pr\"ufer domains (see \ref{Zginc}). As a consequence we get the following theorem.

\smallskip\noindent
\textbf{Theorem.} (See \ref{ZieglertoTinfty})
{\it
Let $R$ be a recursive Pr\"ufer domain such that $R/\mfrak{m}$ is infinite for all maximal ideals $\mfrak{m}$. The theory of $R$-modules is decidable if and only if $\DPR(R)$ is recursive.
}

\smallskip

\noindent
For rings $R$ with a pair of pp-formulae $\nf{\phi}{\psi}$ and an $R$-module $M$ such that $\vertl\phi(M)/\psi(M)\vertr$ is finite but not equal to $1$, in particular for commutative rings with finite non-zero modules, we need to do more than show that there is an algorithm deciding inclusions of Ziegler basic open sets. 

For any ring $R$, if the theory of $R$-modules is decidable then the theory of modules of size $n$ is decidable uniformly in $n$.
In \ref{decfinimpEPP}, we introduce a set $\EPP(R)$, whose recursivity, for a recursive Pr\"ufer domain $R$,
is equivalent to the decidability
of the theory of $R$-modules of size $n$, uniformly in $n$. This is proved in Theorem \ref{EPPfinite}.
The main feature of $\EPP(R)$ is that it is often easier to check in examples that $\EPP(R)$ is recursive than it is to check that the theory of modules of size $n$ is decidable uniformly in $n$.

The  set $\EPP(R)$ is a generalisation of $\PP(R)$, which is defined in \cite{Decdense} and inspired by the characterisation of commutative von Neumann regular rings with decidable theories of modules given in \cite{PointPrest}. In \cite{Decdense}, for a recursive B\'ezout domain $R$, under the condition that for each maximal ideal $\mfrak{m}$, $R_{\mfrak{m}}$ has dense value group, it is shown that the theory of $R$-modules is decidable if and only if $\DPR(R)$ and $\PP(R)$ are recursive. Building heavily on \cite{Decdense}, we show, that
this result also holds for Pr\"ufer domains.


\smallskip\noindent
\textbf{Theorem \ref{densethm}.}
{\it
Let $R$ be a recursive Pr\"ufer domain such that $R_\mfrak{m}$ has dense value group for all maximal ideals $\mfrak{m}$. The theory of $R$-modules is decidable if and only if $\PP(R)$ and $\DPR(R)$ is recursive.
}

\smallskip

\noindent
It follows from this result that the theory of modules of a recursive Pr\"ufer domain with dense value groups (or infinite residue fields) is decidable if and only if there is an algorithm deciding inclusions of Ziegler basic open sets and the theory of finite modules of size $n$ is decidable uniformly in $n$. The same result for commutative von Neumann regular rings easily follows from \cite{PointPrest}. This does not appear to be the case for arbitrary Pr\"ufer domains.

The third set, $X(R)$, captures information about finite Baur-Monk invariants of the, in some sense intrinsically infinite modules, $R_{\mfrak{p}}/I$ where $\mfrak{p}\lhd R$ is a prime ideal and $I\lhd R_{\mfrak{p}}$.

\smallskip\noindent
\textbf{Main Theorem \ref{mainthm}.}
{\it
Let $R$ be a recursive Pr\"ufer domain. The theory of $R$-modules is decidable if and only if the sets $\DPR(R)$, $\EPP(R)$ and $X(R)$ are recursive.
}

Our characterisation is such that it can be easily checked for concrete rings. We illustrate this in section \ref{SIntpoly}
by using our main theorem to show that the ring of integer valued polynomials with rational coefficients, $\text{Int}(\Z)$, has decidable theory of modules.
In order to prove that $\DPR(\text{Int}(\Z))$, $\EPP(\text{Int}(\Z))$ and $X(\text{Int}(\Z))$ are recursive, we use Ax's result, \cite[Thm 17]{Ax}, that the common theory of the $p$-adic valued fields $\Q_p$ as $p$ varies, is decidable.


Section $1$ contains background material and simple preparations for the rest of the paper. Its main purpose is to make the article as
accessible as possible. We postpone a guide to the proof and discussion of what happens in each section to subsection \ref{SSproofguide}.

When it doesn't complicate the proofs, we will state some of our intermediate results for arithmetical rings i.e. commutative rings whose localisations at prime ideals are valuation rings.

\section{Preliminaries}

\noindent
Notation: In this article $\N=\{1,2,3,\ldots\}$, $\N_0=\N\cup \{0\}$, $\N_n=\{m\geq n\}$ for $n\in\N$ and $\P$ denotes the set of prime natural numbers. For a ring $R$, let $\Mod\text{-}R$ denote the category of (right) modules.

\subsection{Model theory of Modules}\label{sectionModThModules}

For general background on model theory of modules see \cite{PrestBluebook}.

Let $R$ be a ring. Let $\mcal{L}_R:=\{0,+,(\cdot r)_{r\in R}\}$ be the language of (right) $R$-modules and $T_R$ be the theory of (right) $R$-modules. A (right) pp-$n$-formula is a formula of the form
\[\exists y_1,\ldots,y_l\bigwedge_{j=1}^m\sum_{i=1}^ly_ir_{ij}+\sum_{i=1}^nx_is_{ij}=0
\] where $r_{ij},s_{ij}\in R$. For $a\in R$, we write $a|x$ for the pp-$1$-formula $\exists y \, x=ya$.

The solution set $\phi(M)$ of a pp-$n$-formula $\phi$ in an $R$-module $M$ is a subgroup of $M^n$. For $\phi,\psi$, pp-$n$-formulae and $M\in\Mod\text{-}R$, we will write $\psi\leq_M \phi$ to mean that $\psi(M)\subseteq \phi(M)$. We will write $\psi\leq \phi$ to mean that $\psi\leq_M\phi$ for all $M\in\Mod\text{-}R$. After identifying equivalent pp-$n$-formulae, the set of pp-$n$-formulae, $\pp_R^n$, equipped with the order $\leq$ is a lattice.

A \textbf{pp-$m$-pair} will simply mean a pair of pp-$m$-formulae and we will write $\nf{\phi}{\psi}$ for such pairs. We write $\nf{\phi}{\psi}(M)$ for the quotient group $\phi(M)/\phi(M)\cap\psi(M)$. For every $n\in\N$ and pp-pair $\nf{\phi}{\psi}$, there is a sentence, denoted $\vertl\nf{\phi}{\psi}\vertr\geq n$, in the language of (right) $R$-modules, which expresses, in every $R$-module $M$, that $\nf{\phi}{\psi}(M)$ has at least $n$ elements. We write $\vertl\nf{\phi}{\psi}\vertr=n$ for the sentence which expresses, in every $R$-module $M$, that $\nf{\phi}{\psi}(M)$ has exactly $n$ elements.

\begin{theorem}[Baur-Monk]
Let $R$ be a ring. Every sentence in $\mcal{L}_R$ is equivalent to a boolean combination of sentences of the form $\vertl\nf{\phi}{\psi}\vertr\geq n$ where $\nf{\phi}{\psi}$ is a pp-$1$-pair.
\end{theorem}

An embedding $f:A\rightarrow B$ is \textbf{pure} if for all pp-$1$-formulae $\phi$ and $m\in A$, $f(m)\in \phi(B)$ implies $m\in\phi(A)$. A module $N$ is \textbf{pure-injective} if for every embedding $f:A\rightarrow B$ and homomorphism $g:A\rightarrow N$, there exists $h:B\rightarrow N$ such that $h\circ f=g$.

\begin{lemma}\cite[4.36]{PrestBluebook}\label{sumpielemeq}
Let $R$ be a ring. For all $M\in \Mod\text{-}R$, there exist indecomposable pure-injective modules $N_i\in\Mod\text{-}R$ for $i\in I$ such that $\oplus_{i\in I}N_i$ is elementary equivalent to $M$.
\end{lemma}

We say a pure-embedding $i:M\rightarrow N$ with $N$ pure-injective is a \textbf{pure-injective hull} of $M$ if for every other pure-embedding $g:M\rightarrow K$ where $K$ is pure-injective, there is a pure-embedding $h:N\rightarrow K$ such that $h\circ i=g$. The pure-injective hull of $M$ is unique up to isomorphism over $M$ and we will write $H(M)$ for any module $N$ such that the inclusion of $M$ in $N$ is a pure-injective hull of $M$. Every module is a elementary substructure of its pure-injective hull \cite[2.27]{PrestBluebook}. So, in particular every module is elementary equivalent to its pure-injective hull.

The (right) \textbf{Ziegler spectrum}, denoted $\Zg_R$, is a topological space whose points are isomorphism classes of indecomposable pure injective (right) modules and which has a basis of compact open sets given by
\[\left(\nf{\phi}{\psi}\right):=\{N\in\pinj_R\st \phi(N)\supsetneq \phi(N)\cap\psi(N)\}\] where $\nf{\phi}{\psi}$ range over pp-$1$-pairs.

For each $n\in\N$, Prest gave a lattice anti-isomorphism $D:\pp_R^n\rightarrow {}_R\pp^n$ (see \cite[8.21]{PrestBluebook}). As is standard, we denote its inverse ${}_R\pp^n\rightarrow \pp_R^n$ also by $D$. We don't recall the full definition of $D$ here but instead note that for all $a\in R$, $Da|x$ is $ax=0$ and $D(xa=0)$ is $a|x$.

Herzog extended this duality to an isomorphism between the lattice of  open sets of the right and left Ziegler spectra of a ring \cite[4.4]{HerDual}, and, to a useful bijection between the complete theories of right and left $R$-modules.

The following proposition is direct consequence of \cite[6.6]{HerDual}.

\begin{proposition}\label{Dualsent}
Let $R$ be a ring. Let $n,m\in\N$ be such that $n\leq m$ and for $1\leq i\leq m$, let $N_i\in\N$ and let $\nicefrac{\phi_i}{\psi_i}$ be a pp-pair. For
\[\chi:=\bigwedge_{i=1}^n\vertl\nicefrac{\phi_i}{\psi_i}\vertr=N_i\wedge\bigwedge_{i=n+1}^m\vertl \nicefrac{\phi_i}{\psi_i}\vertr\geq N_i,\] define
\[D\chi:=\bigwedge_{i=1}^n\vertl \nicefrac{D\psi_i}{D\phi_i}\vertr=N_i\wedge\bigwedge_{i=n+1}^m\vertl \nicefrac{D\psi_i}{D\phi_i}\vertr\geq N_i.\]
There exists a right $R$-module satisfying $\chi$ if and only if there exits a left $R$-module satisfying $D\chi$.
\end{proposition}

A priori, duality may not appear particularly relevant to an article about commutative rings. However, its use significantly simplifies some of the proofs in this paper and, as in \cite{DecVal}, the fact that it exchanges formulae $xb=0$ with $b|x$ allows us to reduce the number of calculations.

\subsection{Decidability and recursive Pr\"ufer domains}

\

\noindent
A \textbf{recursive ring} is either a finite ring or a ring $R$ together with a bijection $\pi:\N\rightarrow R$ such that addition and multiplication in $R$ induce recursive functions on $\N$ via $\pi$.

Note that if $R$ is a ring and $\pi: \N\rightarrow R$ is a bijection then $T_R$ is recursively axiomatisable with respect to $\pi$ if and only if $R$ together with $\pi$ is a recursive ring.

When proving decidability results about theories of modules, it is common to work with an ``effectively given'' ring rather than just a recursive one. Usually, a ring of a particular type is called effectively given if $R$ is a recursive ring and the bijection $\pi$ satisfies some extra conditions which are equivalent, for that particular type of ring, to Prest's condition $(D)$ holding (see \cite[pg 334]{PrestBluebook}). Recall that a recursive ring satisfies condition $(D)$ if there is an algorithm which, given $\phi,\psi\in\pp_R^1$ answers whether $\psi\leq \phi$.
For example, a recursive valuation domain $V$ is said to be \textbf{effectively given} if the preimage under $\pi$ of the set of units of $V$ is recursive. A recursive Pr\"ufer domain $R$ is said to be \textbf{effectively given} if the preimage under $\pi$ of the set of $(a,b)\in R^2$ such that $a\in bR$ is recursive.

Whether one works with recursive rings or effectively given ones is largely a matter of personal taste. For simplicity, we choose to work with recursive rings.

It was remarked in \cite[paragraph before 2.4]{Decdense}, that the property that $a\in bR$ is recursive is never used in \cite{Decpruf} and \cite{Decdense}. Moreover, see \cite[2.4]{Decdense}, if $R$ is a recursive Pr\"ufer domain and the set $\DPR(R)\subseteq R^4$ is recursive then $a\in bR$ is recursive. In particular, even though they are stated for effectively given Pr\"ufer domains, all results in \cite{Decpruf} and \cite{Decdense} in fact hold for recursive Pr\"ufer domains.

The next theorem is a well-known and easy to prove consequence of the Baur-Monk Theorem. Note that since $T_R$ is recursively axiomatisable when $R$ is recursive, given a sentence $\chi$ in $\mcal{L}_R$, we can always find, using a proof algorithm, a sentence $\chi'$ as in the statement of the Baur-Monk theorem which is $T_R$-equivalent to $\chi$.

\begin{theorem}\label{DecconBM}
Let $R$ be a recursive ring. The theory of $R$-modules is decidable if and only if there is an algorithm which, given a sentence $\chi$ of the form
\[\tag{$\dagger$}\label{sentform}\chi:=\bigwedge_{i=1}^n\vertl\nicefrac{\phi_i}{\psi_i}\vertr=N_i\wedge\bigwedge_{i=n+1}^m\vertl \nicefrac{\phi_i}{\psi_i}\vertr\geq N_i,\] where $N_i\in\N$ and $\nicefrac{\phi_i}{\psi_i}$ is a pp-$1$-pair for $1\leq i\leq m$, answers whether there exists an $R$-module satisfying $\chi$.
\end{theorem}

\begin{remark}\label{Zgsent}
Let $R$ be a recursive ring. There is an algorithm deciding inclusions of Ziegler basic open sets if and only if there is an algorithm which, given a sentence
\[\chi:=\bigwedge_{i=1}^n\vertl \nf{\phi_i}{\psi_i}\vertr\geq E_i\wedge\bigwedge_{j=1}^m \vertl \nf{\sigma_j}{\tau_j}\vertr=1,\] where $\nf{\phi_i}{\psi_i},\nf{\sigma_j}{\tau_j}$ are pp-$1$-pairs and $E_i\in\N$ for $1\leq i\leq n$ and $1\leq j\leq m$, answers whether there exists an $R$-module satisfying $\chi$.

\end{remark}
\begin{proof}
There is a module satisfying $\chi$ as in the statement if and only if for each $1\leq i\leq n$, there exists an indecomposable pure-injective module $N_i$ such that $N_i\in \left(\nf{\phi_i}{\psi_i}\right)\backslash \bigcup_{j=1}^m\left(\nf{\sigma_j}{\tau_j}\right)$. This is a standard argument. For the forward direction use \ref{sumpielemeq}. For the reverse, observe $\oplus_{i=1}^nN_i^{E_i}$ satisfies $\chi$.
\end{proof}

\subsection{Arithmetical rings and Pr\"ufer domains}
A commutative ring is \textbf{arithmetical}\footnote{This condition is often referred to as Pr\"ufer in papers on Model Theory of Modules. However, algebraists tend to use the term Pr\"ufer for the stronger condition that every regular ideal is invertible. To avoid confusion we choose the term with a unique definition.} if all its localisations at maximal ideals are valuation rings. Equivalently, \cite[Thm 1]{JenAri}, a commutative ring $R$ is arithmetical if its lattice of ideals is distributive. A \textbf{Pr\"ufer domain} is an integral domain which is arithmetical.

The following is a consequence of \cite[1.3]{Tug}.

\begin{lemma}\label{Tuganbaev}
If $R$ is an arithmetical ring then for all $a,b\in R$, there exist $\alpha,r,s\in R$ such that $a\alpha=br$ and $b(\alpha-1)=as$.
\end{lemma}
Note that if $R$ is a recursive arithmetical ring then there is an algorithm which, given $a,b\in R$, finds $\alpha,r,s$ satisfying the above equations. We will frequently use this fact without note.

An $R$-module is \textbf{pp-uniserial} if its lattice of pp-definable subgroups is totally ordered. Over a commutative ring all pp-definable subgroups are submodules. Thus, all uniserial modules over a commutative ring are pp-uniserial.

The lattice of pp-$1$-formulae of a commutative ring is distributive if and only if it is arithmetical \cite[3.1]{EklHer}. Thus, the following is a direct consequence of \cite[3.3]{KGser}.

\begin{lemma}\label{ppuniserial}
Let $R$ be a commutative ring. All indecomposable pure-injective $R$-modules are pp-uniserial if and only if $R$ is arithmetical.
\end{lemma}

The endomorphism rings of indecomposable pure-injective modules are local \cite[4.3.43]{PSL}. Therefore, if $R$ is a commutative ring and $N$ is an indecomposable pure-injective module then the set, $\Att N$, of $r\in R$ acting on $N$  non-bijectively form a prime ideal.
Thus, if $N$ is an indecomposable pure-injective module over a commutative ring $R$ then $N$ may be equipped with the structure of an $R_{\Att N}$-module. Moreover, $N$ remains indecomposable and pure-injective as an $R_{\Att N}$-module.  Conversely, if $N$ is an indecomposable pure-injective $R_{\mfrak{p}}$-module for some prime ideal $\mfrak{p}\lhd R$  then the restriction of $N$ to $R$ remains indecomposable and pure-injective.

The following lemma is now an easy consequence of the fact that indecomposable pure-injective modules over arithmetical rings are pp-uniserial (a proof appears in \cite[2.8]{Decdense}).

\begin{lemma}\label{DivAssAtt}
Let $R$ be an arithmetical ring and $N$ an indecomposable pure-injective $R$-module. The sets
\[\Div N:=\{r\in R\st Nr\subsetneq N\},\]
\[\Ass N:=\{r\in R \st \text{ there exists }m\in N\backslash\{0\} \text{ such that }mr=0\}\] and
\[\Att N:=\Div N\cup\Ass N\] are prime ideals.
\end{lemma}

\begin{lemma}\label{decomposeorder}
Let $R$ be an arithmetical ring and $M\in \Mod\text{-}R$.
\begin{enumerate}
\item For all $\alpha\in R$, there exist $M_1,M_2\in \Mod\text{-}R$ such that $M_1\oplus M_2\equiv M$, $\vertl \nicefrac{x\alpha=0}{x=0}(M_1)\vertr=1$, $\vertl \nicefrac{x=x}{\alpha|x}(M_1)\vertr=1$, $\vertl \nicefrac{x(\alpha-1)=0}{x=0}(M_2)\vertr=1$ and  $\vertl \nicefrac{x=x}{(\alpha-1)|x}(M_2)\vertr=1$.
\item For all $a,b\in R$, there exist $M_1,M_2\in M$ such that $M_1\oplus M_2\equiv M$, $\vertl \nicefrac{ab|x}{x=0}(M_1)\vertr=1$ and $\vertl \nicefrac{xa=0}{b|x}(M_2)\vertr=1$.
\end{enumerate}
\end{lemma}
\begin{proof}

\noindent
(1) By \ref{sumpielemeq}, there exist indecomposable pure-injective $R$-modules $N_i$ for $i\in I$ such that $M\equiv\oplus_{i\in I}N_i$. Since $\Att N_i$ is a proper ideal for each $N_i$, for all $i\in I$, either $\alpha\notin \Att N_i$ or $\alpha-1\notin \Att N_i$. Let $I_\alpha$ be the set of $i\in I$ such that $\alpha\notin\Att N_i$ and let $I_{\alpha-1}=I\backslash I_\alpha$. So, for all $i\in I_{\alpha-1}$, $\alpha-1\notin \Att N_i$.
For each $\beta\in R$ and $N$ indecomposable pure-injective, $\beta\notin\Att N$ if and only if $\vertl \nicefrac{x=x}{\beta|x}(N)\vertr=1$ and $\vertl \nicefrac{x\beta=0}{x=0}(M_1)\vertr=1$. Therefore $\vertl \nicefrac{x=x}{\alpha|x}(\oplus_{i\in I_\alpha}N_i)\vertr=1$, $\vertl \nicefrac{x\alpha=0}{x=0}(\oplus_{i\in I_\alpha}N_i)\vertr=1$, $\vertl \nicefrac{x=x}{(\alpha-1)|x}(\oplus_{i\in I_{\alpha-1}}N_i)\vertr=1$ and $\vertl \nicefrac{x(\alpha-1)=0}{x=0}(\oplus_{i\in I_{\alpha-1}}N_i)\vertr=1$.

\noindent
(2) For any $L\in \Mod\text{-}R$, $xb=0\geq_La|x$ if and only if $ab\in\ann_R L$. So $\vertl\nicefrac{ab|x}{x=0}(L)\vertr=1$ if and only if $xb=0\geq_L a|x$. Let $N$ be an indecomposable pure-injective $R$-module. By \ref{ppuniserial}, either $xb=0\geq_Na|x$ or $a|x\geq_Nxb=0$. So either $\vertl \nicefrac{ab|x}{x=0}(N)\vertr=1$ or $\vertl \nicefrac{xb=0}{a|x}(N)\vertr=1$. The proof is now as in $(1)$.
%
%
\end{proof}

\noindent
It is easy to see that if $R$ is a commutative ring, $\mfrak{p}\lhd R$ is a prime ideal and $M$ is an $R_{\mfrak{p}}$-module then the restriction to $R$ of the pure-injective hull of $M$ as an $R_{\mfrak{p}}$-module is equal to the pure-injective hull of $M$ as an $R$-module.

Recall that a module is called \textbf{uniserial} if its lattice of submodules form a chain.

\begin{theorem}\label{standarduni}\cite{Zie}
Let $V$ be a valuation domain with field of fractions $Q$. Every indecomposable pure-injective $V$-module is the pure-injective hull of a module $J/I$ where $I\subsetneq J\subseteq Q$ are submodules of $Q$.
\end{theorem}

So, in particular, over a valuation domain, all indecomposable pure-injective $R$-modules are pure-injective hulls of uniserial modules. It is not known if all indecomposable pure-injective modules over valuation rings are pure-injective hulls of uniserial modules (see \cite[\S 4]{EklHer}).

\begin{lemma}\label{REFuni}
Let $R$ be a Pr\"ufer domain. For any sentence $\chi\in\mcal{L}_R$, there exists $M\in\Mod\text{-}R$ such that $M\models \chi$ if and only if there exist $n\in\N$, prime ideals $\mfrak{p}_i\lhd R$ and uniserial $R_{\mfrak{p}_i}$-modules $U_i$ for $1\leq i\leq n$ such that $\oplus_{i=1}^nU_i\models \chi$.
\end{lemma}
\begin{proof}
For any ring $R$, there exists $M\in\Mod\text{-}R$ such that $M\models \chi$ if and only if there exist $n\in\N$ and indecomposable pure-injective $R$-modules $N_i$ for $1\leq i\leq n$ such that $\oplus_{i=1}^n N_i\models \chi$. The result now follows from \ref{standarduni}.
\end{proof}
\noindent
We will frequently use the following easy lemma.

\begin{lemma}\label{unifinitefg}
Let $V$ be a valuation domain, $\phi$ a pp-$1$-formula and $U$ a uniserial $V$-module. If $\vertl U/\phi(U)\vertr$ is finite but not equal to $1$ then $U\cong V/I$ for some ideal $I\lhd V$.
\end{lemma}
\begin{proof}
Since $U$ is uniserial, so is $U/\phi(U)$. Therefore, since $U/\phi(U)$ is finite, there exists $u\in U$ such that $u+\phi(U)$ generates $U/\phi(U)$ as a $V$-module. Since $U/\phi(U)\neq 0$, $uR\supseteq \phi(U)$. Therefore $uR=U$.
\end{proof}

\noindent
We finish this subsection by reviewing material about ideals of valuation domains.

For any commutative ring $R$, $r\in R$ and ideal $I\lhd R$, define
\[(I:r):=\{a\in R \st ar\in I\}.\]
Note that $(I:r)$ is an ideal of $R$.

\begin{definition}
For $V$ a valuation domain and $I\lhd R$ a proper ideal, define $I^\#:=\bigcup_{a\notin I}(I:a)$. By convention, we define $V^\#$ to be the unique maximal ideal of $V$.
\end{definition}

Note that this definition agrees with the definition given in \cite[ChII \S 4]{MONND}, that is, for $I\neq 0$, $r\in I^\#$ if and only if $rI\subsetneq I$.


\begin{lemma}\label{propIhash}
Let $V$ be a valuation domain.
\begin{enumerate}[(i)]
\item For any ideal $I\lhd V$, $I^\#$ is a prime ideal.
\item If $\mfrak{p}\lhd V$ is a prime ideal then $\mfrak{p}^\#=\mfrak{p}$.
\item If $I\lhd V$ and $a\in V\backslash\{0\}$ then $(aI)^\#=I^\#$.
\item If $I\lhd V$ and $V/I$ is finite then $I^\#$ is the unique maximal ideal of $V$.\label{finmodulespr}
\end{enumerate}
\end{lemma}
\begin{proof}
The first $3$ statements are in \cite[II.4]{MONND}. We prove (iv). If $I=V$ then the statement follows directly from the definition. Otherwise, $I\subseteq I^\#$ and hence $V/I^\#$ is also finite. Since $V/I^\#$ is a finite integral domain, it is a field. Therefore $I^\#$ is maximal.
\end{proof}

\begin{lemma}\label{Rp/IAssDiv}
Let $R$ be a Pr\"ufer domain, $\mfrak{p}\lhd R$ a prime ideal and $I\lhd R_{\mfrak{p}}$. Then, for all $\delta,\gamma\in R$,
\begin{enumerate}
\item $\vertl\nf{x\delta=0}{x=0}(R_{\mfrak{p}}/I)\vertr=1$ if and only if $\delta\notin I^\#$ or $I=R_{\mfrak{p}}$, and,
\item $\vertl\nf{x=x}{\gamma|x}(R_{\mfrak{p}}/I)\vertr=1$ if and only if $\gamma\notin \mfrak{p}$ or $I=R_{\mfrak{p}}$.
\end{enumerate}
\end{lemma}
\begin{proof}
(1) For any $\delta\in R$, $\vertl\nf{x\delta=0}{x=0}(R_{\mfrak{p}}/I)\vertr=1$ if and only if $(I:\delta)\subseteq I$. Now $(I:\delta)\subseteq I$ if and only if $I=R_{\mfrak{p}}$, or, for all $a\notin I$, $\delta a\notin I$. Therefore $\vertl\nf{x\delta=0}{x=0}(R_{\mfrak{p}}/I)\vertr=1$ if and only if $\delta\notin I^\#$ or $I=R_{\mfrak{p}}$.

\noindent
(2) For any $\gamma\in R$, $\vertl\nf{x=x}{\gamma|x}(R_{\mfrak{p}}/I)\vertr=1$ if and only if $\gamma R_{\mfrak{p}}+I=R_{\mfrak{p}}$. This is true if and only if $\gamma\notin\mfrak{p}$ or $I=R_{\mfrak{p}}$.
\end{proof}

\begin{remark}\label{I/bI} Let $b\in R\backslash\{0\}$, $\mfrak{p}\lhd R$ be a prime ideal and $I\lhd R_{\mfrak{p}}$ be an ideal. If $R_{\mfrak{p}}/I$ is finite then $\vertl I/bI\vertr=\vertl R_{\mfrak{p}}/bR_{\mfrak{p}}\vertr$.
\end{remark}
\begin{proof}
Since $b\neq 0$,
\[\vertl R_{\mfrak{p}}/I\vertr\cdot\vertl I/bI\vertr=\vertl R_{\mfrak{p}}/bI\vertr=\vertl R_{\mfrak{p}}/bR_{\mfrak{p}}\vertr\cdot\vertl bR_{\mfrak{p}}/bI\vertr=\vertl R_{\mfrak{p}}/bR_{\mfrak{p}}\vertr\cdot\vertl R_{\mfrak{p}}/I\vertr.\] So, since $\vertl R/I\vertr$ is non-zero, $\vertl I/bI\vertr=\vertl R_{\mfrak{p}}/bR_{\mfrak{p}}\vertr$.
\end{proof}

We will frequently use the following lemma which becomes particularly useful when $R$ is a valuation ring because then for all $r\in R$ and $I\lhd R$ either $r\in I$ or $rR\supseteq I$.

\begin{lemma}
Let $R$ be a commutative ring, $r\in R$ and $I\lhd R$. Then $rR\supseteq I$ if and only there exists $J\lhd R$ such that $I=rJ$.
\end{lemma}
\begin{proof}
The reverse direction is clear. For the forward direction, take $J=(I:r)$.
\end{proof}

\begin{lemma}\label{(I:a)useful}
Let $V$ be a valuation domain, $I,J\lhd V$ and $a\in R\backslash\{0\}$. Then $J\supseteq (I:a)$ if and only $aJ\supseteq I$ or $J=V$.
\end{lemma}
\begin{proof}
$(\Rightarrow)$ Suppose $J\supseteq (I:a)$. Since $V$ is a valuation domain, either $a\in I$ or $aV\supseteq I$. If $a\in I$ then $J\supseteq (I:a)=V$. So $J=V$. Suppose $aV\supseteq I$. Take $c\in I$. There exists $b\in V$ such that $ab=c$. So $b\in (I:a)\subseteq J$. Hence $c=ab\in aJ$ as required.

\noindent
$(\Leftarrow)$ If $J=V$ then $J\supseteq (I:a)$. So, suppose $aJ\supseteq I$. Take $c\in (I:a)$. Then $ac\in I\supseteq aJ$. Since $a\neq 0$, $c\in J$. Therefore $J\supseteq (I:a)$.
\end{proof}

\subsection{Guide to the proof}\label{SSproofguide}

The reverse direction of the main theorem is proved in \ref{Srecsets}. That is, we show that if $T_R$ is decidable then $\DPR(R)$, $\EPP(R)$ and $X(R)$ are recursive (see \cite[6.4]{Decpruf} (or \ref{DPRZg}), \ref{decimpEPPrec} and \ref{decimpXrec}).

The proof of the forward direction of the main theorem has $3$ principal ingredients.
\begin{enumerate}[(A)]
\item Consequences of $\DPR(R)$, $\EPP(R)$ and $X(R)$ being recursive. (\S\ref{Srecsets} and \S\ref{Sfinitemodules})
\item Syntactic reductions. (\S \ref{Sform}, \S \ref{1stsyn} and \S \ref{Furthersyn})
\item Semantic input. (\S \ref{sectuni}, \S\ref{Sremoving} and \S \ref{notred})

\end{enumerate}

\medskip
\noindent
(A) In section \ref{Srecsets} we introduce and analyse the sets $\DPR(R)$, $\EPP(R)$ and $X(R)$. These sets are chosen to be as simple as possible so that our theorem is as easy as possible to apply to concrete rings. For this reason, work needs to be done to obtain more elaborate consequences of them being recursive.

For each $n\in\N$, a set $\DPR_n(R)$ was introduced in \cite{Decpruf}. We show that, \ref{DPRimpDPRn}, for $R$ a recursive arithmetical ring, if $\DPR(R)$ is recursive then the sets $\DPR_n(R)$ are recursive (uniformly in $n$). Combining this with \cite[7.1]{Decpruf}, or more precisely its proof, we conclude that we can effectively decide inclusions of Ziegler basic open sets if and only if $\DPR(R)$ is recursive.


In section \ref{Sfinitemodules}, we investigate the consequences of $\EPP(R)$ being recursive, and of
$\EPP(R)$ and the radical relation being recursive (which is deployed in section \ref{Sremoving}).
We show that for a recursive Pr\"ufer domain the theory of $R$-modules of size $n$ is decidable
uniformly in $n$ if and only if $\EPP(R)$ is recursive, \ref{EPPfinite}.

In the proof of the forward direction of the main theorem, the set $X(R)$ is only ever used in section \ref{notred}.

\medskip
\noindent
(B) Given a sentence $\chi$ as in (\ref{sentform}) (from \ref{DecconBM}) we often produce a finite set  $S$ of tuples of sentences $(\chi_1,\ldots,\chi_n)$ with each $\chi_i$ having a ``better'' form than $\chi$ such that there exists $M\models \chi$ if and only if there exist $(\chi_1,\ldots,\chi_n)\in S$ and modules $M_i$ for $1\leq i\leq n$ with $M_i\models \chi_i$. This is roughly what happens in the proof of \cite[4.1]{Decdense}.
Section \ref{Sform} introduces two important formalisms (and ideas) which used in combination are key to the proof. Essentially they allow us to ``automate'' some reductions similar to those in the proof of \cite[4.1]{Decpruf} which in this article become too complicated to perform entirely by hand.
%

It is shown in \cite[4.1]{Decdense} that for arithmetical rings\footnote{The result is stated there only for Pr\"ufer domains but the same proof implies the result for all arithmetical rings (see section \ref{1stsyn}).} it is enough to consider sentences like (\ref{sentform}) where the pp-pairs involved are of the form $\nf{d|x}{x=0}$ or $\nf{xb=0}{c|x}$.  Section \ref{1stsyn} uses the formalisms in section \ref{Sform}, to show that it is enough to consider sentences like (\ref{sentform}) where at most one conjunct of the form $\vertl \nf{d|x}{x=0}\vertr=D$ or $\vertl \nf{d|x}{x=0}\vertr\geq D$ with $D\geq 2$ occurs and where at most one conjunct of the form $\vertl\nf{xb=0}{c|x}\vertr=G$ or $\vertl\nf{xb=0}{c|x}\vertr\geq G$ with $G\geq 2$ and $b,c\neq 0$ occurs.

In section \ref{Furthersyn} a notion of complexity, called the extended signature, is put on the set $W$ of sentences as in (\ref{sentform}) as reduced to in section \ref{1stsyn}. The set of extended signatures is equipped with an artinian partial order.
The reduction processes in section \ref{Furthersyn} terminate at expressions whose extended signatures are not reducible.
Some of the sentences which are not reducible are of a form for which we can answer whether there  exists a module satisfying them, because
there is an algorithm deciding inclusions of Ziegler basic open sets.
The remaining sentences are of a particular simple form and we deal with them in section \ref{notred}.

\medskip
\noindent
(C) In section \ref{sectuni}, for pp-pairs $\nf{\phi}{\psi}$ of the form $\nf{d|x}{x=0}$, $\nf{xb=0}{c|x}$ with $b,c\neq 0$ and $\nf{x=x}{c|x}$ with $c\neq 0$, we give a description of the uniserial module $U$ over a valuation domain $V$, such that $\nf{\phi}{\psi}(U)$ is finite but non-zero. Unlike the descriptions of such modules in \cite{DecVal} and \cite{PunPunTof}, the results we prove do not depend on whether the value group of $V$ is dense or not.
We now describe how we use semantic input to deal with sentences as in (\ref{sentform}) with a conjunct of the form
$\vertl\nf{d|x}{x=0}\vertr=D$, or of the form $\vertl\nf{xb=0}{c|x}\vertr=G$ where $b,c\neq 0$.

For instance, if $\nf{d|x}{x=0}(U)$ is finite but non-zero then it is easy to show that $U\cong V/dI$ for some ideal $I\lhd V$. In view of \ref{REFuni}, for $R$ a Pr\"ufer domain, this means that for $\chi$ a sentence as in (\ref{sentform}), $D\in\N_2$ and $d\in R\backslash\{0\}$, there exists $M\in \Mod\text{-}R$ such that $M\models \vertl\nf{d|x}{x=0}\vertr=D\wedge \chi$ if and only if there exist $h\in\N$, prime ideals $\mfrak{p}_i\lhd R$ and ideals $I_i\lhd R_{\mfrak{p}_i}$ for $1\leq i\leq h$, and, $M'\in\Mod\text{-}R$ with $\vertl\nf{d|x}{x=0}(M')\vertr=1$ such that $\bigoplus_{i=1}^hR_{\mfrak{p}_i}/dI_i\oplus M'\models \vertl\nf{d|x}{x=0}\vertr=D\wedge \chi$.

In section \ref{Sremoving}, we show that if $\EPP(R)$ and the radical relation are recursive then there is an algorithm which given $D\in \N$, $d\in R\backslash\{0\}$ and a sentence $\Theta$ as in (\ref{sentform}), answers whether there exists a direct sum $\oplus_{i=1}^hR_{\mfrak{p}_i}/dI_i$ satisfying $\vertl\nf{d|x}{x=0}\vertr=D\wedge \Theta$. This is used in \ref{ssrem=} to produce sentences $\chi_1,\ldots \chi_n$ as in (\ref{sentform}) such that there exists an $R$-module satisfying $\vertl\nf{d|x}{x=0}\vertr=D\wedge\chi$ if and only if there exists an $R$-module satisfying $\vertl \nf{d|x}{x=0}\vertr=1\wedge\chi_i$ for some $1\leq i\leq n$.
The sentences $\chi_i$ are less complex than $\chi$ in a way precisely defined in section \ref{Furthersyn}.

Similar, but slightly more complicated, reductions are made for pp-pairs of the form $\nf{xb=0}{c|x}$ where $b,c\neq 0$.

For pp-pairs of the form $\nf{x=x}{c|x}$ we need to do something different. It is easy to see, \ref{x=x/c|xfinnonzero}, that if $U$ is a uniserial module over a valuation domain $V$ then $\nf{x=x}{c|x}(U)$ is finite but non-zero if and only if $U\cong V/cI$ for some $I\lhd V$ or $U$ is finite and $c\in\ann_R U$.
However, it does not seem possible, in this case, to make a reduction as for $\vertl\nf{d|x}{x=0}\vertr=D$ and sums of modules of the form $R_{\mfrak{p}}/dI$. This is the main reason that we need to make the syntactic reductions in section \ref{1stsyn} and \ref{Furthersyn}. In particular, the set of sentences that are not reducible in the sense of section \ref{Furthersyn}, contains only a small number of forms of sentences with a conjunct of the form $\vertl\nf{x=x}{c|x}\vertr=C$. These sentences are considered individually in section \ref{notred}.

%
%
%
%

\section{Recursive sets}\label{Srecsets}

\noindent
In this section we consider the sets $\DPR(R)$, $\EPP(R)$ and $X(R)$. In each case, we show that if $T_R$ is decidable then they are recursive.

\subsection{The set $\DPR(R)$}

\

In \cite{Decpruf}, a family of relations $\DPR_n(R)$ were defined. Although not directly stated there, see \cite[7.1]{Decpruf}, it was shown that, for $R$ a recursive\footnote{Recall $R$ recursive and $\DPR(R)$ recursive imply $R$ is effectively given} Pr\"ufer domain, if the sets $\DPR_n(R)$ are recursive (uniformly in $n$) then there is an algorithm deciding inclusions of finite unions of Ziegler basic open sets. However, it was not known if this condition was necessary for the existence of such an algorithm, or even if the decidability of the theory of modules of Pr\"ufer domain implied that $\DPR_n(R)$ is recursive for any $n>1$.  It is a consequence of \ref{DPRZg} that the existence of an algorithm deciding inclusions of finite unions of Ziegler basic open sets implies that the sets $\DPR_n(R)$ are recursive (uniformly in $n$).

For a recursive B\'ezout domain, it was shown that if $\DPR(R):=\DPR_1(R)$ is recursive then there is an algorithm deciding inclusions of finite unions of Ziegler basic open sets. For Pr\"ufer domains, it was not known if $\DPR_1(R)$ being recursive is sufficient to imply that there is an algorithm deciding inclusions of finite unions of Ziegler basic open sets.  We show, \ref{DPRimpDPRn}, that, for $R$ a Pr\"ufer domain, $\DPR_1(R)$ recursive implies $\DPR_n(R)$ is recursive uniformly in $n$.

\begin{definition}Let $R$ be a commutative ring.

\begin{itemize}
\item  For each $l\in\N$, let $\DPR_l(R)$ be the set of $2l+2$-tuples $(a,b_1,\ldots,b_l,c,d_1,\ldot,d_l)\in R^{2l+2}$ such that, for all prime ideals $\mfrak{p},\mfrak{q}\lhd R$ with $\mfrak{p}+\mfrak{q}\neq R$, either $a\in\mfrak{p}$, $c\in \mfrak{q}$, $b_i\notin \mfrak{p}$ for some $1\leq i\leq l$ or $d_i\notin \mfrak{q}$ for some $1\leq i\leq l$.
\item  Let $\DPR_{*}(R)$ be the set of $4$-tuples $(a,B,c,D)$, where $a,c\in R$ and $B,D\lhd R$ are finitely generated ideals, such that for all prime ideals $\mfrak{p},\mfrak{q}\lhd R$ with $\mfrak{p}+\mfrak{q}\neq R$, either $a\in \mfrak{p}$, $c\in \mfrak{q}$, $B\nsubseteq\mfrak{p}$ or $D\nsubseteq \mfrak{q}$.
\end{itemize}
\end{definition}

\noindent
Note that $(a,b_1,\ldots,b_l,c,d_1,\ldots,d_l)\in \DPR_l(R)$ if and only if
\[(a,\sum_{i=1}^lb_iR,c,\sum_{i=1}^ld_iR)\in \DPR_{*}(R).\]
In the following lemma, for $R$ a commutative ring, $I\lhd R$ an ideal and $X\subseteq \Spec R$, let $V(I)$ denote the closed set, in the Zariski topology, of prime ideals $\mfrak{p}$ such that $\mfrak{p}\supseteq I$ and $\overline{X}$ the closure of $X$ in the Zariski topology.

\begin{lemma}\label{radcolon}
Let $R$ be a commutative ring, $a\in R$ and $B\lhd R$. Then
\[(\rad B:a)=\bigcap_{\substack{\mfrak{p}\in\Spec R \\ \mfrak{p}\supseteq B, \ a\notin \mfrak{p}}}\mfrak{p}.\] Hence, \[V((\rad B:a))=\overline{V(B)\backslash V(aR)}.\]

\end{lemma}
\begin{proof}
Suppose $r\in(\rad B:a)$. Then $(ra)^n\in B$ for some $n\in\N$. Let $\mfrak{p}\lhd R$ be a prime ideal with $B\subseteq \mfrak{p}$ and $a\notin \mfrak{p}$. Then $r^na^n=(ra)^n\in \mfrak{p}$. Therefore $r\in \mfrak{p}$.

Conversely, suppose that $r\notin (\rad B:a)$ i.e. $(ra)^n\notin B$ for all $n\in\N$. A standard argument using Zorn's lemma produces a prime ideal $\mfrak{p}\lhd R$ such that $\mfrak{p}\supseteq B$ and $(ra)^n\notin \mfrak{p}$ for all $n\in \N$. Now $(ra)^n\notin \mfrak{p}$ implies $r\notin \mfrak{p}$ and $a\notin \mfrak{p}$. So we have proved the first statement. For any $X\subseteq \Spec R$, $V(\bigcap_{\mfrak{p}\in X}\mfrak{p})=\overline{X}$. Therefore, the second statement follows from the first.
\end{proof}

\noindent
The following statement with $\DPR(R):=\DPR_1(R)$ in place of $\DPR_{*}(R)$ is proved in \cite[6.3]{Decpruf} for Pr\"ufer domains. We use \ref{radcolon} to further extend it to all commutative rings.

\begin{proposition}\label{DPRradical}
Let $R$ be a commutative ring. The following are equivalent for $a,c\in R$ and $B,D\lhd R$ finitely generated.
\begin{enumerate}
\item $(a,B,c,D)\in \DPR_{*}(R)$
\item $1\in (\rad(B):a)+(\rad(D):c)$
\end{enumerate}
\end{proposition}
\begin{proof}
We prove this proposition topologically.

For all $a, c\in R$ and $B, D\lhd R$, $1\in (\rad(B):a)+(\rad(D):c)$ if and only if $V((\rad B:a))\cap V((\rad D:c))=\emptyset$. By \ref{radcolon}, $V((\rad B:a))=\overline{V(B)\backslash V(aR)}$ and $V((\rad D:c))=\overline{V(D)\backslash V(cR)}$. Thus
\[1\in (\rad(B):a)+(\rad(D):c)\] if and only if
\[\overline{V(B)\backslash V(aR)}\cap \overline{V(D)\backslash V(cR)}=\emptyset.\]
Now, by \cite[1.5.4 (i)]{SpecSp},
\[\overline{V(B)\backslash V(aR)}=\bigcup_{\mfrak{p}\in V(B)\backslash V(aR)}V(\mfrak{p}) \ \ \text{ and } \ \ \overline{V(D)\backslash V(cR)}=\bigcup_{\mfrak{q}\in V(D)\backslash V(cR)}V(\mfrak{q}).\]
So $1\in (\rad(B):a)+(\rad(D):c)$ if and only if
\begin{equation}\label{equ1}
    (\bigcup_{\mfrak{p}\in V(B)\backslash V(aR)}V(\mfrak{p}))\cap (\bigcup_{\mfrak{q}\in V(D)\backslash V(cR)}V(\mfrak{q}))=\emptyset.\tag{$\dagger$}
\end{equation}
Now (\ref{equ1}) holds if and only if for all prime ideals $\mfrak{p},\mfrak{q}$ such that $a\notin \mfrak{p}$, $B\subseteq \mfrak{p}$, $c\notin \mfrak{q}$ and $D\subseteq \mfrak{q}$, we have $V(\mfrak{p})\cap V(\mfrak{q})=\emptyset$ (i.e. $\mfrak{p}+\mfrak{q}=R$).
\end{proof}
\noindent
We refer to the relation $a\in\rad bR$ as the \textbf{radical relation}.

\begin{remark}
Let $R$ be a commutative ring. For $a,b\in R$, $a\in\rad bR$ if and only if $(a,b,a,b)\in\DPR(R)$. In particular, if $R$ is a recursive ring and $\DPR(R)$ is recursive then the radical relation is recursive.
\end{remark}

\noindent
For $R$ a commutative ring and $B,D\lhd R$ finitely generated ideals, let $xB=0$ denote the pp-formula $\bigwedge_{i=1}^n xb_i=0$ where $B=\sum_{i=1}^n b_iR$ and let $D|x$ denote the pp-formula $\sum_{i=1}^n d_i|x$ where $D=\sum_{i=1}^nd_iR$. Note that, up to $T_R$ equivalence, these formulae don't depend on the choice of generators of $B$ and $D$.

The following statement with $\DPR_{*}(R)$ replaced by $\DPR(R)=\DPR_1(R)$ is proved in \cite[6.4]{Decpruf} for Pr\"ufer domains.

\begin{proposition}\label{DPRZg}
Let $R$ be an arithmetical ring. The following are equivalent for $a,c\in R$ and $B,D\lhd R$ finitely generated ideals.
\begin{enumerate}
\item $(a,B,c,D)\in \DPR_{*}(R)$
\item $\left(\nicefrac{xB=0}{D|x}\right)\subseteq \left(\nicefrac{xa=0}{x=0}\right)\cup\left(\nicefrac{x=x}{c|x}\right)$
\end{enumerate}
\end{proposition}
\begin{proof}
$(1)\Rightarrow(2)$: Suppose $(1)$ holds. By \ref{DPRradical}, $1\in (\rad B:a)+(\rad D:c)$. Hence there exist $u\in R$ and $n\in\N$ such that $(ua)^n\in B$ and $((u-1)c)^n\in D$.

Suppose $N\in \left(\nicefrac{xB=0}{D|x}\right)$. Take $m\in N$ such that $mb=0$ for all $b\in B$ and $m\notin ND$. So $mu^na^n=0$ and $m\notin N(u-1)^nc^n$. Since $\text{Att}N$ is a proper ideal, either $u\notin \text{Att}N$ or $u-1\notin \text{Att}N$. Suppose $u\notin \text{Att}N$. Then $mu^na^n=0$ implies $ma^n=0$ and hence $N\in \left(\nicefrac{xa^n=0}{x=0}\right)=\left(\nicefrac{xa=0}{x=0}\right)$ as required. Now suppose $u-1\notin \text{Att}N$. Then $N(u-1)^nc^n=Nc^n$ and hence $N\in \left(\nicefrac{x=x}{c^n|x}\right)=\left(\nicefrac{x=x}{c|x}\right)$.

\noindent
$(2)\Rightarrow (1)$: Suppose $(2)$ holds. Note that if $\left(\nicefrac{xB=0}{D|x}\right)\subseteq \left(\nicefrac{xa=0}{x=0}\right)\cup\left(\nicefrac{x=x}{c|x}\right)$ then, applying Herzog's duality, $\left(\nicefrac{xD=0}{B|x}\right)\subseteq\left(\nicefrac{xc=0}{x=0}\right)\cup\left(\nicefrac{x=x}{a|x}\right)$.

Suppose $\mfrak{p},\mfrak{q}$ are prime ideals with $\mfrak{p}+\mfrak{q}\neq R$. Further, suppose that $a\notin \mfrak{p}$, $B\subseteq \mfrak{p}$ and $D\subseteq \mfrak{q}$. We need to show that $c\in \mfrak{q}$.

Since $R$ is an arithmetical ring, either $\mfrak{p}\supseteq \mfrak{q}$ or $\mfrak{q}\supseteq \mfrak{p}$. Suppose $\mfrak{p}\supseteq \mfrak{q}$. Then, again since $R$ is an arithmetical ring, $R_{\mfrak{p}}/\mfrak{q}R_{\mfrak{p}}$ is uniserial and its pure-injective hull $N:=H(R_{\mfrak{p}}/\mfrak{q}R_{\mfrak{p}})$ is indecomposable. Since $D\subseteq \mfrak{q}$, $xD=0$ is equivalent to $x=x$ in $R_{\mfrak{p}}/\mfrak{q}R_{\mfrak{p}}$ and hence in $N$. Since $B\subseteq \mfrak{p}$, $B|x$ is not equivalent to $x=x$ in $R_{\mfrak{p}}/\mfrak{q}R_{\mfrak{p}}$ and hence in $N$. Thus $N\in \left(\nicefrac{xD=0}{B|x}\right)$. Since $a\notin \mfrak{p}$, $N\notin \left(\nicefrac{x=x}{a|x}\right)$ and hence $N\notin \left(\nicefrac{xc=0}{x=0}\right)$. Therefore $c\notin \mfrak{q}$, for otherwise $xc=0$ is equivalent to $x=0$ in $N$. The argument when $\mfrak{q}\supseteq \mfrak{p}$ is very similar except this time $N:=H(R_{\mfrak{q}}/\mfrak{p}R_{\mfrak{q}})$ and we use that $\left(\nicefrac{xD=0}{B|x}\right)\subseteq\left(\nicefrac{xc=0}{x=0}\right)\cup\left(\nicefrac{x=x}{a|x}\right)$.
\end{proof}

\noindent
Thus, for $R$ an arithmetical ring, if there exists an algorithm deciding inclusions of finite unions of Ziegler basic open sets then $\DPR_{*}(R)$ is recursive. Combining this with \cite[7.1]{Decpruf} (and its proof) we get the following theorem.

\begin{theorem}
Let $R$ be a recursive Pr\"ufer domain. There is an algorithm deciding inclusions of finite unions of Ziegler basic open sets if and only if $\DPR_{*}(R)$ is recursive.
Moreover, if $R/\mfrak{m}$ is infinite for every maximal ideal $\mfrak{m}\lhd R$ then $T_R$ is decidable if and only if $\DPR_{*}(R)$ is recursive.
\end{theorem}

\noindent
We finish off the work in \cite{Decpruf} by showing that for $R$ a Pr\"ufer domain, if $\DPR(R)$ is recursive then $\DPR_n(R)$ is recursive uniformly in $n$ (equivalently $\DPR_{*}(R)$ is recursive).

\begin{proposition}\label{DPRimpDPRn}
Let $R$ be a recursive arithmetical ring. If $\DPR(R)$ is recursive then $\DPR_n(R)$ is recursive uniformly in $n$.
\end{proposition}

\begin{proof}
Let $n\in\N_2$ and $a,c,b_1,\ldots,b_n,d_1,\ldots,d_n\in R$. Suppose that $\alpha,\beta, r_1,r_2, \allowbreak s_1, \allowbreak s_2\in R$ are such that
\[b_1\alpha=b_2r_2, \ \ \ \ \ \ b_2(\alpha-1)=b_1r_1\]
\[d_1\beta=d_2s_2  \ \text{ and } \ d_2(\beta-1)=d_1s_1. \]
\noindent
\textbf{Claim:}
\[
\tag{$\star$}    (a,c,b_1,\ldot,b_n,d_1,\ldots,d_n)\in \text{DPR}_n(R)
\]
if and only if
\begin{enumerate}[(i)]
\item $(a\alpha, c\beta,b_2,\ldots,b_n,d_2,\ldots,d_n)\in \text{DPR}_{n-1}(R)$,
\item $(a\alpha,c(\beta-1),b_2,\ldots,b_n,d_1,d_3,\ldots,d_n)\in \text{DPR}_{n-1}(R)$,
\item $(a(\alpha-1),c\beta,b_1,b_3\ldots,b_n,d_2,\ldots,d_n)\in \text{DPR}_{n-1}(R)$ and
\item $(a(\alpha-1),c(\beta-1),b_1,b_3,\ldots,b_n,d_1,d_3,\ldots,d_n)\in \text{DPR}_{n-1}(R)$.
\end{enumerate}

\noindent
This claim plus the fact that we can always find appropriate $\alpha, \beta, r_1,r_2,s_1,s_2\in R$ implies the proposition.

To prove the forward direction, we show that $(\star)$ implies $(i)$ and note that the remaining conditions $(ii)$, $(iii)$, $(iv)$ are the same as $(i)$ but with the roles of $b_1$ and $b_2$, and of $\alpha$ and $\alpha-1$ interchanged (respectively of $d_1$ and $d_2$, and of $\beta$ and $\beta-1$ interchanged).

Let $\mfrak{p},\mfrak{q}$ be prime ideals such that $\mfrak{p}+\mfrak{q}\neq R$. Assuming $(\star)$, we need to show that $a\alpha\in \mfrak{p}$, $c\beta\in \mfrak{q}$, $b_i\notin \mfrak{p}$ for some $2\leq i\leq n$ or $d_i\notin \mfrak{q}$ for some $2\leq i\leq n$.

Now $(\star)$ implies $a\in \mfrak{p}$, $c\in \mfrak{q}$, $b_i\notin \mfrak{p}$ for some $1\leq i\leq n$ or $d_i\notin \mfrak{q}$ for some $1\leq i\leq n$. So the only problem is when $b_1\notin \mfrak{p}$ or $d_1\notin \mfrak{q}$. Suppose $b_1\notin \mfrak{p}$. If $\alpha\in \mfrak{p}$ then $a\alpha\in \mfrak{p}$ as required. So suppose further that $\alpha\notin \mfrak{p}$. Then $b_1\alpha\notin \mfrak{p}$ and hence $b_2r_2\notin \mfrak{p}$. So $b_2\notin \mfrak{p}$ as required. The argument when $d_1\notin \mfrak{q}$ is the same with the roles of $b_1$ and $d_1$, of $b_2$ and $d_2$, and of $\alpha$ and $\beta$ interchanged.

We now show that if $(\star)$ is not true then one of $(i), (ii), (iii), (iv)$ is not true. If $(\star)$ is not true then there exist prime ideals $\mfrak{p},\mfrak{q}$ such that $\mfrak{p}+\mfrak{q}\neq R$ and $a\notin \mfrak{p}$, $c\notin \mfrak{q}$, $b_i\in \mfrak{p}$ for all $1\leq i\leq n$ and $d_i\in \mfrak{q}$ for all $1\leq i\leq n$. For any proper ideal $I$, either $\alpha\notin I$ or $\alpha-1\notin I$ (respectively $\beta\notin I$ or $\beta-1\notin I$). Without loss of generality, we may assume that $\alpha\notin \mfrak{p}$ and $\beta\notin \mfrak{q}$. Hence $a\alpha\notin \mfrak{p}$, $c\beta\notin \mfrak{q}$, $b_i\in \mfrak{p}$ for all $1\leq i\leq n$ and $d_i\in \mfrak{q}$ for all $1\leq i\leq n$. So $(a\alpha, c\beta,b_2,\ldots,b_n,d_2,\ldots,d_n)\notin \text{DPR}_{n-1}(R)$ as required.
\end{proof}
\noindent
Combining this with the results in \cite{Decpruf} we get the following.

\begin{theorem}\label{Zginc}
Let $R$ be a recursive Pr\"ufer domain. There is an algorithm answering whether one Ziegler basic open set is contained in a finite union of others if and only if the relation $\DPR(R)$ is recursive.
\end{theorem}

\begin{cor}\label{ZieglertoTinfty}
Let $R$ be a recursive Pr\"ufer domain. The relation $\DPR(R)$ is recursive if and only if there is an algorithm which, given a sentence
\[\chi:=\bigwedge_{i=1}^n\vertl \nf{\phi_i}{\psi_i}\vertr\geq E_i\wedge\bigwedge_{j=1}^m \vertl \nf{\sigma_j}{\tau_j}\vertr=1,\] where $\phi_i,\psi_i,\sigma_j,\tau_j$ are pp-$1$-formulae and $E_i\in\N$ for $1\leq i\leq n$ and $1\leq j\leq m$, answers whether there exists an $R$-module $M$ with $M\models \chi$.
Moreover, if $R/\mfrak{m}$ is infinite for every maximal ideal $\mfrak{m}\lhd R$ then $T_R$ is decidable if and only if $\DPR(R)$ is recursive.
\end{cor}
\begin{proof}

The first statement follows from \ref{Zginc} and \ref{Zgsent}.

For any $R$-module $M$ and pp-pair $\nf{\phi}{\psi}$, $\nf{\phi}{\psi}(M)$ is an $R$-module. By \cite[3.1]{Decdense}, if $R/\mfrak{m}$ is infinite for every maximal ideal $\mfrak{m}\lhd R$ then the only finite $R$-module is the zero module. So, the second statement follows from the first by a standard argument using the Baur-Monk theorem.
\end{proof}

\begin{question}
It was shown in \cite[3.2]{DecVal} that, for any commutative ring $R$, if $T_R$ is recursive then the radical relation is recursive. For $R$ an arbitrary commutative ring, does $T_R$ decidable imply $\DPR(R)$, or more generally $\DPR_*(R)$, is recursive?
\end{question}

\subsection{The sets $\PP(R)$ and $\EPP(R)$}

\

\noindent
In \cite{Decdense}, a family of relations $\PP_n(R)$ were introduced. It is shown, \cite[3.2]{Decdense}, that if a recursive Pr\"ufer domain has decidable theory of modules then $\PP_n(R)$ is recursive uniformly in $n$. Conversely, it was shown that if $R$ is a recursive\footnote{It was stated there for effectively given Pr\"ufer domains. However, recall, if $R$ is a recursive Pr\"ufer domain with $\DPR(R)$ recursive then $R$ is effectively given.} Pr\"ufer domain such that the value group of each localisation of $R$ at a maximal ideal is dense then if $\DPR_n(R)$ and $\PP_n(R)$ are recursive uniformly in $n$ then the theory of $R$-modules is recursive.

\begin{definition}
For $l\in\N$, let $\PP_l(R)$ consist of the tuples $(p,n,c_1,\ldots,c_l,d)\in\P\times\N\times R^{l+1}$ such that there exist positive integers $s,k_1,\ldots,k_s$ and maximal ideals $\mfrak{m}_1,\ldots,\mfrak{m}_s$ of $R$ for which $n\in\mathrm{Span}_{\N_0}\{k_1,\ldots,k_s\}$ and for $1\leq i\leq s$
\begin{enumerate}
\item $\vertl R/\mfrak{m}_i\vertr=p^{k_i}$,
\item $c_j\in\mfrak{m}_i$ for $1\leq j\leq l$,
\item $d\notin \mfrak{m}_i$.
\end{enumerate}
\end{definition}

\noindent
It is clear that for $R$ a B\'ezout domain, if $\PP_1(R)$ is recursive then $\PP_l(R)$ is recursive uniformly in $l$. This is because $(p,n,c_1,\ldots,c_l,d)\in \PP_l(R)$ if and only if $(p,n,\gcd\{c_1,\ldots,c_l\},d)\in \PP_1(R)$. However, with a bit more work one can show this is also true for Pr\"ufer domains.

\begin{proposition}\label{PPimpPPn}
Let $R$ be a recursive Pr\"ufer domain. If $\PP_1(R)$ is recursive then $\PP_l(R)$ is recursive uniformly in $l$.
\end{proposition}
\begin{proof}
We skip this proof as it is very similar to the proof of \ref{epp1enough}.
\end{proof}

\noindent
As a direct consequence of \cite[6.1]{Decdense}, \cite[3.2]{Decdense}, \cite[6.4]{Decpruf}, \ref{PPimpPPn} and \ref{DPRimpDPRn}, we get the following theorem.

\begin{theorem}\label{densethm}
Let $R$ be a recursive Pr\"ufer domain such that for all maximal ideals $\mfrak{m}$, the value group of $R_{\mfrak{m}}$ is dense. The theory of $R$-modules is decidable if and only if $\DPR(R)$ and $\PP(R)$ are recursive.
\end{theorem}

\noindent
We generalise $\PP_l(R)$ to $\EPP_l(R)$ in order to deal with Pr\"ufer domains with maximal ideals $\mfrak{m}$ such that $R_{\mfrak{m}}$ is a valuation domain with non-dense value group.

\begin{definition}\label{defnEPP}
For $l\in \N$, let $\EPP_l(R)$ consist of tuples \[(p,n;a_1,\ldots,a_l;\gamma;e,m)\in \mathbb{P}\times\N_0\times R^l\times R\times R\times \N_0\] such that there exist $h\in \N_0$ and, for $1\leq i\leq h$, prime ideals $\mfrak{p}_i\lhd R$ and ideals $I_i\lhd R_{\mfrak{p}_i}$ such that $\gamma\notin \mfrak{p}_i$ and $a_1,\ldots,a_l\in I_i$ for $1\leq i\leq h$, $|\oplus_{i=1}^h R_{\mfrak{p}_i}/I_i|=p^n$ and $|\oplus_{i=1}^h R_{\mfrak{p}_i}/eR_{\mfrak{p}_i}|=p^m$.

We say the sequence $(\mfrak{p}_j,I_j)_{1\leq j\leq h}$ witnesses $(p,n;a_1,\ldots,a_l;\gamma;e,m)\in \EPP_l(R)$. By convention, $(p,0;a_1,\ldots,a_l;\gamma;e,0)\in\EPP_l(R)$ and the empty sequence is a witness for it. We will often write $\EPP(R)$ for $\EPP_1(R)$.
\end{definition}

\begin{remark}
We may replace prime ideals with maximal ideals in \ref{defnEPP} without changing the definition since if $\vertl R_{\mfrak{p}}/eR_{\mfrak{p}}\vertr$ is finite then either $\mfrak{p}$ is maximal or $R_{\mfrak{p}}=eR_{\mfrak{p}}$ and if $\vertl R_{\mfrak{p}}/I\vertr$ is finite then either $\mfrak{p}$ is maximal or $I=R_{\mfrak{p}}$.
\end{remark}

\noindent
The relation $\EPP_l(R)$ is an extension of the relation $\PP_l(R)$.

\begin{lemma}
Let $p\in \mathbb{P}$, $n\in\N$ and $c_1,\ldots,c_n,d\in R$. Then $(p,n,c_1,\ldots,c_l,d)\in \PP_l(R)$ if and only if $(p,n;c_1,\ldots,c_l;d;1,0)\in \EPP_l(R)$.
\end{lemma}
\begin{proof}
Suppose $(p,n,c_1,\ldots,c_l,d)\in \PP_l(R)$. Let $s\in\N$, $k_1,\ldots,k_s\in \N$ and  $\mfrak{m}_1,\ldots,\mfrak{m}_s$ be as in the definition of $\PP_l(R)$. Let $\lambda_1,\ldots,\lambda_s\in\N_0$ be such that $n=\sum_{i=1}^s\lambda_ik_i$. Then $c_j\in \mfrak{m}_i$ for $1\leq j\leq l$, $d\notin \mfrak{m}_i$, $|\oplus_{i=1}^s(R_{\mfrak{m}_i}/\mfrak{m}_iR_{\mfrak{m}_i})^{\lambda_i}|=p^n$ and $|\oplus_{i=1}^s(R_{\mfrak{m}_i}/1\cdot R_{\mfrak{m}_i})^{\lambda_i}|=p^0$. So $(p,n;c_1,\ldots,c_l;d;1,0)\in \EPP_l(R)$.

Suppose $(\mfrak{m}_j,I_j)_{1\leq j\leq s}$ witnesses $(p,n;c_1,\ldots,c_l;d;1,0)\in \EPP_l(R)$. For $1\leq i\leq s$, let $\lambda_i\in \N_0$ be such that $|R_{\mfrak{m}_i}/I_i|=|R/\mfrak{m}_i|^{\lambda_i}$. If $\lambda_i\neq 0$ then $I_i\subseteq \mfrak{m}_iR_{\mfrak{m}_i}$ and so $c_j\in \mfrak{m}_i$ for $1\leq j\leq l$. For $1\leq i\leq s$, let $k_i$ be such that $|R/\mfrak{m}_i|=p^{k_i}$. Now $n=\sum_{i=1, \ \lambda_i\neq 0}^s\lambda_ik_i$ and hence $(p,n,c_1,\ldots,c_l,d)\in \PP_l(R)$.
\end{proof}

\noindent
If the value group of $R_{\mfrak{m}}$ is dense then, for all $e\in R$, $\vertl R_{\mfrak{m}}/eR_{\mfrak{m}}\vertr$ is either $1$ or infinite. Moreover, $\vertl R_{\mfrak{m}}/eR_{\mfrak{m}}\vertr=1$ if and only if $e\notin\mfrak{m}$.

\begin{remark}
Let $R$ be a Pr\"ufer domain such that the value group of $R_{\mfrak{m}}$ is dense for all maximal ideals $\mfrak{m}$. Then $(p,n;a_1,\ldots,a_l;\gamma;e,m)\in \EPP_l(R)$ if and only if $(p,n,a_1,\ldots,a_l,\gamma\cdot e)\in \PP_l(R)$ and $m=0$.
\end{remark}

\noindent
In particular, if $R$ is a recursive Pr\"ufer domain such that the value group of $R_{\mfrak{m}}$ is dense for all maximal ideals $\mfrak{m}$ then $\EPP_l(R)$ is recursive if and only if $\PP_l(R)$ is recursive.

\begin{proposition}\label{epp1enough}
Let $R$ be a recursive Pr\"ufer domain. If $\EPP_1(R)$ is recursive then $\EPP_l(R)$ is recursive uniformly in $l$.
\end{proposition}
\begin{proof}
We show that for all $p\in\mathbb{P}$, $n,m\in\N_0$, $a_1,\ldots,a_l,\gamma,e\in R$, if $\alpha,r,s\in R$ are such that $a_1\alpha=a_2r$ and $a_2(\alpha-1)=a_1s$ then $(p,n;a_1,\ldots,a_l;\gamma;e,m)\in \EPP_l(R)$ if and only if there exist $n_1,n_2\in\N_0$ and $m_1,m_2\in\N_0$ such that $n_1+n_2=n$, $m_1+m_2=m$, $(p,n_1;a_2,\ldots,a_l;\gamma\alpha;e,m_1)\in \EPP_{l-1}(R)$ and $(p,n_2,a_1,a_3,\ldots,a_l,\gamma(\alpha-1);e,m_2)\in \EPP_{l-1}(R)$. This is enough since we can always effectively find appropriate $\alpha,r,s\in R$.

Suppose that $(\mfrak{p}_j,I_j)_{1\leq j\leq s}$ witnesses $(p,n;a_1,\ldots,a_l;\gamma;e,m)\in \EPP_l(R)$. For all $1\leq j\leq s$, either $\alpha\notin \mfrak{p}_j$ or $\alpha-1\notin\mfrak{p}_j$. By reordering, we may assume that $\alpha\notin \mfrak{p}_j$ for $1\leq j\leq t$ and $\alpha-1\notin \mfrak{p}_j$ for $t+1\leq j\leq s$. Let $n_1 =\log_p\vertl\oplus_{i=1}^tR_{\mfrak{p}_i}/I_i\vertr$, $n_2=\log_p\vertl\oplus_{i=t+1}^sR_{\mfrak{p}_i}/I_i\vertr$, $m_1=\log_p\vertl \oplus_{i=1}^t R_{\mfrak{p}_i}/eR_{\mfrak{p}_i}\vertr$ and $m_2=\log_p\vertl \oplus_{i=t+1}^s R_{\mfrak{p}_i}/eR_{\mfrak{p}_i}\vertr$. Now $\alpha \gamma\notin \mfrak{p}_j$ and $a_2,\ldots,a_l\in I_j$ for $1\leq j\leq t$, $\vertl\oplus_{i=1}^tR_{\mfrak{p}_i}/I_i\vertr=p^{n_1}$ and $\vertl \oplus_{i=1}^t R_{\mfrak{p}_i}/eR_{\mfrak{p}_i}\vertr=p^{m_1}$. So $(p,n_1;a_2,\ldots,a_l;\gamma\alpha;e,m_1)\in \EPP_{l-1}(R)$. Similarly, $(p,n_2,a_1,a_3,\ldots,a_l,\allowbreak \gamma(\alpha-1);e,m_2)\in \EPP_{l-1}(R)$.

Conversely, suppose that $n_1,n_2,m_1,m_2\in\N_0$ are such that $n_1+n_2=n$, $m_1+m_2=m$, $(p,n_1;a_2,\ldots,a_l;\allowbreak\gamma\alpha;\allowbreak e,m_1)\in \EPP_{l-1}(R)$ and $(p,n_2,a_1,a_3,\ldots,a_l,\allowbreak\gamma(\alpha-1);e,m_2)\in \EPP_{l-1}(R)$. Let $(\mfrak{p}_j,I_j)_{1\leq j\leq t}$ witness $(p,n_1;a_2,\ldots,a_l;\gamma\alpha;e,m_1)\allowbreak\in \EPP_{l-1}(R)$ and let $(\mfrak{p}_j,I_j)_{t+1\leq j\leq s}$ witness $(p,n_2,a_1,a_3,\ldots,a_l,\allowbreak\gamma(\alpha-1);e,m_2)\in \EPP_{l-1}(R)$. Then
\[\vertl\oplus_{i=1}^sR_{\mfrak{p}_i}/I_i\vertr=\vertl\oplus_{i=1}^tR_{\mfrak{p}_i}/I_i\vertr\cdot\vertl\oplus_{i=t+1}^{s}R_{\mfrak{p}_i}/I_i\vertr=p^{n_1}p^{n_2}=p^n\]
and
\[
\vertl \oplus_{i=1}^s R_{\mfrak{p}_i}/eR_{\mfrak{p}_i}\vertr=\vertl \oplus_{i=1}^t R_{\mfrak{p}_i}/eR_{\mfrak{p}_i}\vertr\cdot\vertl \oplus_{i=t+1}^s R_{\mfrak{p}_i}/eR_{\mfrak{p}_i}\vertr=p^{m_1}p^{m_2}=p^m.
\]
For $1\leq j\leq t$, $\gamma\alpha\notin I_j$ and hence $\gamma\notin I_j$ and $\alpha\notin I_j$. Since $a_2\in I_j$ and $\alpha\notin \mfrak{p}_j$, $a_1\alpha=a_2r\in I_j$ implies $a_1\in I_j$. Similarly, $\gamma\notin \mfrak{m}_j$ and $a_2\in I_j$ for $t+1\leq j\leq s$. Therefore $(\mfrak{p}_j,I_j)_{1\leq j\leq s}$ witnesses $(p,n;a_1,\ldots,a_l;\gamma;e,m)\in \EPP_l(R)$.
\end{proof}

\begin{lemma}\label{decimpEPPrec}
Let $R$ be a Pr\"ufer domain. If $T_R$ is decidable then $\EPP_1(R)$ (and hence $\EPP_l(R)$ is recursive uniformly in $l$).
\end{lemma}
\begin{proof}
For $p\in\mathbb{P}$, $n,m\in \N_0$ and $a,\gamma,e\in R$, let
$\Theta_{(p,n;a;\gamma;e,m)}$ be the $\mcal{L}_R$-sentence
\[\vertl \nicefrac{x=x}{x=0}\vertr=p^{2m+n}\wedge\vertl \nicefrac{x=x}{e|x}\vertr=p^m\wedge \vertl \nicefrac{xe=0}{e|x}\vertr=1\wedge \vertl \nicefrac{x=x}{\gamma|x}\vertr=1\wedge \vertl \nicefrac{e^2a|x}{x=0}\vertr=1.\]
We show that, for $m,n$ not both zero, $(p,n;a;\gamma;e,m)\in \EPP_1(R)$ if and only if there is an $R$-module satisfying $\Theta_{(p,n;a;\gamma;e,m)}$. Hence the lemma holds.

Suppose $(p,n;a;\gamma;e,m)\in \EPP_1(R)$. By definition, there exist prime ideals $\mfrak{p}_1,\ldots\mfrak{p}_s\lhd R$ and ideals $I_i\lhd R_{\mfrak{p}_i}$ such that $\gamma\notin \mfrak{p}_j$ and $a\in I_j$ for $1\leq j\leq s$, $|\oplus_{i=1}^s R_{\mfrak{p}_i}/I_i|=p^n$ and $|\oplus_{i=1}^s R_{\mfrak{p}_i}/eR_{\mfrak{p}_i}|=p^m$. Let $M:=\oplus_{i=1}^s R_{\mfrak{p}_i}/e^2I_i$. Then $M$ satisfies $\Theta_{(p,n;a;\gamma;e,m)}$.

Conversely, suppose there exists an $R$-module $M$ satisfying $\Theta_{(p,n;a;\gamma;e,m)}$. Since $M$ is finite and non-zero, there exist maximal ideals $\mfrak{m}_1,\ldots,\mfrak{m}_s\lhd R$ and proper ideals $J_i\lhd R_{\mfrak{m}_i}$ such that $\oplus_{i=1}^sR_{\mfrak{m}_i}/J_i\cong M$. Since $M$ satisfies $\Theta_{(p,n;a;\gamma;e,m)}$, for $1\leq i\leq s$, $(J_i:e)\subseteq eR_{\mfrak{m}_i}+J_i$.

By \ref{(I:a)useful}, either $eR_{\mfrak{m}}+J_i=R_{\mfrak{m}_i}$ or $J_i\subseteq e^2R_{\mfrak{m}_i}+eJ_i$. So, either $e\notin\mfrak{m}_i$, $J_i\subseteq e^2R_{\mfrak{m}_i}$ or $J_i\subseteq eJ_i$. So, for each $1\leq i\leq s$, there exists $I_i\lhd R_{\mfrak{m}_i}$ with $J_i=e^2I_i$.


Since $|e^2a|x/x=0(R_{\mfrak{m}_i}/e^2I_i)|=1$, $a\in I_i$ and since $\vert x=x/\gamma|x(R_{\mfrak{m}_i}/e^2I_i)\vert=1$, $\gamma\notin \mfrak{m}_i$. Therefore $\mfrak{m}_1,\ldots,\mfrak{m}_s$ and $I_i\lhd R_{\mfrak{m}_i}$ are such that $a\in I_i$ and $\gamma\notin \mfrak{m}_i$ for $1\leq i\leq s$, and $|\oplus_{i=1}^sR_{\mfrak{m}_i}/I_i|=p^n$ and $|\oplus_{i=1}^sR_{\mfrak{m}_i}/eR_{\mfrak{m}_i}|=p^m$. Hence $(p,n;a;\gamma;e,m)\in \EPP_1(R)$.
\end{proof}

\noindent
The following corollary is a direct consequence of the proof of \ref{decimpEPPrec}. We will later see, \ref{EPPfinite}, that the converse also holds.

\begin{cor}\label{decfinimpEPP}
If the theory of $R$-modules of size $n$ is decidable uniformly in $n$ then $\EPP_1(R)$ is recursive.
\end{cor}

\subsection{The set $X(R)$}

\

\noindent

\begin{definition}
Let $X(R)$ be the set of $(p,n;e,\gamma,a,\delta)\in\mathbb{P}\times\N\times (R\backslash\{0\})\times R^3$ such that there exist integers $h\in\N$ and prime ideals $\mfrak{p}_1,\ldots,\mfrak{p}_h$ such that  $\vertl \oplus_{i=1}^hR_{\mfrak{p}_i}/e R_{\mfrak{p}_i}\vertr=p^n$ and for $1\leq i\leq h$, $\gamma\notin \mfrak{m}_i$, and,
 there exists an ideal $I_i\lhd R_{\mfrak{p}_i}$  such that $a\in I_i$ and $\delta\notin (I_i)^\#$.
\end{definition}

It is often easier to check that $X(R)$ is recursive in concrete rings using the following reformulation.
\begin{remark}\label{X2ndformulation}
Let $(p,n;e,\gamma,a,\delta)\in\mathbb{P}\times\N\times (R\backslash\{0\})\times R^3$. Then $(p,n;e,\gamma,a,\delta)\in X(R)$ if and only if there exist $1\leq h\leq n$ and maximal ideals $\mfrak{m}_1,\ldots,\mfrak{m}_h$ such that $\vertl \oplus_{i=1}^hR_{\mfrak{m}_i}/e R_{\mfrak{m}_i}\vertr=p^n$, and, for $1\leq i\leq h$
\begin{enumerate}
\item $\gamma\notin \mfrak{m}_i$, and,
\item either $\delta\notin \mfrak{m}_i$, or, there exists a prime ideal $\mfrak{q}_i\subseteq \mfrak{m}_i$ such that $a\in \mfrak{q}_i$ and $\delta\notin \mfrak{q}_i$.
\end{enumerate}
\end{remark}

\begin{proof}
Note that if $\mfrak{p}_i$ in the definition of $X(R)$ is such that $\vertl R_{\mfrak{p}_i}/eR_{\mfrak{p}_i}\vertr=1$ then we may drop $\mfrak{p}_i$ from the sequence of prime ideals witnessing $(p,n;e,\gamma,a,\delta)\in X(R)$. Therefore, we may assume each $\mfrak{p}_i$ is maximal and that $1\leq h\leq n$. Now, if $a\in I_i$ and $\delta\notin (I_i)^\#$ then either $I_i=R_{\mfrak{p}_i}$ and $\delta\notin \mfrak{p}_i$, or, $I_i\lhd R_{\mfrak{p}_i}$ is a proper ideal. If $I_i$ is a proper ideal then $a\in I_i$ implies $a\in (I_i)^\#$.

Therefore, if $(p,n;e,\gamma,a,\delta)\in X(R)$ then the conditions in the statement hold with, $\mfrak{m}_i:=\mfrak{p}_i$, and $\mfrak{q}_i:=(I_i)^\#$ if $\delta\in\mfrak{m}_i$. Conversely, if the conditions in the statement hold for $(p,n;e,\gamma,a,\delta)$ then set $\mfrak{p}_i:=\mfrak{m}_i$ and, $I_i:=R_{\mfrak{m}_i}$ if $\delta\notin \mfrak{m}_i$ and $I_i:=\mfrak{q}_i$ otherwise.
\end{proof}

\begin{proposition}\label{decimpXrec}
Let $R$ be a Pr\"ufer domain. If $T_R$ is decidable then $X(R)$ is recursive.
\end{proposition}
\begin{proof}
Let $(p,n;e,\gamma,a,\delta)\in\mathbb{P}\times\N\times(R\backslash\{0\})\times R^3$. We show that $(p,n;e,\gamma,a,\delta)\in X(R)$ if and only if there exists an $R$-modules satisfying
\[\chi:=\vertl \nicefrac{x=x}{e|x} \vertr=p^n\wedge \vertl \nicefrac{xe=0}{e|x}\vertr=1\wedge \vertl \nicefrac{e^2a|x}{x=0}\vertr=1\wedge\vertl \nicefrac{x=x}{\gamma|x}\vertr=1\wedge \vertl \nicefrac{x\delta=0}{x=0}\vertr=1.\]
First suppose that there exist $h\in\N$ and prime ideals $\mfrak{p}_1,\ldots,\mfrak{p}_h\lhd R$ such that
$\vertl \oplus_{i=1}^hR_{\mfrak{p}_i}/e R_{\mfrak{p}_i}\vertr=p^n$ and for $1\leq i\leq h$, $\gamma\notin \mfrak{p}_i$, and,
 there exists an ideal $I_i\lhd R_{\mfrak{p}_i}$  such that $a\in I_i$ and $\delta\notin (I_i)^\#$. Then $\bigoplus_{i=1}^h R_{\mfrak{p}_i}/e^2I_i\models \chi$.

Conversely, suppose there exists an $R$-module satisfying $\chi$. Then, \ref{REFuni}, there exists a finite direct sum of modules $U_i$ such that $\oplus_{i=1}^hU_i\models \chi$ and each $U_i$ is the restriction to $R$ of a uniserial module over $R_{\mfrak{p}_i}$ for some prime ideal $\mfrak{p}_i\lhd R$. We may assume that $U_i/U_ie$ is non-zero for each $U_i$, for otherwise the direct sum with $U_i$ omitted also satisfies $\chi$. Since $U_i$ is uniserial as an $R_{\mfrak{p}_i}$-module and $U_i/U_ie$ is non-zero and finite, $U_i$ is finitely generated over $R_{\mfrak{p}_i}$. Therefore $U_i\cong R_{\mfrak{p}_i}/J_i$ for some ideal $J_i\lhd R_{\mfrak{p}_i}$. Since $|xe=0/e|x(U_i)|=1$, $(J_i:e)\subseteq J_i+eR_{\mfrak{p}_i}$. So, as in \ref{decimpEPPrec}, there exists $I_i\lhd R_{\mfrak{p}_i}$ such that $J_i=e^2I_i$.

Now, since $\oplus_{i=1}^hU_i\models \chi$, $\vertl\oplus_{i=1}^h R_{\mfrak{p}_i}/e R_{\mfrak{p}_i}\vertr=p^n$, $e^2a\in e^2I$ and hence $a\in I$. Moreover, by \ref{Rp/IAssDiv}, $\gamma\notin \mfrak{p}_i$ and $\delta\notin I_i^\#$.
\end{proof}

\section{Formalisms}\label{Sform}

\noindent
The formalisms introduced in this section will be used throughout the paper to allow us to make reductions in complexity
of certain sets of conditions later on.

\subsection{Sets of functions}

\

\noindent
Let $\Delta$ be a set and $\mcal{E}$ a set of functions from $\Delta$ to $\N\cup\{\infty\}$ such that if $f,g\in \mcal{E}$ then $f\cdot g\in \mcal{E}$ and such that the function which has constant value $1$ is in $\mcal{E}$. Let $n\in\N$, $X,Y\subseteq \Delta$ be finite sets and let $f:X\rightarrow \N$ and $g:Y\rightarrow \N$. 

Define $\Omega_{f,g,n}$ to be the set of all tuples of functions $(f_1,\ldots,f_n,g_1,\ldots,g_n)$ where $f_i:X\cup (Y\backslash Y_i)\rightarrow \N$ and $g_i:Y_i\rightarrow \N$ are such that $Y_i\subseteq Y$ and
\begin{itemize}
\item $\prod_{i=1}^nf_i(x)=f(x)$ for all $x\in X$,
\item $f_i(y)<g(y)$ for all $y\in Y\backslash Y_i$,
\item $g_i(y)=g(y)$ for all $y\in Y_i$, and
\item for all $y\in Y$,\[\left(\prod_{i \text{ with } y\in Y\backslash Y_i}f_i(y)\cdot \prod_{i \text{ with }y\in Y_i}g_i(y)\right)\geq g(y).\]
\end{itemize}
The most important instance of this set up in this paper is when $\Delta$ is a set of pairs $\phi/\psi$ of pp-$1$-formulae over a ring $R$ and $\mcal{E}$ is the set of $R$-modules $M$ viewed as functions on $\Delta$ by setting $M(\phi/\psi):=|\phi(M)/\psi(M)|$.

For $\mcal{E}_1,\mcal{E}_2$ sets of functions from $\Delta$ to $\N\cup\{\infty\}$, define $\mcal{E}_1\cdot\mcal{E}_2$ to be the set of $f:\Delta\rightarrow \N\cup\{\infty\}$ such that there exist $f_1\in \mcal{E}_1$ and $f_2\in\mcal{E}_2$ such that $f=f_1\cdot f_2$.

\begin{lemma}\label{feathering}
Let $\mcal{E}=\prod_{i=1}^n\mcal{E}_i$, $X,Y\subseteq \Delta$ be finite sets  and let $f:X\rightarrow \N$ and $g:Y\rightarrow \N$. There exists $h\in \mcal{E}$ such that $h(x)=f(x)$ for all $x\in X$ and $h(y)\geq g(y)$ for all $y\in Y$ if and only if for some $(f_1,\ldots,f_n,g_1,\ldots,g_n)\in\Omega_{f,g,n}$ there exist $h_i\in\mcal{E}_i$ for $1\leq i\leq n$ such that $h_i(x)=f_i(x)$ for all $x\in X\cup (Y\backslash Y_i)$ and $h_i(y)\geq g_i(y)$ for all $y\in Y_i$. 
\end{lemma}
\begin{proof}
Let $h\in \mcal{E}$ be such that $h(x)=f(x)$ for all $x\in X$ and $h(y)\geq g(y)$ for all $y\in Y$. Since $\mcal{E}=\prod_{i=1}^n\mcal{E}_i$, there exist $h_i\in \mcal{E}_i$ for $1\leq i\leq n$ such that $\prod_{i=1}^nh_i(x)=h(x)$ for all $x\in \Delta$. For each $1\leq i\leq n$, let $Y_i:=\{y\st h_i(y)\geq g(y)\}$, let $f_i(x)=h_i(x)$ for all $x\in X\cup (Y\backslash Y_i)$, and let $g_i(y)=g(y)$ for all $y\in Y_i$. By definition $\prod_{i=1}^nf_i(x)=\prod_{i=1}^nh_i(x)=h(x)=f(x)$ for all $x\in X$, $f_i(y)=h_i(y)<g(y)$ for all $y\in Y\backslash Y_i$ and $g_i(y)=g(y)$ for all $y\in Y_i$. Now if $y\in Y\backslash Y_i$ then $f_i(y)=h_i(y)$ and if $y\in Y_i$ then $g_i(y)=g(y)$. Therefore, for all $y\in Y$, either $y\in Y_i$ for some $1\leq i\leq n$ and hence the $4$th condition holds or $y\notin Y_i$ for all $1\leq i\leq n$ and so $g(y)\leq h(y)=\prod_{i=1}^nh_i(y)=\prod_{i}f_i(y)$. So $(f_1,\ldots,f_n,g_1,\ldots,g_n)\in \Omega_{f,g,n}$.

Conversely, suppose $h_1,\ldots,h_n:\Delta\rightarrow \N$ are such that there exists $(f_1,\ldots,f_n,\allowbreak g_1,\ldots,g_n)\in \Omega_{f,g,n}$ with $h_i(x)=f_i(x)$ for all $x\in X\cup (Y\backslash Y_i)$ and $h_i(y)\geq g_i(y)$ for all $y\in Y_i$. Define $h:\Delta\rightarrow \N$ by $h(x)=\prod_{i=1}^nh_i(x)$ for all $x\in \Delta$. Then $h(x)=f(x)$ for all $x\in X$ and $h(y)\geq g(y)$ for all $y\in Y$ as required.
\end{proof}

\begin{definition}
Let $X,Y\subseteq \Delta$ be finite sets  and let $f:X\rightarrow \N$ and $g:Y\rightarrow \N$. Let $g:=\max\{g(y)\st y\in Y\}$. Define $\Theta_{f,g}$ to be the set of pairs of functions $(f',g')$ such that $f':X\cup (Y\backslash Y')\rightarrow \N$, $g':Y'\rightarrow \N$, $f'(x)=f(x)$ for all $x\in X$, $g'(y)=g$ for all $y\in Y'$ and $g(y)\leq f'(y)<g$ for all $y\in Y\backslash Y'$. 
\end{definition}

\begin{remark}\label{reducegconstant}
A function $h\in \mcal{E}$ is such that $h(x)=f(x)$ for all $x\in X$ and $h(y)\geq g(y)$ for all $y\in Y$ if and only if there exists $(f',g')\in \Theta_{f,g}$ such that  $h(x)=f'(x)$ for all $x\in X\cup (Y\backslash Y')$ and $h(y)\geq g'(y)$ for all $y\in Y'$.
\end{remark}

\subsection{Lattices generated by conditions}\label{lattcomb}

\

\noindent
Let $W$ be an infinite set. Let $\mathbb{W}$ be the free bounded distributive lattice generated by $W$\footnote{See \cite{Graetzer} for the definition of a free distributive lattice and add a largest and smallest element.}. We use $\sqcup$ for the supremum in this lattice and $\sqcap$ for the infimum in this lattice. Any element of $\mathbb{W}$ may be expressed as $\bigsqcup_{i\in I}\bigsqcap_{j\in J_i} w_{ij}$ where $I$ and $J_i$ for $i\in I$ are finite sets and $w_{ij}\in W$. Moreover, for $v_k,w_{ij}\in W$,
\[\bigsqcap_{k\in K} v_{k}\leq \bigsqcup_{i\in I}\bigsqcap_{j\in J_i} w_{ij} \] if and only if there exists $i\in I$ such that
\[\bigsqcap_{k\in K} v_{k}\leq \bigsqcap_{j\in J_i} w_{ij}\] if and only if there exists $i\in I$ such that
\[\{v_k\st k\in K\}\supseteq \{w_{ij}\st j\in J_i\}.\]
We make the convention that the empty infimum is the largest element $\top$ and the empty supremum is the least element $\bot$.

We call an expression of the form $\bigsqcup_{i\in I}\bigsqcap_{j\in J_i} w_{ij}$, where $w_{ij}\in W$, \textbf{irredundant} if for each $i\in I$, $w_{ij_1}= w_{ij_2}$ implies $j_i=j_2$ and the sets $w_i:=\{w_{ij}\st j\in J_i\}$ for $i\in I$ are pairwise incomparable by inclusion. If $\bigsqcup_{i\in I}\bigsqcap_{j\in J_i} w_{ij}$ and $\bigsqcup_{i\in I'}\bigsqcap_{j\in J'_i} w'_{ij}$ are in irredundant form then $\bigsqcup_{i\in I}\bigsqcap_{j\in J_i} w_{ij}=\bigsqcup_{i\in I'}\bigsqcap_{j\in J'_i} w'_{ij}$ if and only if there exist bijections $\sigma:I\rightarrow I'$ and $\sigma_i:J_i\rightarrow J'_{\sigma(i)}$ for each $i\in I$ such that $w_{ij}=w'_{\sigma(i),\sigma_i(j)}$ for all $i\in I$ and $j\in I_i$.

Given a recursive presentation of $W$ (i.e. a bijection with $\N$), this presentation of $W$ gives rise to a recursive presentation of $\mathbb{W}$ (i.e a presentation where the inclusion of $W$ in $\mathbb{W}$ is recursive and $\sqcup$ and $\sqcap$ are recursive functions).

For any $V\subseteq W$, define $\mathbb{V}$ to be the filter generated by $V$ in $\mathbb{W}$. Note for $w_{ij}\in W$, $\bigsqcup_{i\in I}\bigsqcap_{j\in J_i} w_{ij}\in\mathbb{V}$ if and only if there exists $i\in I$ such that $w_{ij}\in V$ for all $j\in J_i$. So, in particular $\mathbb{V}$ is prime filter. It follows that $V$ is a recursive subset of $W$ if and only if $\mathbb{V}$ is a recursive subset of $\mathbb{W}$.

Suppose that $\text{clx}:W \rightarrow \alpha$ where $\alpha$ is a partially ordered set with the descending chain condition. Let $\underline{w}\in \mathbb{W}$ and let $\underline{w}=\bigsqcup_{i\in I}\bigsqcap_{j\in J_i} w_{ij}$ be in irredundant form. For $\beta\in \alpha$, we write $\clx \underline{w}\leq \beta$ if $w_{ij}\leq \beta$ for all $i\in I$ and $j\in J_i$ and $\clx \underline{w}< \beta$ if $\clx w_{ij}<\beta$ for all $i\in I$ and $j\in J_i$. Note that if $\underline{w}$ is a lattice combination of elements $w_i\in W$ for $1\leq i\leq n$ then $\clx w_i\leq \beta$ (respectively $\clx w_i<\beta$) for $1\leq i\leq n$ implies $\clx \underline{w}\leq \beta$ (respectively $\clx \underline{w}<\beta$).

\begin{remark}
Let $W$ be an infinite recursively presented set and $V\subseteq W$. Suppose that $\alpha$ is an artinian recursive partially ordered set, $\clx: W\rightarrow \alpha$ is recursive and $W_0\subseteq W$ is recursive. Suppose that there is an algorithm which given $w\notin W_0$ computes $\underline{w}\in\mathbb{W}$ such that $\text{clx} \underline{w}<\text{clx} w$ and such that $w\in V$ if and only if $\underline{w}\in\mathbb{V}$. Then $V$ is a recursive subset of $W$ if and only if $V\cap W_0$ is a recursive subset of $W$.
\end{remark}
\noindent
The precise choice of $W$ and $V$ varies throughout this article.

\medskip\noindent
To illustrate how this setup is used, let
$R$ be a recursive ring. Let $W$ be the set of $\mcal{L}_R$-sentences
\[\bigwedge_{\nf{\phi}{\psi}\in X}\vertl\nf{\phi}{\psi}\vertr=f(\nf{\phi}{\psi})\wedge\bigwedge_{\nf{\phi}{\psi}\in Y}\vertl\nf{\phi}{\psi}\vertr\geq g(\nf{\phi}{\psi})\]
where $X,Y$ are finite sets of pp-$1$-pairs, $f:X\rightarrow \N$ and $g:Y\rightarrow \N$.
Let $V$ be the set of $w\in W$ such that there exists $M\in\Mod\text{-}R$ with $M\models w$. Then, by \ref{DecconBM}, $T_R$ is decidable if and only if $V$ is recursive. Working with $\mathbb{W}$ and $\mathbb{V}$, rather than $W$ and $V$ directly, allows us to talk about more than one module at a time. For instance, for $w_1,\ldots,w_n\in W$, the condition
\(w_1\sqcap\ldots\sqcap w_n\in \mathbb{V}\)
says that there exist $R$-modules $M_i\in\Mod\text{-}R$ with $M_i\models w_i$ for $1\leq i\leq n$.


\section{First syntactic reductions}\label{1stsyn}

\noindent
Recall that, for a recursive ring $R$, in order to show that the theory of $R$-modules is decidable, it is enough to show that there is an algorithm which, given a sentence of the form
\[\bigwedge_{i=1}^s \vertl\nicefrac{\phi_i}{\psi_i} \vertr= F_i \wedge \bigwedge_{j=1}^t \vert \nicefrac{\sigma_j}{\tau_j} \vert \geq G_j \tag{$\star$},\]
where, for $1\leq i \leq s$ and $1\leq j\leq t$, $\phi_i$, $\psi_i$, $\sigma_j$ and $\tau_j$ are pp-$1$-formulae and $F_i,G_j\in \N$,
answers whether there exists an $R$-module satisfying it. 

In \cite{Decdense}, it was shown that if $R$ is a recursive Pr\"ufer domain then it is enough to consider sentences where the pairs of pp-$1$-formulae in $(\star)$ are all of the form $\nicefrac{d|x}{x=0}$ and $\nicefrac{xb=0}{c|x}$.
The proof of this statement relies on \ref{Tuganbaev}, \cite[2.2]{Bezwid} and
 the fact, by \ref{sumpielemeq} and \ref{ppuniserial}, that every $R$-module is elementary equivalent to a direct sum of pp-uniserial modules. This is also true for arithmetical rings, and so, although not stated in \cite{Decdense}, the result, with the same proof, also holds for arithmetical rings.

\begin{theorem}\label{4.1Denseimproved} \cite[4.1]{Decdense}
Let $R$ be a recursive arithmetical ring. If there exists an algorithm which, given a sentence
\[ \chi:=\bigwedge_{i=1}^m\vertl\nf{\phi_i}{\psi_i}\vertr=G_i\wedge\bigwedge_{i=m+1}^n\vertl\nf{\phi_i}{\psi_i}\vertr\geq H_i, \] where $G_i,H_i\in\N$ and $\nf{\phi_i}{\psi_i}$ are pairs of pp-$1$-formulae of the form $\nf{d|x}{x=0}$ and $\nf{xb=0}{c|x}$ for $1\leq i\leq n$, answers whether there exists $M\in\Mod\text{-}R$ satisfying $\chi$, then $T_R$ is decidable. 
\end{theorem}

%
%
\noindent
We call any conjunction of sentences of the form
\[ \vertl \nicefrac{d|x}{x=0}\vertr=1 \ \ \ \text{ or } \ \ \ \vertl \nicefrac{xb=0}{c|x}\vertr=1\] an \textbf{auxiliary sentence}.

\smallskip

\noindent
\textbf{Convention:} In the sequel, we will use the symbol $\square$ as a variable denoting either $=$ or $\geq$ when talking about conjunctions of sentences like $\vertl\nf{\phi}{\psi}\vertr\square N$. It will be useful for us to extend this notation so that $\square$ can also be the symbol $\emptyset$, where $\square$ being $\emptyset$ indicates that $\vertl\nf{\phi}{\psi}\vertr\square N$ is omitted from the conjunction. For instance, when $\square_1$ is $\emptyset$, the sentence
$\vertl \nicefrac{d|x}{x=0}\vertr \square_1 D\wedge \vertl \nicefrac{xb=0}{c|x}\vertr \square_2 E$ stands for $\vertl \nicefrac{xb=0}{c|x}\vertr \square_2 E$.


\smallskip
\noindent
In this section we improve \cite[4.1]{Decdense} to prove the following.

\begin{theorem}\label{1stsynthm}
Let $R$ be arithmetical ring. If there exists an algorithm which, given a sentence $\chi$ of the form
\[\tag{$\dagger$} \vertl \nicefrac{d|x}{x=0}\vertr\square_1 D\wedge \vertl \nicefrac{xb=0}{c|x}\vertr \square_2 E\wedge\bigwedge_{i=1}^m\vertl\nf{\phi_i}{\psi_i}\vertr=G_i\wedge\bigwedge_{i=m+1}^n\vertl\nf{\phi_i}{\psi_i}\vertr\geq H_i\wedge\Xi ,\] where $\square_1,\square_2\in\{\geq,=,\emptyset\}$, $d,c,b\in R\backslash\{0\}$, $D,E,G_i,H_i\in\N$, $\Xi$ is an auxiliary sentence and $\nf{\phi_i}{\psi_i}$ are pairs of pp-$1$-formulae of the form $\nf{x=x}{c'|x}$ and $\nf{xb'=0}{x=0}$ for $1\leq i\leq n$, answers whether there exists an $R$-module satisfying $\chi$, then $T_R$ is decidable.
%
%
\end{theorem}

\begin{definition}
Let $X,Y$ be finite subsets of pp-pairs of the form $\nf{d|x}{x=0}$ or $\nf{xb=0}{c|x}$, $f:X\rightarrow \N$ and $g:Y\rightarrow \N$ functions. Define $\chi_{f,g}$ to be the sentence
\[\bigwedge_{\nicefrac{\phi}{\psi}\in X}\vertl \nicefrac{\phi}{\psi}\vertr=f(\nicefrac{\phi}{\psi})\wedge\bigwedge_{\nicefrac{\phi}{\psi}\in Y}\vertl \nicefrac{\phi}{\psi}\vertr\geq g(\nicefrac{\phi}{\psi}).\]
For the rest of this section, let $W$ be the set of $\mcal{L}_R$-sentences of the form $\chi_{f,g}$ and let $V$ be the set of $w\in W$ such that there exists $M\in\Mod\text{-}R$ with $M\models w$. If $X$ and $Y$ are both empty then $\chi_{f,g}$ should be read as the true sentence. 
\end{definition}

\noindent
Define $\clx_1 \chi_{f,g}$ to be
\[\vertl \{\nf{d|x}{x=0}\in X \st f(\nf{d|x}{x=0})>1\}\vertr+\vertl \{\nf{d|x}{x=0}\in Y \st g(\nf{d|x}{x=0})>1\}\vertr\]
and
$\clx_2 \chi_{f,g}$ to be
\begin{align*}
    &\vertl \{\nf{xb=0}{c|x}\in X \st f(\nf{xb=0}{c|x})>1 \text{ and } c,b\neq 0\}\vertr+\cr
    &\qquad\qquad\qquad\qquad\qquad\qquad\qquad\vertl \{\nf{xb=0}{c|x}\in Y \st g(\nf{xb=0}{c|x})>1 \text{ and } b,c\neq 0\}\vertr.
\end{align*}
Formally, we extend $\clx_1$ and $\clx_2$ to $\bot,\top\in \mathbb{W}$ by setting $\clx_1\bot=\clx_1\top=0$ and $\clx_2\bot=\clx_2\top=0$. We will use the notation $\clx_i\underline{w}\leq \clx_i w$ and $\clx_i\underline{w}<\clx_i w$, for $i\in\{1,2\}$, $\underline{w}\in\mathbb{W}$ and $w\in W$ as in subsection \ref{lattcomb}.

\begin{remark}
For all $w_1,w_2\in W$ and $i\in\{1,2\}$, \[\clx_i(w_1\wedge w_2)\leq \clx_i(w_1)+\clx_i(w_2).\]
\end{remark}
\noindent
For our purposes, given $w\in W$, we may always assume that $w$ is of the form
\[
\chi_{f,g}\wedge \Xi,
\]
where $f:X\rightarrow \N_2$, $g:Y\rightarrow \N_2$
with $X,Y$ finite disjoint sets of pp-pairs of the form $\nf{d|x}{x=0}$ or $\nf{xb=0}{c|x}$,
and $\Xi$ is an auxiliary sentence.

It is obvious that any $\chi_{f,g}\in W$ may be rewritten as $\chi_{f',g}\wedge \Xi$ where $f':X'\rightarrow \N_2$ and $\Xi$ is an auxiliary sentence. Moreover, for $\chi_{f,g}\wedge \Xi\in W$, let $Y':=\{\nf{\phi}{\psi} \st g(\nf{\phi}{\psi})>1\}$ and $g':=g|_{Y'}$. Then $\chi_{f,g}\wedge \Xi\in V$ if and only if $\chi_{f,g'}\wedge \Xi\in V$. If $\nf{\phi}{\psi}\in X\cap Y$ and $f(\nf{\phi}{\psi})< g(\nf{\phi}{\psi})$ then $T_R\models\lnot\chi_{f,g}$. If  $f(\nf{\phi}{\psi})\geq g(\nf{\phi}{\psi})$ then $\chi_{f,g}\wedge\Xi\in V$ if and only if $\chi_{f,g'}\wedge\Xi\in V$ where $g':=g|_{Y\backslash\{\nf{\phi}{\psi}\}}$. So, given $w\in W$, we can effectively decide that $w\notin V$ or compute $f,g$ and $\Xi$ of the required form such that, $\clx_1 \chi_{f,g}\wedge\Xi\leq \clx_1w$, $\clx_2\chi_{f,g}\wedge\Xi\leq \clx_2 w$, and,  $w\in V$ if and only if $\chi_{f,g}\wedge \Xi\in V$.

\begin{remark}\label{decomp}
Let $X,Y$ be disjoint finite sets of pp-pairs of the form $d|x/x=0$ or $xb=0/c|x$, $f:X\rightarrow \N$ and $g:Y\rightarrow \N$ functions, and $\Xi$ an auxiliary sentence. For each $1\leq j\leq n$, let $\theta_j$ be an auxiliary sentence. Suppose that for all $M\in\Mod\text{-}R$, there exist modules $M_i\models \theta_j$ such that $M\equiv\bigoplus_{j=1}^nM_j$. Then $\chi_{f,g}\wedge\Xi\in V$ if and only if
\[\bigsqcup_{(\overline{f},\overline{g})\in\Omega_{(f,g,n)}}\bigsqcap_{j=1}^n  \chi_{f_j,g_j}\wedge\theta_j\wedge\Xi\in \mathbb{V}.\]
Moreover, for all $(\overline{f},\overline{g})\in\Omega_{f,g,n}$ and $1\leq j\leq n$,
\[ \ \ \ \ \ \ \clx_1(\chi_{f_j,g_j}\wedge\theta_j\wedge\Xi)\leq \clx_1 (\chi_{f,g}\wedge\Xi) \text{ and }\]
\[\clx_2(\chi_{f_j,g_j}\wedge\theta_j\wedge\Xi)\leq \clx_2 (\chi_{f,g}\wedge\Xi).\]
\end{remark}

\noindent
The next lemma is more precise than we will need in this section. However, we will need its full strength in section \ref{Furthersyn}. The total order $\prec$ on the set $\{\emptyset, =,\leq \}$ is defined as $\emptyset \, \prec \,  = \, \prec \, \leq$.

\begin{lemma}\label{redmult}
Let $\phi/\psi,\phi'/\psi',\sigma/\tau$ be pp-pairs, $\square,\square'\in\{=,\geq\}$, $E,E'\in\N_2$ and let $\Sigma$ be an $\mcal{L}_R$-sentence. Suppose that $M\models\Sigma $ implies \[|\phi/\psi(M)|=|\sigma/\tau(M)|\cdot|\phi'/\psi'(M)|\] for all $M\in\Mod\text{-}R$.

There is an algorithm which, given $\phi/\psi,\phi'/\psi',\sigma/\tau$, $\square,\square'\in\{=,\geq\}$, $E,E'\in\N_2$ and $\Sigma$ as above, either returns $\Omega:=\{\bot\}$, in which case
\[T_R\models \lnot(\Sigma\wedge |\phi/\psi|\square E\wedge |\phi'/\psi'|\square' E'),\]
or, returns $\Omega$, a finite set of tuples $(D_1,D_2,\square_1,\square_2)\in \N^2\times \{=,\geq\}^2$ such that
\[T_R\models \Sigma\wedge |\phi/\psi|\square E\wedge |\phi'/\psi'|\square' E'\leftrightarrow\bigvee_{(D_1,D_2,\square_1,\square_2)\in\Omega}\Sigma\wedge |\tau/\sigma|\square_1 D_1\wedge |\phi'/\psi'|\square_2 D_2\]
and $D_1\cdot D_2<E\cdot E'$, $\square_1\preceq\square$ and $\square_2\preceq \square'$ for all $(D_1,D_2,\square_1,\square_2)\in\Omega$.
\end{lemma}

\begin{proof}\

\noindent
\textbf{Case 1:} $\square$ and $\square'$ are both $=$.

\noindent
Let $\Omega:=\{\bot\}$ if $E'$ does not divide $E$, otherwise $\Omega:=\{(E/E',E',=,=)\}$. Note $(E/E')\cdot E'=E<E\cdot E'$ because $E'\geq 2$.

\noindent
\textbf{Case 2:} $\square$ is $\geq$ and $\square'$ is $=$.

\noindent
Let $\Omega:=\{(\lceil E/E'\rceil,E',\geq, =)\}$. Note \[\lceil E/E'\rceil\cdot E'<(E/E'+1)\cdot E=E+E'\leq E\cdot E'.\]

\noindent
\textbf{Case 3:} $\square$ is $=$ and $\square'$ is $\geq$.

\noindent
Let $X:=\{D\in\N \st D|E \text{ and } D\geq E'\}$. Define $\Omega:=\{\bot\}$ if $X=\emptyset$ and $\Omega:=\{(E/D,D,=,=)\st D\in X\}$.
Note $(E/D)\cdot D=E<E\cdot E'$.

\noindent
\textbf{Case 4:} $\square$ and $\square'$ are both $\geq$.

\noindent
If $E'\geq E$ then let $\Omega:=\{(1,E',\geq ,\geq)\}$. If $E>E'$ then set
\[
\Omega:=\{(\lceil E/D\rceil ,D,\geq ,=) \st E>D\geq E'\}\cup\{(1,E,\geq ,\geq)\}.
\]
Note that $E'<E\cdot E'$, $E<E\cdot E'$ and \[\lceil E/D\rceil\cdot D<(E/D+1)\cdot D=E+D\leq E\cdot E'.
\qedhere\]
\end{proof}

The following remark is easy to prove. We record it here because we will use it frequently.

\begin{remark}\label{equivpair1}
Let $R$ be a commutative ring. For all $a,b\in R$ and $M\in\Mod\text{-}R$,
\[\vertl \nicefrac{a|x}{ab|x}(M)\vertr=\vertl \nicefrac{x=x}{xa=0+b|x}(M)\vertr\] and
\[\vertl \nicefrac{xab=0}{xa=0}(M)\vertr=\vertl \nicefrac{a|x\wedge xb=0}{x=0}(M)\vertr.\]

\end{remark}

\begin{proposition}\label{redclx1} Let $R$ be a recursive arithmetical ring.
There is an algorithm which given $w\in W$ with $\clx_1(w)> 1$ outputs $\underline{w}\in\mathbb{W}$ such that $\clx_1(\underline{w})<\clx_1(w)$, $\clx_2(\underline{w})\leq\clx_2(w)$, and, $w\in V$ if and only if $\underline{w}\in\mathbb{V}$.
\end{proposition}
\begin{proof}
Let $X,Y$ be disjoint finite sets of pp-pairs of the form $d|x/x=0$ and $xb=0/c|x$. Let $f:X\rightarrow \N_2$ and $g:Y\rightarrow \N_2$ be functions, and let $\Xi$ be an auxiliary sentence. Let $w$ be
\[\bigwedge_{\nicefrac{\phi}{\psi}\in X}\vertl \nicefrac{\phi}{\psi}\vertr=f(\nicefrac{\phi}{\psi})\wedge\bigwedge_{\nicefrac{\phi}{\psi}\in Y}\vertl \nicefrac{\phi}{\psi}\vertr\geq g(\nicefrac{\phi}{\psi})\wedge \Xi.\]
Suppose that there exist non-equal $a, b\in R$ such that $\nicefrac{a|x}{x=0},\nicefrac{b|x}{x=0}\in X\cup Y$ i.e. $\clx_1(w)>1$. Let $\alpha,r,s\in R$ be such that $a\alpha=br$ and $b(\alpha-1)=as$. Define
\begin{enumerate}
\item $\Sigma_1:=\vertl \nicefrac{x=x}{\alpha|x}\vertr=1\wedge \vertl \nicefrac{rb|x}{x=0}\vertr=1$,
\item $\Sigma_2:=\vertl \nicefrac{x=x}{\alpha|x}\vertr=1\wedge \vertl \nicefrac{xb=0}{r|x}\vertr=1$,
\item $\Sigma_3:=\vertl \nicefrac{x=x}{(\alpha-1)|x}\vertr=1\wedge \vertl \nicefrac{as|x}{x=0}\vertr=1$, and
\item $\Sigma_4:=\vertl \nicefrac{x=x}{(\alpha-1)|x}\vertr=1\wedge \vertl \nicefrac{xa=0}{s|x}\vertr=1$.
\end{enumerate}
It follows directly from \ref{decomposeorder} that, for any $M\in \Mod\text{-}R$, there are $M_i\models \Sigma_i$ for $1\leq i\leq 4$ such that $M\equiv M_1\oplus M_2\oplus M_3\oplus M_4$. Therefore, by \ref{decomp}, $w\in V$ if and only if
\[
\bigsqcup_{(\overline{f},\overline{g})\in\Omega_{f,g,4}}\bigsqcap_{i=1}^4\chi_{f_i,g_i}\wedge\Xi\wedge \Sigma_i\in \mathbb{V}.
\]
For each $(\overline{f},\overline{g})\in\Omega_{f,g,4}$ and $1\leq i\leq 4$, it is enough to compute $\underline{w}_i\in\mathbb{W}$ such that $\clx_1(\underline{w}_i)<\clx_1(\chi_{f,g}\wedge\Xi)$, $\clx_2(\underline{w}_i)\leq\clx_2(\chi_{f,g}\wedge\Xi)$ and $\chi_{f_i,g_i}\wedge\Xi\wedge \Sigma_i\in V$ if and only if $\underline{w}_i\in V$.

Fix $(\overline{f},\overline{g})\in\Omega_{f,g,4}$. For each $1\leq i\leq 4$, let $X_i$ be the domain of $f_i$ and $Y_i$ be the domain of $g_i$.

\smallskip

\noindent
\textbf{Case i=1:} Suppose $M\models \Sigma_1$. Then $M\alpha=M$ and hence $Ma=Mbr=0$. Therefore, if $M\models \Sigma_1$ then $\vertl \nicefrac{a|x}{x=0}(M)\vertr=1$.

If $\nf{a|x}{x=0}\in X_1$ and $f_1(\nf{a|x}{x=0})=1$ then $\clx_1(\chi_{f_1,g_1}\wedge\Xi\wedge \Sigma_1)<\clx_1(\chi_{f,g}\wedge\Xi)$ and, by \ref{decomp}, $\clx_2(\chi_{f_1,g_1}\wedge\Xi\wedge \Sigma_1)\leq\clx_2(\chi_{f,g}\wedge\Xi)$. So, $\underline{w}_i:=\chi_{f_1,g_1}\wedge\Xi\wedge \Sigma_1$ has the required properties.

If $\nf{a|x}{x=0}\in X_1$ and $f_1(\nf{a|x}{x=0})\neq 1$ then, by the first paragraph, $\chi_{f_1,g_1}\wedge\Xi\wedge \Sigma_1$ is not satisfied by any $R$-module. If $\nf{a|x}{x=0}\notin X_1$ then $\nf{a|x}{x=0}\in Y_1$ since $X\cup Y=X_1\cup Y_1$. Moreover $g_1(\nf{a|x}{x=0})=g(\nf{a|x}{x=0})$. So $g_1(\nf{a|x}{x=0})>1$ and hence $\chi_{f_1,g_1}\wedge\Xi\wedge \Sigma_1$ is not satisfied by any $R$-module. In either case, set $\underline{w}_i:=\bot$. Then $\underline{w}_i\in \mathbb{V}$ if and only if $\chi_{f_1,g_1}\wedge\Xi\wedge \Sigma_1\in V$. By definition $\clx_1(\bot)<\clx_1 (\chi_{f,g}\wedge\Xi)$ and $\clx_2(\bot)\leq\clx_2 (\chi_{f,g}\wedge\Xi)$.

\smallskip

\noindent
\textbf{Case i=2:} Suppose $M\models \Sigma_2$. Then $Ma=M\alpha a=Mbr$. So, since $xb=0\leq_Mr|x$, by \ref{equivpair1},
\[\vertl \nicefrac{b|x}{x=0}(M)\vertr=\vertl \nicefrac{b|x}{br|x}(M)\vertr\cdot\vertl \nicefrac{br|x}{x=0}(M)\vertr=\vertl \nicefrac{x=x}{r|x}(M)\vertr\cdot\vertl \nicefrac{a|x}{x=0}(M)\vertr.\]
Let $X'=X_2\backslash\{\nicefrac{a|x}{x=0},\nicefrac{b|x}{x=0}\}$ and $Y'=Y_2\backslash\{\nicefrac{a|x}{x=0},\nicefrac{b|x}{x=0}\}$. Let $\square,\square'\in \{=,\geq\}$ and $A,B$ be such that $\chi_{f_2,g_2}\wedge\Xi\wedge\Sigma_2$ is
\[\vertl \nicefrac{a|x}{x=0}\vertr\square A\wedge\vertl \nicefrac{b|x}{x=0}\vertr\square'B\wedge\bigwedge_{\nicefrac{\phi}{\psi}\in X'}\vertl\nicefrac{\phi}{\psi}\vertr= f_2(\nicefrac{\phi}{\psi})\wedge\bigwedge_{\nicefrac{\phi}{\psi}\in Y'}\vertl\nicefrac{\phi}{\psi}\vertr\geq g_2(\nicefrac{\phi}{\psi})\wedge\Xi\wedge\Sigma_2.\]
By \ref{redmult}, there is an algorithm which either returns $\Omega:=\{\bot\}$, in which case
\[T_R\models\lnot \Sigma_2\wedge \vertl \nicefrac{a|x}{x=0}\vertr\square A\wedge\vertl \nicefrac{b|x}{x=0}\vertr\square'B,\]
or, a set $\Omega\subseteq \N^2\times\{=,\geq\}^2$ such that
\[\Sigma_2\wedge \vertl \nicefrac{a|x}{x=0}\vertr\square A\wedge\vertl \nicefrac{b|x}{x=0}\vertr\square'B\]
is equivalent, with respect to $T_R$, to
\[
\bigvee_{(D_1,D_2,\square_1,\square_2)\in \Omega}\Sigma_2\wedge\vertl\nf{x=x}{r|x}\vertr\square_1 D_1\wedge \vertl\nf{a|x}{x=0}\vertr\square_2D_2.
\]
If $\Omega:=\{\bot\}$ then $\chi_{f_2,g_2}\wedge\Xi\wedge \Sigma_2\in V$ if and only if $\bot\in \mathbb{V}$. By definition $\clx_1(\bot)< \clx_1(\chi_{f,g}\wedge\Xi)$ and $\clx_2(\bot)\leq \clx_2(\chi_{f,g}\wedge\Xi)$. Otherwise, for each $(D_1,D_2,\square_1,\square_2)\in\Omega$, let $u_{(D_1,D_2,\square_1,\square_2)}$ be
\[
\vertl\nf{x=x}{r|x}\vertr\square_1 D_1\wedge \vertl\nf{a|x}{x=0}\vertr\square_2D_2\wedge\bigwedge_{\nicefrac{\phi}{\psi}\in X'}\!\vertl\nicefrac{\phi}{\psi}\vertr= f_2(\nicefrac{\phi}{\psi})\wedge\bigwedge_{\nicefrac{\phi}{\psi}\in Y'}\!\vertl\nicefrac{\phi}{\psi}\vertr\geq g_2(\nicefrac{\phi}{\psi})\wedge\Xi\wedge\Sigma_2.
\]
Then $\chi_{f_2,g_2}\wedge\Xi\wedge\Sigma_2$ is equivalent to
\[\bigvee_{(D_1,D_2,\square_1,\square_2)\in \Omega} u_{(D_1,D_2,\square_1,\square_2)}\] with respect to $T_R$. Therefore $\chi_{f_2,g_2}\wedge\Xi\wedge\Sigma_2\in V$ if and only if
\[\bigsqcup_{(D_1,D_2,\square_1,\square_2)\in \Omega} u_{(D_1,D_2,\square_1,\square_2)}\in\mathbb{V}.\]
Moreover,
\[\clx_1 (u_{(D_1,D_2,\square_1,\square_2)})\leq 1+(\clx_1 (\chi_{f_2,g_2}\wedge\Xi\wedge\Sigma_2) -2)< \clx_1 (\chi_{f_2,g_2}\wedge\Xi\wedge\Sigma_2) \]
and $\clx_2(u_{(D_1,D_2,\square_1,\square_2)})=\clx_2 (\chi_{f_2,g_2}\wedge\Xi\wedge\Sigma_2)$.
So we are done, since, by \ref{decomp}, $\clx_1 (\chi_{f_2,g_2}\wedge\Xi\wedge\Sigma_2)\leq \clx_1 (\chi_{f,g}\wedge\Xi)$ and $\clx_2 (\chi_{f_2,g_2}\wedge\Xi\wedge\Sigma_2)\leq \clx_2 (\chi_{f,g}\wedge\Xi)$.

The case $i=3$ is similar to $i=1$ and the case $i=4$ is similar to $i=2$.
\end{proof}

\noindent
We now start work on showing that there exists an algorithm which given $w\in W$ with $\clx_2 w>1$ outputs $\underline{w}\in\mathbb{W}$ such that $w\in V$ if and only if $\underline{w}\in \mathbb{V}$ and $\clx_2(\underline{w})<\clx_2(w)$. This uses the same ideas as for $\clx_1$ but is significantly more complicated.

\begin{figure}
\newcommand\factor{0.6}
\newcommand\textSizeHere{} 
\newcommand\nodeSizeHere{\LARGE}
\scalebox{0.7}{
\begin{tikzpicture}[scale=1,rotate=90,transform shape]
\node (root) at ( 0,0) [] {};
\node (R) at ( 12*\factor,-2*\factor) [] {};
\node (L) at ( -12*\factor,-2*\factor) [] {};
\node (RR) at ( 18*\factor,-4*\factor)  [] {\nodeSizeHere$\Lambda'_1$};
\node (LL) at ( -18*\factor,-4*\factor)[] {\nodeSizeHere$\Lambda_1$};
\node (RU) at ( 12*\factor,-5*\factor) [] {};
\node (LU) at ( -12*\factor,-5*\factor) [] {};
\node (RUU) at ( 12*\factor,-8*\factor) [] {};
\node (LUU) at ( -12*\factor,-8*\factor)[] {};
\node (LLUU) at ( -18*\factor,-7*\factor)  [] {\nodeSizeHere$\Lambda_2$};
\node (RRUU) at ( 18*\factor,-7*\factor) [] {\nodeSizeHere$\Lambda'_2$};
\node (LLUUU) at ( -18*\factor,-10*\factor) [] {\nodeSizeHere$\Lambda_3$};
\node (RRUUU) at ( 18*\factor,-10*\factor) [] {\nodeSizeHere$\Lambda'_3$};
\node (B) at ( 12*\factor,-11*\factor) [] {};
\node (A) at ( -12*\factor,-11*\factor)  [] {};
\node (BR) at ( 18*\factor,-13*\factor)  [] {};
\node(AL) at ( -18*\factor,-13*\factor)  [] {};
\node (BL) at ( 7*\factor,-13*\factor) [] {};
\node (AR) at ( -7*\factor,-13*\factor)[] {};
\node (BL4) at ( 24*\factor,-15*\factor) [] {\nodeSizeHere$\Lambda'_4$};
\node (AL4) at ( -24*\factor,-15*\factor)[] {\nodeSizeHere$\Lambda_4$};
\node (BP4) at ( 2*\factor,-15*\factor) [] {\nodeSizeHere$P'_4$};
\node (AP4) at ( -2*\factor,-15*\factor) [] {\nodeSizeHere$P_4$};
\node (BLU) at( 7*\factor,-17*\factor)  [] {};
\node (ARU) at ( -7*\factor,-17*\factor) [] {};
\node(BRU) at( 18*\factor,-17*\factor) [] {};
\node (ALU) at ( -18*\factor,-17*\factor) [] {};
\node (BL6) at ( 18*\factor,-21*\factor)  [] {\nodeSizeHere$\Lambda_6'$};
\node (AL6) at ( -18*\factor,-21*\factor) [] {\nodeSizeHere$\Lambda_6$};
\node (BL5) at ( 24*\factor,-19*\factor)  [] {\nodeSizeHere$\Lambda_5'$};
\node (AL5) at ( -24*\factor,-19*\factor) [] {\nodeSizeHere$\Lambda_5$};
\node (BP5) at ( 2*\factor,-19*\factor)  [] {\nodeSizeHere$P'_5$};
\node (AP5) at ( -2*\factor,-19*\factor) [] {\nodeSizeHere$P_5$};
\node(BP6) at ( 7*\factor,-21*\factor)  [] {\nodeSizeHere$P_6'$};
\node (AP6) at ( -7*\factor,-21*\factor) [] {\nodeSizeHere$P_6$};
\draw [->] (root) to node [sloped, above] {\textSizeHere$\vertl \frac{x=x}{\alpha|x}\vertr=1$} (L);
\draw [->] (root) to node [sloped, above] {\textSizeHere$\vertl \frac{x=x}{\alpha-1|x}\vertr=1$} (R);
\draw [->] (L) to node [sloped, above] {\textSizeHere$\vertl \frac{c'u|x}{x=0}\vertr=1$} (LL);
\draw [->] (R) to node [sloped, above] {\textSizeHere$\vertl \frac{cu'|x}{x=0}\vertr=1$} (RR);
\draw [->] (L) to node [right] {\textSizeHere$\vertl\frac{xu=0}{c'|x}\vertr=1$} (LU);
\draw [->] (R) to node [left] {\textSizeHere$\vertl\frac{xu'=0}{c|x}\vertr=1$} (RU);
\draw [->] (LU) to node [right] {\textSizeHere$\vertl \frac{x(\beta-1)=0}{x=0}\vertr=1$} (LUU);
\draw [->] (RU) to node [left] {\textSizeHere$\vertl \frac{x(\beta'-1)=0}{x=0}\vertr=1$} (RUU);
\draw [->] (LU) to node [sloped, above] {\textSizeHere$\vertl \frac{x\beta=0}{x=0}\vertr=1$} (LLUU);
\draw [->] (RU) to node [sloped, above] {\textSizeHere$\vertl \frac{x\beta'=0}{x=0}\vertr=1$} (RRUU);
\draw [->] (LUU) to node [sloped, above] {\textSizeHere$\vertl\frac{us|x}{x=0}\vertr=1$} (LLUUU);
\draw [->] (RUU) to node [sloped, above] {\textSizeHere$\vertl\frac{u's'|x}{x=0}\vertr=1$} (RRUUU);
\draw [->] (LUU) to node [right] {\textSizeHere $\vertl\frac{xs=0}{u|x}\vertr=1$} (A);
\draw [->] (RUU) to node [left] {\textSizeHere$\vertl\frac{xs'=0}{u'|x}\vertr=1$} (B);
\draw [->] (A) to node [sloped, above] {\textSizeHere$\vertl \frac{x\delta=0}{x=0}\vertr=1$} (AL);
\draw [->] (B) to node [sloped, above] {\textSizeHere$\vertl \frac{x\delta'=0}{x=0}\vertr=1$} (BR);
\draw [->] (A) to node [sloped, above] {\textSizeHere$\vertl \frac{x(\delta-1)=0}{x=0}\vertr=1$} (AR);
\draw [->] (B) to node [sloped, above] {\textSizeHere$\vertl \frac{x(\delta'-1)=0}{x=0}\vertr=1$}(BL);
\draw [->] (AL) to node [sloped, above] {\textSizeHere$\vertl\frac{xs=0}{c'u|x}\vertr=1$} (AL4);
\draw [->] (BR) to node [sloped, above] {\textSizeHere$\vertl\frac{xs'=0}{cu'|x}\vertr=1$} (BL4);
\draw [->] (AR) to node [sloped, above] {\textSizeHere $\vertl\frac{xb=0}{c'u|x}\vertr=1$}(AP4);
\draw [->] (BL) to node [sloped, above] {\textSizeHere $\vertl\frac{xb'=0}{cu'|x}\vertr=1$}(BP4);
\draw [->] (AR) to node [left] {\textSizeHere$\vertl\frac{c'bu|x}{x=0}\vertr=1$} (ARU);
\draw [->] (BL) to node [right] {\textSizeHere$\vertl\frac{cb'u'|x}{x=0}\vertr=1$} (BLU);
\draw [->] (BR) to node [left] {\textSizeHere $\vertl\frac{cu's'|x}{x=0}\vertr=1$}(BRU);
\draw [->] (AL) to node [right] {\textSizeHere$\vertl\frac{c'us|x}{x=0}\vertr=1$}(ALU);
\draw [->] (ALU) to node [sloped, above] {\textSizeHere $\vertl\frac{\lambda s|x}{x=0}\vertr=1$} (AL5);
\draw [->] (BRU) to node [sloped, above] {\textSizeHere $\vertl\frac{\lambda's'|x}{x=0}\vertr=1$} (BL5);
\draw [->] (ALU) to node [right] {\textSizeHere $\vertl \frac{x\lambda=0}{s|x}\vertr=1$} (AL6);
\draw [->] (BRU) to node [left] {\textSizeHere $\vertl \frac{x\lambda'=0}{s'|x}\vertr=1$} (BL6);
\draw [->] (ARU) to node [sloped, above] {\textSizeHere $\vertl\frac{b\mu|x}{x=0}\vertr=1$}(AP5);
\draw [->] (BLU) to node [sloped, above] {\textSizeHere $\vertl\frac{b'\mu'|x}{x=0}\vertr=1$}(BP5);
\draw [->] (ARU) to node [left] {\textSizeHere$\vertl\frac{x\mu=0}{b|x}\vertr=1$} (AP6);
\draw [->] (BLU) to node [right] {\textSizeHere$\vertl\frac{x\mu'=0}{b'|x}\vertr=1$} (BP6);
\node at ( -1*\factor,-24*\factor) [] {\nodeSizeHere Figure 1};
\end{tikzpicture}
}
\end{figure}

\begin{lemma}\label{xb=0/c|xsplit}
Let $R$ be an arithmetical ring and let $b,c,b',c'\in R$. Let $\alpha,u,u',\allowbreak\beta,\beta',s,\allowbreak s',\allowbreak r\allowbreak,r',\allowbreak\delta,\delta',\lambda,\lambda',\mu,\mu'\in R$ be such that
\[\begin{array}{ccccccccccccccc}
& & & & c\alpha & \!\!\!\!=\!\!\!\! & c'u, &  & c'(\alpha-1) & \!\!\!\!=\!\!\!\! & cu' & & & &\\
 u\beta & \!\!\!\!=\!\!\!\! & b'r,  &   & b'(\beta-1) & \!\!\!\!=\!\!\!\! & us, & & u'\beta'  & \!\!\!\!=\!\!\!\! &br', & & b(\beta'-1)& \!\!\!\!=\!\!\!\! &u's'\\
 b\delta & \!\!\!\!=\!\!\!\! & s\lambda, &  & s(\delta-1) & \!\!\!\!=\!\!\!\! & b\mu, & & b\delta' &\!\!\!\!=\!\!\!\! &s'\lambda', & & s'(\delta'-1) &\!\!\!\!=\!\!\!\! & b'\mu'.
\end{array}\]

\noindent
Define the sentences $\Lambda_i$ and $\Lambda_i'$ for $1\leq i\leq 6$ and $P_i$ and $P_i'$ for $4\leq i\leq 6$ to be the conjunction of sentences labeling the edges in the path from the root of the tree in Figure 1 to the leaf of the tree with that sentence as label.

Every $R$-module is elementary equivalent to an $R$-module of the form
\[(\bigoplus_{i=1}^6M_i\oplus \bigoplus_{i=4}^6N_i)\oplus(\bigoplus_{i=1}^6M_i'\oplus \bigoplus_{i=4}^6N_i')\]
where $M_i\models\Lambda_i$ and $M_i'\models\Lambda_i'$ for $1\leq i\leq 6$ and $N_i\models P_i$ and $N_i'\models P_i'$ for $4\leq i\leq 6$.

\smallskip
\noindent
Moreover, for all $M\in\Mod\text{-}R$,
\begin{enumerate}[(i)]
\item $M\models \Lambda_1$ implies $c\in\ann_RM$ and hence $\vertl \nicefrac{xb=0}{c|x}(M)\vertr=\vertl \nicefrac{xb=0}{x=0}(M)\vertr$,
\item $M\models \Lambda_2$ implies $\vertl \nicefrac{xb'=0}{c'|x}(M)\vertr=1$,
\item $M\models \Lambda_3$ implies $b'\in\ann_RM$ and hence $\vertl \nicefrac{xb'=0}{c'|x}(M)\vertr =\vertl \nicefrac{x=x}{c'|x}(M)\vertr$,
\item $M\models \Lambda_4$ implies $\vertl \nicefrac{xb'=0}{c'|x}(M)\vertr=1$,
\item $M\models \Lambda_5$ implies $b\in \ann_R M$ and hence $\vertl \nicefrac{xb=0}{c|x}(M)\vertr=\vertl \nicefrac{x=x}{c|x}(M)\vertr$,
\item $M\models \Lambda_6$ implies
\[\vertl \nicefrac{xb=0}{c|x}(M)\vertr=\vertl \nicefrac{xb'=0}{c'|x}(M)\vertr\cdot \vertl \nicefrac{x\lambda=0}{x=0}(M)\vertr,\]
\item $M\models P_4$ implies $\vertl \nicefrac{xb=0}{c|x}(M)\vertr=1,$
\item $M\models P_5$ implies $b'\in\ann_RM$ and hence $\vertl \nicefrac{xb'=0}{c'|x}(M)\vertr=\vertl \nicefrac{x=x}{c'|x}(M)\vertr,$ and
\item $M\models P_6$ implies
\[\vertl \nicefrac{xb'=0}{c'|x}(M)\vertr=\vertl \nicefrac{xb=0}{c|x}(M)\vertr\cdot \vertl \nicefrac{x\mu=0}{x=0}(M)\vertr.\]
\end{enumerate}
Similarly, the symmetry of figure 1, gives 9 conditions for $\Lambda_i'$ and $P_i'$, where $c,b$ and  $c',b'$ are interchanged
and $\lambda,\mu$ are replaced by $\lambda',\mu'$, respectively.

\end{lemma}
\begin{proof}
There are two edges coming out of each node of the tree in Figure $1$. In each instance the two edges are either
\begin{enumerate}
\item $\vertl \nicefrac{x=x}{\gamma|x}\vertr=1$ and $\vertl \nicefrac{x=x}{(\gamma-1)|x}\vertr=1$ for some $\gamma\in R$,
\item $\vertl \nicefrac{x\gamma=0}{x=0}\vertr=1$ and $\vertl \nicefrac{x(\gamma-1)=0}{x=0}\vertr=1$ for some $\gamma\in R$, or
\item $\vertl \nicefrac{ab|x}{x=0}\vertr=1$ and $\vertl \nicefrac{xa=0}{b|x}\vertr=1$.
\end{enumerate}
By \ref{decomposeorder}, in each $(1),(2)$ and $(3)$, for all modules $M\in\Mod\text{-}R$, there exist $M_1$ satisfying the first sentence and $M_2$ satisfying the second sentence such that $M\cong M_1\oplus M_2$. The first claim follows from this fact.

For any $M\in\Mod\text{-}R$, $\alpha\notin \Div M$ implies $c|x$ is equivalent to $c'u|x$ in $M$ because $c\alpha=c'u$.

\smallskip

\noindent
$(i)$ Suppose $M\models \Lambda_1$. Since $\vertl \nicefrac{c'u|x}{x=0}(M)\vertr=1$, $Mc=Mc'u=0$. So $c\in\ann_RM$.

\smallskip
\noindent
$(ii)$ Suppose $M\models \Lambda_2$. Then $\beta\notin \Ass M$ and so, since $u\beta=b'r$, $xb'r=0$ is equivalent to $xu=0$ in $M$. Since $xu=0\leq_Mc'|x$, we conclude \[xb'=0\leq_M xb'r=0\leq_Mxu=0 \leq_Mc'|x.\]

\smallskip
\noindent
$(iii)$ Suppose $M\models \Lambda_3$. Then $\beta-1\notin \Ass M$ and so, since $b'(\beta-1)=us$, $xb'=0$ is equivalent to $xus=0$ in $M$. Since $\vertl \nicefrac{us|x}{x=0}(M)\vertr=1$, $us\in\ann_RM$ and hence $b'\in\ann_RM$.

\smallskip
\noindent
\textbf{Claim:} If
$M\models\vertl \nicefrac{x=x}{\alpha|x}\vertr=1\wedge\vertl \nicefrac{xu=0}{c'|x}\vertr=1\wedge \vertl \nicefrac{x(\beta-1)=0}{x=0}\vertr=1\wedge \vertl \nicefrac{xs=0}{u|x}\vertr=1$ then $\vertl \nicefrac{xb'=0}{c'|x}(M)\vertr=\vertl \nicefrac{xs=0}{c'u|x}(M)\vertr$.

First note that since $\beta-1\notin \Ass M$ and $b'(\beta-1)=us$, $xb'=0$ is equivalent to $xus=0$ in $M$. We show that $\vertl \nicefrac{xus=0}{c'|x}(M)\vertr=\vertl \nicefrac{xs=0}{c'u|x}(M)\vertr$.

Consider the map $f: xus=0/x=0(M)\rightarrow xs=0/c'u|x(M)$ defined by $f(m):=mu+c'u|x(M)$. Now $f$ is surjective since $\vertl \nicefrac{xs=0}{u|x}\vertr=1$. Suppose $f(m)=m'c'u$ for some $m'\in M$. Then $(m-m'c)u=0$. Since $\vertl \nicefrac{xu=0}{c'|x}\vertr=1$, $m-m'c\in c'|x(M)$ and hence $m\in c'|x(M)$. Therefore $\ker f=c'|x(M)$. So we have proved the claim.

We now prove the statements about modules satisfying $\Lambda_4,\Lambda_5,\Lambda_6$. The statements for modules satisfying $P_4,P_5$ and $P_6$ follow similarly.

\smallskip
\noindent
$(iv)$ Suppose $M\models \Lambda_4$. Then $\vertl \nicefrac{xs=0}{c'u|x}(M)\vertr =1$ and by the claim $\vertl \nicefrac{xb'=0}{c'|x}(M)\vertr=\vertl \nicefrac{xs=0}{c'u|x}(M)\vertr$.

\smallskip
\noindent
$(v)$ Suppose $M\models \Lambda_5$. Since $\delta\notin \Ass M$, $xb=0$ is equivalent to $xs\lambda=0$ in $M$. Since $\vertl \nicefrac{s\lambda|x}{x=0}(M)\vertr=1$, $s\lambda\in\ann_RM$ and hence $b\in\ann_RM$.

\smallskip
\noindent
$(vi)$ Suppose $M\models \Lambda_6$. Since $\delta\notin \Ass M$, $xb=0$ is equivalent to $xs\lambda=0$ in $M$. Since $xs=0\geq_Mc'u|x$,
\[\vertl \nicefrac{xb=0}{c|x}(M)\vertr=\vertl \nicefrac{xs\lambda=0}{c'u|x}(M)\vertr=\vertl \nicefrac{xs\lambda=0}{xs=0}(M)\vertr\cdot\vertl \nicefrac{xs=0}{c'u|x}(M)\vertr.\]
By \ref{equivpair1}, $\vertl \nicefrac{xs\lambda=0}{xs=0}(M)\vertr=\vertl \nicefrac{s|x\wedge x\lambda=0}{x=0}(M)\vertr$. So, since $\vertl\nicefrac{x\lambda=0}{s|x}\vertr=1$,
\[\vertl \nicefrac{xb=0}{c|x}\left(M\right)\vertr=\vertl \nicefrac{xs\lambda=0}{xs=0}\left(M\right)\vertr\cdot\vertl \nicefrac{xs=0}{c'u|x}\left(M\right)\vertr=\vertl \nicefrac{x\lambda=0}{x=0}\left(M\right)\vertr\cdot\vertl \nicefrac{xb'=0}{c'|x}\left(M\right)\vertr.\]
\end{proof}

\begin{proposition}\label{redclx2}
There is an algorithm which, given $w\in W$ with $\clx_2(w)> 1$, outputs $\underline{w}\in\mathbb{W}$ such that $\clx_2(\underline{w})<\clx_2(w)$, $\clx_1(\underline{w})\leq\clx_1(w)$, and, $w\in V$ if and only if $\underline{w}\in\mathbb{V}$.
\end{proposition}
\begin{proof}
We start with a special case. Let $b,c,b',c'\in R\backslash\{0\}$. Let $P_6,P'_6,\Lambda_6,\Lambda'_6$ be as in \ref{xb=0/c|xsplit}. Let $\Sigma_1:=\vertl c|x/x=0\vertr=1$, $\Sigma_2:=\vertl b|x/x=0\vertr=1$, $\Sigma_3:=\vertl c'|x/x=0\vertr=1$, $\Sigma_4:=\vertl b'|x/x=0\vertr=1$, $\Sigma_5:=\vertl\nicefrac{xb=0}{c|x}\vertr=1$, $\Sigma_6:=\vertl\nicefrac{xb'=0}{c'|x}\vertr=1$, $\Sigma_7:=\Lambda_6$, $\Sigma_8:=P_6$, $\Sigma_9:=\Lambda'_6$ and $\Sigma_{10}:=P'_6$.

Fix $1\leq i\leq 4$ or $7\leq i\leq 10$. Suppose $w$ is
\[\vertl\nicefrac{xb=0}{c|x}\vertr\square E\wedge\vertl\nicefrac{xb'=0}{c'|x}\vertr\square'E'\wedge\chi_{f,g}\wedge\Sigma_i\wedge\Xi.\]

\smallskip

\noindent
\textbf{Case i=1:} Let $w'$ be
\[\vertl\nicefrac{xb=0}{x=0}\vertr\square E\wedge\vertl\nicefrac{xb'=0}{c'|x}\vertr\square'E'\wedge\chi_{f,g}\wedge\Sigma_i\wedge\Xi.\]
Then $\clx_1 w'= \clx_1 w$ and  $\clx_2 w'<\clx_2 w$. Since $T_R\models w\leftrightarrow w'$, we get $w\in V$ if and only if $w'\in V$.

\smallskip

\noindent
\textbf{Case i=2,3,4:} The same argument as for $i=1$ works.

\smallskip

\noindent
\textbf{Case i=7:} By \ref{xb=0/c|xsplit}, if $M\models \Sigma_7=\Lambda_6$ then
\[\vertl \nicefrac{xb=0}{c|x}(M)\vertr=\vertl \nicefrac{xb'=0}{c'|x}(M)\vertr\cdot \vertl \nicefrac{x\lambda=0}{x=0}(M)\vertr.\]
By \ref{redmult}, there is an algorithm which either returns $\Omega:=\bot$, in which case
\[T_R\models \lnot (\vertl\nf{xb=0}{c|x}\vertr\square E\wedge \vertl\nf{xb'=0}{c'|x}\vertr\square'E'),\] or, a set $\Omega\subseteq \N^2\times\{=,\geq\}^2$ such that
\[
\Sigma_7\wedge \vertl\nf{xb=0}{c|x}\vertr\square E\wedge |\nf{xb'=0}{c'|x}|\square' E'
\] is equivalent, with respect to $T_R$, to
\[
\bigvee_{(D_1,D_2,\square_1,\square_2)\in\Omega}\Sigma_7\wedge \vertl \nf{x\lambda=0}{x=0}\vertr\square_1 D_1\wedge \vertl \nf{xb'=0}{c'|x}\vertr\square_2 D_2.
\]
If $\Omega:=\{\bot\}$ then $w\in V$ if and only if $\bot\in \mathbb{V}$ and by definition $\clx_1\bot\leq \clx_1 w$ and $\clx_2\bot<\clx_2w$. Otherwise,
\[\vertl\nicefrac{xb=0}{c|x}\vertr\square E\wedge\vertl\nicefrac{xb'=0}{c'|x}\vertr\square'E'\wedge\chi_{f,g}\wedge\Sigma_7\wedge\Xi\] is equivalent to
\[\bigvee_{(D_1,D_2,\square_1,\square_2)\in\Omega}\vertl \nf{x\lambda=0}{x=0}\vertr\square_1 D_1\wedge \vertl \nf{xb'=0}{c'|x}\vertr\square_2 D_2\wedge\chi_{f,g}\wedge\Sigma_7\wedge\Xi.\]
For each $(D_1,D_2,\square_1,\square_2)\in\Omega$, let
\[w_{(D_1,D_2,\square_1,\square_2)}:=\vertl \nf{x\lambda=0}{x=0}\vertr\square_1 D_1\wedge \vertl \nf{xb'=0}{c'|x}\vertr\square_2 D_2\wedge\chi_{f,g}\wedge\Sigma_7\wedge\Xi.\]
So $w\in V$ if and only if
\[\bigsqcup_{(D_1,D_2,\square_1,\square_2)\in\Omega}w_{(D_1,D_2,\square_1,\square_2)}\in\mathbb{V}.\] For all $(D_1,D_2,\square_1,\square_2)\in\Omega$, $\clx_1 w_{(D_1,D_2,\square_1,\square_2)}=\clx_1w$ and
 $\clx_2w_{(D_1,D_2,\square_1,\square_2)}<\clx_2 w$.

\smallskip

\noindent
\textbf{Case i=8,9,10:} The same argument as for $i=7$ works.

\smallskip
\noindent
We now deal with the general case. Let
\[w:=\bigwedge_{\nf{\phi}{\psi}\in X}\vertl \nf{\phi}{\psi}\vertr=f(\nf{\phi}{\psi})\wedge\bigwedge_{\nf{\phi}{\psi}\in Y}\vertl \nf{\phi}{\psi}\vertr\geq g(\nf{\phi}{\psi})\wedge \Xi\in W\] where $X,Y$ are disjoint finite sets of pp-pairs of the form $d|x/x=0$ and $xb=0/c|x$, $f:X\rightarrow \N_2$, $g:Y\rightarrow \N_2$ and $\Xi$ is an auxiliary sentence.

Suppose that $\nf{xb=0}{c|x},\nf{xb'=0}{c'|x}\in X\cup Y$ are distinct pp-pairs. Let $\Omega\subseteq \Omega_{f,g,10}$ such that $(\overline{f},\overline{g})\in\Omega$ if and only if $\nf{xb=0}{c|x}\in X_5$, $\nf{xb'=0}{c'|x}\in X_6$, $f_5(\nf{xb=0}{c|x})=1$, and $f_6(\nf{xb'=0}{c'|x})=1$.

Then $w\in V$ if and only if
\[\bigsqcup_{(\overline{f},\overline{g})\in\Omega}\bigsqcap_{i=1}^{10}\bigwedge_{\nf{\phi}{\psi}\in X}\vertl \nf{\phi}{\psi}\vertr=f_i(\nf{\phi}{\psi})\wedge\bigwedge_{\nf{\phi}{\psi}\in Y}\vertl\nf{\phi}{\psi}\vertr\geq g_i(\nf{\phi}{\psi})\wedge\Sigma_i\wedge\Xi\in\mathbb{V}.\]
For each $(\overline{f},\overline{g})\in\Omega$ and $1\leq i\leq 10$, let
\[
w_{i,\overline{f},\overline{g}}:=\bigwedge_{\nf{\phi}{\psi}\in X}\vertl \nf{\phi}{\psi}\vertr=f_i(\nf{\phi}{\psi})\wedge\bigwedge_{\nf{\phi}{\psi}\in Y}\vertl\nf{\phi}{\psi}\vertr\geq g_i(\nf{\phi}{\psi})\wedge\Sigma_i\wedge\Xi.
\]
By definition of $\Omega_{f,g,10}$, for each $1\leq i\leq 10$, $\clx_1 w_{i,\overline{f},\overline{g}}\leq \clx_1 w$, $\clx_2 w_{i,\overline{f},\overline{g}}\leq clx_2 w$.

By assumption $\nf{xb=0}{c|x}\in X$ and $f(\nf{xb=0}{c|x})>1$ or $\nf{xb=0}{c|x}\in Y$ and $g(\nf{xb=0}{c|x})>1$. So, since $\nf{xb=0}{c|x}\in X_5$ and $f_5(\nf{xb=0}{c|x})=1$, for each $(\overline{f},\overline{g})\in\Omega_{f,g,10}$, $\clx_2 w_{5,\overline{f},\overline{g}}< \clx_2 w$. Replacing $\nf{xb=0}{c|x}$ by $\nf{xb'=0}{c'|x}$, the same argument gives $\clx_2 w_{6,\overline{f},\overline{g}}< \clx_2 w$.

If $1\leq i\leq 4$ or $7\leq i\leq 10$ then $w_{i,\overline{f},\overline{g}}$ is of the form of the special case considered at the beginning of the proof. Thus we may replace each $w_{i,\overline{f},\overline{g}}$ by a lattice combination of $w'\in W$ such that $\clx_1 w'\leq \clx_1 w_{i,\overline{f},\overline{g}}$, and, either $\clx_2w'<\clx_2w_{i,\overline{f},\overline{g}}$ or $\clx_2w'\leq\clx_2w_{i,\overline{f},\overline{g}}$.
\end{proof}

\begin{proof}[Proof of \ref{1stsynthm}]
By \ref{4.1Denseimproved}, in order to show that $T_R$ is decidable, it is enough to show that there exists an algorithm which given $w\in W$ answers whether $w\in V$ or not. By assumption, the set of $w'\in V$ with $\clx_1 w'\leq 1$ and $\clx_2w'\leq 1$ is recursive. Thus the set of $\underline{w}\in \mathbb{V}$ with $\clx_1 \underline{w}\leq 1$ and $\clx_2 \underline{w}\leq 1$ is recursive.

Since $\N_0$ is artinian as an order, iteratively applying \ref{redclx1} and \ref{redclx2}, gives an algorithm which given $w\in W$ with either $\clx_1w> 1$ or $\clx_2 w> 1$ outputs $\underline{w}\in\mathbb{W}$ such that $\clx_1\underline{w},\clx_2\underline{w}\leq 1$, and, $w\in V$ if and only if $\underline{w}\in \mathbb{V}$.
\end{proof}

\section{Uniserial modules with finite invariants sentences}\label{sectuni}

\noindent
Descriptions of the uniserial (and hence indecomposable pure-injective) modules, $U$, over a valuation domain which have $\nf{\phi}{\psi}(U)$ finite but non-zero for a given pp-pair $\nf{\phi}{\psi}$ are given for valuation domains with dense value groups in \cite{PunPunTof} and for valuation domains with non-dense value groups in \cite{DecVal}. However, we need a uniform description that works for both valuation domains with dense and non-dense value groups.
This is done in Lemmas \ref{needV/dI}, \ref{needI/bcV} and \ref{x=x/c|xfinnonzero} and used in section \ref{Furthersyn} and \ref{notred}.
The rest of the section is about these modules in preparation for sections \ref{Sfinitemodules}, \ref{Sremoving} and \ref{notred}.

\begin{lemma}\label{needV/dI}
Let $V$ be a valuation domain. If $d\in V$ and $U$ is a uniserial $V$-module such that $\nicefrac{d|x}{x=0}(U)=Ud$ is finite and non-zero then $U\cong V/dI$ for some $I\lhd V$ and $Ud\cong V/I$.
\end{lemma}
\begin{proof}For any module $M$, $\nf{x=x}{xd=0}(M)\cong \nf{d|x}{x=0}(M)$. Thus if $\nf{d|x}{x=0}(U)\cong \nf{x=x}{xd=0}(U)$ is finite but not equal to the zero module then, by \ref{unifinitefg}, $U\cong V/J$ for some ideal $J\lhd V$. Since $Ud\neq 0$, $d\notin J$ and therefore $dV\supseteq J$. So there exists $I\lhd V$ such that $dI=J$.
%
%
\end{proof}

\noindent
Note that in the assumptions of the second clause of the next lemma we are not excluding that $I=V$ or consequently that $xb=0/c|x(I/bcV)=0$.

\begin{lemma}\label{needI/bcV}
Let $V$ be a valuation domain and $b,c\in V\backslash\{0\}$. If $U$ is a uniserial $V$-module such that $b,c\notin \text{ann}_V U$ and $\nicefrac{xb=0}{c|x}(U)$ is finite but non-zero then there exists $I\lhd V$ with $b,c\in I$ such that $U\equiv I/bcV$ and $\nicefrac{xb=0}{c|x}(U)\cong V/I$.

 Conversely, if $0\neq I\lhd V$ is an ideal and $b,c\in I\backslash\{0\}$ then $\nicefrac{xb=0}{c|x}(I/bcV)\cong V/I$.
\end{lemma}
\begin{proof} Let $Q$ be the field of fractions of $V$. By \cite[pg 168]{Zie}, for any non-zero uniserial module $U$, there exist submodules $K\subsetneq J\subseteq Q$ such that $U\equiv J/K$. Now $b,c\notin \text{ann}_V J/K$ imply $(K:b)\subsetneq J$ and $cJ\supsetneq K$ respectively. Therefore, since $\nf{xb=0}{c|x}(J/K)\neq 0$, $xb=0/c|x(J/K)=(K:b)/cJ\cong K/cbJ$.  Since $K/cbJ$ is a non-zero finite uniserial module, it has the form $V/I$ for some proper ideal $I\lhd V$. Therefore $K=\lambda V$ for some $\lambda\in Q\backslash \{0\}$ and $\lambda^{-1}cbJ=I$. Thus $J/K\cong I/bcV$ as required.

Let $0\neq I\lhd V$ be an ideal and $b,c\in I\backslash\{0\}$. Then $\nicefrac{xb=0}{c|x}(I/bcV)=cV/cI\cong V/I$.
\end{proof}

\begin{lemma}\label{x=x/c|xfinnonzero}
Let $V$ be a valuation domain and $c\in V$.  If $U$ is a uniserial $V$-module such that $\nicefrac{x=x}{c|x}(U)$ is finite but non-zero then $U\cong V/K$ for some ideal $K\lhd R$. Moreover, if $\nicefrac{x=x}{c|x}(U)$ is finite but non-zero then either $U\cong V/cI$ for some $I\lhd V$ and $V/cV\cong \nicefrac{x=x}{c|x}(U)$, or, $c\in \ann_RU$ and $U\cong \nicefrac{x=x}{c|x}(U)$.
\end{lemma}
\begin{proof}
The first claim is a consequence of \ref{unifinitefg}. If $c\in K$ then $c\in \ann_RV/K$. If $c\notin K$ then $cV\supseteq K$ and hence there exists $I\lhd V$ such that $K=cI$.
\end{proof}

\noindent
We avoid dealing directly with uniserial $V$-modules $U$ such that $\nicefrac{xb=0}{x=0}(U)$ is finite but non-zero by using duality.

\subsection{$(\mfrak{p},I)\models(r,a,\gamma,\delta)$}

\begin{definition}
Let $r,a,\gamma,\delta\in R$, $\mfrak{p}\lhd R$ a prime ideal and $I\lhd R_{\mfrak{p}}$ an ideal. We write $(\mfrak{p},I)\models (r,a,\gamma,\delta)$ if $rR_{\mfrak{p}}\supseteq I$, $a\in I$, $\gamma\notin \mfrak{p}$ and $\delta\notin I^\#$.
\end{definition}
\noindent
The task of this subsection is to show that given an auxiliary sentence $\Xi$ and $\lambda\in R\backslash\{0\}$, we can compute $(r_i,r_ia_i,\gamma_i,\delta_i)$ for $1\leq i\leq n$ such that for any prime ideal $\mfrak{p}\lhd R$ and ideal $I\lhd R_{\mfrak{p}}$, $R_{\mfrak{p}}/\lambda I\models \Xi$ if and only if $(\mfrak{p},I)\models (r_i,r_ia_i,\gamma_i,\delta_i)$ for some $1\leq i\leq n$. However, in \ref{x=x/c|x=xb=0/x=0geq}, we will need the $(r_i,r_ia_i,\gamma_i,\delta_i)$ that we compute to interact well with the duality defined in \ref{Dualsent}.


\begin{remark}\label{alphafeather}
Let $r,a,\gamma,\delta,\alpha\in R$. For all prime ideals $\mfrak{p}\lhd R$ and ideals $I\lhd R_{\mfrak{p}}$, $(\mfrak{p},I)\models (r,a,\gamma,\delta)$ if and only if $(\mfrak{p},I)\models (r,a,\gamma\alpha,\delta\alpha)$ or $(\mfrak{p},I)\models (r,a,\gamma(\alpha-1),\delta(\alpha-1))$.
\end{remark}
\begin{proof}
This is true because for all prime ideals $\mfrak{p}\lhd R$, either $\alpha\notin\mfrak{p}$ or $\alpha-1\notin\mfrak{p}$.
\end{proof}

\begin{lemma}\label{R/Iclosexbcx}\label{R/Iclosed|xx=0}
Let $R$ be a Pr\"ufer domain and $b,c,d\in R$ with $b\neq 0$. Let $\mfrak{p}\lhd R$ be a prime ideal and $I\lhd R_\mfrak{p}$ be an ideal.
\begin{enumerate}
\item Then $\vertl xb=0/c|x(R_{\mfrak{p}}/I)\vertr=1$ if and only if $b\notin I^\#$, $c\notin \mfrak{p}$, $bcR_{\mfrak{p}}\supseteq I$ or $1\in I$.
\item Then $\vertl d|x/x=0(R_{\mfrak{p}}/I)\vertr=1$ if and only if $d\in I$.
\end{enumerate}

\end{lemma}
\begin{proof}
(1) For any ideal $I\lhd R_{\mfrak{p}}$, $\vertl xb=0/c|x(R_{\mfrak{p}}/I)\vertr=1$ if and only if $cR_{\mfrak{p}}+I\supseteq (I:b)$. Since $R_\mfrak{p}$ is a valuation ring, $cR_{\mfrak{p}}+I\supseteq (I:b)$ if and only if $ I\supseteq (I:b)$ or $ cR_{\mfrak{p}}\supseteq (I:b)$. So it is enough to note that $I\supseteq(I:b)$ if and only if $b\notin I^\#$ or $1\in I$, and, $ cR_{\mfrak{p}}\supseteq (I:b)$ if and only if $ bcR_{\mfrak{p}}\supseteq I$ or $c\notin \mfrak{p}$.

\noindent
(2) is obvious.
\end{proof}

%

\begin{lemma}\label{comb1}
Given $(r,a,\gamma,\delta),(r',a',\gamma',\delta')\in R^4$ we can compute $n\in \N$ and $(r_i,a_i,\gamma_i,\delta_i)\in R^4$ for $1\leq i\leq n$ such that \[(\mfrak{p},I)\models (r,a,\gamma,\delta) \text{ and } (\mfrak{p},I)\models (r',a',\gamma',\delta')\] if and only if
\[(\mfrak{p},I)\models (r_i,a_i,\gamma_i,\delta_i)\] for some $1\leq i\leq n$, and,
\[(\mfrak{p},I)\models (r,a,\delta,\gamma) \text{ and } (\mfrak{p},I)\models (r',a',\delta',\gamma')\] if and only if
\[(\mfrak{p},I)\models (r_i,a_i,\delta_i,\gamma_i)\] for some $1\leq i\leq n$.
\end{lemma}
\begin{proof}
Let $\alpha,u_1,u_2,\beta,v_1,v_2\in R$ be such that $r\alpha=r'u_1$, $r'(\alpha-1)=ru_2$, $a\beta=a'v_1$ and $a'(\beta-1)=av_2$. Then
one verifies easily that
\[(\mfrak{p},I)\models (r,a,\gamma,\delta) \text{ and } (\mfrak{p},I)\models (r',a',\gamma',\delta')\] if and only if
 $(\mfrak{p},I)\models (r,a',\alpha\beta\gamma\gamma',\alpha\beta\delta\delta')$, $(\mfrak{p},I)\models (r,a,\alpha(\beta-1)\gamma\gamma',\alpha(\beta-1)\delta\delta')$, $(\mfrak{p},I)\models(r',a',(\alpha-1)\beta\gamma\gamma',(\alpha-1)\beta\delta\delta')$ or $(\mfrak{p},I)\models (r',a,(\alpha-1)(\beta-1)\gamma\gamma',(\alpha-1)(\beta-1)\delta\delta')$.
\end{proof}

\begin{lemma}\label{shiftlambda}
Given $(r,a,\gamma,\delta)\in R^4$ and $\lambda\in R\backslash\{0\}$, we can compute $n\in\N$ and $(r_j,a_j,\gamma_j,\delta_j)$ for $1\leq j\leq n$ such that, for all prime ideals $\mfrak{p}\lhd R$ and ideals $I\lhd R_{\mfrak{p}}$,
\begin{itemize}
\item $(\mfrak{p},\lambda I)\models (r,a,\gamma,\delta)$ if and only if $(\mfrak{p},I)\models (r_j,a_j,\gamma_j,\delta_j)$ for some $1\leq j\leq n$ and
\item $(\mfrak{p},\lambda I)\models (r,a,\delta,\gamma)$ if and only if $(\mfrak{p},I)\models (r_j,a_j,\delta_j,\gamma_j)$ for some $1\leq j\leq n$.
\end{itemize}
\end{lemma}
\begin{proof}
Let $\alpha,u,v,\beta,u',v'\in R$ be such that $r\alpha=\lambda u$, $\lambda(\alpha-1)=rv$, $a\beta=\lambda u'$ and $\lambda(\beta-1)=av'$.
By \ref{alphafeather}, $(\mfrak{p},\lambda I)\models (r,a,\gamma,\delta)$ if and only if $(\mfrak{p},\lambda I)\models (r,a,\alpha\beta\gamma,\alpha\beta\delta)$, $(\mfrak{p},\lambda I)\models (r,a,\alpha(\beta-1)\gamma,\alpha(\beta-1)\delta)$, $(\mfrak{p},\lambda I)\models (r,a,(\alpha-1)\beta\gamma,(\alpha-1)\beta\delta)$ or $(\mfrak{p},\lambda I)\models (r,a,(\alpha-1)(\beta-1)\gamma,(\alpha-1)(\beta-1)\delta)$.

\smallskip

\noindent
It is straightforward to see that:
\begin{itemize}
\item If $\alpha\notin \mfrak{p}$ then $rR_{\mfrak{p}}\supseteq \lambda I$ if and only if $uR_{\mfrak{p}}\supseteq I$.
\item If $\alpha-1\notin\mfrak{p}$ then $rR_{\mfrak{p}}\supseteq \lambda I$.
\item If $\beta\notin \mfrak{p}$ then $a\in\lambda I$ if and only if $u'\in I$.
\item If $\beta-1\notin \mfrak{p}$ then $a\in \lambda I$ if and only if $1\in I$ and $v'\notin\mfrak{p}$ (and hence $v'\notin I^\#$).
\end{itemize}
Recall that, since $\lambda\neq 0$, $(\lambda I)^\#=I^\#$. Therefore, $(\mfrak{p},\lambda I)\models (r,a,\gamma,\delta)$ if and only if $(\mfrak{p},I)\models (u,u',\alpha\beta\gamma,\alpha\beta\delta)$, $(\mfrak{p},I)\models (u,1,\alpha(\beta-1) v'\gamma,\alpha(\beta-1) v'\delta)$, $(\mfrak{p},I)\models (1,u',(\alpha-1)\beta\gamma,(\alpha-1)\beta\delta)$ or $(\mfrak{p},I)\models (1,1,(\alpha-1)(\beta-1)v'\gamma,(\alpha-1)(\beta-1)v'\delta)$.
%
%
\end{proof}

\begin{lemma}\label{cleanup}
Given $(r,a,\gamma,\delta)\in R^4$, we can compute $n\in\N$ and $(r_j,r_ja_j,\gamma_j,\delta_j)$ for $1\leq j\leq n$ such that for all prime ideals $\mfrak{p}\lhd R$ and ideals $I\lhd R_{\mfrak{p}}$,
\begin{itemize}
\item $(\mfrak{p},I)\models (r,a,\gamma,\delta)$ if and only if $(\mfrak{p},I)\models (r_j,r_ja_j,\gamma_j,\delta_j)$ for some $1\leq j\leq n$ and
\item $(\mfrak{p},I)\models (r,a,\delta,\gamma)$ if and only if $(\mfrak{p},I)\models (r_j,r_ja_j,\delta_j,\gamma_j)$ for some $1\leq j\leq n$.
\end{itemize}
\end{lemma}
\begin{proof}
If $a=0$ then $(r,a,\gamma,\delta)=(r,r\cdot 0,\gamma,\delta)$ is already of the required form. So suppose $a\neq 0$.

Let $\alpha,u,v\in R$ be such that $a\alpha=ru$ and $r(\alpha-1)=av$. By \ref{alphafeather}, $(\mfrak{p},I)\models (r,a,\gamma,\delta)$ if and only if $(\mfrak{p},I)\models (r,a,\gamma\alpha,\delta\alpha)$ or $(\mfrak{p},I)\models (r,a,\gamma(\alpha-1),\delta(\alpha-1))$. Since $a\alpha=ru$, $(\mfrak{p},I)\models (r,a,\gamma\alpha,\delta\alpha)$ if and only if $(\mfrak{p},I)\models (r,ru,\gamma\alpha,\delta\alpha)$. Since $r(\alpha-1)=av$, $(\mfrak{p},I)\models (r,a,\gamma(\alpha-1),\delta(\alpha-1))$ if and only if $(\mfrak{p},I)\models(av,a,\gamma(\alpha-1),\delta(\alpha-1))$. Now, since $a\neq 0$, $av R_{\mfrak{p}}\supseteq I$ and $a\in I$ if and only if $v\notin\mfrak{p}$, $aR_{\mfrak{p}}\supseteq I$ and $a\in I$. So $(\mfrak{p},I)\models (av,a,\gamma(\alpha-1),\delta(\alpha-1))$ if and only if $(\mfrak{p},I)\models (a,a,\gamma\alpha v,\delta \alpha v)$.

Therefore $(\mfrak{p},I)\models (r,a,\gamma,\delta)$ if and only if $(\mfrak{p},I)\models (r,rv,\gamma\alpha,\delta\alpha)$ or $(\mfrak{p},I)\models (a,a,\gamma(\alpha-1) v,\delta (\alpha-1)v)$. The same argument shows that $(\mfrak{p},I)\models (r,a,\delta,\gamma)$ if and only if $(\mfrak{p},I)\models (r,rv,\delta\alpha,\gamma\alpha)$ or $(\mfrak{p},I)\models (a,a,\delta (\alpha-1) v,\gamma(\alpha-1) v)$.
\end{proof}

\begin{proposition}\label{auxtoragammadelta}
Given an auxiliary sentence $\Xi$ and $\lambda\in R\backslash\{0\}$, we can compute $n\in\N$ and for $1\leq j\leq n$, $(r_j,r_ja_j,\gamma_j,\delta_j)\in R^4$ such that, for all prime ideals $\mfrak{p}\lhd R$ and ideals $I\lhd R_{\mfrak{p}}$,
\begin{itemize}
\item $(\mfrak{p},I)\models (r_j,r_ja_j,\gamma_j,\delta_j)$ for some $1\leq j\leq n$ if and only if $R_{\mfrak{p}}/\lambda I\models\Xi$, and
\item $(\mfrak{p},I)\models (r_j,r_ja_j,\delta_j,\gamma_j)$ for some $1\leq j\leq n$ if and only if $R_{\mfrak{p}}/\lambda I\models D\Xi$
\end{itemize}
\end{proposition}
\begin{proof}
Let $\Xi$ be the sentence
\[\bigwedge_{i=1}^{l'} \vertl \nicefrac{d_i|x}{x=0}\vertr=1\wedge\bigwedge_{i=l'}^{l} \vertl \nicefrac{xb_{i}=0}{c_{i}|x}\vertr=1. \]
Using \ref{Rp/IAssDiv} and \ref{R/Iclosed|xx=0}, we can compute $n_i\in\N$ for $1\leq i\leq l$ and $s_{ij},b_{ij},g_{ij},h_{ij}\in R$ for $1\leq i\leq l$ and $1\leq j\leq n_i$ such that for all prime ideals $\mfrak{p}\lhd R$ and ideals $I\lhd R_{\mfrak{p}}$,
\begin{itemize}
\item $R_{\mfrak{p}}/I\models \Xi$ if and only if $(\mfrak{p},I)\models \bigwedge_{i=1}^{l}\bigvee_{j=1}^{n_i}(s_{ij},b_{ij},g_{ij},h_{ij})$, and
\item $R_{\mfrak{p}}/I\models D\Xi$ if and only if $(\mfrak{p},I)\models \bigwedge_{i=1}^l\bigvee_{j=1}^{n_i}(s_{ij},b_{ij},h_{ij},g_{ij})$.
\end{itemize}
Therefore,
for all prime ideals $\mfrak{p}\lhd R$ and ideals $I\lhd R_{\mfrak{p}}$,
\begin{itemize}
\item $R_{\mfrak{p}}/I\models \Xi$ if and only if \[(\mfrak{p},I)\models \bigvee_{\substack{\sigma:\{1,\ldots,l\}\rightarrow\N\\\sigma(i)\leq n_i}}\ \bigwedge_{i=1}^l(s_{i\sigma(i)},b_{i\sigma(i)},g_{i\sigma(i)},h_{i\sigma(i)}),\text{ and}\]
\item $R_{\mfrak{p}}/I\models D\Xi$ if and only if \[(\mfrak{p},I)\models \bigvee_{\substack{\sigma:\{1,\ldots,l\}\rightarrow\N \\ \sigma(i)\leq n_i}}\ \bigwedge_{i=1}^l(s_{i\sigma(i)},b_{i\sigma(i)},h_{i\sigma(i)},g_{i\sigma(i)}).\]
\end{itemize}
We can use \ref{comb1}, to replace the conjunction $\bigwedge_{i=1}^l(s_{i\sigma(i)},b_{i\sigma(i)},g_{i\sigma(i)},h_{i\sigma(i)})$, for each $\sigma$, by a disjunction to produce and $(s'_{\sigma k},b'_{\sigma k},g'_{\sigma k},\delta'_{\sigma k})$ for $1\leq k\leq m_\sigma$ such that
\begin{itemize}
\item $R_{\mfrak{p}}/I\models \Xi$ if and only if $(\mfrak{p},I)\models (s'_{\sigma k},b'_{\sigma k},g'_{\sigma k},h'_{\sigma k})$ for some $1\leq k\leq m_{\sigma}$, and
\item $R_{\mfrak{p}}/I\models D\Xi$ if and only if $(\mfrak{p},I)\models (s'_{\sigma k},b'_{\sigma k},h'_{\sigma k},g'_{\sigma k})$ for some $1\leq k\leq m_{\sigma}$.
\end{itemize}
Applying \ref{shiftlambda} to each $(s'_{\sigma k},b'_{\sigma k},h'_{\sigma k},g'_{\sigma k})$, we compute $(r'_j,a'_j,\gamma'_j,\delta'_j)$ for $1\leq j\leq m$ such that
\begin{itemize}
\item $R_{\mfrak{p}}/\lambda I\models \Xi$ if and only if $(\mfrak{p},I)\models  (r'_j,a'_j,\gamma'_j,\delta'_j)$ for some $1\leq j\leq m$, and
\item $R_{\mfrak{p}}/\lambda I\models D\Xi$ if and only if $(\mfrak{p},I)\models (r'_j,a'_j,\delta'_j,\gamma'_j)$ for some $1\leq j\leq m$.
\end{itemize}
Finally, applying \ref{cleanup} to each $(r'_j,a'_j,\gamma'_j,\delta'_j)$, we can compute $(r_i,r_ia_i,\gamma_i,\delta_i)$ for $1\leq i\leq n$ such that
\begin{itemize}
\item $R_{\mfrak{p}}/\lambda I\models \Xi$ if and only if $(\mfrak{p},I)\models  (r_i,r_ia_i,\gamma_i,\delta_i)$ for some $1\leq i\leq n$, and
\item $R_{\mfrak{p}}/\lambda I\models D\Xi$ if and only if $(\mfrak{p},I)\models (r_i,r_ia_i,\delta_i,\gamma_i)$ for some $1\leq i\leq n$. \qedhere
\end{itemize}
\end{proof}

\subsection{Simplification of $\nf{\phi}{\psi}(R_{\mfrak{p}}/\lambda I)$ and $\nf{\phi}{\psi}(I/\lambda R_{\mfrak{p}})$}\

\medskip\noindent
The results of this subsection will be used in sections \ref{Sfinitemodules} and \ref{Sremoving}.
In this subsection, we no longer need to worry about stability under duality.

\begin{remark}\label{ainbIetc}
Let $a,b\in R$ and $\alpha,u,v\in R$ be such that $a\alpha=bu$ and $b(\alpha-1)=av$. For prime ideals $\mfrak{p}\lhd R$ and ideals $I\lhd R_{\mfrak{p}}$,
\begin{itemize}
\item $a\in bI$ if and only if $(\mfrak{p},I)\models (1,u,\alpha,1)$ or $(\mfrak{p},I)\models (1,1,v(\alpha-1),1)$, and,
\item $aR_{\mfrak{p}}\supseteq bI$ if and only if $(\mfrak{p},I)\models (u,0,\alpha,1)$ or $(\mfrak{p},I)\models (1,0,\alpha-1,1)$.
\end{itemize}
\end{remark}

\begin{lemma}\label{R/lambdaIhelp}
Let $R$ be a recursive Pr\"ufer domain and $\lambda\in R\backslash\{0\}$.
\begin{enumerate}[(a)]
\item Given $d\in R$, we can compute finite sets $S_1,S_2,S_3\subseteq R^4$, $\rho:\bigcup_{i=1}^3S_i\rightarrow \{1,2,3\}$ and $s:\bigcup_{i=1}^3S_i\rightarrow R$ such that for all $q\in\bigcup_{i=1}^3S_i$, $q\in S_{\rho(q)}$, and, for all prime ideals $\mfrak{p}\lhd R$ and ideals $I\lhd R_{\mfrak{p}}$, there exists $q\in \bigcup_{i=1}^3S_i$ such that $(\mfrak{p},I)\models q$, and
\[\vertl\nf{d|x}{x=0}(R_{\mfrak{p}}/\lambda I)\vertr:=\left\{
                                                        \begin{array}{ll}
                                                          \vertl R_{\mfrak{p}}/s(q)I\vertr, & \hbox{if $(\mfrak{p},I)\models q$ and $\rho(q)=1$;} \\
                                                          \vertl s(q)R_{\mfrak{p}}/I\vertr, & \hbox{if $(\mfrak{p},I)\models q$ and $\rho(q)=2$;} \\
                                                          1, & \hbox{if $(\mfrak{p},I)\models q$ and $\rho(q)=3$.}
                                                        \end{array}
                                                      \right.
\]
Furthermore, if $(\mfrak{p},I)\models q$ for some $q\in \bigcup_{i=1}^3S_i$ and $\rho(q)=2$ then $s(q)R_{\mfrak{p}}\supseteq I$.
\item Given $b,c\in R$, we can compute finite sets $S_1,\ldots,S_5\subseteq R^4$, $\rho:\bigcup_{i=1}^5S_i\rightarrow \{1,\ldots,5\}$ and $s:\bigcup_{i=1}^5S_i\rightarrow R$ such that for all $q\in\bigcup_{i=1}^5S_i$, $q\in S_{\rho(q)}$, and, for all prime ideals $\mfrak{p}\lhd R$ and ideals $I\lhd R_{\mfrak{p}}$, there exists $q\in \bigcup_{i=1}^5S_i$ such that $(\mfrak{p},I)\models q$, and
\[\vertl\nf{xb=0}{c|x}(R_{\mfrak{p}}/\lambda I)\vertr:=\left\{
                                                        \begin{array}{ll}
                                                          \vertl R_{\mfrak{p}}/\lambda I\vertr, & \hbox{if $(\mfrak{p},I)\models q$ and $\rho(q)=1$;} \\
                                                          \vertl R_{\mfrak{p}}/cR_{\mfrak{p}}\vertr, & \hbox{if $(\mfrak{p},I)\models q$ and $\rho(q)=2$;} \\
                                                          \vertl I/bI\vertr, & \hbox{if $(\mfrak{p},I)\models q$ and $\rho(q)=3$;} \\
                                                          1, & \hbox{if $(\mfrak{p},I)\models q$ and $\rho(q)=4$;} \\
                                                          \vertl I/s(q)R_{\mfrak{p}}\vertr, & \hbox{if $(\mfrak{p},I)\models q$ and $\rho(q)=5$.}
                                                        \end{array}
                                                      \right.
\]
Furthermore, if $(\mfrak{p},I)\models q$ for some $q\in \bigcup_{i=1}^5S_i$ and $\rho(q)=5$ then $s(q)\in I$. Moreover, if $b=0$ then we may assume that $S_3=S_4=S_5=\emptyset$.
\end{enumerate}
\end{lemma}
\begin{proof}
(a) If $d=0$ then $\nf{d|x}{x=0}(N)=0$ for all $R$-modules $N$. So set $S_1=S_2=\emptyset$ and $S_3:=\{(1,0,1,1)\}$.

Suppose $d\neq 0$. Let $\alpha,u,v\in R$ be such that $d\alpha=\lambda u$ and $\lambda(\alpha-1)=dv$.
\begin{itemize}
\item If $(\mfrak{p},I)\models q_1:=(1,0,\alpha-1,1)$ then \[\vertl \nf{d|x}{x=0}(R_{\mfrak{p}}/\lambda I)\vertr=\vertl dR_{\mfrak{p}}/dvI\vertr=\vertl R_{\mfrak{p}}/vI\vertr.\]
\item If $(\mfrak{p},I)\models q_2:=(u,0,\alpha,1)$ then by definition $uR_{\mfrak{p}}\supseteq I$, and,\[\vertl \nf{d|x}{x=0}(R_{\mfrak{p}}/\lambda I)\vertr=\vertl \lambda uR_{\mfrak{p}}/\lambda I\vertr=\vertl uR_{\mfrak{p}}/I\vertr.\]
\item If $(\mfrak{p},I)\models q_3:=(1,u,\alpha,1)$ then $d\in\lambda I$ and hence $\vertl \nf{d|x}{x=0}(R_{\mfrak{p}}/\lambda I)\vertr=1$.
\end{itemize}
For all prime ideals $\mfrak{p}\lhd R$, either $\alpha\notin\mfrak{p}$ or $\alpha-1\notin\mfrak{p}$, and, for all ideal $I\lhd R_{\mfrak{p}}$, either $uR_{\mfrak{p}}\supseteq I$ or $u\in I$. Therefore for all prime ideals $\mfrak{p}\lhd R$ and ideals $I\lhd R_{\mfrak{p}}$, either $(\mfrak{p},I)\models q_1$, $(\mfrak{p},I)\models q_2$ or $(\mfrak{p},I)\models q_3$. So, set $S_i:=\{q_i\}$ for $1\leq i\leq 3$, and, set $s(q_1):=v$, $s(q_2):=u$ and $s(q_3):=1$.

\noindent
(b)
For all prime ideals $\mfrak{p}\lhd R$ and ideals $I\lhd R_{\mfrak{p}}$,
\[\vertl\nf{xb=0}{c|x}(R_{\mfrak{p}}/\lambda I)\vertr=\vertl\frac{(\lambda I:b)+cR_{\mfrak{p}}}{cR_{\mfrak{p}}+\lambda I}\vertr.\]
First suppose $b=0$. If $c\in \lambda I$ then $\vertl\nf{xb=0}{c|x}(R_{\mfrak{p}}/\lambda I)\vertr=\vertl R_{\mfrak{p}}/\lambda I\vertr$, and, if $cR_{\mfrak{p}}\supseteq \lambda I$ then $\vertl\nf{xb=0}{c|x}(R_{\mfrak{p}}/\lambda I)\vertr=\vertl R_{\mfrak{p}}/cR_{\mfrak{p}}\vertr$. We can use \ref{ainbIetc} to compute finite sets $S_1,S_2\subseteq R^4$ such that for all prime ideals $\mfrak{p}\lhd R$ and ideals $I\lhd R_{\mfrak{p}}$, there exists $q\in S_1$ such that $(\mfrak{p},I)\models q$ if and only if $c\in\lambda I$, and, there exists $q\in S_2$ such that $(\mfrak{p},I)\models q$ if and only if $cR_{\mfrak{p}}\supseteq\lambda I$. Set $S_3=S_4=S_5=\emptyset$.

Now suppose $b\neq 0$. If $bR_{\mfrak{p}}\supseteq \lambda I$ then
\[\vertl\frac{(\lambda I:b)+cR_{\mfrak{p}}}{cR_{\mfrak{p}}+\lambda I}\vertr=\vertl\frac{\lambda I+bcR_{\mfrak{p}}}{bcR_{\mfrak{p}}+b\lambda I}\vertr.\]

So
\[\vertl\nf{xb=0}{c|x}(R_{\mfrak{p}}/\lambda I)\vertr:=\left\{
                                                        \begin{array}{ll}
                                                          \vertl R_{\mfrak{p}}/\lambda I\vertr, & \hbox{if $b,c\in\lambda I$;} \\
                                                          \vertl R_{\mfrak{p}}/cR_{\mfrak{p}}\vertr, & \hbox{if $b\in\lambda I$ and $cR_{\mfrak{p}}\supseteq \lambda I$;} \\
                                                          \vertl I/bI\vertr, & \hbox{if $bR_{\mfrak{p}}\supseteq \lambda I$ and $c\in\lambda I$;} \\
                                                          1, & \hbox{if $bcR_{\mfrak{p}}\supseteq \lambda I$;} \\
                                                          \vertl \lambda I/bcR_{\mfrak{p}}\vertr, & \hbox{if $bR_{\mfrak{p}}\supseteq \lambda I$, $cR_{\mfrak{p}}\supseteq \lambda I$ and $bc\in\lambda I$.}
                                                        \end{array}
                                                      \right.\]

Therefore it is enough to compute:
\begin{itemize}
\item $S_1$ such that $b,c\in \lambda I $ if and only if there exists $q\in S_1$ such that $(\mfrak{p},I)\models q$,
\item  $S_2$ such that $b\in \lambda I $ and $cR_{\mfrak{p}}\supseteq \lambda I$ if and only if exists $q\in S_2$ such that $(\mfrak{p},I)\models q$,
\item $S_3$ such that $bR_{\mfrak{p}}\supseteq \lambda I$ and $c\in \lambda I $ and if and only if exists $q\in S_3$ such that $(\mfrak{p},I)\models q$,
\item $S_4$ such that  $bcR_{\mfrak{p}}\supseteq \lambda I$ if and only if there exists $q\in S_4$ such that $(\mfrak{p},I)\models q$, and,
\item $S_5$ and for each $q\in S_5$, $s_q\in R$ such that $bR_{\mfrak{p}}\supseteq \lambda I $, $cR_{\mfrak{p}}\supseteq \lambda I$ and $bc\in \lambda I$ if and only if there exists $q\in S_5$ such that $(\mfrak{p},I)\models q$ and such that, in this situation $\vertl\nf{xb=0}{c|x}(R_{\mfrak{p}}/\lambda I)\vertr=\vertl I/s_qR_{\mfrak{p}}\vertr$.
\end{itemize}
It is easy to compute $S_1,\ldots, S_4$ using \ref{ainbIetc} and \ref{comb1}.

Let $\epsilon,r,s$ be such that $bc\epsilon =\lambda s$ and $\lambda(\epsilon-1)=bcr$. By \ref{ainbIetc}, $bc\in \lambda I$ if and only if $(\mfrak{p},I)\models (1,s,\epsilon,1)$ or $(\mfrak{p},I)\models(1,1,r(\epsilon-1),1)$. If $bR_{\mfrak{p}}\supseteq \lambda I$, $cR_{\mfrak{p}}\supseteq \lambda I$ and $(\mfrak{p},I)\models (1,s,\epsilon,1)$ then \[\vertl\nf{xb=0}{c|x}(R_{\mfrak{p}}/\lambda I)\vertr=\vertl \lambda I/bcR_{\mfrak{p}}\vertr=\vertl I/sR_{\mfrak{p}}\vertr.\]
Use \ref{ainbIetc} and \ref{comb1} to compute $S_5'$ such that $(\mfrak{p},I)\models q$ for some $q\in S_5'$ if and only if $bR_{\mfrak{p}}\supseteq \lambda I$, $cR_{\mfrak{p}}\supseteq \lambda I$ and $(\mfrak{p},I)\models (1,s,\epsilon,1)$.

If $(\mfrak{p},I)\models(1,1,r(\epsilon-1),1)$ then $I=R_{\mfrak{p}}$ and
\[\vertl\nf{xb=0}{c|x}(R_{\mfrak{p}}/\lambda I)\vertr=\vertl \lambda I/bcR_{\mfrak{p}}\vertr=\vertl I/1R_{\mfrak{p}}\vertr.\]
Let $S_5'':=\{(1,1,r(\epsilon-1),1)\}$ and let $S_5:=S_5'\cup S_5''$. Set $s(q):=s$ if $s\in S_5'$ and $s(q):=1$ otherwise. Note that in both cases, by definition, $s(q)\in I$.
\end{proof}

\begin{proposition}\label{sortR/lambdaI}
Let $R$ be a recursive Pr\"ufer domain, $\lambda\in R\backslash\{0\}$ and $Z$ a finite subset of pp-pairs of the form $\nf{xb=0}{c|x}$ and $\nf{d|x}{x=0}$. Let
\[T_Z:=\{\mu:Z\rightarrow\{1,\ldots,5\} \st \text{ for all }  \nf{d|x}{x=0}\in Z \ \ \mu(\nf{d|x}{x=0})\leq 3\}.\]
We can compute $S_Z$ a finite subset of $R^4$, $\rho_Z: S_Z\rightarrow T_Z$ and $s_Z:S_Z\times Z\rightarrow R$ such that for all prime ideals $\mfrak{p}\lhd R$ and ideals $I\lhd R_{\mfrak{p}}$,
\begin{enumerate}[(a)]
\item there exists $q\in S_Z$ such that $(\mfrak{p},I)\models q$,
\item \begin{enumerate}[(1)]
\item if $\rho_Z(q)(\nf{d|x}{x=0})=1$ then $\vertl \nf{d|x}{x=0}(R_{\mfrak{p}}/\lambda I)\vertr=\vertl R_{\mfrak{p}}/s_Z(q,\nf{d|x}{x=0})I\vertr$
\item if $\rho_Z(q)(\nf{d|x}{x=0})=2$ then $\vertl \nf{d|x}{x=0}(R_{\mfrak{p}}/\lambda I)\vertr=\vertl s_Z(q,\nf{d|x}{x=0})R_{\mfrak{p}}/I\vertr$
\item if $\rho_Z(q)(\nf{d|x}{x=0})=3$ then $\vertl \nf{d|x}{x=0}(R_{\mfrak{p}}/\lambda I)\vertr=1$
\end{enumerate}
\item \begin{enumerate}[(1)]
\item if $\rho_Z(q)(\nf{xb=0}{c|x})=1$ then $\vertl \nf{xb=0}{c|x}(R_{\mfrak{p}}/\lambda I)\vertr= \vertl R_{\mfrak{p}}/\lambda I\vertr$
\item if $\rho_Z(q)(\nf{xb=0}{c|x})=2$ then $\vertl \nf{xb=0}{c|x}(R_{\mfrak{p}}/\lambda I)\vertr= \vertl R_{\mfrak{p}}/cR_{\mfrak{p}}\vertr$
\item if $\rho_Z(q)(\nf{xb=0}{c|x})=3$ then $\vertl \nf{xb=0}{c|x}(R_{\mfrak{p}}/\lambda I)\vertr= \vertl I/bI\vertr$
\item if $\rho_Z(q)(\nf{xb=0}{c|x})=4$ then $\vertl \nf{xb=0}{c|x}(R_{\mfrak{p}}/\lambda I)\vertr= 1$
\item if $\rho_Z(q)(\nf{xb=0}{c|x})=5$ then $\vertl \nf{xb=0}{c|x}(R_{\mfrak{p}}/\lambda I)\vertr= \vertl I/s_Z(q,\nf{xb=0}{c|x})R_{\mfrak{p}}\vertr$
\end{enumerate}
\end{enumerate}
Furthermore, for all prime ideals $\mfrak{p}\lhd R$ and ideals $I\lhd R_{\mfrak{p}}$, if $(\mfrak{p},I)\models q$ and $\rho_Z(q)(\nf{d|x}{x=0})=2$ then $s_Z(q,\nf{d|x}{x=0})R_{\mfrak{p}}\supseteq I$, and, if $(\mfrak{p},I)\models q$ and $\rho_Z(q)(\nf{xb=0}{c|x})=5$ then $s_Z(q,\nf{xb=0}{c|x})\in I$. Moreover, if $b=0$ then we may assume $\rho_Z(q)(\nf{xb=0}{c|x})\in\{1,2\}$ for all $q\in S_Z$.
\end{proposition}
\begin{proof}
We prove the proposition iteratively. Let $\lambda\in R\backslash\{0\}$ and $Z$ be a finite set of pp-pairs of the form $\nf{xb=0}{c|x}$ or $\nf{d|x}{x=0}$. If $|Z|=1$ then \ref{R/lambdaIhelp} gives the required result. Suppose that $S_Z,\rho_Z$ and $s_Z$ are as in the statement. We construct $S_{Z'},\rho_{Z'}$ and $s_{Z'}$ for $Z'=Z\cup \{\nf{\phi}{\psi}\}$.

Suppose that $\nf{\phi}{\psi}$ is either of the form $\nf{d|x}{x=0}$ or $\nf{xb=0}{c|x}$ and $\nf{\phi}{\psi}\notin Z$. Let $S_1,\ldots,S_5$, $s:\bigcup_{i=1}^5S_i\rightarrow R$ and $\rho:\bigcup_{i=1}^5S_i\rightarrow \{1,\ldots,5\}$ be as in \ref{R/lambdaIhelp} (1) or (2), as appropriate (if $\nf{\phi}{\psi}$ is of the form $\nf{d|x}{x=0}$ then set $S_4=S_5=\emptyset$). By \ref{comb1}, for each $q\in S_Z$ and $p\in S_i$, we can compute a finite set $S_{q,p,i}\subseteq R^4$ such that for all prime ideals $\mfrak{p}\lhd R$ and ideals $I\lhd R_{\mfrak{p}}$, $(\mfrak{p},I)\models q$ and $(\mfrak{p},I)\models p$ if and only if there exists $q'\in S_{q,p,i}$ such that $(\mfrak{p},I)\models q'$. Let $S_{q,i}:=\bigcup_{p\in S_i}S_{q,p,i}$, let $S_q:=\bigcup_{i=1}^5S_{q,i}$ and let $S_{Z'}:=\bigcup_{q\in S_Z}S_q$.

For all prime ideals $\mfrak{p}\lhd R$ and ideals $I\lhd R_{\mfrak{p}}$, by assumption, there exists $q\in S_Z$ such that $(\mfrak{p},I)\models q$ and, by \ref{R/lambdaIhelp}, there exists $1\leq i\leq 5$, such that $(\mfrak{p},I)\models q'$ for some $q'\in S_i$. Therefore, $(\mfrak{p},I)\models q''$ for some $q''\in S_{q,i}\subseteq S_{Z'}$. So (a) holds for $S_{Z'}$.

Define $\rho_{Z'}:S_{Z'}\rightarrow T_{Z'}$ by setting
\[\rho_{Z'}(q')(\nf{\sigma}{\tau}):=\left\{
                                      \begin{array}{ll}
                                        \min\{\rho_Z(q)(\nf{\sigma}{\tau})\st q'\in S_q\}, & \hbox{if $\nf{\sigma}{\tau}\in Z$;} \\
                                        \min\{1\leq i\leq 5 \st q'\in S_{q,i} \text{ for some }q\in S_Z\}, & \hbox{ if $\nf{\sigma}{\tau}=\nf{\phi}{\psi}$.}
                                      \end{array}
                                    \right.
\]
For $\sigma/\tau\in Z$, set $s_{Z'}(q',\nf{\sigma}{\tau})$ to be $s_{Z}(q,\sigma/\tau)$ for some $q\in S_Z$ where $q'\in S_q$ and $\rho_{Z'}(q')(\nf{\sigma}{\tau})=\rho_{Z}(q)(\nf{\sigma}{\tau})$. For $q'\in S_{Z'}$, set $s_{Z'}(q',\nf{\phi}{\psi})$ to be $s(p)$ for some $p\in S_j$ where $j=\rho_{Z'}(q')(\nf{\phi}{\psi})$ and $q'\in S_{q,p,j}$.

Now, for $\nf{\sigma}{\tau}\in Z$, properties (b) or (c), as appropriate, are inherited from those properties holding for $\rho_Z$ and $s_Z$ and for $\nf{\phi}{\psi}$, properties (b) or (c), as appropriate, are inherited from $\rho$ and $s$.
\end{proof}

\begin{lemma}\label{I/lambdaRhelp}
Let $R$ be a recursive Pr\"ufer domain and $\lambda\in R\backslash\{0\}$.
\begin{enumerate}[(a)]
\item Given $b,c\in R$, we can compute finite sets $S_1,\ldots,S_6\subseteq R^4$, $\rho:\bigcup_{i=1}^6S_i\rightarrow \{1,\ldots,6\}$ and $s:\bigcup_{i=1}^5S_i\rightarrow R$ such that for all $q\in\bigcup_{i=1}^6S_i$, $q\in S_{\rho(q)}$, for all prime ideals $\mfrak{p}\lhd R$ and ideals $I\lhd R_{\mfrak{p}}$ with $\lambda\in I$, there exists $q\in \bigcup_{i=1}^6S_i$ such that $(\mfrak{p},I)\models q$, and,
\[\vertl\nf{xb=0}{c|x}(I/\lambda R_{\mfrak{p}} )\vertr:=\left\{
                                                        \begin{array}{ll}
                                                          \vertl I/\lambda R_{\mfrak{p}}\vertr, & \hbox{if $(\mfrak{p},I)\models q$ and $\rho(q)=1$;} \\
                                                          \vertl I/cI\vertr, & \hbox{if $(\mfrak{p},I)\models q$ and $\rho(q)=2$;} \\
                                                          \vertl R_{\mfrak{p}}/bR_{\mfrak{p}}\vertr, & \hbox{if $(\mfrak{p},I)\models q$ and $\rho(q)=3$;} \\
                                                          1, & \hbox{if $(\mfrak{p},I)\models q$ and $\rho(q)=4$;} \\
                                                          \vertl s(q)R_{\mfrak{p}}/I\vertr, & \hbox{if $(\mfrak{p},I)\models q$ and $\rho(q)=5$;} \\
                                                          \vertl R_{\mfrak{p}}/s(q)I\vertr, & \hbox{if $(\mfrak{p},I)\models q$ and $\rho(q)=6$.}
                                                        \end{array}
                                                      \right.
\]
Furthermore, if $\rho(q)=5$ and $(\mfrak{p},I)\models q$ then $s(q)R_{\mfrak{p}}\supseteq I$.

\item Given $d\in R$, we can compute finite sets $S_1,S_2\subseteq R^4$, $\rho:\bigcup_{i=1}^2S_i\rightarrow \{1,2\}$ and $s:\bigcup_{i=1}^2S_i\rightarrow R$ such that for all $q\in\bigcup_{i=1}^2S_i$, $q\in S_{\rho(q)}$, for all prime ideals $\mfrak{p}\lhd R$ and ideals $I\lhd R_{\mfrak{p}}$ with $\lambda\in I$, there exists $q\in \bigcup_{i=1}^2S_i$ such that $(\mfrak{p},I)\models q$, and,
\[\vertl\nf{d|x}{x=0}(I/\lambda R_{\mfrak{p}})\vertr:=\left\{
                                                        \begin{array}{ll}
                                                          \vertl I/s(q)R_{\mfrak{p}}\vertr, & \hbox{if $(\mfrak{p},I)\models q$ and $\rho(q)=1$;} \\
                                                                                                                   0, & \hbox{if $(\mfrak{p},I)\models q$ and $\rho(q)=2$.}
                                                        \end{array}
                                                      \right.
\]
Furthermore, if $\rho(q)=1$ and $(\mfrak{p},I)\models q$ then $s(q)\in I$.
\end{enumerate}
\end{lemma}
\begin{proof}

\noindent
(a)Let $\alpha,r,s\in R$ be such that $\lambda \alpha=bcr$ and $bc(\alpha-1)=\lambda s$.

\noindent
\textbf{Claim:} If $\lambda\in I$ then
\[
\vertl\nf{xb=0}{c|x}(I/\lambda R_{\mfrak{p}})\vertr:=\left\{
                                                                                                                                        \begin{array}{ll}
                                                                                                                                          \vertl I/\lambda R_{\mfrak{p}}\vertr, & \hbox{if (1) $bI\subseteq \lambda R_{\mfrak{p}}$ and $cI\subseteq \lambda R_{\mfrak{p}}$;} \\
                                                                                                                                          \vertl I/cI\vertr, & \hbox{if (2) $bI\subseteq \lambda R_{\mfrak{p}}$ and $\lambda\in  cI $;} \\
                                                                                                                                         \vertl R_{\mfrak{p}}/bR_{\mfrak{p}}\vertr, & \hbox{if (3) $b\in \lambda I$ and $cI\subseteq \lambda R_{\mfrak{p}}$;} \\
                                                                                                                                          1, & \hbox{if (4) $\lambda\in bI$, $\lambda\in cI$, $r\in I$ and $\alpha\notin \mfrak{p}$;} \\
\vertl rR_{\mfrak{p}}/I\vertr, & \hbox{if (5) $\lambda\in bI$, $\lambda\in cI$, $rR_{\mfrak{p}}\supseteq I$ and $\alpha\notin \mfrak{p}$;} \\
                                                                                                                                          \vertl R_{\mfrak{p}}/s I\vertr, & \hbox{if (6) $\lambda\in bI$, $\lambda\in cI$ and $\alpha-1\notin \mfrak{p}$.}
                                                                                                                                        \end{array}
                                                                                                                                      \right.
\]
Note that \[\vertl\nf{xb=0}{c|x}(I/\lambda R_{\mfrak{p}})\vertr=\vertl \frac{(\lambda R_{\mfrak{p}}:b)\cap I+cI}{cI+\lambda R_{\mfrak{p}}}\vertr.\]
For all $a\in R$, $a\in\ann_RI/\lambda R_{\mfrak{p}}$ if and only if $aI\subseteq \lambda R_{\mfrak{p}}$. Therefore the equalities for conditions $(1)$ and $(2)$ hold. If $\lambda\in bI\subseteq bR_{\mfrak{p}}$ then $b\neq 0$ since $\lambda\neq 0$ and $(\lambda R_{\mfrak{p}}:b)\cap I/\lambda R_{\mfrak{p}}\cong R_{\mfrak{p}}/bR_{\mfrak{p}}$. So the equality for condition $(3)$ holds.

When proving the equalities for $(4)$, $(5)$ and $(6)$, we may assume $b\neq0$ and $c\neq 0$ since $\lambda\neq 0$, $\lambda\in cI$ and $\lambda\in bI$. Moreover, \[\vertl\nf{xb=0}{c|x}(I/\lambda R_{\mfrak{p}})\vertr=\vertl(\lambda R_{\mfrak{p}}:b)+cI/cI\vertr=\vertl\lambda R_{\mfrak{p}}+bcI/bcI\vertr.\]
If $\alpha\notin \mfrak{p}$ and $r\in I$ then $\lambda\in bcI$.  So the equality for condition $(4)$ holds.

Suppose condition (5) holds. Then $\alpha\notin\mfrak{p}$ implies $\lambda R_{\mfrak{p}}=bc r R_{\mfrak{p}}$. Since $rR_{\mfrak{p}}\supseteq I$, $\vertl\nf{xb=0}{c|x}(R_{\mfrak{p}}/\lambda I)\vertr=\vertl rR_{\mfrak{p}}/I\vertr$. So the equality for condition (5) holds.

Suppose condition (6) holds. Since $\alpha-1 \notin \mfrak{p}$, $bcI=\lambda sI$. So $\lambda R_{\mfrak{p}}+bcI/bcI\cong R_{\mfrak{p}}/sI$. So the equality for condition (6) holds. So we have proved the claim.

Given a finite set of conditions of the form $\beta\notin\mfrak{p}$, $a\in \lambda I$ or $aR_{\mfrak{p}}\supseteq \lambda I$, using \ref{comb1} and \ref{ainbIetc}, we can compute a finite set $S\subseteq R^4$ such that $(\mfrak{p},I)$ satisfies these conditions if and only if $(\mfrak{p},I)\models q$ for some $q\in S$. So for each conditions ($i$) for $1\leq i\leq 6$ in the claim, we can compute $S_i\subseteq R^4$ such that $(\mfrak{p},I)$ satisfies $(i)$ and $\lambda\in I$ if and only if there exists $q\in S_i$ with $(\mfrak{p},I)\models q$. Moreover, it is easy to see that, for all prime ideals $\mfrak{p}\lhd R$ and ideals $I\lhd R_{\mfrak{p}}$, there exists $1\leq i\leq 6$ such that $(\mfrak{p},I)$ satisfies $(i)$.  Let $\rho:\bigcup_{i=1}^6S_i\rightarrow \{1,\ldots,6\}$ be such that $q\in S_{\rho(q)}$. Finally, set $s(q)=1$ if $\rho(q)\leq 4$, $s(q)=r$ if $\rho(q)=5$ and $s(q)=s$ if $\rho(q)=6$.

\smallskip
\noindent
(b) The case $d=0$ is done as in \ref{R/lambdaIhelp}. Suppose $d\neq 0$. Let $\alpha,r,s\in R$ be such that $d\alpha=\lambda r$ and $\lambda(\alpha-1)=ds$. If either $\alpha\notin\mfrak{p}$, or, $\alpha-1\notin\mfrak{p}$ and $sR_{\mfrak{p}}\supseteq I$ then $dI\subseteq \lambda R_{\mfrak{p}}$. If $\alpha-1\notin \mfrak{p}$ and $s\in I$ then $ds R_{\mfrak{p}}=\lambda R_{\mfrak{p}}\subseteq dI$. Since $d\neq 0$, $dI/dsR_{\mfrak{p}}\cong I/sR_{\mfrak{p}}$. Therefore
\[\vertl\nf{d|x}{x=0}(I/\lambda R_{\mfrak{p}})\vertr=\vertl\frac{dI+\lambda R_{\mfrak{p}}}{\lambda R_{\mfrak{p}}}\vertr=\left\{
                                                       \begin{array}{ll}
                                                         1, & \hbox{if  $(\mfrak{p},I)\models (1,0,\alpha,1)$;} \\
                                                         1, & \hbox{if  $(\mfrak{p},I)\models (s,0,\alpha-1,1)$;} \\
                                                         \vertl I/sR_{\mfrak{p}}\vertr, & \hbox{if  $(\mfrak{p},I)\models (1,s,\alpha-1,1)$.}
                                                       \end{array}
                                                     \right.\]
It is now clear how to define, $S_1,S_2$, $\rho$ and $s$.
\end{proof}

\begin{proposition}\label{sortI/lambda}
Let $R$ be a recursive Pr\"ufer domain, $\lambda\in R\backslash\{0\}$ and $Z$ a finite subset of pp-pairs of the form $\nf{xb=0}{c|x}$ and $\nf{d|x}{x=0}$. Let \[T_Z:=\{\mu:Z\rightarrow\{1,\ldots,6\}\st \mu(\nf{d|x}{x=0})\leq 2 \}.\]
We can compute $S_Z$ a finite subset of $R^4$, $\rho_Z: S_Z\rightarrow T_Z$ and $s_Z:S_Z\times Z\rightarrow R$ such that, for all prime ideals $\mfrak{p}\lhd R$ and ideals $I\lhd R_{\mfrak{p}}$ with $\lambda\in I$, if $(\mfrak{p},I)\models q$ for some $q\in S_Z$ then $\lambda\in I$, and, if $\lambda\in I$ then
\begin{enumerate}[(a)]
\item there exists $q\in S_Z$ such that $(\mfrak{p},I)\models q$,
\item \begin{enumerate}[(1)]
\item if $\rho_Z(q)(\nf{d|x}{x=0})=1$ then $\vertl \nf{d|x}{x=0}(I/\lambda R_{\mfrak{p}})\vertr=\vertl I/s_Z(q,\nf{d|x}{x=0})R_{\mfrak{m}}\vertr$
\item if $\rho_Z(q)(\nf{d|x}{x=0})=2$ then $\vertl \nf{d|x}{x=0}(I/\lambda R_{\mfrak{p}})\vertr=1$
\end{enumerate}
\item \begin{enumerate}[(1)]
\item if $\rho_Z(q)(\nf{xb=0}{c|x})=1$ then $\vertl\nf{xb=0}{c|x}(I/\lambda R_{\mfrak{p}})\vertr=\vertl I/\lambda R_{\mfrak{p}}\vertr$
\item if $\rho_Z(q)(\nf{xb=0}{c|x})=2$ then $\vertl\nf{xb=0}{c|x}(I/\lambda R_{\mfrak{p}})\vertr=\vertl I/cI\vertr$
\item if $\rho_Z(q)(\nf{xb=0}{c|x})=3$ then $\vertl\nf{xb=0}{c|x}(I/\lambda R_{\mfrak{p}})\vertr=\vertl R_{\mfrak{p}}/bR_{\mfrak{p}}\vertr$
\item if $\rho_Z(q)(\nf{xb=0}{c|x})=4$ then $\vertl\nf{xb=0}{c|x}(I/\lambda R_{\mfrak{p}})\vertr=1$
\item if $\rho_Z(q)(\nf{xb=0}{c|x})=5$ then $\vertl\nf{xb=0}{c|x}(I/\lambda R_{\mfrak{p}})\vertr=\vertl s(q,\nf{xb=0}{c|x})R_{\mfrak{p}}/I\vertr$
\item if $\rho_Z(q)(\nf{xb=0}{c|x})=6$ then $\vertl\nf{xb=0}{c|x}(I/\lambda R_{\mfrak{p}})\vertr=\vertl R_{\mfrak{p}}/s(q,\nf{xb=0}{c|x})I\vertr$
\end{enumerate}
\end{enumerate}
Furthermore, for all prime ideals $\mfrak{p}\lhd R$ and ideals $I\lhd R_{\mfrak{p}}$, if $\rho_Z(q)(\nf{d|x}{x=0})=1$ and $(\mfrak{p},I)\models q$ then $s_Z(q,\nf{d|x}{x=0})\in I$, and, if   $\rho_Z(q)(\nf{xb=0}{c|x})=5$ and $(\mfrak{p},I)\models q$ then $s_Z(q,\nf{d|x}{x=0})R_{\mfrak{p}}\supseteq I$.
\end{proposition}
\begin{proof}
This can be proved similarly to \ref{sortR/lambdaI} by replacing \ref{R/lambdaIhelp} by \ref{I/lambdaRhelp}. The extra condition that $(\mfrak{p},I)\models q$ for some $q\in S_Z$ implies $\lambda\in I$ can be incorporated using \ref{comb1}, since $(\mfrak{p},I)\models (1,\lambda,1,1)$ if and only if $\lambda\in I$.
\end{proof}

\section{Finite modules}\label{Sfinitemodules}
In this section we investigate the consequences of $\EPP(R)$ being recursive, and of
$\EPP(R)$ and the radical relation being recursive. In particular, we show that for a recursive Pr\"ufer domain the theory of $R$-modules of size $n$ is decidable
uniformly in $n$ if and only if $\EPP(R)$ is recursive, \ref{EPPfinite}.

Observe that finite modules over a Pr\"ufer domain $R$ are finite direct sums of modules of the form $R_{\mfrak{p}}/I$ where $\mfrak{p}\lhd R$ is a prime ideal and $I\lhd R_{\mfrak{p}}$ is an ideal. There are many ways of seeing this. If $M$ is finite then $M$ is pure-injective. So, by \ref{sumpielemeq} and \ref{standarduni}, there exist prime ideals $\mfrak{p}_i\lhd R$ and uniserial $R_{\mfrak{p}_i}$-modules $U_i$ such that $M$ is elementary equivalent, and hence isomorphic, to $\oplus_{i=1}^nU_i$. The desired result now follows from \ref{unifinitefg}.

\medskip

\noindent
Let $W$ be the set of tuples $(f,g,\overline{a},\gamma)$ where
\begin{enumerate}[(i)]
\item $f:X\rightarrow \N$ where $X:=X_0\cup \{\star\}$, $X_0$ is a finite subset of $R$ and $\star\notin R$,
\item $g:Y\rightarrow \N$ where $Y$ is a finite subset of $R$,
\item $\overline{a}:=(a_1,\ldots,a_m)$ is a finite tuple of elements of $R$ and $\gamma\in R$.
\end{enumerate}

\noindent
Let $V$ be the set of $(f,g,\overline{a},\gamma)\in W$ such that, for some $h\in\N$ and for $1\leq i\leq h$, there exist a prime ideal $\mfrak{p}_i$ and an ideal $I_i\lhd R_{\mfrak{p}_i}$ such that $a_j\in I_i$ for $1\leq j\leq m$ and $1\leq i\leq h$, $\gamma\notin\mfrak{p}_i$,
\[|\oplus_{i=1}^hR_{\mfrak{p}_i}/I_i|=f(\star),\]
\[|\oplus_{i=1}^hR_{\mfrak{p}_i}/eR_{\mfrak{p}_i}|=f(e) \ \ \text{ \ for \ } e\in X_0\] and
\[|\oplus_{i=1}^hR_{\mfrak{p}_i}/eR_{\mfrak{p}_i}|\geq g(e) \ \ \text{ \ for \ } e\in Y.\]
We write $(\mfrak{p}_i,I_i)_{1\leq i\leq h}\models (f,g,\overline{a},\gamma)$. By convention, $\emptyset\models (f,g,\overline{a},\gamma)$ if $f(e)=1$ for all $e\in X$ and $g(e)=1$ for all $e\in Y$ and in this situation, $(f,g,\overline{a},\gamma)\in V$.

%
%

As in \ref{lattcomb}, $\mathbb{W}$ will denote bounded distributive lattice generated by $W$ and $\mathbb{V}$ will denote the (prime) filter in $\mathbb{W}$ generated by $V$.

Let $W_0$ be the subset of elements of the form $(f,\emptyset,\overline{a},\gamma)$ with $|X_0|\leq 1$ and let $W_1$ be the subset of elements of the form $(f,\emptyset,\overline{a},\gamma)$ where, in both cases, $\emptyset$ denotes the function from the empty set to $\N$. Let $\mathbb{W}_0$, respectively $\mathbb{W}_1$, denote the lattice generated by $W_0\cup\{\top,\bot\}$ in $\mathbb{W}$, respectively the lattice generated by $W_1\cup\{\top,\bot\}$ in $\mathbb{W}$.

\begin{lemma}\label{reddomf}
Let $R$ be a recursive Pr\"ufer domain. There is an algorithm which given $w\in W_1$ returns $\underline{u}\in \mathbb{W}_0$, such that $w\in V$ if and only if  $\underline{u}\in \mathbb{V}$.
\end{lemma}
\begin{proof}
Define $\text{clx}(f,\emptyset,\overline{a},\gamma):=(|X_0|,\prod_{e\in X_0}f(e))$ and order $\N_0\times\N$ lexicographically.  For any $w\in W_1$ with $|X_0|>1$, we show how to compute $\underline{u}\in \mathbb{W}_1$, such that $\clx\underline{u}<\clx w$ and $w\in V$ if and only if $\underline{u}\in\mathbb{V}$. Since the lexicographic order on $\N_0\times\N$ is artinian, this is sufficient to prove the lemma.

Take $e_1,e_2\in X_0$ non-equal. Let $\alpha,r,s\in R$ be such that $e_1\alpha=e_2r$ and $e_2(\alpha-1)=e_1s$. Let $\Omega$ be the set of pairs of functions $(f_1,f_2)$ such that $f_1:X\cup \{r\}\rightarrow \N$, $f_2:X\cup \{s\}\rightarrow \N$ and
\begin{enumerate}
\item $f_1(e)f_2(e)=f(e)$ for all $e\in X$,
\item $f_1(e_1)=f_1(e_2)f(r)$, and
\item $f_2(e_2)=f_2(e_1)f(s)$.
\end{enumerate}
Let $\underline{u}$ be
\[\bigsqcup_{\substack{(f_1,f_2)\in \Omega\\ f_1(e_2)\neq 1, f_2(e_1)\neq 1  }}(f_1|_{(X\cup\{r\})\backslash\{e_1\}},\emptyset,\overline{a},\gamma\alpha)\sqcap (f_2|_{(X\cup\{s\})\backslash\{e_2\}},\emptyset,\overline{a},\gamma(\alpha-1))\]
\[\sqcup\bigsqcup_{\substack{(f_1,f_2)\in \Omega\\ f_1(e_2)= 1, f_2(e_1)\neq 1  }}(f_1|_{X\backslash\{e_2\}},\emptyset,\overline{a},e_2\gamma\alpha)\sqcap (f_2|_{(X\cup\{s\})\backslash\{e_2\}},\emptyset,\overline{a},\gamma(\alpha-1))\]
\[\sqcup\bigsqcup_{\substack{(f_1,f_2)\in \Omega\\ f_1(e_2)\neq 1, f_2(e_1)= 1  }}(f_1|_{(X\cup\{r\})\backslash\{e_1\}},\emptyset,\overline{a},\gamma\alpha)\sqcap (f_2|_{X\backslash\{e_1\}},\emptyset,\overline{a},e_1\gamma(\alpha-1))\]
\[\sqcup\bigsqcup_{\substack{(f_1,f_2)\in \Omega\\ f_1(e_2)= 1, f_2(e_1)= 1  }}(f_1|_{X\backslash\{e_2\}},\emptyset,\overline{a},e_2\gamma\alpha)\sqcap (f_2|_{X\backslash\{e_1\}},\emptyset,\overline{a},e_1\gamma(\alpha-1)).\]

\noindent
\textbf{Claim:} $w\in V$ if and only if $\underline{u}\in \mathbb{V}$

\noindent
For all prime ideals $\mfrak{p}\lhd R$, $\alpha\notin\mfrak{p}$ implies $e_1R_{\mfrak{p}}=e_1\alpha R_{\mfrak{p}}=e_2rR_{\mfrak{p}}$ and $\alpha-1\notin \mfrak{p}$ implies $e_2R_{\mfrak{p}}=e_2(\alpha-1)R_{\mfrak{p}}=e_1sR_{\mfrak{p}}$. So, $\alpha\notin\mfrak{p}$ implies
\[\vertl R_{\mfrak{p}}/e_1R_{\mfrak{p}}\vertr=\vertl R_{\mfrak{p}}/e_2R_{\mfrak{p}}\vertr\cdot \vertl R_{\mfrak{p}}/rR_{\mfrak{p}}\vertr\] and
$\alpha-1\notin\mfrak{p}$ implies
\[\vertl R_{\mfrak{p}}/e_2R_{\mfrak{p}}\vertr=\vertl R_{\mfrak{p}}/e_1R_{\mfrak{p}}\vertr\cdot \vertl R_{\mfrak{p}}/sR_{\mfrak{p}}\vertr.\]
Suppose that $\mfrak{p}_i\lhd R$ is a prime ideal and $I_i\lhd R_{\mfrak{p}_i}$ is an ideal for $1\leq i\leq h$ such that $(\mfrak{p}_i,I_i)_{1\leq i\leq h}\models (f,\emptyset,\overline{a},\gamma)$. By reordering, we may assume that $\alpha\notin \mfrak{p}_i$ for $1\leq i\leq h'$ and $\alpha-1\notin \mfrak{p}_i$ for $h'+1\leq i\leq h$. Let $f_1(\star):= |\oplus_{i=1}^{h'}R_{\mfrak{p}_i}/I_i|$ and $f_2(\star):=|\oplus_{i=h'+1}^{h}R_{\mfrak{p}_i}/I_i|$. For each $e\in X_0\cup \{r\}$, let $f_1(e)=|\oplus_{i=1}^{h'}R_{\mfrak{p}_i}/eR_{\mfrak{p}_i}|$ and for each $e\in X_0\cup\{s\}$, let $f_2(e)=|\oplus_{i=h'+1}^{h}R_{\mfrak{p}_i}/eR_{\mfrak{p}_i}|$. It follows from the first paragraph of the proof of this claim that $(f_1,f_2)\in \Omega$.

If $f_1(e_2)\neq 1$ then \[(\mfrak{p}_i,I_i)_{1\leq i\leq h'}\models (f_1|_{(X\cup\{r\})\backslash\{e_1\}},\emptyset,\overline{a},\gamma\alpha)\] and if $f_2(e_1)\neq 1$ then \[(\mfrak{p}_i,I_i)_{h'+1\leq i\leq h}\models (f_2|_{(X\cup\{s\})\backslash\{e_2\}},\emptyset,\overline{a},\gamma(\alpha-1)).\]
By definition, if $f_1(e_2)=1$, then $e_2\notin \mfrak{p}_i$ for $1\leq i\leq h'$ and if $f_2(e_1)=1$ then $e_1\notin \mfrak{p}_i$ for $h'+1\leq i\leq h$. So if $f_1(e_2)=1$ then \[(\mfrak{p}_i,I_i)_{1\leq i\leq h'}\models (f_1|_{X\backslash\{e_2\}},\emptyset,\overline{a},e_2\gamma\alpha)\] and if $f_2(e_1)=1$ then \[(\mfrak{p}_i,I_i)_{h'+1\leq i\leq h}\models(f_2|_{X\backslash\{e_1\}},\emptyset,\overline{a},e_1\gamma(\alpha-1)).\] So we have shown that if $w\in V$ then one of the components of the join defining $\underline{u}$ is in $\mathbb{V}$ and hence $\underline{u}\in\mathbb{V}$.

Conversely, take $(f_1,f_2)\in\Omega$. Suppose $f_1(e_2)\neq 1$ and \[(\mfrak{p}_i,I_i)_{1\leq i\leq h'}\models (f_1|_{(X\cup\{r\})\backslash\{e_1\}},\emptyset,\overline{a},\gamma\alpha).\]
Since $\alpha\notin\mfrak{p}_i$ for $1\leq i\leq h'$,
\[\vertl\oplus_{i=1}^{h'}R_{\mfrak{p}_i}/e_1R_{\mfrak{p}_i}\vertr=\vertl\oplus_{i=1}^{h'}R_{\mfrak{p}_i}/e_2R_{\mfrak{p}_i}\vertr\cdot\vertl\oplus_{i=1}^{h'}R_{\mfrak{p}_i}/rR_{\mfrak{p}_i}\vertr=f_1(e_2)\cdot f_1(r)=f_1(e_1)\] by definition of $\Omega$. So
$(\mfrak{p}_i,I_i)_{1\leq i\leq h'}\models (f_1,\emptyset,\overline{a},\gamma)$.

Suppose that $f_2(e_1)= 1$ and \[(\mfrak{q}_i,J_i)_{1\leq i\leq h''}\models (f_2|_{X\backslash\{e_1\}},\emptyset,\overline{a},e_1\gamma(\alpha-1)).\] Then
\[(\mfrak{q}_i,J_i)_{1\leq i\leq h''}\models (f_2,\emptyset,\overline{a},\gamma)\] because $e_1\notin\mfrak{q}_i$ for $1\leq i\leq h''$ implies
\[\vertl\oplus_{i=1}^{h''} R_{\mfrak{q}_i}/e_1R_{\mfrak{q}_i}\vertr=1=f_1(e_1).\]
So, setting $\mfrak{p}_i:=\mfrak{q}_{i-h'}$ and $I_i:=J_{i-h'}$ for $h'+1\leq i\leq h'+h''=h$,
\[(\mfrak{p}_i,I_i)_{1\leq i\leq h}\models (f,\emptyset,\overline{a},\gamma)\] because $f_1(e)f_2(e)=f(e)$ for all $e\in X$.
We leave the case $f_1(e_2)\neq 1$ and $f_2(e_1)\neq 1$, the case $f_1(e_2)= 1$ and $f_2(e_1)\neq 1$ and the case $f_1(e_2)=f_2(e_2)=1$ to the reader.

\smallskip
\noindent
\textbf{Claim:} $\clx \underline{u}<\clx w$

\noindent
We show that each of the components, $u'$, of the lattice combination defining $\underline{u}$ have $\clx u'<\clx w$. We only consider the components involving $f_1$; the result for those involving $f_2$ follows similarly.

If $f_1(e_2)= 1$ then $\clx(f_1|_{X\backslash\{e_2\}},\emptyset,\overline{a},e_2\gamma\alpha)<\clx w$ since $|X_0|>|X_0\backslash\{e_2\}|$. If $f_1(e_2)> 1$ then \[\frac{f_1(r)}{f(e_1)}=\frac{f_1(e_2)f_1(r)}{f_1(e_2)f(e_1)}=\frac{f_1(e_1)}{f_1(e_2)f(e_1)}<1\] since $f_1(e_1)/f(e_1)\leq 1$.
So
\[f_1(r)\cdot\prod_{x\in X_0\backslash\{e_1\}} f_1(x)\leq f_1(r)\cdot\prod_{x\in X_0\backslash\{e_1\}}f(x)=\frac{f_1(r)}{f(e_1)}\cdot\prod_{x\in X_0}f(x)<\prod_{x\in X_0}f(x).\]
Therefore $\text{clx}(f_1|_{(X\cup\{r\})\backslash\{e_1\}},\emptyset,\overline{a},\gamma\alpha)<\text{clx}w$.
\end{proof}

\begin{lemma}\label{USEEPP}
Let $R$ be a recursive Pr\"ufer domain. If $\EPP(R)$ is recursive then there is an algorithm which given $w\in W_1$ answers whether $w\in V$ or not.
\end{lemma}
\begin{proof}
Let $w=(f,\emptyset,\overline{a},\gamma)$ and suppose that $X_0=\{e\}$ i.e. $w\in W_0$. Let $P$ be the set of prime divisors of $f(\star)\cdot f(e)$. If $P$ is empty then $f(\star)=f(e)=1$ and $w\in V$. Otherwise, for each $p\in P$, let $n_p\in \N_0$ and $m_p\in\N_0$ be such that $f(\star)=\prod_{p\in P}p^{n_p}$ and $f(e)=\prod_{p\in P}p^{m_p}$. For any prime ideal $\mfrak{p}\lhd R$, $e\in R$ and ideal $I\lhd R_{\mfrak{p}}$,  if $|R_{\mfrak{p}}/eR_{\mfrak{p}}|$ (respectively $|R_{\mfrak{p}}/I|$) is finite then it is a power of $|R/\mfrak{p}|$. Thus a prime power. Hence $w\in V$ if and only if $(p,n_p;\overline{a};\gamma;e,m_p)\in \EPP_{*}(R)$ for all $p\in P$.

If $w\in W_1$ then by \ref{reddomf}, we can compute $\underline{w}\in\mathbb{W}_0$ such that $w\in V$ if and only if $\underline{w}\in \mathbb{V}$. Therefore, by the previous paragraph there is an algorithm which, given $w\in W_1$ answers whether $w\in V$ or not.
\end{proof}

\begin{cor}\label{fincond}
Let $R$ be a recursive Pr\"ufer domain with $\EPP(R)$ recursive. There is an algorithm which given $r,a,\gamma,\delta\in R$, $n\in\N$, $e_j\in R$ for $1\leq j\leq n$, $N\in\N$ and $N_j\in \N$ for $1\leq j\leq n$, answers whether there exist $h\in\N_0$, prime ideals $\mfrak{p}_i\lhd R$ and ideals $I_i\lhd R_{\mfrak{p}_i}$ for $1\leq i\leq h$ such that
\begin{enumerate}
\item  $(\mfrak{p}_i,I_i)\models (r,a,\gamma,\delta)$ for $1\leq i\leq h$,
\item  $\vertl\oplus_{i=1}^h R_{\mfrak{p}_i}/I_i\vertr=N$, and
\item $\vertl \oplus_{i=1}^h R_{\mfrak{p}_i}/e_jR_{\mfrak{p}_i}\vertr=N_j$ for $1\leq j\leq n$.
\end{enumerate}
\end{cor}
\begin{proof}
Applying \ref{cleanup}, we may reduce to the case where $a=ra'$. We may also assume $r\neq 0$ since $(\mfrak{p},I)\models (0,b',\gamma,\delta)$ implies $I=0$ and hence $\vertl R_{\mfrak{p}}/I\vertr=\vertl R_\mfrak{p}\vertr$ which is infinite.

For any prime ideal $\mfrak{p}\lhd R$ and ideal $I\lhd R_{\mfrak{p}}$, $(\mfrak{p},I)\models (r,ra',\gamma,\delta)$ if and only if there exists $J\lhd R_{\mfrak{p}}$ such that $I=rJ$ and $(\mfrak{p},J)\models (1,a',\gamma,\delta)$. Note that, because $R$ is a domain,  $\vertl R_{\mfrak{p}}/rJ\vertr=\vertl R_{\mfrak{p}}/J\vertr\cdot \vertl R_{\mfrak{p}}/rR_{\mfrak{p}}\vertr$.

Therefore, there exist $h\in\N_0$, prime ideals $\mfrak{p}_i\lhd R$ and ideals $I_i\lhd R_{\mfrak{p}_i}$ for $1\leq i\leq h$ satisfying $(1),(2)$ and $(3)$ if and only if there exist $N',N''\in \N$ with $N'\cdot N''=N$, $h\in\N_0$, prime ideals $\mfrak{p}_i\lhd R$ and ideals $J_i\lhd R_{\mfrak{p}_i}$ for $1\leq i\leq h$ such that
\begin{enumerate}[(a)]
\item  $(\mfrak{p}_i,J_i)\models (1,a',\gamma,\delta)$ for $1\leq i\leq h$,
\item  $\vertl\oplus_{i=1}^h R_{\mfrak{p}_i}/J_i\vertr=N'$, and
\item $\vertl \oplus_{i=1}^h R_{\mfrak{p}_i}/rR_{\mfrak{p}_i}\vertr=N''$, $\vertl \oplus_{i=1}^h R_{\mfrak{p}_i}/e_jR_{\mfrak{p}_i}\vertr=N_j$ for $1\leq j\leq n$.
\end{enumerate}

By \ref{USEEPP}, there is an algorithm which answers whether there exist $h\in\N_0$, prime ideals $\mfrak{p}_i\lhd R$ and ideals $J_i\lhd R_{\mfrak{p}_i}$ for $1\leq i\leq h$ satisfying (a),(b) and (c).
\end{proof}

\begin{proposition}\label{PIconditions}
Let $R$ be a recursive Pr\"ufer domain such that $\EPP(R)$ and the radical relation are recursive. There is an algorithm which, given $(f,g,\overline{a},\gamma)\in W$, answers whether there exist $h\in\N_0$, prime ideals $\mfrak{p}_i$ and ideals $I_i\lhd R_{\mfrak{p}_i}$ for $1\leq i\leq h$ such that $(\mfrak{p}_i,I_i)_{1\leq i\leq h}\models (f,g,\overline{a},\gamma)$.
\end{proposition}
\begin{proof}
Suppose $(f,g,\overline{a},\gamma)\in W$.

\smallskip

\noindent
\textbf{Case $Y=\emptyset$:} This is \ref{USEEPP}.

\smallskip

\noindent
\textbf{Case $|Y|=1$:} Suppose that $X_0:=\{e_1,\ldots,e_n\}$ and $Y:=\{e\}$.

If $\gamma \prod_{j=1}^n e_j\notin \rad eR$ then there exists a prime ideal $\mfrak{p}$ such that $e_j\notin\mfrak{p}$ for $1\leq j\leq n$, $\gamma\notin \mfrak{p}$ and $e\in\mfrak{p}$. Then $|R_{\mfrak{p}}/e_jR_{\mfrak{p}}|=1$ for $1\leq j\leq n$ and $|R_{\mfrak{p}}/eR_{\mfrak{p}}|>1$. So $(\mfrak{p}_i,I_i)_{i=1}^h\models (f,\emptyset,\overline{a},\gamma)$ if and only if $(\mfrak{p}_i,I_i)_{i=1}^{h+g(e)}\models(f,g,\overline{a},\gamma)$ where $\mfrak{p}_i:=\mfrak{p}$ and $I_i:=R_{\mfrak{p}_i}$ for $h<i\leq g(e)$. So $(f,g,\overline{a},\gamma)\in V$ if and only if $(f,\emptyset,\overline{a},\gamma)\in V$. So we are now in the situation of case $Y=\emptyset$.

If $\gamma\prod_{j=1}^n e_j\in \rad eR$ then there exist $l\in\N$ and $r\in R$ such that $(\gamma\prod_{j=1}^n e_j)^l=er$. Thus, for all prime ideals $\mfrak{p}$ with $\gamma\notin \mfrak{p}$, $|R_{\mfrak{p}}/eR_{\mfrak{p}}|\leq \prod_{j=1}^n|R_{\mfrak{p}}/e_jR_{\mfrak{p}}|^l$. Therefore $(\mfrak{p},I)\models (f,g,\overline{a},\gamma)$ if and only if there exists $f':X\cup\{e\}\rightarrow \N$ where $f'(x)=f(x)$ for all $x\in X$ and $g(e)\leq f'(e)\leq \prod_{j=1}^nf(e_j)^l$ and $(\mfrak{p},I)\models (f',\emptyset,\overline{a},\gamma)$. Since the set of $f'$ is finite and computable, we are now in the situation of case $Y=\emptyset$.

\smallskip

\noindent
\textbf{Case $|Y|\geq 2$:} We show how to reduce to the situation where $|Y|\leq 1$. By \ref{reducegconstant}, we may assume that $g$ is a constant function. Take $e_1,e_2\in Y$ non-equal. Let $\alpha,r,s\in R$ be such that $e_1\alpha=e_2r$ and $e_2(\alpha-1)=e_1s$. Since for all prime ideals $\mfrak{p}\lhd R$, either $\alpha\notin\mfrak{p}$ or $\alpha-1\notin\mfrak{p}$, by \ref{feathering}, $(f,g,\overline{a},\gamma)\in V$ if and only if
\[\bigsqcup_{(f_1,f_2,g_1,g_2)\in \Omega_{f,g,2}}(f_1,g_1,\overline{a},\alpha\gamma)\sqcap (f_2,g_2,\overline{a},(\alpha-1)\gamma)\in\mathbb{V}.\]
Note that if $g$ is constant then $g_1$ and $g_2$ are constant for all $(f_1,f_2,g_1,g_2)\in \Omega_{f,g}$.

For each $(f_1,f_2,g_1,g_2)\in \Omega_{f,g}$, either $|Y_1|<|Y|$ or $e_1,e_2\in Y_1$. In the first case we are done. In the second, $(f_1,g_1,\overline{a},\alpha\gamma)\in V$ if and only if $(f_1,g_1|_{Y_1\backslash\{e_1\}},\overline{a},\alpha\gamma)\in V$. This is because, for all prime ideals $\mfrak{p}$ with $\alpha\notin\mfrak{p}$, $|R_{\mfrak{p}}/e_1R_{\mfrak{p}}|\geq |R_{\mfrak{p}}/e_2R_{\mfrak{p}}|$ and hence if $|R_{\mfrak{p}}/e_2R_{\mfrak{p}}|\geq g_1(e_2)$ then $|R_{\mfrak{p}}/e_1R_{\mfrak{p}}|\geq g_1(e_2)=g_1(e_1)$. So we may replace $(f_1,g_1,\overline{a},\alpha\gamma)$ by $(f_1,g_1|_{Y_1\backslash\{e_1\}},\overline{a},\alpha\gamma)$. A similar argument shows that either $|Y_2|<|Y|$ or we can replace $(f_2,g_2,\overline{a},(\alpha-1)\gamma)$ by $(f_2,g_2|_{Y_1\backslash\{e_2\}},\overline{a},(\alpha-1)\gamma)$.
\end{proof}

\begin{cor}\label{fincond2}
Let $R$ be a recursive Pr\"ufer domain with the radical relation and $\EPP(R)$ recursive. There is an algorithm which given $r,a,\gamma,\delta\in R$, $n,n'\in\N$, $e_j\in R$ for $1\leq j\leq n$, $e_j'\in R$ for $1\leq j\leq n'$, $N\in\N$, $N_j\in \N$ for $1\leq j\leq n$ and $N_j'\in \N$ for $1\leq j\leq n'$, answers whether there exist $h\in\N_0$, prime ideals $\mfrak{p}_i\lhd R$ and ideals $I_i\lhd R_{\mfrak{p}_i}$ for $1\leq i\leq h$ such that
\begin{enumerate}
\item  $(\mfrak{p}_i,I_i)\models (r,a,\gamma,\delta)$ for $1\leq i\leq h$,
\item  $\vertl\oplus_{i=1}^h R_{\mfrak{p}_i}/I_i\vertr=N$,
\item $\vertl \oplus_{i=1}^h R_{\mfrak{p}_i}/e_jR_{\mfrak{p}_i}\vertr=N_j$ for $1\leq j\leq n$, and
\item $\vertl \oplus_{i=1}^h R_{\mfrak{p}_i}/e_jR_{\mfrak{p}_i}\vertr\geq N_j'$ for $1\leq j\leq n'$.
\end{enumerate}
\end{cor}
\begin{proof}
The proof is as in \ref{fincond} but we use $\ref{PIconditions}$ in place of $\ref{USEEPP}$.
\end{proof}

\begin{theorem}\label{EPPfinite}
Let $R$ be a recursive Pr\"ufer domain.  The theory of $R$-modules of size $n$ is decidable uniformly
in $n$ if and only if $\EPP(R)$ is recursive.
\end{theorem}
\begin{proof}
The forward direction is \ref{decfinimpEPP}.

Standard arguments using the Baur-Monk theorem and the fact that $T_R$ is recursively axiomatisable imply that the theory of $R$-modules of size $n$ is decidable uniformly in $n$ if and only if there is an algorithm which, given $N\in\N$, pairs of pp-$1$-formulae $\phi_i/\psi_i$ and $A_i\in\N$ for $1\leq i\leq m$, answers whether there exists $M\in\Mod\text{-}R$ such that satisfying
\[\tag{$\star$}\label{temp2} \vertl \nf{x=x}{x=0}\vertr=N\wedge\bigwedge_{i=1}^m\vertl\nf{\phi_i}{\psi_i}\vertr=A_i.\]
Unfortunately, we can't directly apply the statement of \cite[4.1]{Decdense} to reduce to the case that the pp-pairs $\nf{\phi_i}{\psi_i}$ are of the form $\nf{d|x}{x=0}$ or $\nf{xb=0}{c|x}$. However, the proof of \cite[4.1]{Decdense} can be easily modified to allow us to do this. Roughly speaking, starting with a sentence $\chi$ of the form
\[\tag{$\dagger$}\label{temp1}\bigwedge_{i=1}^l \vertl\nf{\phi_i}{\psi_i}\vertr= A_i\wedge \bigwedge^m_{i=l+1}\vertl\nf{\phi_i}{\psi_i}\vertr\geq A_i\] where $A_i\in\N$ and $\nf{\phi_i}{\psi_i}$ is an arbitrary pp-pair for $1\leq i\leq m$, each step of the proof of \cite[4.1]{Decdense} produces a finite set $S$ of finite tuples of sentences $(\chi_1,\ldots,\chi_k)$ as in (\ref{temp1}), but with the form of the pp-pairs involved progressively improved, such that there is an $R$-module satisfying $\chi$ if and only there exist $(\chi_1,\ldots,\chi_k)\in S$ and $R$-modules $M_i\in\Mod\text{-}R$ with $M_i\models \chi_i$ for $1\leq i\leq k$. In order to adapt the proof to our situation, the reader just needs to note that, for each of the steps of the proof of \cite[4.1]{Decdense}, if the initial sentence $\chi$ is as in (\ref{temp2}), then the sentences $\chi_i$ in the finite tuples $(\chi_1,\ldots,\chi_k)\in S$ are of the same form as in (\ref{temp2}).

Let $N\in\N$, $Z$ be a finite set of pp-pairs of the form $\nf{d|x}{x=0}$ and $\nf{xb=0}{c|x}$, and $f:Z\rightarrow \N$. Let $\chi$ be the sentence \[\vertl \nf{x=x}{x=0}\vertr=N\wedge\bigwedge_{\nf{\phi}{\psi}\in Z}\vertl\nf{\phi}{\psi}\vertr=f(\nf{\phi}{\psi}).\]
Recall that every finite module over a Pr\"ufer domain is isomorphic to a direct sum of modules of the form $R_{\mfrak{p}}/I$ for some prime ideal $\mfrak{p}\lhd R$ and ideal $I\lhd R_{\mfrak{p}}$.

Let $S_Z$, $\rho_Z$ and $s_Z$ be as in \ref{sortR/lambdaI} with $\lambda:=1$. Enumerate $S_Z:=\{q_1,\ldots,q_m\}$ and let $q_i:=(r_i,r_ia_i,\gamma_i,\delta_i)$.  By definition, for all prime ideals $\mfrak{p}\lhd R$ and ideals $I\lhd R_{\mfrak{p}}$ there exists $1\leq i\leq m$ such that $(\mfrak{p},I)\models q_i$. Therefore, there exists $M\in \Mod\text{-}R$ satisfying $\chi$ if and only if there exist $N_i\in\N$ and $f_i:Z\rightarrow \N$ for $1\leq i\leq m$ such that $N=\prod_{i=1}^mN_i$, for all $\nf{\phi}{\psi}\in Z$, $f(\nf{\phi}{\psi})=\prod_{i=1}^mf_i(\nf{\phi}{\psi})$ and there exist $h_i\in\N_0$, prime ideals $\mfrak{p}_{ij}\lhd R$ and ideals $I_{ij}\lhd R_{\mfrak{p}_{ij}}$ for $1\leq j\leq h_i$ such that
\begin{itemize}
\item [(a)${}_i$] $(\mfrak{p}_{ij},I_{ij})\models q_i$ for $1\leq j\leq h_i$ and
\item [(b)${}_i$] \[\bigoplus_{j=1}^{h_j}R_{\mfrak{p}_{ij}}/I_{ij}\models \vertl \nicefrac{x=x}{x=0}\vertr=N_i\wedge\bigwedge_{\nf{\phi}{\psi}\in Z}\vertl \nicefrac{\phi}{\psi}\vertr=f_i(\nf{\phi}{\psi}).\]
\end{itemize}

Fix $1\leq i\leq m$, $N_i\in\N$ and $f_i$ as above. If $r_i=0$ and $(\mfrak{p},I)\models (r_i,r_ia_i,\gamma_i,\delta_i)$ then $R_{\mfrak{p}}/I=R_\mfrak{p}$, which is not finite. Hence, (a)${}_i$ and (b)${}_i$ holds if and only if $h_i=0$, $N_i=1$ and $f_i(\nf{\phi}{\psi})=1$ for all $\nf{\phi}{\psi}\in Z$. So, we may assume that $r_i\neq 0$. Note, by \ref{propIhash}\ref{finmodulespr}, if $R_{\mfrak{p}}/I$ is finite then $I^\#=\mfrak{p}$. So, $(\mfrak{p},I)\models (r_i,r_ia_i,\gamma_i,\delta_i)$ if and only if $I=r_iJ$ for some $J\lhd R_{\mfrak{p}}$, $a\in J$, $\gamma\notin \mfrak{p}$ and $\delta\notin\mfrak{p}$.

Therefore $h_i\in\N_0$, $\mfrak{p}_{ij}\lhd R$ and $I_{ij}\lhd R_{\mfrak{p}_{ij}}$ for $1\leq j\leq h_i$ satisfy (a)${}_i$ and (b)${}_i$  if and only if for each $1\leq j\leq h_i$, there exists $J_{ij}\lhd R_{\mfrak{p}_{ij}}$ such that $I_{ij}=r_iJ_{ij}$, $a\in J_{ij}$, $\gamma\delta\notin \mfrak{p}$,
$\vertl\bigoplus_{j=1}^{h_i} R_{\mfrak{p}_{ij}}/I_{ij}\vertr=N_i$, and,
for all $\nf{d|x}{x=0}\in Z$,
\begin{enumerate}
\item if $\rho_Z(q_i)(\nf{d|x}{x=0})=1$ then $\vertl\oplus_{j=1}^{h_i}R_{\mfrak{p}_{ij}}/s_Z(q_i,\nf{d|x}{x=0})R_{\mfrak{p}_{ij}}\vertr=f_i(\nf{d|x}{x=0})\cdot N_i^{-1}$
\item if $\rho_Z(q_i)(\nf{d|x}{x=0})=2$ then $\vertl\oplus_{j=1}^{h_i}R_{\mfrak{p}_{ij}}/s_Z(q_i,\nf{d|x}{x=0})R_{\mfrak{p}_{ij}}\vertr=f_i(\nf{d|x}{x=0})\cdot N_i$
\item if $\rho_Z(q_i)(\nf{d|x}{x=0})=3$ then $f_i(\nf{d|x}{x=0})=1$
\end{enumerate}
and for all $\nf{xb=0}{c|x}\in Z$,
\begin{enumerate}
\item if $\rho_Z(q_i)(\nf{xb=0}{c|x})=1$ then $f_i(\nf{xb=0}{c|x})=N_i$
\item if $\rho_Z(q_i)(\nf{xb=0}{c|x})=2$ then $\vertl \oplus_{j=1}^{h_i}R_{\mfrak{p_{ij}}}/cR_{\mfrak{p_{ij}}}\vertr=f_i(\nf{xb=0}{c|x})$
\item if $\rho_Z(q_i)(\nf{xb=0}{c|x})=3$ then $\vertl \oplus_{j=1}^{h_i}I_{ij}/bI_{ij}\vertr=f_i(\nf{xb=0}{c|x})$
\item if $\rho_Z(q_i)(\nf{xb=0}{c|x})=4$ then $f_i(\nf{xb=0}{c|x})=1$
\item if $\rho_Z(q_i)(\nf{xb=0}{c|x})=5$ then $\vertl\oplus_{j=1}^{h_i} R_{\mfrak{p}_{ij}}/s_Z(q_i)(\nf{xb=0}{c|x})R_{\mfrak{p}_{ij}}\vertr=f_i(\nf{xb=0}{c|x})\cdot N_i$.
\end{enumerate}
Since $R_{\mfrak{p}_{ij}}/I_{ij}$ is finite, by \ref{I/bI}, if $b\neq 0$ then $\vertl I_{ij}/bI_{ij}\vertr =\vertl R_{\mfrak{p}_{ij}}/bR_{\mfrak{p}_{ij}}\vertr$. By definition, $\rho_Z(\nf{xb=0}{c|x})=3$ implies $b\neq 0$. So ($3$) in the second list of conditions may be replaced by
\begin{itemize}
\item [($3$')] if $\rho_Z(q_i)(\nf{xb=0}{c|x})=3$ then $\vertl \oplus_{j=1}^{h_i}R_{\mfrak{p}_{ij}}/bR_{\mfrak{p}_{ij}}\vertr=f_i(\nf{xb=0}{c|x})$.
\end{itemize}
Finally, the condition that $\vertl\bigoplus_{j=1}^{h_i} R_{\mfrak{p}_{ij}}/I_{ij}\vertr=N_i$ can be replaced by the condition that there exist $N_i',N_i''\in\N$ such that $N_i=N_i'N_i''$, $\vertl\bigoplus_{j=1}^{h_i} R_{\mfrak{p}_{ij}}/r_iR_{\mfrak{p}_{ij}}\vertr=N_i'$ and $\vertl\bigoplus_{j=1}^{h_i} R_{\mfrak{p}_{ij}}/J_{ij}\vertr=N_i''$.
The proof can now be finished using \ref{fincond}.
\end{proof}

\section{Removing $|\nf{d|x}{x=0}|=D$ and $|\nf{xb=0}{c|x}|=G$.}\label{Sremoving}

\noindent
This section uses results from sections \ref{sectuni} and \ref{Sfinitemodules}. For $d\in R\backslash\{0\}$, respectively $b,c\in R\backslash\{0\}$, we show that there is an algorithm answering whether there exists a sum of modules $\oplus_{i=1}^hR_{\mfrak{p}_i}/dI_i$, respectively $\oplus_{i=1}^hI_i/bcR_{\mfrak{p}_i}$, satisfying a sentence as in \ref{4.1Denseimproved} under the assumption that one of the conjuncts is $\vertl d|x/x=0\vertr=D$, respectively $\vertl xb=0/c|x\vertr=G$.
These results will be used in section \ref{notred} to eliminate expressions of the form $|\nf{d|x}{x=0}|=D$ and $|\nf{xb=0}{c|x}|=G$, where $D,G\geq 2$.

\begin{proposition}\label{removed|x/x=0=}
Let $R$ be a recursive Pr\"ufer domain such that $\EPP(R)$ and the radical relation are recursive. There exists an algorithm which, given a sentence $\chi$ of the form
\[\vertl \nicefrac{d|x}{x=0}\vertr= \ D\wedge\bigwedge_{\nf{\phi}{\psi}\in X} \vertl\nf{\phi}{\psi}\vertr=f(\nf{\phi}{\psi})\wedge\bigwedge_{\nf{\phi}{\psi}\in Y} \vertl\nf{\phi}{\psi}\vertr\geq g(\nf{\phi}{\psi}),\] where $d\in R\backslash\{0\}$, $D\in\N$, $f:X\rightarrow \N$, $g:Y\rightarrow \N$ and $X,Y$ are disjoint finite sets of pp-pairs of the form $\nf{xb=0}{c|x}$ and $\nf{a|x}{x=0}$, answers whether there exist $h\in \N_0$, prime ideals $\mfrak{p}_i\lhd R$ and ideals $I_i\lhd R_{\mfrak{p}_i}$ for $1\leq i\leq h$ such that $\oplus_{i=1}^hR_{\mfrak{p}_i}/dI_i\models \chi$.
\end{proposition}
\begin{proof}
Let $Z:=X\cup Y$ and let $S_Z$, $\rho_Z$ and $s_Z$ be as in \ref{sortR/lambdaI} with $\lambda:=d$. Enumerate $S_Z:=\{q_1,\ldots,q_m\}$ and let $q_i:=(r_i,r_ia_i,\gamma_i,\delta_i)$. By definition, for all prime ideals $\mfrak{p}\lhd R$ and ideals $I\lhd R_{\mfrak{p}}$, there exists $1\leq i\leq m$ such that $(\mfrak{p},I)\models q_i$.

Therefore, there exist $h\in \N_0$, prime ideals $\mfrak{p}_i\lhd R$ and ideals $I_i\lhd R_{\mfrak{p}_i}$ for $1\leq i\leq h$ such that $\oplus_{i=1}^h R_{\mfrak{p}_i}/dI_i\models \chi$ if and only if there exist $D_i\in \N$ for $1\leq i\leq m$ such that $\prod_{i=1}^mD_i=D$, $(\overline{f},\overline{g})\in\Omega_{f,g,m}$ and, for $1\leq i\leq m$, $h_i\in\N_0$, prime ideals $\mfrak{p}_{ij}\lhd R$ and ideals $I_{ij}\lhd R_{\mfrak{p}_{ij}}$ such that, for $1\leq i\leq m$,
\begin{itemize}
\item [(a)${}_i$] $(\mfrak{p}_{ij},I_{ij})\models q_i$ for $1\leq j\leq h_i$
\item [(b)${}_i$] $\oplus_{j=1}^{h_j}R_{\mfrak{p}_{ij}}/dI_{ij}$ satisfies
\[
\vertl \nicefrac{d|x}{x=0}\vertr= \ D_i\wedge\bigwedge_{\nf{\phi}{\psi}\in X_i} \vertl\nf{\phi}{\psi}\vertr=f_i(\nf{\phi}{\psi})\wedge\bigwedge_{\nf{\phi}{\psi}\in Y_i} \vertl\nf{\phi}{\psi}\vertr\geq g_i(\nf{\phi}{\psi}).
\]
\end{itemize}
Fix $1\leq i\leq m$, $D_i\in\N$ and $(\overline{f},\overline{g})\in \Omega_{f,g,m}$ as above.
Note that \[\vertl\nf{d|x}{x=0}(\oplus_{j=1}^{h_i}R_{\mfrak{p}_{ij}}/dI_{ij})\vertr=D_i \ \text{ if and only if } \ \vertl\oplus_{j=1}^{h_i}R_{\mfrak{p}_{ij}}/I_{ij}\vertr=D_i.\] So, exactly as in \ref{EPPfinite}, we may assume $r_i\neq 0$ for $1\leq i\leq m$. Moreover, by \ref{propIhash}\ref{finmodulespr}, $I_{ij}^\#=\mfrak{p}_{ij}$.

Now $h_i\in\N_0$, $\mfrak{p}_{ij}\lhd R$ and $I_{ij}\lhd R_{\mfrak{p}_{ij}}$ for $1\leq j\leq h_i$ satisfy (a)${}_i$ if and only if for $1\leq j\leq h_i$, there exists $J_{ij}\lhd R_{\mfrak{p}_i}$ such that $I_{ij}=r_iJ_{ij}$, $a_i\in J_{ij}$, $\gamma_i\notin \mfrak{p}_{ij}$ and $\delta_i\notin J_{ij}^\#=I_{ij}^\#=\mfrak{p}_{ij}$.

Now $h_i\in\N_0$, $\mfrak{p}_{ij}\lhd R$ and $J_{ij}\lhd R_{\mfrak{p}_{ij}}$ with $I_{ij}:=r_iJ_{ij}$ for $1\leq j\leq h_i$ satisfy (b)${}_i$ if and only if there exists $D'_i,D''_i\in\N$ with $D'_iD''_i=D_i$ such that
\[\vertl \oplus_{j=1}^{h_i}R_{\mfrak{p}_{ij}}/J_{ij}\vertr=D'_i \ \text{ and } \ \vertl \oplus_{j=1}^{h_i}R_{\mfrak{p}_{ij}}/r_iR_{\mfrak{p}_{ij}}\vertr=D''_i,\]
for all $\nf{xb=0}{c|x}\in X_i$ (respectively $\nf{xb=0}{c|x}\in Y_i$),
\begin{enumerate}
\item if $\rho_Z(q_i)(\nf{xb=0}{c|x})=1$ then $f_i(\nf{xb=0}{c|x})=D_i$ (respectively $D_i\geq g_i(\nf{xb=0}{c|x})$)
\item if $\rho_Z(q_i)(\nf{xb=0}{c|x})=2$ then $\vertl \oplus_{j=1}^{h_i}R_{\mfrak{p}_{ij}}/cR_{\mfrak{p}_{ij}}\vertr=f_i(\nf{xb=0}{c|x})$ (respectively $\vertl \oplus_{j=1}^{h_i}R_{\mfrak{p}_{ij}}/cR_{\mfrak{p}_{ij}}\vertr\geq g_i(\nf{xb=0}{c|x})$ )
\item if $\rho_Z(q_i)(\nf{xb=0}{c|x})=3$ then $\vertl \oplus_{j=1}^{h_i}R_{\mfrak{p}_{ij}}/bR_{\mfrak{p}_{ij}}\vertr=f_i(\nf{xb=0}{c|x})$ (respectively $\vertl \oplus_{j=1}^{h_i}R_{\mfrak{p}_{ij}}/bR_{\mfrak{p}_{ij}}\vertr\geq g_i(\nf{xb=0}{c|x})$)
\item if $\rho_Z(q_i)(\nf{xb=0}{c|x})=4$ then $f_i(\nf{xb=0}{c|x})=1$ (respectively $g_i(\nf{xb=0}{c|x})=1$)
\item if $\rho_Z(q_i)(\nf{xb=0}{c|x})=5$ then $\vertl\oplus_{j=1}^{h_i} R_{\mfrak{p}_{ij}}/s_Z(q_i,\nf{xb=0}{c|x})R_{\mfrak{p}_{ij}}\vertr=f_i(\nf{xb=0}{c|x})\cdot D_i$ (respectively $\vertl\oplus_{j=1}^{h_i} R_{\mfrak{p}_{ij}}/s_Z(q_i,\nf{xb=0}{c|x})R_{\mfrak{p}_{ij}}\vertr\geq g_i(\nf{xb=0}{c|x})\cdot D_i$).
\end{enumerate}
and
for all $\nf{a|x}{x=0}\in X_i$ (respectively $\nf{a|x}{x=0}\in Y_i$),
\begin{enumerate}
\item if $\rho_Z(q_i)(\nf{a|x}{x=0})=1$ then $\vertl\oplus_{j=1}^{h_i}R_{\mfrak{p}_{ij}}/s_Z(q_i,\nf{a|x}{x=0})R_{\mfrak{p}_{ij}}\vertr=f_i(\nf{a|x}{x=0})\cdot D_i^{-1}$ (respectively $\vertl\oplus_{j=1}^{h_i}R_{\mfrak{p}_{ij}}/s_Z(q_i,\nf{a|x}{x=0})R_{\mfrak{p}_{ij}}\vertr\geq g_i(\nf{a|x}{x=0})\cdot D_i^{-1}$),
\item if $\rho_Z(q_i)(\nf{a|x}{x=0})=2$ then $\vertl\oplus_{j=1}^{h_i}R_{\mfrak{p}_{ij}}/s_Z(q_i,\nf{a|x}{x=0})R_{\mfrak{p}_{ij}}\vertr=f_i(\nf{a|x}{x=0})\cdot D_i$ (respectively $\vertl\oplus_{j=1}^{h_i}R_{\mfrak{p}_{ij}}/s_Z(q_i,\nf{a|x}{x=0})R_{\mfrak{p}_{ij}}\vertr\geq g_i(\nf{a|x}{x=0})\cdot D_i$), and,
\item if $\rho_Z(q_i)(\nf{d|x}{x=0})=3$ then $f_i(\nf{a|x}{x=0})=1$ (respectively $g_i(\nf{a|x}{x=0})=1$).
\end{enumerate}
So we are now done by \ref{fincond2}.
\end{proof}

\begin{proposition}\label{I/bcR}
Let $R$ be a recursive Pr\"ufer domain such that $\EPP(R)$ and the radical relation are recursive. There exists an algorithm which, given a sentence $\chi$ of the form
\[\vertl \nicefrac{xb=0}{c|x}\vertr= \ G\wedge\bigwedge_{\nf{\phi}{\psi}\in X} \vertl\nf{\phi}{\psi}\vertr=f(\nf{\phi}{\psi})\wedge\bigwedge_{\nf{\phi}{\psi}\in Y} \vertl\nf{\phi}{\psi}\vertr\geq g(\nf{\phi}{\psi}),\] where $b,c\in R\backslash\{0\}$, $G\in\N$, $f:X\rightarrow \N$, $g:Y\rightarrow \N$ and $X,Y$ are disjoint finite sets of pp-pairs of the form $\nf{xb'=0}{c'|x}$ and $\nf{d|x}{x=0}$, answers whether there exist $h\in \N_0$, prime ideals $\mfrak{p}_i\lhd R$ and ideals $I\lhd R_{\mfrak{p}_i}$ with $b,c\in I_i$ for $1\leq i\leq h$ such that $\oplus_{i=1}^hI_i/bc R_{\mfrak{p}_i}\models \chi$.
%
\end{proposition}
\begin{proof}
The proof, which we leave to the reader, is very similar to \ref{removed|x/x=0=}, except we use \ref{sortI/lambda} in place of \ref{sortR/lambdaI} and the fact that $\vertl\nicefrac{xb=0}{c|x}(I/bcR_{\mfrak{p}})\vertr=\vertl R_{\mfrak{p}}/I\vertr$ in place of $\vertl\nf{d|x}{x=0}(R_\mfrak{p}/dI)\vertr=\vertl R_{\mfrak{p}}/I\vertr$.
\end{proof}

\section{Further syntactic reductions}\label{Furthersyn}

\noindent
In this section we continue work of section \ref{1stsyn} to improve the form of the conjunction in \ref{4.1Denseimproved}. Some of these reductions use results in sections \ref{Sfinitemodules} and \ref{Sremoving}, which only apply to Pr\"ufer domains (i.e. they do not  apply to arbitrary arithmetical rings).

Let $W$ be the set of $\mcal{L}_R$-sentences of the form
\[\tag{$\dagger$}\label{flaW2}\vertl \nicefrac{d|x}{x=0}\vertr\square_1 D\wedge \vertl \nicefrac{xb=0}{c|x}\vertr \square_2 E\wedge\bigwedge_{\nf{\phi}{\psi}\in X}\vertl\nf{\phi}{\psi}\vertr=f(\nf{\phi}{\psi})\wedge \bigwedge_{\nf{\phi}{\psi}\in Y}\vertl\nf{\phi}{\psi}\vertr\geq g(\nf{\phi}{\psi})\wedge\Xi  \] where $\square_1,\square_2\in\{\geq,=,\emptyset\}$, $d,c,b\in R\backslash\{0\}$, $D,E\in\N$, $f:X\rightarrow \N$, $g:Y\rightarrow \N$, $X,Y$ are finite subsets of pp-pairs of the form $\nf{xb'=0}{x=0}$ and $\nf{x=x}{c'|x}$, and $\Xi$ an auxiliary sentence.

Let $V$ be the set of $w\in W$ such that there exists an $R$-module which satisfies $w$.

\begin{definition}
Let $w\in W$ be as in \textnormal{(\ref{flaW2})}. Define
\begin{align*}
z_1 &:=\vertl\{\nf{xb'=0}{x=0}\in X \st f(\nf{xb'=0}{x=0})>1\}\vertr,\\
z_2 &:=\vertl\{\nf{xb'=0}{x=0}\in Y \st g(\nf{xb'=0}{x=0})>1\}\vertr,\\
z_3 &:=\vertl\{\nf{x=x}{c'|x}\in X \st f(\nf{x=x}{c'|x})>1\}\vertr, \text{ and }\\
z_4 &:=\vertl\{\nf{x=x}{c'|x}\in Y \st g(\nf{x=x}{c'|x})>1\}\vertr.
\end{align*}
The \textbf{short signature} is defined to be the tuple $(\square_1,\square_2)$ and the \textbf{extended signature}, $\exsig w$, is defined to be the tuple $((\square_1,\square_2),(z_1,z_2),(z_3,z_4))$.
\end{definition}

\noindent
We equip the set $\{\geq,=,\emptyset\}$ with a total order $\succ$ by putting $\geq \ \succ \ = \ \succ \ \emptyset$. We partially order the set of short signatures $\{\geq,=,\emptyset\}^2$ by setting $(\square_1,\square_2)\geq (\square'_1,\square'_2)$ whenever $\square_1\succeq \square'_1$ and $\square_2\succeq \square'_2$. Finally, we partially order on the set of extended signatures by setting
\[((\square_1,\square_2),(z_1,z_2),(z_3,z_4))\geq ((\square'_1,\square'_2),(z'_1,z'_2),(z'_3,z'_4))\] if and only if
$(\square_1,\square_2)> (\square'_1,\square'_2)$ or $(\square_1,\square_2)= (\square'_1,\square'_2)$ and
\begin{itemize}
\item $z_1+z_2>z'_1+z'_2$ or $z_1+z_2= z'_1+z'_2$ and $z_2\geq z'_2$, and
\item $z_3+z_4>z'_3+z'_4$ or $z_3+z_4= z'_3+z'_4$ and $z_4\geq z'_4$.
\end{itemize}
We now present various algorithms which, given $w\in W$ of a particular form, output $\underline{w}\in\mathbb{W}$ such that $w\in V$ if and only if $\underline{w}\in \mathbb{V}$, and, $\exsig \underline{w}<\exsig w$. By convention, we give both $\top\in\mathbb{W}$ and $\bot\in\mathbb{W}$ extended signature $((\emptyset,\emptyset),(0,0),(0,0))$.

\begin{remark}
The order on the set of extended signatures is artinian.
\end{remark}

\noindent
As in \S \ref{1stsyn}, given $w\in W$, we may always assume that $w$ is of the form
\[\chi_{f,g}\wedge\Xi\] where $X,Y$ are disjoint finite sets of pp-pairs of the form $\nf{d|x}{x=0}$ or $\nf{xb=0}{c|x}$, $f:X\rightarrow \N_2$, $g:Y\rightarrow \N_2$ and $\Xi$ is an auxiliary sentence.

%
%
%

\begin{remark}\label{dualexsig}
If $w\in W$ has extended signature $((\square_1,\square_2),(z_1,z_2),(z_3,z_4))$ then $Dw$ has extended signature $((\square_1,\square_2),(z_3,z_4),(z_1,z_2))$ where $D$ denotes the duality defined in \ref{Dualsent}.
\end{remark}

\begin{remark}\label{featheringextsig}
Let $X,Y$ be disjoint finite subsets of pp-pairs, $f:X\rightarrow \N_2$, $g:Y\rightarrow \N_2$ and $\Xi$ an auxiliary sentence be such that
\[\chi_{f,g}\wedge\Xi=\bigwedge_{\nicefrac{\phi}{\psi}\in X}\vertl\nicefrac{\phi}{\psi}\vertr=f(\nicefrac{\phi}{\psi})\wedge\bigwedge_{\nicefrac{\phi}{\psi}\in Y}\vertl \nicefrac{\phi}{\psi}\vertr\geq g(\nicefrac{\phi}{\psi})\wedge\Xi\in W.\] Then, for each $(f_1,\ldots,f_n,g_1,\ldots,g_n)\in\Omega_{f,g,n}$ and $1\leq i\leq n$,
\[\chi_{f_i,g_i}\wedge\Xi =\bigwedge_{\nicefrac{\phi}{\psi}\in X_i}\vertl\nicefrac{\phi}{\psi}\vertr=f_i(\nicefrac{\phi}{\psi})\wedge\bigwedge_{\nicefrac{\phi}{\psi}\in Y_i}\vertl \nicefrac{\phi}{\psi}\vertr\geq g_i(\nicefrac{\phi}{\psi})\wedge\Xi\in W,\] and, $\exsig\chi_{f_i,g_i}\wedge\Xi\leq \exsig\chi_{f,g}\wedge\Xi$. Moreover, for each $(\overline{f},\overline{g})\in \Omega_{f,g,n}$ and $1\leq i\leq n$, either $X=X_i$ and $Y=Y_i$, or, $\exsig\chi_{f_i,g_i}\wedge\Xi<\chi_{f,g}\wedge\Xi$.
\end{remark}
\begin{proof}
Fix $(\overline{f},\overline{g})\in\Omega_{f,g,n}$ and $1\leq i\leq n$. Let $((\square_1,\square_2),(z_1,z_2),(z_3,z_4))$ be the extended signature of $\chi_{f,g}\wedge \Xi$ and $((\square'_1,\square'_2),(z'_1,z'_2),(z'_3,z'_4))$ the extended signature of $\chi_{f_i,g_i}\wedge\Xi$. Note that since $X$ and $Y$ are disjoint, so are $X_i=X\cup (Y\backslash Y_i)$ and $Y_i$.

That the short signature of $\chi_{f_i,g_i}$ is less than the short signature of $\chi_{f,g}$ follows from the fact that $Y_i\subseteq Y$ and $X_i=X\cup(Y\backslash Y_i)$.

%

We show that $z_1+z_2\geq z_1'+z_2'$ and $z_2\geq z_2'$. By definition, $z_1+z_2$ is the number of $\nf{x=x}{c|x}\in X\cup Y$ and $z_2$ is the number of $\nf{x=x}{c|x}\in Y$. Since $g_i(\nf{x=x}{c|x})=g(\nf{x=x}{c|x})>1$ for all $\nf{x=x}{c|x}\in Y_i$, we see that $z_2'$ is the number of $\nf{x=x}{c|x}\in Y_i$. So $z_2'\leq z_2$ because $Y_i\subseteq Y$. Since $X$ and $Y$ are disjoint, so are $X_i=X\cup (Y\backslash Y_i)$ and $Y_i$. Therefore $z_1'+z_2'$ is equal to the number of $\nf{x=x}{c|x}\in X_i\cup Y_i=X\cup Y$ with $\nf{x=x}{c|x}\in Y_i$, or, $\nf{x=x}{c|x}\in X_i$ and $f_i(\nf{x=x}{c|x})>1$. Since $X\cup Y=X_i\cup Y_i$, $z_1+z_2\geq z_1'+z_2'$. A similar argument shows that $z_3+z_4\geq z_3'+z_4'$ and $z_4\geq z_4'$. So the extended signature of $\chi_{f,g}$ is less than the extended signature of $\chi_{f_i,g_i}$.

%

We now prove the moreover. Since $X,Y$ are disjoint, $X\neq X_i$ if and only if $Y\neq Y_i$. Suppose $X\neq X_i$. Then there exists $\phi/\psi\in Y\backslash Y_i$. By assumption, $g(\phi/\psi)>1$.

If $\phi/\psi$ is $d|x/x=0$ then the short signature of $\chi_{f,g}$ is $(\geq,\square)$ and the short signature of $\chi_{f_i,g_i}$ is $(=,\square')$ or $(\emptyset,\square')$, and, by what we have already proved, $\square'\leq \square$. So the short signature of $\chi_{f_i,g_i}$ is strictly less than the short signature of $\chi_{f,g}$. The case of $xb=0/c|x$ for $b,c\neq 0$ is similar.


If $\phi/\psi$ is $x=x/c|x$ then
\[\vertl\{x=x/c'|x\in Y_i\st g_i(x=x/c'|x)>1\}\vertr<\vertl\{x=x/c'|x\in Y\}\vertr.\] So $\exsig\chi_{f_i,g_i}<\exsig\chi_{f,g}$. The case of $\phi/\psi$ equal to $xb=0/x=0$ is similar.
\end{proof}

\subsection{Reducing the short signature}

\begin{proposition}\label{ssrem=}
Let $R$ be a recursive Pr\"ufer domain with $\EPP(R)$ recursive. There is an algorithm which given $w\in W$ with short signature $(=,\square)$ or $(\square,=)$, for some $\square\in\{\emptyset,=,\geq\}$, outputs $\underline{w}\in\mathbb{W}$ such that $w\in V$ if and only if $\underline{w}\in \mathbb{V}$, and, $\underline{w}$ is a lattice combination of elements $w'\in W$ such that the short signature of $w'$ is strictly less than the short signature of $w$.
\end{proposition}
\begin{proof}
Suppose $w$ has short signature $(=,\square)$. Then $w$ is of the form

\[\vertl d|x/x=0\vertr = D\wedge\chi_{f,g}\wedge\Xi\]
where $d\in R\backslash\{0\}$, $D\in\N_2$, $f:X\rightarrow \N_2$, $g:Y\rightarrow \N_2$, $X$ and $Y$ are disjoint finite subsets of $\{x=x/c'|x,xb'=0/x=0\st c',b'\in R\}\cup\{xb=0/c|x\}$ for some $b,c\in R\backslash\{0\}$, and $\Xi$ is an auxiliary sentence.

If $M\models w$ then, by \ref{REFuni} and \ref{needV/dI}, there exist $h\in\N$, prime ideals $\mfrak{p}_i\lhd R$ and ideals $I_i\lhd R_{\mfrak{p}_i}$ for $1\leq i\leq h$ and $M'\in\Mod\text{-}R$ such that $M\equiv M'\oplus\oplus_{i=1}^hR_{\mfrak{p}_i}/dI_i$ and $\vertl d|x/x=0(M')\vertr=1$.

Thus there exists $M\in\Mod\text{-}R$ such that $M\models w$ if and only if there exist $(f_1,f_2,g_1,g_2)\in\Omega_{f,g,2}$, $h\in\N$, prime ideals $\mfrak{p}_i\lhd R$ and $I_i\lhd R_{\mfrak{p}_i}$ for $1\leq i\leq h$, such that $\oplus_{i=1}^h R_{\mfrak{p}_i}/dI_i$ satisfies
\[w_{(f_1,f_2,g_1,g_2)}:=\vertl d|x/x=0\vertr = D\wedge\chi_{f_1,g_1}\wedge \Xi\]
and $M'$ which satisfies
\[w'_{(f_1,f_2,g_1,g_2)}:=\vertl d|x/x=0\vertr = 1\wedge \chi_{f_2,g_2}\wedge \Xi.\]
Let $\Omega\subseteq \Omega_{f,g,2}$ be the set of $(f_1,f_2,g_1,g_2)\in\Omega_{f,g,2}$, such that there exist $h\in\N_0$, prime ideals $\mfrak{p}_i\lhd R$ and $I_i\lhd R_{\mfrak{p}_i}$ for $1\leq i\leq h$, such that $\oplus_{i=1}^h R_{\mfrak{p}_i}/dI_i$ satisfies $w_{(f_1,f_2,g_1,g_2)}$. So $w\in V$ if and only if
\[\underline{w}:=\bigsqcup_{(f_1,f_2,g_1,g_2)\in \Omega}w'_{(f_1,f_2,g_1,g_2)}\in\mathbb{V}.\]
By \ref{removed|x/x=0=}, given $w$, we can compute $\Omega$, and, so, we can compute $\underline{w}$. The short signature of each $w'_{(f_1,f_2,g_1,g_2)}$ is $(\emptyset,\square')$ for some $\square'\in\{\emptyset,=,\leq\}$ and, by \ref{featheringextsig}, $\square\succeq\square'$. So $(=,\square)>(\emptyset,\square')$ as required.

Suppose $w$ has short signature $(\square,=)$. Then $w$ is

\[\vertl xb=0/c|x\vertr = E\wedge\chi_{f,g}\wedge\Xi\]
where $b,c\in R\backslash\{0\}$, $E\in\N_2$, $f:X\rightarrow \N_2$, $g:Y\rightarrow \N_2$, $X$ and $Y$ are disjoint finite subsets of $\{x=x/c'|x,xb'=0/x=0\st c',b'\in R\}\cup\{d|x/x=0\}$ for some $d\in R\backslash\{0\}$, and $\Xi$ is an auxiliary sentence.

If $M\models w$ then, by \ref{REFuni} and \ref{needI/bcV}, there exist $M_1,M_2,M'\in\Mod\text{-}R$ such that $b\in\ann_RM_1$, $c\in\ann_RM_2$ and $\vertl xb=0/c|x(M')\vertr=1$ and there exist $h\in\N_0$, prime ideals $\mfrak{p}_i$ and ideals $I_i\lhd R_{\mfrak{p}_i}$ for $1\leq i\leq h$ such that $b, c\in I_i$ and \[\oplus_{i=1}^hI_i/bcR_{\mfrak{p}_i}\oplus M_1\oplus M_2\oplus M'\models w.\] Therefore, $w\in V$ if and only if there exist $E_1,E_2,E_4\in \N$  with $E_1E_2E_4=E$ and $(\overline{f},\overline{g})\in\Omega_{f,g,4}$ such that
\[w_{1,E_1,f_1,g_1}:=\vertl x=x/c|x\vertr=E_1\wedge \chi_{f_1,g_1}\wedge \Xi\wedge \vertl b|x/x=0\vertr=1\in V,\]
\[w_{2,E_2,f_2,g_2}:=\vertl xb=0/x=0\vertr=E_2\wedge \chi_{f_2,g_2}\wedge \Xi\wedge \vertl c|x/x=0\vertr=1\in V,\]
\[w_{3,f_3,g_3}:=\chi_{f_3,g_3}\wedge\Xi\wedge \vertl xb=0/c|x\vertr=1\in V\] and
there exist $h\in\N_0$, prime ideals $\mfrak{p}_i$ and ideals $I_i\lhd R_{\mfrak{p}_i}$ for $1\leq i\leq h$ such that $b,c\in I_i$ and \[\oplus_{i=1}^hI_i/bcR_{\mfrak{p}_i}\models \vertl xb=0/c|x\vertr=E_4\wedge\chi_{f_4,g_4}\wedge\Xi.\]
Let $\mcal{H}$ be the set of $(E_1,E_2,(\overline{f},\overline{g}))\in \N^2\times\Omega_{f,g,4}$ such that there exists $E_4$ with $E_1E_2E_4=E$ and there exist $h\in\N_0$, prime ideals $\mfrak{p}_i$ and ideals $I_i\lhd R_{\mfrak{p}_i}$ for $1\leq i\leq h$ such that $b,c\in I_i$ and \[\oplus_{i=1}^hI_i/bcR_{\mfrak{p}_i}\models \vertl xb=0/c|x\vertr=E_4\wedge\chi_{f_4,g_4}\wedge\Xi.\]
Then $w\in V$ if and only if
\[\underline{w}:=\bigsqcup_{(E_1,E_2,(\overline{f},\overline{g}))\in\mcal{H}} w_{1,E_1,f_1,g_1}\sqcap w_{2,E_2,f_3,g_3}\sqcap w_{3,f_3,g_3}\in V.\]
By \ref{I/bcR}, given $w$, we can compute $\mcal{H}$ and so we can compute $\underline{w}$.

Now if $w$ has short signature $(\square,=)$ then $w_{1,E_1,f_1,g_1},w_{2,E_2,f_3,g_3}$ and $w_{3,f_3,g_3}$ have short signature $(\square',\emptyset)$ and by \ref{featheringextsig}, $\square'\preceq \square$. So $(\square',\emptyset)<(\square,=)$, as required.
\end{proof}

\begin{lemma}\label{d|x/x=0x=x/c|x}
Let $R$ be an arithmetical ring. Let $c,d\in R$ and $D\in\N$. For all $C\in\N_2$,
\begin{multline*}
T_R\models \vertl \nicefrac{xc^{D-1}d=0}{c|x}\vertr=1\wedge \vertl \nicefrac{d|x}{x=0}\vertr\geq D\wedge\vertl \nicefrac{x=x}{c|x}\vertr=C \\ \leftrightarrow\vertl \nicefrac{xc^{D-1}d=0}{c|x}\vertr=1\wedge\vertl \nicefrac{x=x}{c|x}\vertr=C
\end{multline*}
and
\begin{multline*}
T_R\models \vertl \nicefrac{c^Dd|x}{x=0}\vertr=1\wedge \vertl \nicefrac{d|x}{x=0}\vertr\geq D\wedge \vertl \nicefrac{x=x}{c|x}\vertr=C \\ \leftrightarrow \bigvee_{D\leq E\leq C^D}\vertl \nicefrac{d|x}{x=0}\vertr= E\wedge \vertl \nicefrac{x=x}{c|x}\vertr=C.
\end{multline*}

\end{lemma}
\begin{proof}
Suppose $M\models\vertl \nf{xc^{D-1}d=0}{c|x}\vertr=1$. Then $xc^id=0\leq_Mc|x$ for $0\leq i\leq D-1$. So, by \ref{equivpair1}, for $0\leq i\leq D-1$, \[\vertl\nicefrac{c^id|x}{c^{i+1}d|x}(M)\vertr=\vertl\nicefrac{x=x}{c|x}(M)\vertr.\] Hence
\[\vertl \nicefrac{d|x}{x=0}(M)\vertr=\vertl \nicefrac{x=x}{c|x}(M)\vertr^D\vertl \nicefrac{c^Dd|x}{x=0}(M)\vertr.\]
%
Therefore, if $C\geq 2$ then
\begin{multline*}
T_R\models \vertl \nicefrac{xc^{D-1}d=0}{c|x}\vertr=1\wedge \vertl \nicefrac{d|x}{x=0}\vertr\geq D\wedge\vertl \nicefrac{x=x}{c|x}\vertr=C\leftrightarrow \\ \vertl \nicefrac{xc^{D-1}d=0}{c|x}\vertr=1\wedge\vertl \nicefrac{x=x}{c|x}\vertr=C.
\end{multline*}
Now suppose that $c^Dd\in \ann_RM$. Then $c^Dd|x$ is equivalent to $x=0$ in $M$. Therefore, by \ref{equivpair1},
\[\vertl \nicefrac{d|x}{c^Dd|x}(M)\vertr=\vertl \nicefrac{x=x}{c^D|x+xd=0}(M)\vertr\leq \vertl \nicefrac{x=x}{c^D|x}(M)\vertr\leq \vertl \nicefrac{x=x}{c|x}\vertr^D.\] Hence
\begin{multline*}
T_R\models \vertl c^Dd|x/x=0\vertr=1\wedge \vertl d|x/x=0\vertr\geq D\wedge \vertl x=x/c|x\vertr=C \leftrightarrow \\ \bigvee_{D\leq E\leq C^D}\vertl d|x/x=0\vertr= E\wedge \vertl x=x/c|x\vertr=C. \qedhere
\end{multline*}
\end{proof}

\begin{proposition}\label{ssremgeqdx}
Let $R$ be a recursive arithmetical ring. There is an algorithm, which given $w\in W$ with extended signature $((\geq,\square),(z_1,z_2),(z_3,z_4))$ with $z_1\geq 1$ or $z_3\geq 1$, returns $\underline{w}\in\mathbb{W}$ such that $w\in V$ if and only if $\underline{w}\in \mathbb{V}$, and, $\exsig\underline{w}<\exsig w$.
\end{proposition}

\begin{proof}
Suppose $w:=\chi_{f,g}\wedge\Xi\in W$, where $f:X\rightarrow \N_2$, $g:Y\rightarrow \N_2$, $X,Y$ are disjoint finite sets of appropriate pp-pairs and $\Xi$ is an auxiliary sentence, has extended signature $((\geq,\square), (z_1,z_2),(z_3,z_4))$ with $z_1\geq 1$. Then $\nf{d|x}{x=0}\in Y$ for some $d\in R\backslash\{0\}$ and $\nf{x=x}{c|x}\in X$ for some $c\in R$. Let $D=g(\nf{x=x}{c|x})$.

Then, by \ref{decomposeorder}, $w\in V$ if and only if
\[
\bigsqcup_{(\overline{f},\overline{g})\in\Omega_{f,g,2}}\left(\chi_{f_1,g_1}\wedge\vertl\nf{xc^{D-1}d=0}{c|x}\vertr=1\wedge \Xi\right)\sqcap\left(\chi_{f_2,g_2}\wedge\vertl\nf{c^Dd|x}{x=0}\vertr=1\wedge \Xi\right)\in\mathbb{V}.
\]
For each $(\overline{f},\overline{g})\in\Omega_{f,g,2}$, we define
$\underline{w}_{\overline{f},\overline{g}}, \underline{w}'_{\overline{f},\overline{g}}\in\mathbb{W}$ such that
\[
\exsig \underline{w}_{\overline{f},\overline{g}}, \exsig\underline{w}'_{\overline{f},\overline{g}}<\exsig w,
\]
\[
\underline{w}_{\overline{f},\overline{g}}\in \mathbb{V} \ \text{ if and only if } \ \chi_{f_1,g_1}\wedge\vertl\nf{xc^{D-1}d=0}{c|x}\vertr=1\wedge \Xi\in V, \text{ and,}
\]
\[\underline{w}'_{\overline{f},\overline{g}}\in \mathbb{V} \ \text{ if and only if } \ \chi_{f_2,g_2}\wedge\vertl\nf{c^Dd|x}{x=0}\vertr=1\wedge \Xi\in V.
\]
Once we have done this, the proof is complete since then
\[
\underline{w}:=\bigsqcup_{(\overline{f},\overline{g})\in\Omega_{f,g,2}}\underline{w}_{\overline{f},\overline{g}}\sqcap\underline{w}'_{\overline{f},\overline{g}}
\] has the properties required by the statement.
%

If $\exsig \chi_{f_1,g_1}<\exsig \chi_{f,g}$ then set $\underline{w}_{\overline{f},\overline{g}}:=\chi_{f_1,g_1}\wedge\vertl\nf{xc^{D-1}d=0}{c|x}\vertr=1\wedge \Xi$. Otherwise, by \ref{featheringextsig}, $X=X_1$ and $Y=Y_1$. Further, $f_1(\nf{x=x}{c|x})>1$ and, by definition of $\Omega_{f,g,2}$, $g_1(\nf{d|x}{x=0})=g(\nf{d|x}{x=0})=D$. Let $Y_1':=Y_1\backslash\{\nf{d|x}{x=0}\}$ and $g_1':=g_1|_{Y_1'}$. Then, by \ref{d|x/x=0x=x/c|x},
\[\chi_{f_1,g_1}\wedge\vertl\nf{xc^{D-1}d=0}{c|x}\vertr=1\wedge \Xi\in V \] if and only if
\[\underline{w}_{\overline{f},\overline{g}}:= \chi_{f_1,g_1'}\wedge\vertl\nf{xc^{D-1}d=0}{c|x}\vertr=1\wedge\Xi\in \mathbb{V}.\]
Moreover, $\underline{w}_{\overline{f},\overline{g}}$ has short signature $(\emptyset,\square')$, where, by \ref{featheringextsig}, $\square'\preceq\square$. Therefore $\exsig \underline{w}_{\overline{f},\overline{g}}<\exsig w$ because $(\emptyset,\square')<(\emptyset,\square)$.

If $\exsig \chi_{f_2,g_2}<\exsig \chi_{f,g}$ then let $\underline{w}'_{\overline{f},\overline{g}}:=\chi_{f_2,g_2}\wedge\vertl\nf{c^Dd|x}{x=0}\vertr=1\wedge \Xi$. Otherwise, by \ref{featheringextsig}, $X=X_1$ and $Y=Y_1$. Further, $f_2(\nf{x=x}{c|x})>1$ and, by definition of $\Omega_{f,g,2}$, $g_2(\nf{d|x}{x=0})=g(\nf{d|x}{x=0})=D$.

Let $X_2':=X_2\backslash \{\nf{x=x}{c|x}\}$ and $Y_2':=Y_2\backslash\{\nf{d|x}{x=0}\}$. Let $f_2':=f_2|_{X_2'}$, $g_2':=g_2|_{Y_2'}$ and $C:=f_2(\nf{x=x}{c|x})$. Then $\chi_{f_2,g_2}\wedge\vertl\nf{c^Dd|x}{x=0}\vertr=1\wedge \Xi$ is
\[\vertl\nf{d|x}{x=0}\vertr\geq D\wedge\vertl\nf{x=x}{c|x}\vertr=C\wedge\chi_{f_2',g_2'}\wedge\Xi\wedge \vertl\nf{c^Dd|x}{x=0}\vertr=1.\]
Let
\[
\underline{w}'_{\overline{f},\overline{g}}:=\bigsqcup_{D\leq E\leq C^D}\vertl\nf{d|x}{x=0}\vertr=E\wedge\vertl\nf{x=x}{c|x}\vertr=C\wedge\chi_{f_2',g_2'}\wedge\vertl\nf{c^Dd|x}{x=0}\vertr=1\wedge\Xi.
\]
By \ref{d|x/x=0x=x/c|x}, $\underline{w}'_{\overline{f},\overline{g}}\in \mathbb{V}$ if and only if $\chi_{f_2,g_2}\wedge\vertl\nf{c^Dd|x}{x=0}\vertr=1\wedge \Xi\in V$. The short signature of each component of the join defining $\underline{w}'_{\overline{f},\overline{g}}$ is $(=,\square')$ where, by \ref{featheringextsig}, $\square'\preceq\square$. Therefore $\exsig \underline{w}'_{\overline{f},\overline{g}}<\exsig w$ because $(=,\square')<(\geq,\square)$.

Suppose $w\in W$ has extended signature $((\geq,\square),(z_1,z_2),(z_3,z_4))$ with $z_3\geq 1$. Then $Dw$ has extended signature $((\geq,\square),(z_3,z_4),(z_1,z_2))$. By the previous case, we can compute $\underline{w}$, a lattice combination of elements of $W$ with extended signatures strictly less than $Dw$ such that $Dw\in V$ if and only if $\underline{w}\in\mathbb{V}$. Now $w\in V$ if and only if $D\underline{w}\in\mathbb{V}$ and $\exsig D\underline{w}<\exsig w$.
\end{proof}

\begin{lemma}\label{redxb=0c|xwithx=xc|x}
Let $R$ be a Pr\"ufer domain. Let $a,b,c\in R$ and $E,C\in\N_2$. Suppose that $r,s,\alpha\in R$ are such that $c\alpha=ar$ and $a(\alpha-1)=cs$. Define
\begin{enumerate}
\item $\Sigma_1$ to be the formula $\vertl \nicefrac{x=x}{\alpha-1|x}\vertr=1$
\item $\Sigma_2$ to be the formula $\vertl \nicefrac{x=x}{\alpha|x}\vertr=1\wedge\vertl \nicefrac{ar|x}{x=0}\vertr=1$
\item $\Sigma_3$ to be the formula $\vertl \nicefrac{x=x}{\alpha|x}\vertr=1\wedge\vertl \nicefrac{xr=0}{a|x}\vertr=1\wedge\vertl \nicefrac{rb|x}{x=0}\vertr=1$
\item $\Sigma_4$ to be the formula $\vertl \nicefrac{x=x}{\alpha|x}\vertr=1\wedge\vertl \nicefrac{xr=0}{a|x}\vertr=1\wedge\vertl \nicefrac{xb=0}{r|x}\vertr=1$.
\end{enumerate}
Then for all $M\in\Mod\text{-}R$ there exist $M_1,\ldots,M_4\in\Mod\text{-}R$ such that $M_i\models \Sigma_i$ for $1\leq i\leq 4$ and $M\equiv M_1\oplus\ldots\oplus M_4$.

\begin{enumerate}[(i)]
\item If $i\in\{1,4\}$ and $C<E$ then
\[T_R\models\lnot(\Sigma_i\wedge \vertl \nicefrac{x=x}{a|x}\vertr=C\wedge \vertl \nicefrac{xb=0}{c|x}\vertr\geq E),
\]
and if $i\in\{1,4\}$ and $E\leq C$ then
\begin{multline*}
T_R\models\Sigma_i\wedge \vertl \nicefrac{x=x}{a|x}\vertr=C\wedge \vertl \nicefrac{xb=0}{c|x}\vertr\geq E\leftrightarrow \\ \bigvee_{E\leq E'\leq C}\Sigma_i\wedge \vertl \nicefrac{x=x}{a|x}\vertr=C\wedge \vertl \nicefrac{xb=0}{c|x}\vertr= E'.
\end{multline*}


\item In $T_R$ the following equivalence holds.
\begin{multline*} \Sigma_2\wedge \vertl \nicefrac{x=x}{a|x}\vertr=C\wedge \vertl \nicefrac{xb=0}{c|x}\vertr\geq E\leftrightarrow \Sigma_2\wedge \vertl \nicefrac{x=x}{a|x}\vertr=C\wedge \vertl \nicefrac{xb=0}{x=0}\vertr\geq E
\end{multline*}


\item  In $T_R$ the following equivalence holds.
\begin{multline*}
\Sigma_3\wedge \vertl \nicefrac{x=x}{a|x}\vertr=C\wedge \vertl \nicefrac{xb=0}{c|x}\vertr\geq E\leftrightarrow
\Sigma_3\wedge \vertl \nicefrac{x=x}{a|x}\vertr=C\wedge \vertl \nicefrac{xb=0}{r|x}\vertr\geq \lceil E/C\rceil
 \end{multline*}

\end{enumerate}
\end{lemma}
\begin{proof}
The first claim follows from \ref{decomposeorder}.

\noindent
(i) Suppose $M\models \Sigma_1$. Then
\[\vertl x=x/a|x(M)\vertr=\vertl x=x/cs|x (M)\vertr\geq \vertl xb=0/c|x(M)\vertr.\]
Suppose $M\models \Sigma_4$. Then
\[\vertl r|x/ar|x(M)\vertr=\vertl x=x/xr=0+a|x(M)\vertr=\vertl x=x/a|x(M)\vertr\]
because $\vertl xr=0/a|x(M)\vertr=1$ and hence $xr=0\leq_Ma|x$.
So
\[\vertl xb=0/c|x(M)\vertr=\vertl xb=0/ar|x(M)\vertr\leq \vertl r|x/ar|x(M)\vertr=\vertl x=x/a|x(M)\vertr.\] The first equality holds because $\vertl x=x/\alpha|x(M)\vertr=1$ and the inequality holds because $\vertl xb=0/r|x\vertr=1$ and hence $xb=0\leq_M r|x$.

Therefore, if $M\models \Sigma_i$ for $i\in\{1,4\}$ then
\[\vertl xb=0/c|x(M)\vertr\leq \vertl x=x/a|x(M)\vertr.\] The first claim follows from this.

\noindent
(ii) Suppose $M\models \Sigma_2$. Then $c|x$ is equivalent to $ar|x$ in $M$ and $ar|x$ is equivalent to $x=0$ in $M$. So $\vertl xb=0/c|x(M)\vertr=\vertl xb=0/x=0(M)\vertr$.

\noindent
(iii) Suppose $M\models \Sigma_3$. Then
\[\vertl xb=0/c|x(M)\vertr=\vertl xb=0/ar|x(M)\vertr=\vertl xb=0/r|x(M)\vertr\cdot \vertl r|x/ar|x(M)\vertr.\]
Since $\vertl xr=0/a|x(M)\vertr=1$, by \ref{equivpair1}, $\vertl r|x/ar|x(M)\vertr=\vertl x=x/a|x(M)\vertr$. The claim now follows.
\end{proof}

\begin{proposition}\label{ssremgeqxbcx}
Let $R$ be a recursive Pr\"ufer domain. There is algorithm which given $w\in W$ with extended signature $((\square,\geq),(z_1,z_2),(z_3,z_4))$ with $z_1\geq 1$ or $z_3\geq 1$ outputs $\underline{w}\in\mathbb{W}$ such that $w\in V$ if and only if $\underline{w}\in \mathbb{V}$ and $\exsig\underline{w}<\exsig w$.
\end{proposition}
\begin{proof}
For $w\in W$ as in (\ref{flaW2}), define $\deg w= E$. Given $w\in W$ with extended signature $((\square,\geq),(z_1,z_2),(z_3,z_4))$ with $z_1\geq 1$, we will show how to compute $\underline{w}\in\mathbb{W}$, a lattice combination of $w'\in W$, such that for each $w'$, either $\exsig w'<\exsig w$ or $\exsig w'=\exsig w$ and $\deg w'<\deg w$, and, such that $w\in V$ if and only if $\underline{w}\in\mathbb{V}$. Since $\deg$ takes values in $\N$, by iterating this process we will eventually compute $\underline{w}\in \mathbb{W}$ such that $\exsig \underline{w}<\exsig w$.

Suppose $w:=\chi_{f,g}\wedge\Xi\in W$, where $f:X\rightarrow \N_2$, $g:Y\rightarrow \N_2$, $X,Y$ are disjoint finite sets of appropriate pp-pairs and $\Xi$ is an auxiliary sentence, has extended signature $((\square,\geq), (z_1,z_2),(z_3,z_4))$ with $z_1\geq 1$. Since $w$ has extended signature $((\square,\geq),(z_1,z_2),(z_3,z_4))$ with $z_1\geq 1$, there exist $\nf{x=x}{a|x}\in X$ and $\nf{xb=0}{c|x}\in Y$.

Let $\Sigma_1,\Sigma_2,\Sigma_3,\Sigma_4$ be as in \ref{redxb=0c|xwithx=xc|x}. Then $w\in V$ if and only if
\[\bigsqcup_{(\overline{f},\overline{g})\in\Omega_{f,g,4}}\bigsqcap_{i=1}^4\chi_{f_i,g_i}\wedge \Xi\wedge\Sigma_i\in\mathbb{V}.\]
We show that for each $(\overline{f},\overline{g})\in\Omega_{f,g,4}$, either the extended signature of $\chi_{f_i,g_i}\wedge \Xi$ is strictly less than the extended signature of $\chi_{f,g}\wedge\Xi$ or we will show that we can compute $\underline{w}_{i,(\overline{f},\overline{g})}\in \mathbb{W}$ such that $\underline{w}_{i,(\overline{f},\overline{g})}$ is a lattice combination of $w'\in W$ with either $\exsig w'<\exsig w$, or, $\exsig w'\leq\exsig w$ and $\deg w'<w$.

Fix $(\overline{f},\overline{g})\in\Omega_{f,g,4}$ and $1\leq i\leq 4$. By \ref{featheringextsig}, $\exsig\chi_{f_i,g_i}\wedge \Xi<\exsig\chi_{f,g}\wedge\Xi$  unless $X_i=X$ and $Y_i=Y$. Moreover, if $f_i(\nf{x=x}{a|x})=1$ then $\exsig\chi_{f_i,g_i}\wedge \Xi<\exsig\chi_{f,g}\wedge\Xi$. So we may assume $X_i=X$, $Y_i=Y$ and $f_i(\nf{x=x}{a|x})>1$.

Let $X':= X\backslash\{\nf{x=x}{a|x}\}$, $Y':=Y\backslash\{\nf{xb=0}{c|X}\}$, $f_i':=f_i|_{X'}$, $g_i':=g_i|_{Y'}$, $C:=f_i(\nf{x=x}{a|x})$ and $E:=g_i(\nf{xb=0}{c|x})=g(\nf{xb=0}{c|x})$. Then $\chi_{f_i,g_i}\wedge \Xi$ is
\[\vertl\nf{x=x}{a|x}\vertr=C\wedge \vertl\nf{xb=0}{c|x}\vertr\geq E\wedge \chi_{f_i',g_i'}\wedge\Xi.\]

\noindent
\textbf{Case  $i=1$ or $i=4$:}
If $f_1(\nf{x=x}{a|x})<g_1(\nf{xb=0}{c|x})$ then, by \ref{redxb=0c|xwithx=xc|x},
\[T_R\models \lnot(\chi_{f_i,g_i}\wedge \Xi\wedge \Sigma_i).\] So set $\underline{w}_{i,(\overline{f},\overline{g})}:=\bot$.

Suppose $f_i(\nf{x=x}{a|x})\geq g_i(\nf{xb=0}{c|x})$. Then, by \ref{redxb=0c|xwithx=xc|x},
\[\vertl\nf{x=x}{a|x}\vertr=C\wedge \vertl\nf{xb=0}{c|x}\vertr\geq E\wedge \chi_{f_i',g_i'}\wedge\Xi\in V\] if and only
\[\underline{w}_{i,(\overline{f},\overline{g})}:=\bigsqcup_{E\leq E'\leq C}\vertl\nf{x=x}{a|x}\vertr=C\wedge \vertl\nf{xb=0}{c|x}\vertr= E'\wedge \chi_{f_i',g_i'}\wedge\Xi\wedge\Sigma_i\in\mathbb{V}.\]
So we are done since the short signatures of the components of the join defining $\underline{w}_{i,(\overline{f},\overline{g})}$ are $(\square',=)$ where, by \ref{featheringextsig}, $\square'\preceq\square$.

\noindent
\textbf{Case $i=2$:}
By \ref{redxb=0c|xwithx=xc|x},
\[\vertl\nf{x=x}{a|x}\vertr=C\wedge \vertl\nf{xb=0}{c|x}\vertr\geq E\wedge \chi_{f_2',g_2'}\wedge\Xi\wedge\Sigma_2\in V\] if and only
\[\underline{w}_{i,(\overline{f},\overline{g})}:=\vertl\nf{x=x}{a|x}\vertr=C\wedge \vertl\nf{xb=0}{x=0}\vertr\geq E\wedge \chi_{f_2',g_2'}\wedge\Xi\wedge\Sigma_2\in V.\]
So we are done since the short signature of $\underline{w}_{i,(\overline{f},\overline{g})}$ is $(\square',\emptyset)$ where, by \ref{featheringextsig}, $\square'\preceq\square$.

\noindent
\textbf{Case $i=3$:}
By \ref{redxb=0c|xwithx=xc|x},
\[\vertl\nf{x=x}{a|x}\vertr=C\wedge \vertl\nf{xb=0}{c|x}\vertr\geq E\wedge \chi_{f_2',g_2'}\wedge\Xi\wedge\Sigma_2\in V\] if and only
\[\underline{w}_{i,(\overline{f},\overline{g})}:=\vertl\nf{x=x}{a|x}\vertr=C\wedge \vertl\nf{xb=0}{r|x}\vertr\geq \lceil E/C\rceil\wedge \chi_{f_2',g_2'}\wedge\Xi\wedge\Sigma_3\in V.\] Since $\lceil E/C\rceil<E$, we have $\deg \underline{w}_{i,(\overline{f},\overline{g})}<E=\deg\chi_{f,g}\wedge \Xi$.

\smallskip

So we have proved the lemma for the case $z_1\geq 1$. Suppose $w\in W$ has extended signature $((\square,\geq),(z_1,z_2),(z_3,z_4))$ with $z_3\geq1$. Then $Dw$ has extended signature $((\square,\geq),(z_3,z_4),(z_1,z_2))$. So by the version of the lemma just proved, we can compute $\underline{w}\in W$ with $\exsig \underline{w}<\exsig Dw$ such that $Dw\in V$ if and only if $\underline{w}\in\mathbb{V}$. Now, $w\in V$ if and only if $D\underline{w}\in\mathbb{V}$ and $\exsig D\underline{w}<\exsig w$, as required.
\end{proof}

\subsection{Reducing the extended signature}

\begin{lemma}\label{a|xx=0b|xx=0help}
Let $a,b\in R$ and $A,B\in\N_2$. Suppose that $r,s,\alpha\in R$ are such that $a\alpha=br$ and $b(\alpha-1)=as$. Define
\begin{enumerate}
\item $\Sigma_1$ to be the sentence $|\nf{x=x}{\alpha|x}|=1\wedge |\nf{rb|x}{x=0}|=1\wedge |\nf{b^A|x}{x=0}|=1$,
\item $\Sigma_2$ to be the sentence $|\nf{x=x}{\alpha|x}|=1\wedge |\nf{rb|x}{x=0}|=1\wedge |\nf{xb^{A-1}=0}{b|x}|=1$,
\item $\Sigma_3$ to be the sentence $|\nf{x=x}{\alpha|x}|=1\wedge |\nf{xb=0}{r|x}|=1$,
\item $\Sigma_4$ to be the sentence $|\nf{x=x}{(\alpha-1)|x}|=1\wedge |\nf{sa|x}{x=0}|=1\wedge |\nf{a^B|x}{x=0}|=1$,
\item $\Sigma_5$ to be the sentence $|\nf{x=x}{(\alpha-1)|x}|=1\wedge |\nf{sa|x}{x=0}|=1\wedge |\nf{xa^{B-1}=0}{a|x}|=1$, and
\item $\Sigma_6$ to be the sentence $|\nf{x=x}{(\alpha-1)|x}|=1\wedge |\nf{xa=0}{s|x}|=1$.
\end{enumerate}
Then, for all $M\in\Mod\text{-}R$, there exist $M_i\in\Mod\text{-}R$ for $1\leq i\leq 6$ such that $M\equiv \oplus_{i=1}^6M_i$ and for $1\leq i\leq 6$, $M_i\models \Sigma_i$.

Moreover, there is an algorithm which, given $a,b,r,s,\alpha,A,B$ as above, $1\leq i\leq 6$ and $\square,\square'\in\{=,\geq\}$, either returns
$\bot$, in which case,
\[T_R\models \lnot(\vertl\nf{x=x}{a|x}\vertr\square A\wedge \vertl\nf{x=x}{b|x}\vertr\square' B\wedge \Sigma_i)\]
or returns $n\in \N$ and  $\sigma_1,\ldots,\sigma_n$ such that
\[T_R\models \vertl \nf{x=x}{a|x}\vertr\square A\wedge \vertl\nf{x=x}{b|x}\vertr\square' B\wedge \Sigma_i\leftrightarrow \bigvee_{j=1}^{n}(\sigma_{j}\wedge \Sigma_i)\]
and each $\sigma_{j}$ is either of the form
\[\vertl \nf{x=x}{a'|x}\vertr\square_{j}A'\wedge \vertl\nf{x=x}{b'|x}\vertr\square'_{j}B',\] where $\square_{j},\square'_{j}\in \{\emptyset,=,\geq\}$, $\square_j\preceq \square$, $\square'_j\preceq \square'$, $a',b'\in R$, $A',B'\in\N$ and $A'B'<AB$ or of the form
\[ \vertl\nf{x=x}{a'|x}\vertr\square_{j}A'\wedge \vertl \nf{x=x}{b'|x}\vertr\square'_{j}B'\wedge \vertl\nf{x=x}{x=0}\vertr=N,\] where $\square_{j},\square'_{j}\in \{=,\geq\}$, $a',b'\in R$ and $A'',B'',N\in \N.$
\end{lemma}
\begin{proof}
Note that
\[T_R\vdash|\nf{x=x}{\alpha|x}|=1\wedge |\nf{rb|x}{x=0}|=1\rightarrow (a|x \leftrightarrow x=0).\]

\noindent
\textbf{Case 1:} $\Sigma \in\{\Sigma_1,\Sigma_2\}$ and $\square$ is $=$.

\noindent
In this case, $\vertl\nf{x=x}{a|x}\vertr\square A\wedge \vertl\nf{x=x}{b|x}\vertr\square' B\wedge \Sigma$ is equivalent to $\vertl\nf{x=x}{x=0}\vertr= A\wedge \vertl\nf{x=x}{b|x}\vertr\square' B\wedge \Sigma$ as required.

\noindent
\textbf{Case 2:} $\Sigma \in\{\Sigma_1,\Sigma_2\}$ and $\square$ is $\geq$ and $B\geq A$.

\noindent
In this case $\vertl \nf{x=x}{a|x}\vertr\square A\wedge \vertl\nf{x=x}{b|x}\vertr\square' B\wedge \Sigma$ is equivalent to $\vertl \nf{x=x}{b|x}\vertr\square' B\wedge \Sigma$.

\noindent
\textbf{Case 3:} $\Sigma=\Sigma_1$, $\square$ is $\geq$ and $\square'$ is $=$.

\noindent
If $b^A\in \ann_RM$ and $\vertl \nf{x=x}{b|x}(M)\vertr=B$ then $B\leq \vertl M\vertr\leq B^A$. Thus $\vertl\nf{x=x}{a|x}\vertr\geq A\wedge\vertl\nf{x=x}{b|x}\vertr= B\wedge \Sigma_1$ is equivalent to
\[\bigvee_{B^A\geq A'\geq A}\vertl \nf{x=x}{x=0} \vertr=A'\wedge \vertl \nf{x=x}{b|x} \vertr\geq B\wedge \Sigma_1\]

\noindent
\textbf{Case 4:} $\Sigma=\Sigma_2$, $\square$ is $\geq$ and $\square'$ is $=$.

\noindent
Suppose $M$ satisfies $\vertl\nf{xb^{A-1}=0}{b|x}\vertr=1$. Then $\vertl \nf{b^i|x}{b^{i+1}|x}(M)\vertr=\vertl \nf{x=x}{b|x}(M)\vertr$ for $i\geq A-1$. So $\vertl \nf{x=x}{b^A|x}(M)\vertr =\vertl \nf{x=x}{b|x}(M)\vertr^A$. Thus if $\vertl \nf{x=x}{b|x}(M)\vertr\geq 2$ then $\vertl M \vertr\geq A$. Therefore
$\vertl\nf{x=x}{a|x}\vertr\geq A\wedge \vertl\nf{x=x}{b|x}\vertr= B\wedge \Sigma_2$ is equivalent to $\vertl\nf{x=x}{b|x}\vertr= B\wedge \Sigma_2$.

\noindent
\textbf{Case 5:} $\Sigma \in\{\Sigma_1,\Sigma_2\}$ and $\square$ is $\geq$ and $B<A$.

\noindent
If $\square'$ is $\geq$ then
\[\vertl \nf{x=x}{a|x}\vertr\geq A\wedge \vertl \nf{x=x}{b|x}\vertr\geq B\wedge\Sigma\] is equivalent to
\[(\vertl \nf{x=x}{b|x}\vertr\geq A\vee \bigvee_{B\leq B'<A}\vertl \nf{x=x}{x=0}\vertr\geq A\wedge \vertl \nf{x=x}{b|x}\vertr=B')\wedge\Sigma\]

So we may reduce to the case where $\square'$ is $=$ and hence to either case $3$ or $4$ at the expense of replacing $B$ by $B'$ with $B\leq B'<A$. This is not a problem since in case $3$ we show that $\sigma_j$ has the form $\vertl\nf{x=x}{a'|x}\vertr\square_{j}A'\wedge \vertl\nf{x=x}{b'|x}\vertr\square'_{j}B'\wedge \vertl\nf{x=x}{x=0}\vertr=N$ i.e. there is no restriction on $A'$ or $B'$ and in case $4$ we show that $\vertl\nf{x=x}{a|x}\vertr\geq A\wedge \vertl\nf{x=x}{b|x}\vertr= B'\wedge \Sigma_2$ is equivalent to $\vertl\nf{x=x}{b|x}\vertr= B'\wedge \Sigma_2$ and $B'\leq A$.

\noindent
\textbf{Case 6:} $\Sigma = \Sigma_3$.

\noindent
Suppose that $M$ satisfies $\Sigma_3$. Then $a|x$ is equivalent to $br|x$ in $M$ and
\[\vertl \nf{x=x}{br|x}(M)\vertr=\vertl \nf{x=x}{b|x}(M)\vertr\cdot\vertl \nf{b|x}{br|x}(M)\vertr=\vertl \nf{x=x}{b|x}(M)\vertr\cdot\vertl \nf{x=x}{r|x}(M)\vertr.\] So, in this case, the result now follows from \ref{redmult}.

\smallskip

\noindent
The remaining cases follow from the cases we have already covered by exchanging the roles of $\alpha$ and $\alpha-1$, $a$ and $b$, $r$ and $s$, and $A$ and $B$.
\end{proof}

\begin{proposition}\label{redexsig}
Let $R$ be a recursive Pr\"ufer domain. There is an algorithm which, given $w\in W$ with extended signature $((\square_1,\square_2),(z_1,z_2),(z_3,z_4))$ with $z_1+z_2>1$ or $z_3+z_4>1$, outputs $\underline{w}\in\mathbb{W}$ such that $w\in V$ if and only if $\underline{w}\in \mathbb{V}$ and $\exsig\underline{w}<\exsig w$.
\end{proposition}
\begin{proof}
For $w\in W$, as in (\ref{flaW2}), define \[\deg_1 w:=\prod_{\nf{x=x}{c|x}\in X}f(\nf{x=x}{c|x})\cdot\prod_{\nf{x=x}{c|x}\in Y}g(\nf{x=x}{c|x}).\]
Given $w\in W$ with extended signature $((\square_1,\square_2),(z_1,z_2),(z_3,z_4))$ with $z_1+z_2>1$, we will show how to compute $\underline{w}\in\mathbb{W}$, a lattice combination of $w'\in W$, such that $w\in V$ if and only if $\underline{w}\in\mathbb{V}$ and such that for each $w'$, either $\exsig w'<\exsig w$, or, $\exsig w'\leq\exsig w$ and $\deg_1 w'<\deg_1 w$. Since $\deg_1$ takes values in $\N$, by iterating this process, we will eventually compute $\underline{w}\in \mathbb{W}$ which is a lattice combination of $w'\in W$ such that $\exsig w'<\exsig w$.

We start with a special case. Let $a,b,\alpha,r,s\in R$ be such that $a\alpha=br$ and $b(\alpha-1)=as$. Let $\Sigma_i$, for $1\leq i\leq 6$, be as in \ref{a|xx=0b|xx=0help}. Suppose that, for some $1\leq i\leq 6$, $w$ is
\[\vertl \nicefrac{x=x}{a|x}\vertr\square A\wedge \vertl \nicefrac{x=x}{b|x}\vertr\square'B\wedge \chi_{f,g}\wedge \Sigma_i\wedge\Xi\in W\]
where $A,B\in\N_2$, $\square,\square'\in\{=,\geq\}$, $X$ and $Y$ are finite sets of appropriate pp-pairs, $f:X\rightarrow \N$, $g:Y\rightarrow \N$ and $\Xi$ is an auxiliary sentence.

We will compute $u=\bigsqcup_{j=1}^n u_j$ such that for each $1\leq j\leq n$, either $\exsig u_j<\exsig w$, or, $\exsig u_j\leq \exsig w$ and $\deg_1 u_j<\deg_1 w$.

If the algorithm from \ref{a|xx=0b|xx=0help} returns $\bot$, then
\[T_R\models \lnot (\vertl \nicefrac{x=x}{a|x}\vertr\square A\wedge \vertl \nicefrac{x=x}{b|x}\vertr\square'B\wedge \chi_{f,g}\wedge \Sigma_i\wedge\Xi).\] In this case, set $u:=\bot$. Otherwise, let $\sigma_1,\ldots\sigma_n$ be, as in \ref{a|xx=0b|xx=0help}, such that
\[T_R\models \vertl\nicefrac{x=x}{a|x}\vertr\square A\wedge \vertl\nicefrac{x=x}{b|x}\vertr\square' B\wedge \Sigma_i\leftrightarrow \bigvee_{j=1}^{n}(\sigma_{j}\wedge \Sigma_i).\]
Therefore
\[T_R\models \vertl \nicefrac{x=x}{a|x}\vertr\square A\wedge \vertl \nicefrac{x=x}{b|x}\vertr\square'B\wedge \chi_{f,g}\wedge \Sigma_i\wedge\Xi\leftrightarrow \bigvee_{j=1}^{n}(\sigma_{j}\wedge \chi_{f,g}\wedge \Sigma_i\wedge\Xi)\]
So $w\in V$ if and only if
\[ \bigsqcup_{j=1}^{n}(\sigma_{j}\wedge \chi_{f,g}\wedge \Sigma_i\wedge\Xi)\in\mathbb{V}.\]
If $\sigma_j$ is of the form
\[ |x=x/a'|x|\square_{j}A'\wedge |x=x/b'|x|\square'_{j}B'\wedge |x=x/x=0|=N,\] where $\square_{j},\square'_{j}\in \{=,\geq\}$, $a',b'\in R$ and $A'',B'',N\in \N$ then $\sigma_{j}\wedge \chi_{f,g}\wedge \Sigma_i\wedge\Xi$ is a sentence about an $R$-module of fixed finite size. So, since $\EPP(R)$ is recursive, by \ref{EPPfinite}, we can effectively decide whether $\sigma_{j}\wedge \chi_{f,g}\wedge \Sigma_i\wedge\Xi$ holds in some $R$-module. Set $u_j:=\top$ if $\sigma_{j}\wedge \chi_{f,g}\wedge \Sigma_i\wedge\Xi$ is true in some $R$-module and $u_j:=\bot$ otherwise. Finally, if $\sigma_j$ is of the form
\[|x=x/a'|x|\square_{j}A'\wedge |x=x/b'|x|\square'_{j}B',\] where $\square_{j},\square'_{j}\in \{\emptyset,=,\geq\}$, $a',b'\in R$ and $A',B'\in\N$ with $\square_j\preceq\square$, $\square'_j\preceq\square'$ and $A'B'<AB$ then set $u_j:=\sigma_{j}\wedge \chi_{f,g}\wedge \Sigma_i\wedge\Xi$. The condition that $\square_j\preceq\square$ and $\square'_j\preceq\square'$ ensures that $\exsig u_j\leq \exsig w$. The condition that $A'B'<AB$ implies that $\deg_1u_j<\deg_1w$. So $w\in V$ if and only if $u:=\bigsqcup_{j=1}^nu_j\in \mathbb{V}$ and $\deg_1 u_j<\deg_1 w$ for $1\leq j\leq n$ as required.

\smallskip

We now consider the general case. Suppose
\[w:=\chi_{f,g}\wedge\Xi,\] where $X$ and $Y$ are disjoint finite sets of appropriate pp-pairs, $f:X\rightarrow\N_2$, $g:Y\rightarrow \N_2$ and $\Xi$ is an auxiliary sentence, has extended signature $((\square_1,\square_2),(z_1,z_2),(z_3,z_4))$ with $z_1+z_2>1$. There exist $a,b\in R$ with $a\neq b$ such that $\nf{x=x}{a|x},\nf{x=x}{b|x}\in X\cup Y$. Then $w\in V$ if and only if \[\bigsqcup_{(\overline{f},\overline{g})\in\Omega_{f,g,6}}\bigsqcap_{i=1}^6\Sigma_i\wedge\chi_{f_i,g_i}\wedge\Xi\in\mathbb{V}.\]

\smallskip
\noindent
\textbf{Claim:} For all $(\overline{f},\overline{g})\in\Omega_{f,g,6}$ and $1\leq i\leq 6$, either \[\exsig\Sigma_i\wedge\chi_{f_i,g_i}\wedge\Xi<\exsig\chi_{f,g}\wedge \Xi, \text{ or, }   \deg_1\Sigma_i\wedge\chi_{f_i,g_i}\wedge\Xi<\deg_1 \chi_{f,g},\] or, $X_i=X$, $Y_i=Y$, $f_i(\nicefrac{x=x}{c|x})=f(\nicefrac{x=x}{c|x})$ for all $\nicefrac{x=x}{c|x}\in X$ and $g_i(\nicefrac{x=x}{c|x})=g(\nicefrac{x=x}{c|x})$ for all $\nicefrac{x=x}{c|x}\in Y$.

\smallskip

By \ref{featheringextsig}, if $X_i\neq X$ or $Y_i\neq Y$ then the extended signature of $\chi_{f_i,g_i}$ is strictly less that the extended signature of $\chi_{f,g}$. So suppose that $X=X_i$, $Y=Y_i$. By definition, $g(\nicefrac{x=x}{c|x})=g_i(\nicefrac{x=x}{c|x})$ for all $\nicefrac{x=x}{c|x}\in Y=Y_i$ and $f_i(\nicefrac{x=x}{c|x})\leq f(\nicefrac{x=x}{c|x})$ for all $\nf{x=x}{c|x}\in X=X_i$. So, if $f_i(\nicefrac{x=x}{c|x})<f(\nicefrac{x=x}{c|x})$, for some $\nf{x=x}{c|x}\in X$, then $\deg_1\Sigma_i\wedge\chi_{f_i,g_i}\wedge\Xi<\deg_1 \chi_{f,g}$. So we have proved the claim.

\smallskip

Therefore for each $(\overline{f},\overline{g})\in\Omega_{f,g,6}$ and $1\leq i\leq 6$, either $\exsig\Sigma_i\wedge\chi_{f_i,g_i}\wedge\Xi<\exsig\chi_{f,g}\wedge\Xi$, $\deg_1 \chi_{f_i,g_i}\wedge\Sigma_i\wedge\Xi<\deg_1 \chi_{f,g}\wedge\Xi$ or $\chi_{f_i,g_i}\wedge\Sigma_i\wedge\Xi$ is of the form of the special case. Since, by \ref{featheringextsig}, we always have $\exsig\Sigma_i\wedge\chi_{f_i,g_i}\wedge\Xi\leq \exsig\chi_{f,g}\wedge\Xi$, we are done.

The version of the lemma with $z_3+z_4>1$ follows from the one we have just proved by applying duality as in \ref{ssremgeqxbcx}.
\end{proof}

\noindent
Say $w\in W$ is \textbf{reducible} if $w$ has
\begin{itemize}
\item short signature $(=,\square)$ or $(\square,=)$, or
\item extended signature $((\geq,\square),(z_1,z_2),(z_3,z_4))$ or $((\square,\geq),(z_1,z_2),(z_3,z_4))$ with $z_1\geq 1$ or $z_3\geq 1$, or
\item extended signature $((\square,\square'),(z_1,z_2),(z_3,z_4))$ with $z_1+z_2>1$ or $z_3+z_4>1$.
\end{itemize}

\begin{remark}
If $w\in W$ is reducible then $w$ is of the form required by one of the lemmas \ref{ssrem=}, \ref{ssremgeqdx}, \ref{ssremgeqxbcx} or \ref{redexsig}.
\end{remark}

\noindent
Thus $w\in W$ is not reducible if and only if $w$ has extended signature
\begin{itemize}
\item $(\emptyset,\emptyset,(z_1,z_2),(z_3,z_4))$ with $(z_1,z_2),(z_3,z_4)\in\{(1,0),(0,1),(0,0)\}$ or
\item $((\square,\square'),(z_1,z_2),(z_3,z_4))$ with $(\square,\square')\in\{(\geq,\emptyset),(\emptyset,\geq), (\geq,\geq),(\emptyset,\emptyset)\}$ and

    \noindent
    $(z_1,z_2),(z_3,z_4)\in\{(0,1),(0,0)\}$.
\end{itemize}
The next remark follows directly from \ref{ZieglertoTinfty}.
\begin{remark}\label{nonredtoZg}
Let $R$ be a Pr\"ufer domain. There is an algorithm which, given $w\in W$ with extended signature $((\square_1,\square_2),(z_1,z_2),(z_3,z_4))$ with $(\square_1,\square_2)\in\{(\geq,\emptyset),(\emptyset,\geq),(\geq,\geq),(\emptyset,\emptyset),\}$ and $(z_1,z_2),(z_3,z_4)\in\{(0,1),(0,0)\}$, answers whether $w\in V$  or not.
\end{remark}

\section{Algorithms for sentences which are not reducible}\label{notred}

%
%
%

\begin{lemma}\label{Xusefulform}
Let $R$ be a recursive Pr\"ufer domain with $X(R)$ recursive. There is an algorithm which, given $\lambda \in R$, $C\in\N$ and $(r,ra,\gamma,\delta)\in R^4$, answers whether there exist $h\in\N_0$, prime ideals $\mfrak{p}_i\lhd R$ and ideals $I_i\lhd R_{\mfrak{p}_i}$ for $1\leq i\leq h$ such that $(\mfrak{p}_i,I_i)\models (r,ra,\gamma,\delta)$ and $\vertl\oplus_{i=1}^hR_{\mfrak{p}_i}/\lambda R_{\mfrak{p}_i}\vertr=C$.
\end{lemma}
\begin{proof}
If $C=1$ then there is nothing to prove, so suppose that $C\neq 1$. Let $p_1,\ldots,p_l\in \mathbb{P}$ be distinct primes and $n_1,\ldots,n_l\in\N$ be such that $C=\prod_{j=1}^lp_j^{n_j}$. For each prime ideal $\mfrak{p}\lhd R$, if $\vertl R_{\mfrak{p}}/\lambda R_{\mfrak{p}}\vertr$ is finite then it is a prime power. Therefore, there exist prime ideals $\mfrak{p}_i\lhd R$ and ideals $I_i\lhd R_{\mfrak{p}_i}$ for $1\leq i\leq h$ such that $(\mfrak{p}_i,I_i)\models (r,ra,\gamma,\delta)$ and $\vertl\oplus_iR_{\mfrak{p}_i}/\lambda R_{\mfrak{p}_i}\vertr=C$ if and only if for each $1\leq j\leq l$, there exist $h_j\in\N_0$, prime ideals $\mfrak{p}_{ij}\lhd R$ and ideals $I_{ij}\lhd R_{\mfrak{p}_{ij}}$ for $1\leq i\leq h_j$ such that $(\mfrak{p}_{ij},I_{ij})\models (r,ra,\gamma,\delta)$ and $\vertl\oplus_{i=1}^{h_j}R_{\mfrak{p}_{ij}}/\lambda R_{\mfrak{p}_{ij}}\vertr=p^{n_j}$. Thus we may reduce to the case that $C=p^n$ for some $p\in\mathbb{P}$ and $n\in\N$.

We consider the cases $r=0$ and $r\neq 0$ separately.

\noindent
\textbf{Case $r=0$:}
For $\mfrak{p}\lhd R$ prime and $I\lhd R_{\mfrak{p}}$, $(\mfrak{p},I)\models (0,0,\gamma,\delta)$ if and only if $I=0$, $\gamma\notin \mfrak{p}$ and $\delta\neq 0$. So, there exist $h\in\N_0$, prime ideals $\mfrak{p}_i\lhd R$ and ideals $I_i\lhd R_{\mfrak{p}_i}$ for $1\leq i\leq h$ such that $(\mfrak{p}_i,I_i)\models (0,0,\gamma,\delta)$ and $\vertl\oplus_{i=1}^hR_{\mfrak{p}_i}/\lambda R_{\mfrak{p}_i}\vertr=p^n$ if and only if $(p,n;\lambda,\gamma,0,1)\in X(R)$ and $\delta\neq 0$.

\noindent
\textbf{Case $r\neq 0$:}
If $r\neq 0$ then $(\mfrak{p},I)\models (r,ra,\gamma,\delta)$ if and only if $I=rJ$ for some $J\lhd R_{\mfrak{p}}$ and $(\mfrak{p},J)\models (1,a,\gamma,\delta)$. So, there exist $h\in\N_0$, prime ideals $\mfrak{p}_i\lhd R$ and ideals $I_i\lhd R_{\mfrak{p}_i}$ for $1\leq i\leq h$ such that $(\mfrak{p}_i,I_i)\models (r,ra,\gamma,\delta)$ and $\vertl\oplus_{i=1}^hR_{\mfrak{p}_i}/\lambda R_{\mfrak{p}_i}\vertr=p^n$ if and only if there exist $h\in\N_0$, prime ideals $\mfrak{p}_i\lhd R$ and ideals $J_i\lhd R_{\mfrak{p}_i}$ for $1\leq i\leq h$ such that $(\mfrak{p}_i,J_i)\models (1,a,\gamma,\delta)$ and $\vertl\oplus_{i=1}^hR_{\mfrak{p}_i}/\lambda R_{\mfrak{p}_i}\vertr=p^n$. The last statement is equivalent to $(p,n;\lambda,\gamma,a,\delta)\in X(R)$.
\end{proof}

\begin{proposition}\label{x=x/c|x=}
Let $R$ be a recursive Pr\"ufer domain with $\EPP(R)$ and $X(R)$ recursive. There is an algorithm which, given $c\in R$, $C\in \N$ and $\Xi$ an auxiliary sentence, answers whether there exists $M\in\Mod\text{-}R$ such that \[M\models \vertl \nicefrac{x=x}{c|x}\vertr=C\wedge \Xi.\]
\end{proposition}
\begin{proof}
Let $\chi$ be the sentence $\vertl \nicefrac{x=x}{c|x}\vertr=C\wedge \Xi$. By \ref{REFuni}, there exists $M\models \chi$ if and only if there exist prime ideals $\mfrak{p}_i\lhd R$ and uniserial $R_{\mfrak{p}_i}$-modules $U_i$ for $1\leq i\leq l$ such that $\oplus_{i=1}^lU_i\models \chi$. Moreover, we may assume that $U_i/U_ic\neq 0$ for all $1\leq i\leq l$. 
By \ref{x=x/c|xfinnonzero}, for each $1\leq i\leq l$, either $c\in \ann_R U_i$ or, for some ideal $I_i\lhd R_{\mfrak{p}_i}$, $U_i\cong R_{\mfrak{p}_i}/cI_i$. Thus, there exists $M\models \chi$ if and only if there exist $A,B\in \N$ with $AB=C$, $F\in\Mod\text{-}R$ such that
\[F\models \vertl \nf{x=x}{x=0}\vertr=A\wedge \vertl \nf{c|x}{x=0}\vertr=1\wedge \Xi\] and $h\in\N_0$, prime ideals $\mfrak{p}_i\lhd R$ and ideals $I_i\lhd R_{\mfrak{p}_i}$ for $1\leq i\leq h$ such that $\vertl \oplus_{i=1}^hR_{\mfrak{p}_i}/cR_{\mfrak{p}_i}\vertr=B$ and $R_{\mfrak{p}_i}/cI_i\models \Xi$ for $1\leq i\leq h$. So, since $\EPP(R)$ is recursive, by \ref{EPPfinite},
it is enough to show that there is an algorithm which answers whether there exist $h\in\N_0$, prime ideals $\mfrak{p}_i\lhd R$ and ideals $I_i\lhd R_{\mfrak{p}_i}$ for $1\leq i\leq h$ such that $\vertl \oplus_{i=1}^h R_{\mfrak{p}_i}/cR_{\mfrak{p}_i}\vertr=B$ and $R_{\mfrak{p}_i}/cI_i\models \Xi$ for $1\leq i\leq h$.

By \ref{auxtoragammadelta}, we can compute $(r_j,r_ja_j,\gamma_j,\delta_j)$ for $1\leq j\leq n$ such that $(\mfrak{p},I)\models (r_j,r_ja_j,\gamma_j,\delta_j)$ for some $1\leq j\leq n$ if and only if $R_{\mfrak{p}}/cI\models \Xi$. Thus there exist $h\in\N_0$, prime ideals $\mfrak{p}_i\lhd R$ and ideals
$I_i\lhd R_{\mfrak{p}_i}$ for $1\leq i\leq h$ such that $\vertl \oplus_{i=1}^h R_{\mfrak{p}_i}/cR_{\mfrak{p}_i}\vertr=B$ and $R_{\mfrak{p}_i}/cI_i\models \Xi$ for $1\leq i\leq h$ if and only if there exist $B_j\in\N$ for $1\leq j\leq n$ such that $B=\prod_{j=1}^nB_j$ and for $1\leq j\leq n$, there exist $h_j\in\N_0$, prime ideals $\mfrak{p}_{ij}\lhd R$ and ideals $I_{ij}\lhd R_{\mfrak{p}_{ij}}$ for $1\leq i\leq h_j$ such that $\vertl \oplus_{i=1}^{h_j} R_{\mfrak{p}_{ij}}/cR_{\mfrak{p}_{ij}}\vertr=B_j$ and $(\mfrak{p}_{ij},I_{ij})\models (r_j,r_ja_j,\gamma_j,\delta_j)$ for $1\leq i\leq h_j$. The result now follows from \ref{Xusefulform}.
\end{proof}

\begin{cor}\label{xb=0x=0=}
Let $R$ be a recursive Pr\"ufer domain with $\EPP(R)$ and $X(R)$ recursive. There is an algorithm which, given $b\in R$, $B\in \N$ and $\Xi$ an auxiliary sentence, answers whether there exists $M\in\Mod\text{-}R$ such that \[M\models \vertl \nicefrac{xb=0}{x=0}\vertr=B\wedge \Xi.\]
\end{cor}
\begin{proof}
Apply \ref{Dualsent} to \ref{x=x/c|x=}.
\end{proof}

\begin{proposition}\label{x=x/c|x=xb=0/x=0=}
Let $R$ be a recursive Pr\"ufer domain with $\EPP(R)$ and $X(R)$ recursive. There is an algorithm which, given $c,b\in R$, $C,B\in \N$ and $\Xi$ an auxiliary sentence, answers whether there exists $M\in\Mod\text{-}R$ such that \[M\models \vertl \nicefrac{x=x}{c|x}\vertr=C\wedge \vertl \nicefrac{xb=0}{x=0}\vertr= B\wedge \Xi.\]
\end{proposition}
\begin{proof}
For $\alpha\in R$, we write $\alpha\notin \Att$ for the sentence $\vertl \nf{x\alpha=0}{x=x}\vertr=1\wedge \vertl\nf{x=x}{\alpha|x}\vertr=1$. Recall, \ref{decomposeorder}, that for all $M\in\Mod\text{-}R$, there are $M_1,M_2\in\Mod\text{-}R$ such that $M\equiv M_1\oplus M_2$, $M_1$ satisfies $\alpha\notin \Att$ and $M_2$ satisfies $\alpha-1\notin \Att$.

Let $\alpha,u,v\in R$ be such that $c\alpha=bu$ and $b(\alpha-1)=cv$. There exists an $R$-module which satisfies $\vertl \nicefrac{x=x}{c|x}\vertr=C\wedge \vertl \nicefrac{xb=0}{x=0}\vertr= B\wedge \Xi$ if and only if there exist $C_1,C_2,B_1,B_2\in \N$ with $C=C_1C_2$ and $B=B_1B_2$ and there exists an $R$-module satisfying
\[\vertl \nicefrac{x=x}{bu|x}\vertr=C_1\wedge \vertl \nicefrac{xb=0}{x=0}\vertr= B_1\wedge \alpha\notin \textrm{Att} \wedge \Xi\]

\noindent
and an $R$-module satisfying
\[\vertl \nicefrac{x=x}{c|x}\vertr=C_2\wedge \vertl \nicefrac{xcv=0}{x=0}\vertr= B_2\wedge (\alpha-1)\notin\textrm{Att}\wedge \Xi.\]
By \ref{Dualsent}, there exists an $R$-module satisfying
\[\vertl \nicefrac{x=x}{c|x}\vertr=C_2\wedge \vertl \nicefrac{xcv=0}{x=0}\vertr= B_2\wedge \alpha\notin\textrm{Att}\wedge \Xi\] if and only if there exists an $R$-module satisfying
\[\vertl \nicefrac{xc=0}{x=0}\vertr=C_2\wedge \vertl \nicefrac{x=x}{cv|x}\vertr= B_2\wedge \alpha\notin\textrm{Att}\wedge D\Xi.\] Thus, in order to prove the proposition, it is enough to show that there is an algorithm which, given $b,u\in R$, $C,B\in \N$ and $\Xi$ an auxiliary sentence, answers whether there exists an $R$-module satisfying the sentence $\chi$ defined as
\[\vertl \nicefrac{x=x}{bu|x}\vertr=C\wedge \vertl \nicefrac{xb=0}{x=0}\vertr= B\wedge \Xi.\]
We may assume that $bu\neq 0$, for otherwise $\chi$ is a sentence about an $R$-module of fixed finite size and since $\EPP(R)$ is recursive, we can decide whether there exist $R$-modules satisfying such sentences.

By  \ref{REFuni} and \ref{x=x/c|xfinnonzero}, there exists an $R$-module satisfying $\chi$ if and only if there exists $F\in\Mod\text{-}R$ with $bu\in\ann_RF$, $h\in\N_0$, prime ideals $\mfrak{p}_i\lhd R$ and ideals $I_i\lhd R_{\mfrak{p}_i}$ for $1\leq i\leq h$ and $M\in\Mod\text{-}R$ with $Mbu=M$ such that $F\oplus \bigoplus_{i=1}^h R_{\mfrak{p}_i}/buI_i \oplus M$ satisfies $\chi$. Now, this happens if and only if there exist $C_1,C_2\in \N$ and $B_1,B_2,B_3\in \N$ with $C=C_1C_2$ and $B=B_1B_2B_3$ such that
\[
F\models \vertl \nicefrac{x=x}{x=0}\vertr=C_1\wedge \vertl \nicefrac{xb=0}{x=0}\vertr= B_1\wedge \vertl \nicefrac{bu|x}{x=0}\vertr=1\wedge  \Xi,
\]
\[
\oplus_{i=1}^h R_{\mfrak{p}_i}/buI_i\models \vertl \nicefrac{x=x}{bu|x}\vertr=C_2\wedge \vertl \nicefrac{xb=0}{x=0}\vertr= B_2\wedge \Xi\text{, and}
\]
\[
M\models  \vertl \nicefrac{xb=0}{x=0}\vertr= B_3\wedge\vertl \nicefrac{x=x}{bu|x}\vertr=1\wedge \Xi.
\]
In view of \ref{EPPfinite} and \ref{xb=0x=0=}, it is therefore enough to show that there is an algorithm which answers whether there exists $h\in\N_0$, prime ideals $\mfrak{p}_i\lhd R$ and ideals $I_i\lhd R_{\mfrak{p}_i}$ for $1\leq i\leq h$ such that
\[
\oplus_{i=1}^h R_{\mfrak{p}_i}/buI_i\models \vertl \nicefrac{x=x}{bu|x}\vertr=C_2\wedge \vertl \nicefrac{xb=0}{x=0}\vertr= B_2\wedge \Xi.
\]
By \ref{auxtoragammadelta}, we can compute $n\in\N$ and $(r_j,r_ja_j,\gamma_j,\delta_j)$ for $1\leq j\leq n$ such that $R_{\mfrak{p}}/buI\models \Xi$ if and only if $(\mfrak{p},I)\models (r_j,r_ja_j,\gamma_j,\delta_j)$ for some $1\leq j\leq n$. It is therefore enough to show that there is an algorithm which, given $(r,ra,\gamma,\delta)$, $b,u\in R$ and $C,B\in \N$, answers whether there exist $h\in\N_0$, prime ideals $\mfrak{p}_i$ for $1\leq i\leq h$ and ideals $I_i\lhd R_{\mfrak{p}_i}$ for $1\leq i\leq h$ such that $(\mfrak{p}_i,I_i)\models (r,ra,\gamma,\delta)$ and \[\oplus_{i=1}^h R_{\mfrak{p}_i}/buI_i\models \vertl \nicefrac{x=x}{bu|x}\vertr=C\wedge \vertl \nicefrac{xb=0}{x=0}\vertr= B.\]

\smallskip
\noindent
\textbf{Case $r=0$:}
In this case $(\mfrak{p},I)\models (r,ra,\gamma,\delta)$ implies $I=0$. Moreover $(\mfrak{p},0)\models (r,ra,\gamma,\delta)$ if and only if $\gamma\notin \mfrak{p}$ and $\delta\neq 0$. Thus, there exist prime ideals $\mfrak{p}_i$ for $1\leq i\leq h$ such that $(\mfrak{p}_i,0)\models (r,ra,\gamma,\delta)$ and $\oplus_{i=1}^h R_{\mfrak{p}_i}\models \vertl \nicefrac{x=x}{bu|x}\vertr=C\wedge \vertl \nicefrac{xb=0}{x=0}\vertr= B$ if and only if $\delta\neq 0$, $B=1$ and there exist prime ideals $\mfrak{p}_i$ for $1\leq i\leq h$ such that $\gamma\notin \mfrak{p}_i$ and $|\oplus_{i=1}^h R_{\mfrak{p}_i}/buR_{\mfrak{p}_i}|=C$. Such an algorithm exists since $\EPP(R)$ is recursive.

\smallskip
\noindent
\textbf{Case $r\neq 0$:}
For all prime ideals $\mfrak{p}$ and ideals $I\lhd R_{\mfrak{p}}$, since $bu\neq 0$,
\[
\vertl xb=0/x=0(R_{\mfrak{p}}/buI)\vertr= \vertl (buI:b)/buI\vertr=\vertl I/bI\vertr.
\]
Now, if $\vertl I/bI\vertr$ is finite but not equal to $1$ then $I=\lambda R_{\mfrak{p}}$ for some $\lambda\neq 0$ and $\vertl I/bI\vertr=\vertl R_{\mfrak{p}}/bR_{\mfrak{p}}\vertr$. Since $b\neq 0$, $\vertl I/bI\vertr =1$ if and only if $b\notin I^\#$.

Therefore, there exist prime ideals $\mfrak{p}_i$ and ideals $I_i\lhd R_{\mfrak{p}_i}$ for $1\leq i\leq h$ such that
$\oplus_{i=1}^h R_{\mfrak{p}_i}/buI_i\models \vertl \nicefrac{x=x}{bu|x}\vertr=C\wedge \vertl \nicefrac{xb=0}{x=0}\vertr= B\wedge \Xi$ if and only if there exist $C',C''\in\N$ with $C'C''=C$ such that the following conditions hold.
\begin{itemize}
\item [(i)] There exist prime ideals $\mfrak{p}_i\lhd R$ and $\lambda_i\in R\backslash\{0\}$ for $1\leq i\leq h$ such that $(\mfrak{p}_i,\lambda_iR_{\mfrak{p}_i})\models (r,ra,\gamma,\delta)$, $\vertl \oplus_{i=1}^{h}R_{\mfrak{p}_i}/buR_{\mfrak{p}_i}\vertr =C'$
    and $\vertl\oplus_{i=1}^{h} R_{\mfrak{p}_i}/bR_{\mfrak{p}_i}\vertr =B$.


\item [(ii)] There exist prime ideals $\mfrak{p}_i\lhd R$ and ideals $I_i\lhd R_{\mfrak{p}_i}$ for $1\leq i\leq h$ such that $(\mfrak{p}_i,I_i)\models (r,ra,\gamma,\delta b)$ and $\vertl \oplus_{i=1}^hR_{\mfrak{p}_i}/buR_{\mfrak{p}_i}\vertr =C''$.
\end{itemize}
Note that, if $(\mfrak{p},\lambda R_{\mfrak{p}})\models (r,ra,\gamma,\delta)$ then $(\mfrak{p}, rR_{\mfrak{p}})\models (r,ra,\gamma,\delta)$. Since $(rR_{\mfrak{p}})^\#=\mfrak{p}R_{\mfrak{p}}$, $(\mfrak{p}, rR_{\mfrak{p}})\models (r,ra,\gamma,\delta)$ if and only if $\gamma\delta\notin\mfrak{p}$. So $(i)$ holds if and only if there exist prime ideals $\mfrak{p}_i$ for $1\leq i\leq h'$ such that $\gamma\delta\notin\mfrak{p}_i$ and $\vertl\oplus_{i=1}^{h'} R_{\mfrak{p}_i}/buR_{\mfrak{p}_i}\vertr =C'$ and $\vertl\oplus_{i=1}^{h'} R_{\mfrak{p}_i}/bR_{\mfrak{p}_i}\vertr =B$. So, since $\EPP(R)$ is recursive, by \ref{USEEPP}, there is an algorithm which answers whether (i) holds or not.

By \ref{Xusefulform}, since $X(R)$ is recursive, there is an algorithm which answers whether (ii) holds or not.
\end{proof}

\noindent
The rest of this section is spent proving the following proposition.

\begin{proposition}\label{x=x/c|x=xb=0/x=0geq}
Let $R$ be a recursive Pr\"ufer domain with $\EPP(R)$ and $\DPR(R)$ recursive. There is an algorithm which, given $c,b\in R$, $C,B\in \N$ and $\Xi$ an auxiliary sentence, answers whether there exists $M\in\Mod\text{-}R$ such that \[M\models \vertl \nicefrac{x=x}{c|x}\vertr=C\wedge \vertl \nicefrac{xb=0}{x=0}\vertr\geq B\wedge \Xi.\]
\end{proposition}

We could choose the module in the following definition uniquely.  For instance, one can show that when $\mfrak{q}\supsetneq \mfrak{p}$ the uniserial module $[\mfrak{q}R_{\mfrak{q}}:\lambda\mfrak{p}R_{\mfrak{q}}]/\mfrak{q}R_{\mfrak{q}}$ has the required theory where $[\mfrak{q}R_{\mfrak{q}}:\lambda\mfrak{p}R_{\mfrak{q}}]$ is the set of elements $a\in Q$, the fraction field of $R$, such that $a\lambda\mfrak{q}R_{\mfrak{q}}\subseteq \mfrak{q}R_{\mfrak{q}}$. However, we are only ever interested in modules up to elementary equivalence.

\begin{definition}
For $\lambda\in R\backslash\{0\}$ and prime ideals $\mfrak{p},\mfrak{q}\lhd R$, let $M(\mfrak{p},\mfrak{q},\lambda)$ be $R_{\mfrak{p}}/\lambda \mfrak{q}R_{\mfrak{p}}$ if $\mfrak{p}\supseteq \mfrak{q}$ and a module with theory dual, in the sense of \cite[6.6]{HerDual}, to the theory of $R_{\mfrak{q}}/\lambda \mfrak{p}R_{\mfrak{q}}$ if $\mfrak{q}\supsetneq \mfrak{p}$.
%
%
\end{definition}


\begin{lemma}\label{MpqXi}
Suppose that $\lambda,a,\gamma,\delta\in R$ with $\lambda\neq 0$ are such that if $(\mfrak{p},I)\models (\lambda,\lambda a,\gamma,\delta)$ then $R_{\mfrak{p}}/I\models \Xi$ and if $(\mfrak{p},I)\models (\lambda,\lambda a,\delta,\gamma)$ then $R_{\mfrak{p}}/I\models D\Xi$. Then $\gamma\notin \mfrak{p}$, $\delta\notin \mfrak{q}$ and $a\in\mfrak{p}\cap\mfrak{q}$ implies $M(\mfrak{p},\mfrak{q},\lambda)\models \Xi$.
\end{lemma}
\begin{proof}
Let $(\lambda,\lambda a,\gamma,\delta)$ be such that if $(\mfrak{p},I)\models (\lambda,\lambda a,\gamma,\delta)$ then $R_{\mfrak{p}}/I\models \Xi$ and if $(\mfrak{p},I)\models (\lambda,\lambda a,\delta,\gamma)$ then $R_{\mfrak{p}}/I\models D\Xi$.

Suppose that $\gamma\notin \mfrak{p}$, $\delta\notin \mfrak{q}$ and $a\in\mfrak{p}\cap\mfrak{q}$. If $\mfrak{p}\supseteq \mfrak{q}$ then $(\mfrak{p},\lambda\mfrak{q}R_{\mfrak{p}})\models (\lambda,\lambda a,\gamma,\delta)$. So $M(\mfrak{p},\mfrak{q},\lambda)\models \Xi$. Now suppose that $\mfrak{q}\supsetneq \mfrak{p}$. Then $(\mfrak{q},\lambda\mfrak{p}R_{\mfrak{q}})\models (\lambda,\lambda a, \delta,\gamma)$. So $R_{\mfrak{q}}/\lambda\mfrak{p}R_{\mfrak{q}}\models D\Xi$. Hence, by definition, $M(\mfrak{p},\mfrak{q},\lambda)\models \Xi$.
\end{proof}

\begin{lemma}\label{reint}
Let $R$ be a recursive Pr\"{u}fer domain with $\EPP(R)$ recursive. There is an algorithm which, given $r,c\in R\backslash\{0\}$, $a,b,\gamma,\delta\in R$ and $A,B,C\in \N$, answers whether there exist $h\in\N_0$, prime ideals $\mfrak{p}_i\lhd R$ and ideals $I_i\lhd R_{\mfrak{p}_i}$ for $1\leq i\leq h$ such that $(\mfrak{p}_i,I_i)\models (1,a,\gamma,\delta)$, $\vertl\oplus_{i=1}^hR_{\mfrak{p}_i}/I_i\vertr=A$ and
\[\oplus_{i=1}^hR_{\mfrak{p}_i}/cI_i\models\vertl \nicefrac{x=x}{c|x}\vertr=C\wedge\vertl \nicefrac{xb=0}{x=0}\vertr\geq B.\]
\end{lemma}
\begin{proof}
Let $\mfrak{p}\lhd R$ be a prime ideal and $I\lhd R_{\mfrak{p}}$ be an ideal. Then $a\in I_i$ if and only if $rca\in rcI_i$. So $a\in I_i$ if and only if $\vertl \nicefrac{rca|x}{x=0}(R_{\mfrak{p}}/rcI)\vertr=1$. If $R_{\mfrak{p}}/rcI\neq 0$ then $\gamma\notin\mfrak{p}$ if and only if $\vertl \nicefrac{x=x}{\gamma|x}\vertr=1$, and, $\delta\notin I^\#$ if and only if $\vertl \nicefrac{x\delta=0}{x=0}\vertr=1$. Note that
\[\vertl \nicefrac{rc|x}{x=0}(\oplus_{i=1}^hR_{\mfrak{p}_i}/rcI_i)\vertr=\vertl\oplus_{i=1}^hR_{\mfrak{p}_i}/I_i\vertr.\]
Therefore, there exist $h\in \N_0$, prime ideals $\mfrak{p}_i\lhd R$ and ideals $I_i\lhd R_{\mfrak{p}_i}$ for $1\leq i\leq h$ as in the statement if and only if there exist $h\in \N_0$, prime ideals $\mfrak{p}_i\lhd R$ and ideals $I_i\lhd R_{\mfrak{p}_i}$ for $1\leq i\leq h$ such that $\oplus_{i=1}^hR_{\mfrak{p}_i}/rcI_i$ satisfies
\begin{multline*}
  \chi:=\vertl \nicefrac{rc|x}{x=0}\vertr=A\wedge\vertl \nicefrac{x=x}{c|x}\vertr=C\wedge\vertl \nicefrac{xb=0}{x=0}\vertr\geq B\\ \wedge\vertl \nicefrac{rca|x}{x=0}\vertr=1 \wedge\vertl \nicefrac{x=x}{\gamma|x}\vertr=1\wedge\vertl \nicefrac{x\delta=0}{x=0}\vertr=1.
\end{multline*}
%
%
By \ref{removed|x/x=0=}, there is an algorithm answering whether there exists $\oplus_{i=1}^hR_{\mfrak{p}_i}/rcI_i$ satisfying $\chi$.
\end{proof}

\begin{proposition}\label{=geqhelp}
Let $R$ be a recursive Pr\"ufer domain with $\EPP(R)$, $\DPR(R)$ and $X(R)$ recursive. There is an algorithm which, given $C,B\in\N$, $c,b\in R$ with $c\neq 0$ and $r,a,\gamma,\delta\in R$ with $r\neq 0$, answers whether there exists $M\in\Mod\text{-}R$ satisfying
\[
\vertl \nicefrac{x=x}{c|x}\vertr =C\wedge \vertl \nicefrac{xb=0}{x=0}\vertr\geq B,
\]
such that $M$ is a direct sum of
\begin{itemize}
\item  modules of the form $R_\mfrak{p}/rcI$ where $\mfrak{p}$ is prime, $I\lhd R_{\mfrak{p}}$ is an ideal and $(\mfrak{p},I)\models (1,a,\gamma,\delta)$, and,
\item  modules of the form $M(\mfrak{p},\mfrak{q},r)$ where $\mfrak{p},\mfrak{q}\lhd R$ are prime ideals such that $\mfrak{p}+\mfrak{q}\neq R$, $c\gamma\notin \mfrak{p}$, $\delta\notin \mfrak{q}$, $a\in \mfrak{p}$ and $a\in \mfrak{q}$.
\end{itemize}
\end{proposition}

\begin{proof}

Recall that, \ref{DPRimpDPRn}, if $\text{DPR}(R)$ is recursive then so is $\text{DPR}_2(R)$.

\smallskip

\noindent
\textbf{Case 1:} $(c\gamma,\delta,a,a,a,b)\notin \text{DPR}_2(R)$

\noindent
There exist prime ideals, $\mfrak{p},\mfrak{q}$ with $\mfrak{p}+\mfrak{q}\neq R$ such that $c\gamma\notin \mfrak{p}$, $\delta\notin \mfrak{q}$, $a\in\mfrak{p}$ and $a,b\in \mfrak{q}$. So, $|\nicefrac{x=x}{c|x}(M(\mfrak{p},\mfrak{q},r))|=1$, since $c\notin\mfrak{p}$, and, $|\nicefrac{xb=0}{x=0}(M(\mfrak{p},\mfrak{q},r))|>1$, since $b\in\mfrak{q}$.

Therefore, there exists $M\in\Mod\text{-}R$ as in the statement if and only if there exists $h\in\N_0$, prime ideals $\mfrak{p}_i\lhd R$ and ideals $I_i\lhd R_{\mfrak{p}_i}$ such that $(\mfrak{p},I_i)\models (1,a,\gamma,\delta)$ for $1\leq i\leq h$ and such that
\[\vertl \oplus_{i=1}^hR_{\mfrak{p}_i}/cR_{\mfrak{p}_i}\vertr=\vertl \nicefrac{x=x}{c|x}(\oplus_{i=1}^hR_{\mfrak{p}_i}/rcI_i)\vertr=C.\] Since $X(R)$ is recursive, we are done by \ref{Xusefulform}.

\smallskip

\noindent
\textbf{Case 2:} $(c\gamma,\delta,a,a,a,b)\in \text{DPR}_2(R)$

\noindent
Then $|\nicefrac{xb=0}{x=0}(M(\mfrak{p},\mfrak{q},r))|=1$ and $|\nicefrac{x=x}{c|x}(M(\mfrak{p},\mfrak{q},r))|=1$ for all primes $\mfrak{p},\mfrak{q}$ such that $\mfrak{p}+\mfrak{q}\neq R$, $c\gamma\notin \mfrak{p}$, $\delta\notin \mfrak{q}$, $a\in \mfrak{p}$ and $a\in \mfrak{q}$.

By \ref{DPRradical}, there exist $n\in\N$, $\epsilon,t,s_1,s_2\in R$ such that \[(\epsilon\gamma c)^n=at \ \ \text{and} \ \ ((\epsilon-1)\delta)^n=as_1+bs_2.\] For all prime ideals $\mfrak{p}\lhd R$, either $\epsilon\notin \mfrak{p}$ or $\epsilon-1\notin \mfrak{p}$. Thus $(\mfrak{p},I)\models (1,a,\gamma,\delta)$ if and only if $(\mfrak{p},I)\models (1,a,\epsilon\gamma,\delta)$ or $(\mfrak{p},I)\models (1,a,(\epsilon-1)\gamma,\delta)$.

Therefore, it is enough to be able to effectively answer whether there exist $C_1,C_2,B_1,B_2\in\N$ with $C_1\cdot C_2=C$, $B_1\cdot B_2\geq B$ and $B_1,B_2\leq B$, such that
\begin{enumerate}
\item there is a sum of modules of the form $R_\mfrak{p}/rcI$ where $(\mfrak{p},I)\models (1,a,\epsilon\gamma,\delta)$ which satisfy $\vertl \nicefrac{x=x}{c|x}\vertr =C_1\wedge \vertl \nicefrac{xb=0}{x=0}\vertr\geq B_1$, and
\item there is a sum of modules of the form $R_\mfrak{p}/rcI$ where $(\mfrak{p},I)\models (1,a,(\epsilon-1)\gamma,\delta)$ which satisfy $\vertl \nicefrac{x=x}{c|x}\vertr =C_2\wedge \vertl \nicefrac{xb=0}{x=0}\vertr\geq B_2$.
\end{enumerate}
Suppose that $(\mfrak{p},I)\models (1,a,\epsilon\gamma,\delta)$ and $\vertl \nicefrac{x=x}{c|x}(R_{\mfrak{p}}/rcI)\vertr\leq C_1$. Then
\[\vertl R_{\mfrak{p}}/I\vertr\leq \vertl R_{\mfrak{p}}/atR_{\mfrak{p}}\vertr=\vertl R_{\mfrak{p}}/(\epsilon \gamma c)^nR_{\mfrak{p}}\vertr=\vertl R_{\mfrak{p}}/c^nR_{\mfrak{p}}\vertr\leq \vertl R_{\mfrak{p}}/cR_{\mfrak{p}}\vertr^n\] because $a\in I$ and $\gamma\epsilon\notin \mfrak{p}$.
So, $\vertl R_{\mfrak{p}}/cR_{\mfrak{p}}\vertr=\vertl \nicefrac{x=x}{c|x}(R_{\mfrak{p}}/rcI)\vertr\leq C_1$ implies $\vertl R_{\mfrak{p}}/I\vertr\leq C_1^n$. Therefore, there exists a sum of modules as in $(1)$ if and only if there is $A\leq C_1^n$ such that there exist $h\in\N_0$, prime ideals $\mfrak{p}_i\lhd R$ and ideals $I_i\lhd R_{\mfrak{p}_i}$ for $1\leq i\leq h$ with $(\mfrak{p}_i,I_i)\models (1,a,\epsilon\gamma,\delta)$, $\vertl\oplus_{i=1}^hR_{\mfrak{p}_i}/I_i\vertr=A$ and
\[\oplus_{i=1}^hR_{\mfrak{p}_i}/rcI_i\models \vertl \nicefrac{x=x}{c|x}\vertr =C_1\wedge \vertl \nicefrac{xb=0}{x=0}\vertr\geq B_1.\] By \ref{reint} and \ref{removed|x/x=0=}, there is an algorithm which answers this question.

Suppose that $(\mfrak{p},I)\models (1,a,(\epsilon-1)\gamma,\delta)$. Since $a\in I$, either $a\notin \mfrak{p}$ and $I=R_{\mfrak{p}}$, or, $a\in I^\#$.
If $a\in I^\#$ then, since $(\epsilon-1)\delta\notin I^\#$, $bs_2=((\epsilon-1)\delta)^n-as_1\notin I^\#$. So $b\notin I^\#$ and hence $\vertl \nf{xb=0}{x=0}(R_{\mfrak{p}}/rcI)\vertr=1$.

%

Thus, there exists a sum of modules as required in $(2)$ if and only if there exist $C_2',C_2''\in \N$ with $C_2=C_2'C_2''$ such that
\begin{enumerate}[(i)]
\item there exist $h\in \N_0$, prime ideals $\mfrak{p}_i\lhd R$ and ideals $I_i\lhd R_{\mfrak{p}_i}$ for $1\leq i\leq h$ such that $(\mfrak{p}_i,I_i)\models(1,a,(\epsilon-1)\delta b,\gamma)$ and
\[\oplus_{i=1}^hR_{\mfrak{p}_i}/rcI_i\models \vertl \nicefrac{x=x}{c|x}\vertr=C_2'\text{, and}\]
\item there exist $h\in \N_0$ and prime ideals $\mfrak{p}_i\lhd R$ for $1\leq i\leq h$ such that $a\notin\mfrak{p}_i$, $\gamma\notin\mfrak{p}_i$, $\delta\notin\mfrak{p}_i$ and
\[\oplus_{i=1}^hR_{\mfrak{p}_i}/rcR_{\mfrak{p}_i}\models \vertl \nicefrac{x=x}{c|x} \vertr=C_2''\wedge \vertl\nicefrac{xb=0}{x=0}\vertr\geq B_2.\]
\end{enumerate}
Since for $\mfrak{p}\lhd R$ prime and $I\lhd R_{\mfrak{p}}$, $\vertl\nicefrac{x=x}{c|x}(R_{p}/rcI)\vertr=\vertl R_{\mfrak{p}}/cR_{\mfrak{p}}\vertr$, by \ref{Xusefulform}, there is an algorithm which answers whether (i) holds.

To conclude the proof we need to show that we can effectively answer whether (ii) holds or not. Let $\alpha,u,v\in R$ be such that $b\alpha=rcu$ and $rc(\alpha-1)=bv$. If $\alpha\notin\mfrak{p}$ then
\[
\vertl\nicefrac{x=x}{c|x}(R_{\mfrak{p}}/crR_{\mfrak{p}})\vertr=\vertl R_{\mfrak{p}}/cR_{\mfrak{p}}\vertr \ \ \text{ and }\ \ \vertl\nicefrac{xb=0}{x=0}(R_{\mfrak{p}}/crR_{\mfrak{p}})\vertr=\vertl R_{\mfrak{p}}/crR_{\mfrak{p}}\vertr.
\]
If $\alpha-1\notin\mfrak{p}$ then
\[\vertl\nicefrac{x=x}{c|x}(R_{\mfrak{p}}/crR_{\mfrak{p}})\vertr=\vertl R_{\mfrak{p}}/cR_{\mfrak{p}}\vertr \ \ \text{ and } \ \ \vertl\nicefrac{xb=0}{x=0}(R_{\mfrak{p}}/crR_{\mfrak{p}})\vertr=\vertl R_{\mfrak{p}}/bR_{\mfrak{p}}\vertr.\]
Since for all prime ideals $\mfrak{p}\lhd R$, either $\alpha\notin \mfrak{p}$ or $\alpha-1\notin\mfrak{p}$, by \ref{fincond}, there is an algorithm which answers whether (ii) holds or not.
\end{proof}

\begin{lemma}\label{Rp}
Let $R$ be a recursive Pr\"ufer domain with $\EPP(R)$ and $\DPR(R)$ recursive. There is an algorithm which, given $b,c,\gamma\in R$ and $B,C\in\N$, answers whether there exist $h\in\N_0$ and prime ideals $\mfrak{p}_i\lhd R$ with $\gamma\notin\mfrak{p}_i$ for $1\leq i\leq h$ such that
\[\bigoplus_{i=1}^hR_{\mfrak{p}_i}\models \vertl\nf{x=x}{c|x}\vertr=C\wedge \vertl \nf{xb=0}{x=0}\vertr\geq B.\]
\end{lemma}
\begin{proof}
We split the proof into $3$ cases. Let $\chi$ be the sentence
\[
\vertl\nf{x=x}{c|x}\vertr=C\wedge \vertl \nf{xb=0}{x=0}\vertr\geq B.
\]

\noindent
\textbf{Case $b\neq 0$:}
Then $\vertl\nf{xb=0}{x=0}(R_{\mfrak{p}})\vertr=1$ for all prime ideals $\mfrak{p}\lhd R$. So
$\oplus_{i=1}^hR_{\mfrak{p}_i}\models \chi$ if and only if $B=1$ and $\vertl \oplus_{i=1}^hR_{\mfrak{p}_i}/cR_{\mfrak{p}_i}\vertr=C$. Since $\EPP(R)$ is recursive, we are done by \ref{fincond}.

\noindent
\textbf{Case $b=0$ and $C>1$:}
Then $\vertl \nf{xb=0}{x=0}(\oplus_{i=1}^hR_{\mfrak{p}_i})\vertr=\vertl\oplus_{i=1}^hR_{\mfrak{p}_i}\vertr$. So $\vertl\oplus_{i=1}^hR_{\mfrak{p}_i}/cR_{\mfrak{p}_i}\vertr=C>1$ implies $\vertl \nf{xb=0}{x=0}(\oplus_{i=1}^hR_{\mfrak{p}_i})\vertr$ is infinite. So
$\oplus_{i=1}^hR_{\mfrak{p}_i}\models \chi$ if and only if $\vertl\oplus_{i=1}^hR_{\mfrak{p}_i}/cR_{\mfrak{p}_i}\vertr=C$. So, since $\EPP(R)$ is recursive, we are done by \ref{fincond}.

\noindent
\textbf{Case $b=0$ and $C=1$:}
If $B=1$ then the zero module satisfies $\chi$ i.e. $h=0$. Otherwise, if $\oplus_{i=1}^hR_{\mfrak{p}_i}\models\chi$ and $\gamma\notin\mfrak{p}_i$ for $1\leq i\leq h$ then $h\geq 1$ and $c\gamma\notin \mfrak{p}_i$ for all $1\leq i\leq h$. So there exists a prime ideal $\mfrak{p}\lhd R$ such that $\gamma c\notin\mfrak{p}$. Conversely, if $\mfrak{p}\lhd R$ is a prime ideal such that $c\gamma\notin\mfrak{p}$ then $R_{\mfrak{p}_i}\models\chi$. There exists a prime ideal $\mfrak{p}\lhd R$ such that $\gamma c\notin \mfrak{p}$ if and only if $(\gamma c,1,0,0)\notin\DPR(R)$.
\end{proof}

\begin{proof}[Proof of \ref{x=x/c|x=xb=0/x=0geq}]
We may assume that $c\neq 0$ since if $c=0$ then $\vertl \nicefrac{x=x}{c|x}\vertr=C\wedge \vertl \nicefrac{xb=0}{x=0}\vertr\geq B\wedge \Xi$ is a statement about an $R$-module of a fixed finite size and in this case we know such an algorithm exists, by \ref{EPPfinite}, since $\EPP(R)$ is recursive.

By \ref{auxtoragammadelta}, we can compute  $n\in\N$ and $(r_j,r_ja_j,\gamma_j,\delta_j)\in R^4$ for $1\leq j\leq n$ such that $R_{\mfrak{p}}/cI\models \Xi$ if and only if $(\mfrak{p},I)\models (r_j,r_ja_j,\gamma_j,\delta_j)$ for some $1\leq j\leq n$ and such that $R_{\mfrak{p}}/cI\models D\Xi$ if and only if $(\mfrak{p},I)\models (r_j,r_ja_j,\delta_j,\gamma_j)$ for some $1\leq j\leq n$.

\smallskip

\noindent
\textbf{Claim:} There exists $M\in\Mod\text{-}R$ such that
\[M\models \vertl \nicefrac{x=x}{c|x}\vertr=C\wedge \vertl \nicefrac{xb=0}{x=0}\vertr\geq B\wedge \Xi\] if and only if there exist $C_j\in\N$ for $0\leq j\leq n$ and $B_j\in\N$, $B_j\leq B$ for $0\leq j\leq n+1$
with $\prod_{j=0}^nC_j=C$ and $\prod_{j=0}^{n+1}B_j\geq B$,
satisfying the following conditions.
\begin{enumerate}


\item There exists $F\in\Mod\text{-}R$ such that
\[
F\models \vertl\nf{x=x}{x=0}\vertr=C_0\wedge \vertl\nf{xb=0}{x=0}\vertr\geq B_0\wedge \vertl \nf{c|x}{x=0}\vertr=1\wedge \Xi.
\]

\item There exists $M'\in\Mod\text{-}R$ such that
\[
M'\models \vertl \nicefrac{x=x}{c|x}\vertr =1\wedge \vertl \nicefrac{xb=0}{x=0}\vertr\geq B_{n+1}\wedge\Xi.
\]

\item For $1\leq j\leq n$,
\begin{itemize}
\item [(a)${}_j$] if $r_j=0$ then there exist $h_j\in\N_0$ and prime ideals $\mfrak{p}_{ij}\lhd R$ for $1\leq i\leq h_j$ such that $\gamma_j\notin \mfrak{p}_{ij}$, $\delta_j\neq 0$ and
\[M_j:=\bigoplus_{i=1}^{h_j}R_{\mfrak{p}_{ij}}\models \vertl \nicefrac{x=x}{c|x}\vertr =C_j\wedge \vertl \nicefrac{xb=0}{x=0}\vertr\geq B_j,\text{ and}\]
\item [(b)${}_j$]if $r_j\neq 0$ then there exist $h_j,k_j\in\N_0$, prime ideals $\mfrak{p}_j,\mfrak{q}_j,\mfrak{p}_{ij}\lhd R$ and ideals $I_{ij}\lhd R_{\mfrak{p}_{ij}}$ for $1\leq i\leq h_j$ such that $(\mfrak{p}_{ij},I_{ij})\models (1,a_j,\gamma_j,\delta_j)$ for $1\leq i\leq h_j$, $\gamma_j\notin\mfrak{p}_j$, $\delta_j\notin\mfrak{q}_j$, $a\in\mfrak{p}_j$, $a\in \mfrak{q}_j$, and
\[M_j:= M(\mfrak{p}_j,\mfrak{q}_j,r_j)^{k_j}\oplus\bigoplus_{i=1}^{h_j}R_{\mfrak{p}_{ij}}/r_jcI_{ij}\models \vertl \nicefrac{x=x}{c|x}\vertr =C_j\wedge \vertl \nicefrac{xb=0}{x=0}\vertr\geq B_j.\]
\end{itemize}
\end{enumerate}

\noindent
\textit{Proof of claim.}
($\Rightarrow$) By \ref{x=x/c|xfinnonzero}, if $U$ is a uniserial module with $\nf{x=x}{c|x}(U)$ finite but non-zero then either $c\in\ann_R U$ or $U\cong R_{\mfrak{p}}/cI$ for some prime ideal $\mfrak{p}\lhd R$ and ideal $I\lhd R_{\mfrak{p}}$. Therefore, by \ref{REFuni}, there exists
\[M\models \vertl \nicefrac{x=x}{c|x}\vertr=C\wedge \vertl \nicefrac{xb=0}{x=0}\vertr\geq B\wedge \Xi\] if and only if there exists $F\in\Mod\text{-}R$ with $c\in\ann_RF$, prime ideals $\mfrak{p}_i\lhd R$ and ideals $J_i\lhd R_{\mfrak{p}_i}$ for $1\leq i\leq h$ and $M'\in\Mod\text{-}R$ with $\vertl\nf{x=x}{c|x}(M')\vertr=1$ such that
\[F\oplus M'\oplus \oplus_{i=1}^hR_{\mfrak{p}_i}/cJ_i\models \vertl \nicefrac{x=x}{c|x}\vertr=C\wedge \vertl \nicefrac{xb=0}{x=0}\vertr\geq B\wedge \Xi.\]
Since $R_{\mfrak{p}}/cJ\models \Xi$ if and only if $(\mfrak{p},J)\models (r_j,r_ja_j,\gamma_j,\delta_j)$ for some $1\leq j\leq n$, we may rewrite $\oplus_{i=1}^h R_{\mfrak{p}_i}/cJ_i$ as $\oplus_{j=1}^{n}\oplus_{i=1}^{h_j}R_{\mfrak{p}_{ij}}/cJ_{ij}$ where $(\mfrak{p}_{ij},J_{ij})\models (r_j,r_ja_j,\gamma_j,\delta_j)$ for $1\leq j\leq n$ and $1\leq i\leq h_j$.

If $r_j=0$ then $(\mfrak{p}_{ij},J_{ij})\models (r_j,r_ja_j,\gamma_j,\delta_j)$ if and only if $J_{ij}=0$, $\delta_j\neq 0$ and $\gamma_j\notin\mfrak{p}_{ij}$. If $r_j\neq 0$ then $(\mfrak{p}_{ij},J_{ij})\models (r_j,r_ja_j,\gamma_j,\delta_j)$ if and only if there exists $I_{ij}\lhd R_{\mfrak{p}_{ij}}$ such that $J_{ij}=rI_{ij}$ and $(\mfrak{p}_{ij},I_{ij})\models (1,a_j,\gamma_j,\delta_j)$.

Let $C_0:=\vertl \nf{x=x}{c|x}(F)\vertr=\vertl F\vertr$ and
$B_0:=\min\{\vertl\nf{xb=0}{x=0}(F)\vertr,B\}$. Let $B_{n+1}:=\min\{\vertl \nf{xb=0}{x=0}(M')\vertr,B\}$. If $r_j=0$ then let
$C_j:=\vertl \nf{x=x}{c|x}(\oplus_{i=1}^{h_j}R_{\mfrak{p}_{ij}})\vertr$ and $B_j:=\min\{\vertl \nf{xb=0}{x=0}(\oplus_{i=1}^{h_j}R_{\mfrak{p}_{ij}})\vertr,B\}$.
If $r_j\neq 0$ then let $C_j:=\vertl \nf{x=x}{c|x}(\oplus_{i=1}^{h_j}R_{\mfrak{p}_{ij}}/r_jcI_{ij})\vertr$ and $B_j:=\min\{\vertl \nf{xb=0}{x=0}(\oplus_{i=1}^{h_j}R_{\mfrak{p}_{ij}}/r_jcI_{ij})\vertr,B\}$. Now, setting $k_j=0$ for $1\leq j\leq n$, we are done.

\smallskip
\noindent
($\Leftarrow$) Fix $1\leq j\leq n$. Suppose $r_j=0$. Then $(\mfrak{p}_{ij},0)\models(r_j,r_ja_j,\gamma_j,\delta_j)$ for each $1\leq i\leq h_j$ and hence $R_{\mfrak{p}_{ij}}\models \Xi$. Suppose $r_j\neq 0$. Then $(\mfrak{p}_{ij},r_jI_{ij})\models (r_j,r_ja_j,\gamma_j,\delta_j)$ for each $1\leq i\leq h_j$ and hence $R_{\mfrak{p}_{ij}}/r_jcI_{ij}\models \Xi$.

By \ref{MpqXi}, $\gamma_j\notin\mfrak{p}_j$, $\delta_j\notin \mfrak{q}_j$, $a\in\mfrak{p}_j$ and $a\in\mfrak{q}_j$ implies that $M(\mfrak{p}_j,\mfrak{q}_j,r_j)\models \Xi$.
Thus
\[F\oplus M'\oplus\bigoplus_{j=1}^{n}M_j\models \Xi.\]
Therefore
\[F\oplus M'\oplus\bigoplus_{j=1}^{n}M_j\models \vertl \nf{x=x}{c|x}\vertr=\prod_{j=0}^n C_i\wedge \vertl \nf{xb=0}{x=0}\vertr\geq \prod_{j=0}^{n+1}B_i\wedge \Xi.\] Since $C=\prod_{j=0}^n C_i$ and $\prod_{j=0}^{n+1}B_i\geq B$, we are done.

Since the set of $C_j\in \N$ for $0\leq j\leq n$ and $B_j\in\N$ for $0\leq j\leq n+1 $ satisfying $(1)$ is finite, it is enough to show that for fixed $C_j\in \N$ for $0\leq j\leq n$ and $B_j\in\N$ for $0\leq j\leq n+1$, there are algorithms answering whether $(2),(3)$ and $(4)$ hold. By \ref{EPPfinite}, since $\EPP(R)$ is recursive, there is an algorithm which answers whether $(2)$ holds. By \ref{ZieglertoTinfty}, since $\DPR(R)$ is recursive there is an algorithm which answers whether $(3)$ holds. Since $\DPR(R),\EPP(R)$ and $X(R)$ are recursive, by \ref{=geqhelp}, if $r_j\neq 0$ then there is an algorithm which answers whether (b)${}_j$ holds. Since $\DPR(R)$ and $\EPP(R)$ are recursive, by \ref{Rp}, if $r_j=0$ then there is an algorithm which answers whether (a)${}_j$ holds.
\end{proof}

\section{The main theorem}

\begin{theorem}\label{mainthm}
Let $R$ be a Pr\"ufer domain. The theory of $R$-modules is decidable if and only if $\DPR(R)$, $\EPP(R)$ and $X(R)$ are recursive.
\end{theorem}
\begin{proof}
The forward direction follows from \cite[6.4]{Decpruf} (or \ref{DPRZg}), \ref{decimpEPPrec} and \ref{decimpXrec}.

By \ref{1stsynthm}, in order to show that $T_R$ is decidable, it is enough to show that there is an algorithm which, given a sentence $\chi$ of the form
\[\vertl \nicefrac{d|x}{x=0}\vertr\square_1 D\wedge \vertl \nicefrac{xb=0}{c|x}\vertr \square_2 E\wedge\chi_{f,g}\wedge\Xi,\]
where $\square_1,\square_2\in\{\geq,=,\emptyset\}$, $d,c,b\in R\backslash\{0\}$, $D,E\in\N$, $f:X\rightarrow \N$, $g:Y\rightarrow \N$, $X,Y$ are finite sets of pp-$1$-pairs of the form $\nf{xb'=0}{x=0}$ or $\nf{x=x}{c'|x}$ and $\Xi$ is an auxiliary sentence, answers whether there exists an $R$-module which satisfies $\chi$ or not.


Let $W$ and $V$ be as in \S \ref{Furthersyn}. By \ref{ssrem=}, \ref{ssremgeqdx}, \ref{ssremgeqxbcx} and \ref{redexsig}, there is an algorithm which given $w\in W$ reducible, returns $\underline{w}\in \mathbb{W}$ such that $w\in V$ if and only if $\underline{w}\in\mathbb{V}$, and, $\exsig\underline{w}<\exsig w$. Since the set of extended signatures is artinian, it is enough to show that there is an algorithm which, given $w\in W$ not reducible, answers whether $w\in V$ or not. By \ref{nonredtoZg} and the statement just before that, it is enough to show that there is an algorithm which, given $w\in W$ with extended signature in \[S:=\{((\emptyset,\emptyset),(z_1,z_2),(z_3,z_4))\st z_1+z_2\leq 1 \text{ and } z_3+z_4\leq 1\},\] answers whether $w\in V$ or not. Now, $w\in V$ if and only if $Dw\in V$. So, by \ref{dualexsig}, we can reduce $S$ further to \[S:=\{((\emptyset,\emptyset),(1,0),(1,0)),((\emptyset,\emptyset),(1,0),(0,0)),((\emptyset,\emptyset),(1,0),(0,1))\}.\]  By \ref{nonredtoZg}, \ref{x=x/c|x=}, \ref{x=x/c|x=xb=0/x=0=} and \ref{x=x/c|x=xb=0/x=0geq}, such an algorithm exists.
%
%
\end{proof}

\section{Integer-valued polynomials}\label{SIntpoly}

\noindent
We use our main theorem, \ref{mainthm}, to show that the theory of modules over the ring of integer valued polynomials with rational valued coefficients, $\text{Int}(\Z)$, is decidable.

First we fix some notation: For all $p\in\mathbb{P}$, $\Z_{(p)}$ denotes $\Z$ localised at the ideal generated by $p$, $\Q_p$ denote the field of $p$-adic numbers, $v_p:\Q_p\rightarrow \Z\cup\{\infty\}$ denotes the $p$-adic valuation on $\Q_p$ and $\widehat{\Z}_p$ denote the $p$-adic integers.

The ring $\text{Int}(\Z)$ is the subring of $\Q[x]$ consisting of all polynomials $a\in \Q[x]$ such that $a(\Z)\subseteq \Z$. Recall, \cite[I.1.1]{IntPoly}, that the polynomials
\[\left(
    \begin{array}{c}
      x \\
      n \\
    \end{array}
  \right):=\frac{x(x-1)\ldots(x-(n-1))}{n!}
\] are a basis for $\text{Int}(\Z)$ as a $\Z$-module. This readily gives us a recursive presentation of $\text{Int}(\Z)$.  The ring $\text{Int}(\Z)$ is a Pr\"{u}fer domain \cite[VI.1.7]{IntPoly}.

Note, \cite[I.2.1]{IntPoly}, that, for all $p\in\mathbb{P}$, if $f\in \Q[x]$ and $f(\Z)\subseteq \Z$ then $f(\Z_{(p)})\subseteq \Z_{(p)}$. Further, for any $p\in\mathbb{P}$, by continuity of polynomials over $\Q_p$, $f(\Z_{(p)})\subseteq \Z_{(p)}$ implies $f(\widehat{\Z_{p}})\subseteq \widehat{\Z_{p}}$.

The prime spectrum of $\text{Int}(\Z)$ is described in \cite[V.2.7]{IntPoly}. We recall the information we need.

\begin{itemize}
\item For any $p\in \mathbb{P}$, the prime ideals of $\text{Int}(\Z)$ containing $p$ are in bijective correspondence with the elements of $\widehat{\Z_{p}}$ by mapping $\alpha\in \widehat{\Z_{p}}$ to
\[\mfrak{m}_{p,\alpha}:=\{f\in\text{Int}(\Z) \st f(\alpha)\in p\widehat{\Z_{p}}\}.\] The prime ideals $\mfrak{m}_{p,\alpha}$ are exactly the maximal ideals of $\text{Int}(\Z)$ and the quotient $\text{Int}(\Z)/\mfrak{m}_{p,\alpha}$ has size $p$.
\item The non-zero prime ideals $\mfrak{p}$ of $\text{Int}(\Z)$ such that $\Z\cap\mfrak{p}=\{0\}$ are in bijective correspondence with the monic irreducible polynomials $q\in\Q[x]$ via the mapping
\[q\mapsto\mfrak{p}_{q}:=q\Q[x]\cap\text{Int}(\Z).\]
Note that $\mfrak{p}_q\subseteq \mfrak{m}_{p,\alpha}$ if and only if $q(\alpha)=0$ in $\Q_p$.
\end{itemize}
It will sometimes be useful to have some alternate notation for the non-maximal prime ideals. For $\alpha\in \widehat{\Z_{p}}$, let
\[\mfrak{p}_{\alpha}:=\left\{
                        \begin{array}{ll}
                          \mfrak{p}_q, & \hbox{if $\alpha$ is algebraic and $q\in\Z[x]$ is its monic minimal polynomial;} \\
                          \{0\}, & \hbox{if $\alpha$ is transcendental.}
                        \end{array}
                      \right.
\]
This notation has the disadvantage that $\mfrak{p}_{\alpha}=\mfrak{p}_\beta$ does not imply $\alpha=\beta$. However, it allows us to work with $\alpha\in\widehat{\Z_p}$ algebraic and transcendental uniformly in the following ways: Firstly, for $a\in\text{Int}(\Z)$ and $\alpha\in\widehat{\Z_p}$, $a\in\mfrak{p}_\alpha$ if and only if $a(\alpha)=0$. Secondly, for $\mfrak{q}\lhd R$ a prime ideal and $\alpha\in\widehat{\Z_p}$, $\mfrak{q}\subseteq \mfrak{m}_{p,\alpha}$ if and only if $\mfrak{q}=\mfrak{m}_{p,\alpha}$, $\mfrak{q}=\mfrak{p}_\alpha$ or $\mfrak{q}=\{0\}$

By \ref{mainthm}, we need to show that $\DPR(\text{Int}(\Z))$, $\EPP(\text{Int}(\Z))$ and $X(\text{Int}(\Z))$ are recursive. In order to do this, we will use the fact, \cite[Thm 17]{Ax}, that the common theory $T_{\text{adic}}$ of the valued fields $\Q_p$, as $p$ varies, is decidable. We shall work in a two-sorted language $\mcal{L}_{val}$ of valued fields with a sort for the field $K$, a sort $\Gamma$ for the value group extended by $\infty$ and a function symbol $v:K\rightarrow \Gamma$ which will be interpreted as $v_p$ in each $\Q_p$. For convenience, we add a constant symbol $1$ to the value group sort $\Gamma$, which for each valued field $\Q_p$ will be interpreted as the least strictly positive element of the value group.

Let $\mcal{L}_{val}^0$ be the set of sentences in $\mcal{L}_{val}$. The sets
\[T_{adic}:=\{\phi\in\mcal{L}_{val}^0 \st \text{ for all } p\in\mathbb{P}, \ \Q_p\models \phi\}\] and
\[\{\phi\in\mcal{L}_{val}^0 \st \text{ there exists } p\in\mathbb{P} \text{ such that } \Q_p\models \phi\}\] are recursive. Hence, since $\Q_p\models\phi$ if and only if $\Q_q\models v(p)=0\vee \phi$ for all $q\in \mathbb{P}$, the set
\[\{(p,\phi)\in \mathbb{P}\times \mcal{L}_{val}^0 \st \Q_p\models \phi\}\] is recursive.

\begin{proposition}\label{DPRInt}
The set $\text{DPR}(\text{Int}(\Z))$ is recursive.
\end{proposition}
\begin{proof}
Let $a,b,c,d\in \text{Int}(\Z)$. Then $(a,b,c,d)\in \text{DPR}(\text{Int}(\Z))$ if and only if
\begin{enumerate}
\item for all $p\in\mathbb{P}$ and $\alpha\in \widehat{\Z_p}$, $a\in \mfrak{m}_{p,\alpha}$,  $b\notin \mfrak{m}_{p,\alpha}$, $c\in \mfrak{m}_{p,\alpha}$  or $d\notin\mfrak{m}_{p,\alpha}$;
\item for all $p\in\mathbb{P}$ and $\alpha\in \widehat{\Z_p}$, $a\in\mfrak{m}_{p,\alpha}$, $b\notin \mfrak{m}_{p,\alpha}$, $c\in \mfrak{p}_\alpha$ or $d\notin \mfrak{p}_\alpha$;
\item for all $p\in\mathbb{P}$ and $\alpha\in \widehat{\Z_p}$, $a\in\mfrak{m}_{p,\alpha}$, $b\notin\mfrak{m}_{p,\alpha}$, $c=0$ or $d\neq 0$;
\item for all $p\in\mathbb{P}$ and $\alpha\in \widehat{\Z_p}$, $a\in\mfrak{p}_\alpha$, $b\notin \mfrak{p}_\alpha$, $c\in\mfrak{m}_{p,\alpha} $ or $d\notin \mfrak{m}_{p,\alpha} $;
\item for all $p\in\mathbb{P}$ and $\alpha\in \widehat{\Z_p}$, $a=0$, $b\neq 0$, $c\in\mfrak{m}_{p,\alpha}$ or $d\notin\mfrak{m}_{p,\alpha}$;
\item for all $q\in\Q[x]$ irreducible and monic, $a\in\mfrak{p}_q$, $b\notin\mfrak{p}_q$, $c\in\mfrak{p}_q$ or $d\notin\mfrak{p}_q$;
\item for all $q\in\Q[x]$ irreducible and monic, $a\in\mfrak{p}_q$, $b\notin\mfrak{p}_q$, $c=0$ or $d\neq 0$;
\item for all $q\in\Q[x]$ irreducible and monic, $a=0$, $b\neq 0$, $c\in\mfrak{p}_q$ or $d\notin\mfrak{p}_q$; and
\item $a=0$ or $b\neq 0$ or $c=0$ or $d\neq 0$
\end{enumerate}
Define $\chi_1,\chi_2,\chi_3,\chi_4,\chi_5\in \mcal{L}^0_{val}$ to be
\[\chi_1:=\forall x \ (v(x)<0 \vee v(a(x))\geq 1\vee v(b(x))=0\vee v(c(x))\geq 1\vee v(d(x))=0),\]
\[\chi_2:=\forall x \ (v(x)<0 \vee v(a(x))\geq 1\vee v(b(x))=0\vee c(x)=\vee d(x)\neq 0),\]
\[\chi_3:=\forall x \ (v(x)<0 \vee v(a(x))\geq 1\vee v(b(x))=0),\]
\[\chi_4:=\forall x \ (v(x)<0 \vee a(x)=0\vee b(x)\neq 0\vee v(c(x))\geq 1\vee v(d(x))=0),\] and
\[\chi_5:=\forall x \ (v(x)<0 \vee v(c(x))\geq 1\vee v(d(x))=0).\]

\smallskip

\noindent
\textbf{Claim:} $(a,b,c,d)\in \DPR(\text{Int}(\Z))$ if and only if
\begin{enumerate}[(i)]
\item $\chi_1,\chi_2,\chi_4\in T_{adic}$,
\item either $c=0$, $d\neq 0$ or $\chi_3\in T_{adic}$,
\item either $a=0$, $b\neq 0$ or $\chi_5\in T_{adic}$,
\item $ac\in \rad_{\Q[x]}(b\Q[x]+d\Q[x])$,
\item $a\in\rad_{\Q[x]}(b\Q[x])$ or $c=0$ or $d\neq 0$,
\item either $c=0$, $d\neq 0$ or $c\in\rad_{\Q[x]}(d\Q[x])$, and
\item either $a=0$, $b\neq 0$, $c=0$ or $d\neq 0$.
\end{enumerate}
Recall that $a\in \mfrak{m}_{p,\alpha}$ if and only if $v_p(a(\alpha))\geq 1$ and $a\in \mfrak{p}_\alpha$ if and only if $a(\alpha)=0$. So, for $j\in\{1,2,4\}$, $(j)$ holds if and only if $\chi_j\in T_{adic}$, $(3)$ holds if and only if $\chi_3\in T_{adic}$, $c=0$ and $d\neq 0$ and, and $(5)$ holds if and only if $\chi_5\in T_{adic}$, $a=0$ and $b\neq 0$.

The statement that, for all $\mfrak{p}$ such that $\mfrak{p}=0$ or $\mfrak{p}=\mfrak{p}_q$ with  $q\in \Q[x]$ monic irreducible,
\[a\in \mfrak{p}\vee c\in \mfrak{p}\vee b\notin \mfrak{p}\vee d\notin\mfrak{p}\] is equivalent to $ac\in \text{rad}_{\Q[x]}(b\Q[x]+d\Q[x])$.
So $(6)$ and $(9)$ hold if and only if (iv) holds. Similarly, $(7)$ and $(9)$ holds if and only if (v) holds and, $(8)$ and $(9)$ holds if and only if (vi) holds. Finally $(9)$ holds if and only if (vii) holds. So the claim holds.

Since $T_{adic}$ is decidable, we can effectively decide whether $(i),(ii)$ and $(iii)$ hold. If $a,b_1,b_2\in \Q[x]$ then $a\in \rad_{\Q[x]}(b_1\Q[x]+b_2\Q[x])$ if and only if for all $q\in\Q[x]$ irreducible, $q$ divides $b_1$ and $q$ divides $b_2$ implies $q$ divides $a$. Since $\Q$ has a splitting algorithm, there is an algorithm which, given $a,b_1,b_2\in\Q[x]$, decides whether $a\in \rad_{\Q[x]}(b_1\Q[x]+b_2\Q[x])$. Therefore, we can effectively decide whether $(iv)$-$(vi)$ holds. It is obvious that we can effectively decide whether $(vii)$ holds.
\end{proof}

\noindent
In order to analyse $\EPP(\text{Int}(\Z))$, we need to understand the valuation overrings of $\text{Int}(\Z)$. For $p\in\mathbb{P}$, let $v_p:\Q_p\rightarrow \Z\cup\{\infty\}$ denote the standard valuation on $\Q_p$.

\begin{itemize}
\item For each $p\in\mathbb{P}$ and $\alpha\in \widehat{\Z_p}$ transcendental, define $v_{p,\alpha}:\Q(x)\rightarrow \Z\cup\{\infty\}$ by setting $v_{p,\alpha}(f/g)=v_p(f(\alpha)/g(\alpha))$.
\item For each $p\in\mathbb{P}$ and $\alpha\in \widehat{\Z_p}$ algebraic with monic minimal polynomial $q\in\Z[x]$, define $v_{p,\alpha}:\Q(x)\rightarrow \Z\times\Z\cup\{\infty\}$ by setting $v_{p,\alpha}(h)=(k,v_p(f(\alpha)/g(\alpha)))$ where $h=q^k\cdot f/g$, $f(\alpha)\neq 0$ and $g(\alpha)\neq 0$.
\end{itemize}
By \cite[VI.1.9]{IntPoly}, $\text{Int}(\Z)_{\mfrak{m}_{p,\alpha}}$ is the valuation ring of $v_{p,\alpha}$.

Let $e\in \text{Int}(\Z)$ and $N\in\N_0$.
\begin{itemize}
\item For $p\in\mathbb{P}$ and $\alpha\in\widehat{\Z_p}$ transcendental, $v_{p,\alpha}(e)=v_{p}(e(\alpha))=N$ if and only if
\[|\text{Int}(\Z)_{\mfrak{m}_{p,\alpha}}/e\text{Int}(\Z)_{\mfrak{m}_{p,\alpha}}|=p^N.\]
\item For $p\in\mathbb{P}$ and $\alpha\in\widehat{\Z_p}$ algebraic, $v_{p,\alpha}(e(\alpha))=N$ if and only if
$v_{p,\alpha}(e)= (0,N)$ if and only if
\[|\text{Int}(\Z)_{\mfrak{m}_{p,\alpha}}/e\text{Int}(\Z)_{\mfrak{m}_{p,\alpha}}|=p^N.\]
\end{itemize}

\begin{proposition}\label{EPPInt}
The set $\EPP(\text{Int}(\Z))$  is recursive.
\end{proposition}
\begin{proof}
\textbf{Claim:} $(p,M;a;\gamma;e,N)\in \EPP_1(Int(\Z))$ if and only if
there exist $h\in\{1,\ldots,N+M\}$,  $N_i,M_i\in \N_0$ for $1\leq i\leq h$ with $\sum_{i=1}^hN_i=N$ and $\sum_{i=1}^hM_i=M$ such that for each $1\leq i\leq h$,
\[\Q_p\models\exists x \ v(x)\geq 0\wedge v(e(x))=N_i\wedge v(a(x))\geq M_i\wedge v(\gamma(x))=0.\]
Suppose $(p,M;a;\gamma;e,N)\in \EPP_1(Int(\Z))$. There exists $\alpha_1,\ldots, \alpha_h$  and $l_1,\ldots, l_h\in \N_0$ such that
$\gamma\notin \mfrak{m}_{p,\alpha_i}$, $a\in \mfrak{m}_{p,\alpha_i}^{l_i}$, \[p^{\sum_{i=1}^hl_i}=\prod_{i=1}^h|\text{Int}(\Z)_{\mfrak{m}_{p,\alpha_i}}/\mfrak{m}_{p,\alpha_i}^{l_i}|=p^M\] and
\[p^{\sum_{i=1}^h v_p(e(\alpha_i))}=\prod_{i=1}^h|\text{Int}(\Z)_{\mfrak{m}_{p,\alpha_i}}/e\text{Int}(\Z)_{\mfrak{m}_{p,\alpha_i}}|=p^N.\]
We may assume that $h\leq N+M$ since the size of the set of $1\leq i\leq h$ such that $v_p(e(\alpha_i))>0$ or $l_i>0$ is at most $N+M$. Set $N_i:=v_{p}(e(\alpha_i))$ and $M_i:=l_i$. Then $N=\sum_{i=1}^hN_i$ and $M=\sum_{i=1}^hM_i$. Since $a\in \mfrak{m}_{p,\alpha_i}^{M_i}$, $v_p(a(\alpha_i))\geq M_i$ and since $\gamma\notin\mfrak{m}_{p,\alpha_i}$, $v_p(\gamma(\alpha_i))=0$, as required.

Conversely, suppose that there exist $h\in\{1,\ldots,N+M\}$,  $N_i,M_i\in \N_0$ for $1\leq i\leq h$ with $\sum_{i=1}^hN_i=N$ and $\sum_{i=1}^hM_i=M$ such that for each $1\leq i\leq h$,
\[\exists x \ v(e(x))=N_i\wedge v(a(x))\geq M_i\wedge v(\gamma(x))=0\] holds in $\Q_p$.

Let $\alpha_i$ witness the truth of the sentence for each $1\leq i\leq h$. Set $\mfrak{p}_i:=\mfrak{m}_{p,\alpha_i}$ for $1\leq i\leq h$ and $I_i:=\mfrak{m}_{p,\alpha_i}^{M_i}$. Then $v_p(a(\alpha_i))\geq M_i$ implies $a\in \mfrak{m}_{p,\alpha_i}^{M_i}$. Since $v_p(\gamma(\alpha_i))=0$, $\gamma\notin \mfrak{m}_{p,\alpha_i}$ and $v_p(e(\alpha_i))=N_i$ implies \[|\text{Int}(\Z)_{\mfrak{m}_{p,\alpha_i}}/e\text{Int}(\Z)_{\mfrak{m}_{p,\alpha_i}}|=p^{N_i}.\] So $(p,M;a;\gamma;e,N)\in \EPP(Int(\Z))$ as required.

The proposition now follows from the claim since the set of $(p,\phi)\in \mathbb{P}\times\mcal{L}^0_{val}$ such that $\Q_p\models \phi$ is recursive.
\end{proof}

\begin{proposition}\label{XInt}
The set $X(\text{Int}(\Z))$ is recursive.
\end{proposition}
\begin{proof}\textbf{Claim:} $(p,n;e,\gamma,a,\delta)\in X(R)$ if and only if there exist $h\in\{0,\ldots,n\}$ and $\alpha_i\in \widehat{\Z_{p}}$ for $1\leq i\leq h$ such that $\sum_{i=1}^hv_p(e(\alpha_i))=n$, $v_p(\gamma(\alpha_i))=0$ and either
\begin{enumerate}[(i)]
\item $v_{p}(\delta(\alpha_i))=0$,
\item $a(\alpha_i)=0$ and $\delta(\alpha_i)\neq 0$, or
\item $a=0$ and $\delta\neq 0$.
\end{enumerate}
Recall, \ref{X2ndformulation}, $(p,n;\lambda,\gamma,a,\delta)\in X(R)$ if and only if there exist $h\in\N$ and maximal ideals $\mfrak{m}_{p,\alpha_i}$ for $1\leq i\leq h$ such that
\[p^{\sum_{i=1}^hv_p(e(\alpha_i))}=|\oplus_{i=1}^h\text{Int}(\Z)_{\mfrak{m}_{p,\alpha_i}}/e\text{Int}(\Z)_{\mfrak{m}_{p,\alpha_i}}|=p^{n},\] $\gamma\notin\mfrak{m}_{p,\alpha_i}$ and for each $1\leq i\leq h$, either $\delta\notin\mfrak{m}_{p,\alpha_i}$ or there exists $\mfrak{q}_i$ a prime ideal such that $\mfrak{q}_i\subsetneq\mfrak{m}_{p,\alpha_i}$, $a\in \mfrak{q}_i$ and $\delta\notin \mfrak{q}_i$.

If $\mfrak{q}$ is a prime ideal strictly contained in $\mfrak{m}_{p,\alpha}$ then $\mfrak{q}=\mfrak{p}_\alpha$ or $\mfrak{q}=\{0\}$. Thus there exists $\mfrak{q}\subsetneq \mfrak{m}_{p,\alpha}$ such that $a\in \mfrak{q}$ and $\delta\notin \mfrak{q}$ if and only if $a(\alpha)=0$ and $\delta(\alpha)\neq 0$, or, $a=0$ and $\delta\neq 0$. For $\alpha\in \widehat{\Z_p}$, $\gamma\notin\mfrak{m}_{p,\alpha}$ if and only if $v_p(\gamma(\alpha))= 0$.

Therefore $(p,n;e,\gamma,a,\delta)\in X(R)$ if and only if there exist $h\in\N$ and maximal ideals $\mfrak{m}_{p,\alpha_i}$ for $1\leq i\leq h$ such that $n=\sum_{i=1}^hv_p(e(\alpha_i))$, $\gamma\notin\mfrak{m}_{p,\alpha_i}$ and for each $1\leq i\leq h$, either $v_p(\delta(\alpha_i))=0$, or, $a(\alpha)=0$ and $\delta(\alpha)\neq 0$, or, $a=0$ and $\delta\neq 0$. So we have proved the claim.

The right hand side of the claim can be written as a sentence $\phi\in\mcal{L}^0_{val}$ so that $\Q_p\models \phi$ if and only if $(p,n;\lambda,\gamma,a,\delta)\in X(R)$. By convention $(p,0;\lambda,\gamma,a,\delta)\in X(R)$ for all $p,\lambda,\gamma,a,\delta$. So, since the set of $(p,\phi)\in\mathbb{P}\times \mcal{L}^0_{val}$ such that $\Q_p\models \phi$ is recursive, we are done.
\end{proof}

\begin{theorem}
The theory of modules of the ring of integer valued polynomials is decidable.
\end{theorem}
\begin{proof}
This follows from \ref{mainthm}, \ref{DPRInt}, \ref{EPPInt} and $\ref{XInt}$.
\end{proof}


\end{document}

\section{excess}

Duality

The following proposition is direct consequence of \cite[6.6]{HerDual}.

\begin{proposition}\label{Dualsent}
Let $R$ be a ring. Let $n,m\in\N$ be such that $n\leq m$ and for $1\leq i\leq m$, let $N_i\in\N$ and let $\nicefrac{\phi_i}{\psi_i}$ be a pp-pair. There exists a right $R$-module satisfying
\[\bigwedge_{i=1}^n\vertl\nicefrac{\phi_i}{\psi_i}\vertr=N_i\wedge\bigwedge_{i=n+1}^m\vertl \nicefrac{\phi_i}{\psi_i}\vertr\geq N_i\]
if and only if there exists a left $R$-module satisfying
\[\bigwedge_{i=1}^n\vertl \nicefrac{D\psi_i}{D\phi_i}\vertr=N_i\wedge\bigwedge_{i=n+1}^m\vertl \nicefrac{D\psi_i}{D\phi_i}\vertr\geq N_i.\]
\end{proposition}

\subsection{Decidability and recursive Pr\"ufer domains}

\

\noindent
A \textbf{recursive ring} is either a finite ring or a ring $R$ together with a bijection $\pi:\N\rightarrow R$ such that addition and multiplication in $R$ induce recursive functions on $\N$ via $\pi$.

Note that if $R$ is a ring and $\pi: \N\rightarrow R$ is a bijection then $T_R$ is recursively axiomatisable with respect to $\pi$ if and only if $R$ together with $\pi$ is a recursive ring.

When proving decidability results about theories of modules, it is common to work with an ``effectively given'' ring rather than just a recursive one. Usually, a ring of a particular type if called effectively given if $R$ is a recursive ring and the bijection $\pi$ satisfies some extra conditions which are equivalent, for that particular type of ring, to Prest's condition $(D)$ holding (see \cite[pg 334]{PrestBluebook}). Recall that a recursive ring satisfies condition $(D)$ if there is an algorithm which, given $\phi,\psi\in\pp_R^1$ answers whether $\psi\leq \phi$.
For example, a recursive valuation domain $V$ is said to be \textbf{effectively given} if the preimage under $\pi$ of the set of units of $V$ is recursive. A recursive Pr\"ufer domain $R$ is said to be \textbf{effectively given} if the preimage under $\pi$ of the set of $(a,b)\in R^2$ such that $a\in bR$ is recursive.

Whether one works with recursive rings or effectively given ones is largely a matter of personal taste. For simplicity, we choose to work with recursive rings.

It was remarked in \cite[paragraph before 2.4]{Decdense}, that the property that $a\in bR$ is recursive is never used in \cite{Decpruf} and \cite{Decdense}. Moreover, see \cite[2.4]{Decdense}, if $R$ is a recursive Pr\"ufer domain and the set $\DPR(R)\subseteq R^4$ (see REF for a definition) is recursive then $a\in bR$ is recursive. In particular, even though they are stated for effectively given Pr\"ufer domains, all results in \cite{Decpruf} and \cite{Decdense} in fact hold for recursive Pr\"ufer domains.

\subsection{pp-fla Pr\"ufer}

\begin{proposition}
Let $R$ be a arithmetical ring. Every pp-$1$-formula over $R$ is a sum of formulae of the form $xa_1=0\wedge xa_2=0\wedge b|x$.
\end{proposition}
\begin{proof}
We know, REF, that every pp-$1$-formula over $R$ is equivalent to a sum of formulae of the form $\exists y \ yc=x\wedge yd=0$. Take $c,d\in R$ and let $\alpha,r,s\in R$ be such that $c\alpha=dr$ and $d(\alpha-1)=cs$. We will show that $\exists y \ yc=x\wedge yd=0$ is equivalent to $x\alpha=0\wedge xs=0\wedge c|x$.

Suppose $M\in\Mod\text{-}R$ and $m\in M$ is such that $mc\alpha=mcs=0$. Then $mc=(m-m\alpha)c$ and $(m-m\alpha)d=-mcs=0$.

Suppose $M\in\Mod\text{-}R$ and $m\in M$ is such that $md=0$. Then $mc\alpha=mdr=0$ and $mcs=md(\alpha-1)=0$. So $mc$ satisfies $x\alpha=0\wedge xs=0\wedge c|x$.
\end{proof}

\subsection{Uniserial and indecomposable pure-injective modules}

For $V$ a valuation domain and a proper ideal $I\lhd R$, define \[I^\#:=\bigcup_{a\notin I}(I:a).\] By convention, we define $V^\#$ to be the unique maximal ideal of $V$. Note that this definition agrees with the definition given in \cite[ChII \S 4]{MONND}, that is, for $I\neq 0$, $r\in I^\#$ if and only if $rI\subsetneq I$. For all ideal $I\lhd V$, $I^\#$ is a prime ideal and for $\mfrak{p}\lhd V$ a prime ideal, $\mfrak{p}^\#=\mfrak{p}$. Moreover, if $a\in V\backslash\{0\}$ and $I\lhd V$ then $(aI)^\#=I^\#$.

%
%

\bibliographystyle{alpha}\bibliography{finresgen}
\begin{thebibliography}{GLPT18}

\bibitem[Ax68]{Ax}
James Ax.
\newblock The elementary theory of finite fields.
\newblock {\em Ann. of Math. (2)}, 88:239--271, 1968.

\bibitem[Bec82]{Becker}
Eberhard Becker.
\newblock Valuations and real places in the theory of formally real fields.
\newblock In {\em Real algebraic geometry and quadratic forms ({R}ennes,
  1981)}, volume 959 of {\em Lecture Notes in Math.}, pages 1--40. Springer,
  Berlin-New York, 1982.

\bibitem[CC97]{IntPoly}
Paul-Jean Cahen and Jean-Luc Chabert.
\newblock {\em Integer-valued polynomials}, volume~48 of {\em Mathematical
  Surveys and Monographs}.
\newblock American Mathematical Society, Providence, RI, 1997.

\bibitem[DST19]{SpecSp}
Max Dickmann, Niels Schwartz, and Marcus Tressl.
\newblock {\em Spectral spaces}, volume~35 of {\em New Mathematical
  Monographs}.
\newblock Cambridge University Press, Cambridge, 2019.

\bibitem[EF72]{EkFis}
Paul~C. Eklof and Edward~R. Fischer.
\newblock The elementary theory of abelian groups.
\newblock {\em Ann. Math. Logic}, 4:115--171, 1972.

\bibitem[EH95]{EklHer}
Paul~C. Eklof and Ivo Herzog.
\newblock Model theory of modules over a serial ring.
\newblock {\em Ann. Pure Appl. Logic}, 72(2):145--176, 1995.

\bibitem[FS01]{MONND}
L\'{a}szl\'{o} Fuchs and Luigi Salce.
\newblock {\em Modules over non-{N}oetherian domains}, volume~84 of {\em
  Mathematical Surveys and Monographs}.
\newblock American Mathematical Society, Providence, RI, 2001.

\bibitem[GLPT18]{Decpruf}
Lorna Gregory, Sonia L'Innocente, Gena Puninski, and Carlo Toffalori.
\newblock Decidability of the theory of modules over {P}r\"{u}fer domains with
  infinite residue fields.
\newblock {\em J. Symb. Log.}, 83(4):1391--1412, 2018.

\bibitem[GLT19]{Decdense}
Lorna Gregory, Sonia L'Innocente, and Carlo Toffalori.
\newblock Decidability of the theory of modules over {P}r\"{u}fer domains with
  dense value groups.
\newblock {\em Ann. Pure Appl. Logic}, 170(12):102719, 23, 2019.

\bibitem[Gr{\"a}11]{Graetzer}
George Gr{\"a}tzer.
\newblock {\em Lattice theory: foundation}.
\newblock Birkh\"{a}user/Springer Basel AG, Basel, 2011.

\bibitem[Gre15]{DecVal}
L.~Gregory.
\newblock Decidability for theories of modules over valuation domains.
\newblock {\em J. Symb. Log.}, 80(2):684--711, 2015.

\bibitem[Hel40]{Helmer}
Olaf Helmer.
\newblock Divisibility properties of integral functions.
\newblock {\em Duke Math. J.}, 6:345--356, 1940.

\bibitem[Her93]{HerDual}
Ivo Herzog.
\newblock Elementary duality of modules.
\newblock {\em Trans. Amer. Math. Soc.}, 340(1):37--69, 1993.

\bibitem[Jen66]{JenAri}
C.~U. Jensen.
\newblock Arithmetical rings.
\newblock {\em Acta Math. Acad. Sci. Hungar.}, 17:115--123, 1966.

\bibitem[Kap74]{Kapcomm}
Irving Kaplansky.
\newblock {\em Commutative rings}.
\newblock The University of Chicago Press, Chicago, Ill.-London, revised
  edition, 1974.

\bibitem[LTP17]{AlgIntDec}
Sonia L'Innocente, Carlo Toffalori, and Gena Puninski.
\newblock On the decidability of the theory of modules over the ring of
  algebraic integers.
\newblock {\em Ann. Pure Appl. Logic}, 168(8):1507--1516, 2017.

\bibitem[PP88]{PointPrest}
Fran\c{c}oise Point and Mike Prest.
\newblock Decidability for theories of modules.
\newblock {\em J. London Math. Soc. (2)}, 38(2):193--206, 1988.

\bibitem[PPT07]{PunPunTof}
G.~Puninski, V.~Puninskaya, and C.~Toffalori.
\newblock Decidability of the theory of modules over commutative valuation
  domains.
\newblock {\em Ann. Pure Appl. Logic}, 145(3):258--275, 2007.

\bibitem[Pre88]{PrestBluebook}
Mike Prest.
\newblock {\em Model theory and modules}, volume 130 of {\em London
  Mathematical Society Lecture Note Series}.
\newblock Cambridge University Press, Cambridge, 1988.

\bibitem[Pre09]{PSL}
Mike Prest.
\newblock {\em Purity, spectra and localisation}, volume 121 of {\em
  Encyclopedia of Mathematics and its Applications}.
\newblock Cambridge University Press, Cambridge, 2009.

\bibitem[PT14]{D+M}
Gena Puninski and Carlo Toffalori.
\newblock Decidability of modules over a {B}\'{e}zout domain {$D+XQ[X]$} with
  {$D$} a principal ideal domain and {$Q$} its field of fractions.
\newblock {\em J. Symb. Log.}, 79(1):296--305, 2014.

\bibitem[PT15]{Bezwid}
Gena Puninski and Carlo Toffalori.
\newblock Some model theory of modules over {B}\'{e}zout domains. {T}he width.
\newblock {\em J. Pure Appl. Algebra}, 219(4):807--829, 2015.

\bibitem[Pun03]{KGser}
Gena Puninski.
\newblock The {K}rull-{G}abriel dimension of a serial ring.
\newblock {\em Comm. Algebra}, 31(12):5977--5993, 2003.

\bibitem[Szm55]{Szmie}
W.~Szmielew.
\newblock Elementary properties of {A}belian groups.
\newblock {\em Fund. Math.}, 41:203--271, 1955.

\bibitem[Tug03]{Tug}
A.~A. Tuganbaev.
\newblock Distributive rings, uniserial rings of fractions, and endo-{B}ezout
  modules.
\newblock volume 114, pages 1185--1203. 2003.
\newblock Algebra, 22.

\bibitem[Zie84]{Zie}
Martin Ziegler.
\newblock Model theory of modules.
\newblock {\em Ann. Pure Appl. Logic}, 26(2):149--213, 1984.

\end{thebibliography}
\end{document}